\documentclass[a4paper, 10pt]{amsart}
\usepackage[top=80pt,bottom=65pt,left=70pt,right=70pt]{geometry}
\usepackage{graphicx, eepic}
\usepackage{times}

\normalfont
\usepackage{bbm}
\usepackage{xcolor}
\usepackage{verbatim}
\usepackage{kotex}
\usepackage{mathrsfs}
\usepackage{amscd}
\usepackage{hhline}
\usepackage{float}
\usepackage[all]{xy}
\usepackage[subnum]{cases}
\ifpdf
\usepackage{tikz}
\fi
\usetikzlibrary{arrows,matrix,positioning}

\usepackage{makecell}
\usepackage{adjustbox}
\usepackage{longtable}
\usepackage{caption}

\usepackage{amsthm,amsmath,amsfonts,amssymb}
\usepackage{hyperref}
\hypersetup{colorlinks}
\hypersetup{
    bookmarks=true,         
    unicode=false,          
    pdftoolbar=true,        
    pdfmenubar=true,        
    pdffitwindow=false,     
    pdfstartview={FitH},    
    pdftitle={My title},    
    pdfauthor={Author},     
    pdfsubject={Subject},   
    pdfcreator={Creator},   
    pdfproducer={Producer}, 
    pdfkeywords={keyword1} {key2} {key3}, 
    pdfnewwindow=true,      
    colorlinks=true,       
    linkcolor=blue,          
    citecolor=red,        
    filecolor=violet,      
    urlcolor=violet       
}
\usepackage{multirow, url}
\usepackage{color}
\usepackage[all]{xy}
\theoremstyle{plain}
\newtheorem{theorem}{Theorem}[section]

\newtheorem{proposition}[theorem]{Proposition}
\newtheorem{lemma}[theorem]{Lemma}

\newtheorem{corollary}[theorem]{Corollary}

\theoremstyle{definition}

\newtheorem{conjecture}[theorem]{Conjecture}
\newtheorem{definition}[theorem]{Definition}

\newtheorem{example}[theorem]{Example}

\theoremstyle{remark}
\newtheorem{remark}[theorem]{Remark}

\newtheorem{notation}[theorem]{Notation}

\renewcommand{\bar}{\overline}

\newcommand{\C}{\mathbb{C}}

\newcommand{\Q}{\mathbb{Q}}
\newcommand{\R}{\mathbb{R}}

\newcommand{\Z}{\mathbb{Z}}
\newcommand{\p}{\mathbb{C}P}

\newcommand{\pp}{\mathbb{P}}

\newcommand{\mcal}{\mathcal}
\newcommand{\pd}{\mathrm{PD}}

\def\mcal{\mathcal}
\def\frak{\mathfrak}

\newcommand{\ds}{\displaystyle}
\newcommand{\vs}{\vspace}
\newcommand{\hs}{\hspace}

\numberwithin{equation}{section} \numberwithin{table}{section}

\begin{document}                                                                          

\title{Classification of six dimensional monotone symplectic manifolds admitting semifree circle actions}
\author{Yunhyung Cho}
\address{Department of Mathematics Education, Sungkyunkwan University, Seoul, Republic of Korea. }
\email{yunhyung@skku.edu}

\begin{abstract}
	Let $(M,\omega_M)$ be a six dimensional closed monotone symplectic manifold admitting an effective semifree Hamiltonian $S^1$-action.
	We show that $(M,\omega_M)$ is $S^1$-equivariant symplectomorphic 
	to some K\"{a}hler Fano manifold $(X,\omega_X, J)$ with a certain holomorphic $\C^*$-action.  
	We also give a complete list of all such Fano manifolds and describe all semifree $\C^*$-actions on them specifically. 
\end{abstract}
\maketitle
\setcounter{tocdepth}{1} 
\tableofcontents
\captionsetup[longtable]{skip=1em}

\section{Introduction}
\label{secIntroduction}

According to Koll\'{a}r-Miyaoka-Mori \cite{KMM}, there are only finitely many deformation types of smooth Fano varieties for each dimension. 
For example, there is only one $1$-dimensional smooth Fano variety $\C P^1.$ 
In dimension two, there are 10 types of smooth Fano surfaces, called \textit{del Pezzo surfaces}, classified as
$\C P^2$, $\C P^1 \times \C P^1$, and the blow-up of $\C P^2$ at $k$ generic points for $1 \leq k \leq 8.$ 
For the 3-dimensional case, Iskovskih \cite{I1} \cite{I2} classified all smooth Fano 3-folds having Picard number one. Later, Mori and Mukai \cite{MM} completed the classification of smooth Fano 3-folds.
(There are 105 types of smooth Fano 3-folds overall.) Note that any smooth Fano variety $X$ admits a K\"{a}hler form $\omega_X$ such that $[\omega_X] = c_1(TX)$. 

A \textit{monotone symplectic manifold $(M,\omega)$} is a symplectic analogue of a smooth Fano variety in the sense that it satisfies 
$\langle c_1(M), [\Sigma] \rangle > 0$ for any symplectic surface 
$\Sigma \subset M$. (See Section \ref{secMonotoneSemifreeHamiltonianS1Manifolds} for the precise definition.)
Then it is straightforward that the category of monotone symplectic manifolds contains all smooth Fano varieties. It is natural to ask whether a given monotone symplectic manifold is
K\"{a}hler (and hence Fano) with respect to some integrable almost complex structure compatible with the given symplectic form. It turned out that the answer for the question is negative in general, where a counter-example was found in dimension twelve by Fine and Panov \cite{FP}. 

In the low dimensional case, where $\dim M = 2$ or $4$, the answer is positive. Ohta and Ono \cite[Theorem 1.3]{OO2} proved that if $\dim M = 4$, then $M$ is diffeomorphic to a del Pezzo surface. Thus, from the uniqueness of a symplectic structure on a rational surface (due to McDuff \cite{McD3}), it follows that  every closed monotone symplectic four manifold is K\"{a}hler. 
As far as the author knows, the existence of a closed monotone symplectic manifold which is not K\"{a}hler is not known for dimension 6, 8, 10. 

The aim of this paper is to study six-dimensional monotone symplectic manifolds admitting Hamiltonian circle actions. More specifically, we deal with the following conjecture. 

\begin{conjecture}\cite[Conjecture 1.1]{LinP}\cite[Conjecture 1.4]{FP2}\label{conjecture_main}
	Let $(M,\omega)$ be a six dimensional closed monotone symplectic manifold equipped with an effective Hamiltonian circle action.  Then $(M,\omega)$ is $S^1$-equivariantly symplectomorphic
	to some K\"{a}hler manifold $(X,\omega_X, J)$ with some holomorphic Hamiltonian $S^1$-action.
\end{conjecture}

Note that Conjecture \ref{conjecture_main} is known to be true when $b_2(M) = 1$ by McDuff \cite{McD2} and Tolman \cite{Tol}. 
(There are four types of such manifolds up to symplectomorphism.)
Recently, Lindsay and Panov \cite{LinP} provided some evidences that Conjecture \ref{conjecture_main} is possibly true.  For instance, they proved that $M$ given in Conjecture
\ref{conjecture_main} is simply-connected as other smooth Fano varieties are. 
In this article, we prove Conjecture \ref{conjecture_main} provided that the action is {\em seimfree}\footnote{An $S^1$-action is called {\em semifree} 
		if the action is free outside the fixed point set.}.

\begin{theorem}\label{theorem_main}
		Let $(M,\omega)$ be a six-dimensional closed monotone symplectic manifold equipped with a semifree Hamiltonian 
		circle action. 
		Then $(M,\omega)$ is $S^1$-equivariantly symplectomorphic to some 
		K\"{a}hler Fano manifold with some holomorphic Hamiltonian circle action. 
\end{theorem}

The proof of Theorem \ref{theorem_main} is essentially based on Gonzalez's approach \cite{G}. He introduced a notion so-called a {\em fixed point data} for a semifree Hamiltonian circle action,  
which is a collection of {\em a symplectic reduction\footnote{A reduced space at a critical level is not a smooth manifold nor an orbifold in general. However, if $\dim M = 6$ and the action is semifree, then a symplectic reduction at any (critical) level is a smooth manifold with the induced symplectic form. See 
Proposition \ref{proposition_topology_reduced_space}.} at each critical level} together with an information of 
critical submanifolds (or equivalently fixed components) as embedded symplectic submanifolds in the reduced space. (See Definition \ref{definition_fixed_point_data} or \cite[Definition 1.2]{G}.)
He then proved that a fixed point data determines a 
semifree Hamiltonian $S^1$-manifold up to $S^1$-equivariant symplectomorphism under the assumption that every reduced space is {\em symplectically rigid}\footnote{See Section \ref{secFixedPointData} for the definition.}. 

\begin{theorem}\cite[Theorem 1.5]{G}\label{theorem_Gonzalez}
		Let $(M,\omega)$ be a six-dimensional closed semifree Hamiltonian $S^1$-manifold. 
		Suppose that every reduced space is symplectically rigid.
		Then $(M,\omega)$ is determined by its fixed point data up to $S^1$-equivariant symplectomorphism.
\end{theorem}

The proof of Theorem \ref{theorem_main} goes as follows : if $(M,\omega)$ is a closed six-dimensional monotone semifree Hamiltonian $S^1$-manifold with an isolated fixed point, 
then we show that
\begin{itemize}
	\item ({\bf Step 1} :) every reduced space of $(M,\omega)$ is symplectically rigid, and 
	\item ({\bf Step 2} :) the fixed point data of $(M,\omega)$ coincides with that of some smooth Fano variety equipped with some holomorphic semifree Hamiltonian $S^1$-action. 
\end{itemize}
The main difficulty in the second step is that it is almost impossible to determine whether two given fixed point data coincide or not in general. 
To overcome the difficulty, we first classify all possible {\em topological fixed point data}\footnote{See Definition \ref{definition_topological_fixed_point_data}.}(or TFD shortly) of $(M,\omega)$. 
A {\em topological fixed point data} of $(M,\omega)$ is a topological version of a fixed point data 
in the sense that it records ``homology classes'', 
not embeddings themselves, of fixed components in reduced spaces. With the aid of the Duistermaat-Heckman theorem (Theorem \ref{theorem_DH}), 
the ABBV-localization theorem (Theorem \ref{theorem_localization}), and some theorems about symplectic four manifolds (cf. \cite{LL}, \cite{Li}), we classify all possible 
TFD, see in Table \ref{table_list}, \ref{table_list_2}, and \ref{table_list_3}.

An immediate consequence of our classification of TFD is that every reduced space of $(M,\omega)$, except for very few cases (see Table \ref{table_list_3}: {\bf (II-1-4.k)}), 
is either $\p^2$, $\p^1 \times \p^1$, or $\p^2 \# ~k\overline{\p^2}$
for $k \leq 4$, where those manifolds are known to be symplectically rigid. (See Theorem \ref{theorem_uniqueness} and Theorem \ref{theorem_symplectomorphism_group}.)
Moreover, each topological fixed point data determines the first Chern number $\langle c_1(TM)^3, [M] \rangle$ as well as the Betti numbers of $M$. This enables us to expect a candidate for $(M,\omega)$
in the list of smooth Fano 3-folds given by Mori-Mukai \cite{MM}. Indeed, we could succeed in finding holomorphic Hamiltonian $S^1$-actions on those Fano candidates whose TFD match up with 
ours listed in Table \ref{table_list}, \ref{table_list_2}, and \ref{table_list_3}. 
(See the examples given in Section \ref{secClassificationOfTopologicalFixedPointDataDimZMin}, \ref{secClassificationOfTopologicalFixedPointDataDimZMinDimZMax2}, 
\ref{secClassificationOfTopologicalFixedPointDataDimZMinGeq2AndDimZMax4}.) 

Next, we will show that each TFD in our tables determines a fixed point data uniquely. The following two facts, due to Siebert-Tian \cite{ST} and Zhang \cite{Z}, 
are essentially used in this process. 
\begin{itemize}
	\item Any possible fixed point data whose topological type is given in Table \ref{table_list}, \ref{table_list_2}, and \ref{table_list_3} 
	is algebraic, i.e., any fixed component as an embedded symplectic submanifold in a reduced space is 
	symplectically isotopic to an algebraic curve. (See Theorem \ref{theorem_ST} and Theorem \ref{theorem_Z}.)
	\item Any two algebraic curves in a reduced space representing a same fixed point data are symplectically isotopic to each other. (See Lemma \ref{lemma_isotopic}.)
\end{itemize}

This paper is organized as follows. In Section \ref{secHamiltonianCircleActions}, 
we give a brief introduction to Hamiltonian $S^1$-actions including the Duistermaat-Heckman theorem.
An equivariant cohomology theory for Hamiltonian $S^1$-actions, especially about the Atiyah-Bott-Berline-Vergne localization theorem and equivariant Chern classes, 
is explained in Section \ref{secEquivariantCohomology}. In Section \ref{secMonotoneSemifreeHamiltonianS1Manifolds},
we restrict our attention to a closed monotone semifree Hamiltonian $S^1$-manifold and explain
how the topology of a reduced space and a reduced symplectic form change when crossing critical values of a moment map. We also explain how a reduced space 
inherits a monotone reduced symplectic form from $\omega$. In Section \ref{secFixedPointData}, we give a definition of (topological) fixed point data and introduce the Gonzalez's Theorem 
\cite[Theorem 1.5]{G}. In Section \ref{secClassificationOfTopologicalFixedPointDataDimZMin}, \ref{secClassificationOfTopologicalFixedPointDataDimZMinDimZMax2}, and 
\ref{secClassificationOfTopologicalFixedPointDataDimZMinGeq2AndDimZMax4}, we classify all topological 
fixed point data and describe the corresponding Fano candidates with specific holomorphic circle actions. In Section \ref{secMainTheorem}, we prove Theorem \ref{theorem_main}.

We have two appendices. Section \ref{secMonotoneSymplecticFourManifoldsWithSemifreeS1Actions}
is about a classification of closed monotone semifree Hamiltonian four manifolds. We apply our arguments used in this paper
to four dimensional cases and obtain a complete list of such manifolds. See Table \ref{table_list_4dim}. Finally in Section \ref{secSymplecticCapacitiesOfSmoothFano3Folds},
as a by-product of our classification, we calculate the Gromov width and the Hofer-Zehnder capacity for each manifold in Table \ref{table_list} by applying theorem of Hwang-Suh \cite{HS}.

\subsection*{Acknowledgements} 
The author would like to thank Dmitri Panov for bringing the paper \cite{Z} to my attention.
The author would also like to thank Jinhyung Park for helpful comments. 
This work is supported by the National Research Foundation of Korea(NRF) grant funded by the Korea government(MSIP; Ministry of Science, ICT \& Future Planning) (NRF-2017R1C1B5018168).

\section{Hamiltonian circle actions}
\label{secHamiltonianCircleActions}
    
    In this section, we briefly review some facts about Hamiltonian circle actions.
    Throughout this section, we assume that $(M,\omega)$ is a $2n$-dimensional closed symplectic manifold and $S^1$ is the unit circle group in $\C$ 
    acting on $M$ smoothly with the fixed point set $M^{S^1}$. 

\subsection{Hamiltonian actions}
\label{ssecHamiltonianActions} 

	Let $\frak{t}$ be the Lie algebra of $S^1$ and $X \in \frak{t}$. The vector field $\underbar{X}$ on $M$ defined by
	\[
		\underbar{X}_p = \left.\frac{d}{dt}\right|_{t=0} \exp(tX) \cdot p 
	\] is called the \textit{fundamental vector field with respect to} $X$.
	We say that the $S^1$-action on $(M,\omega)$ is \textit{symplectic} if it preserves the symplectic form $\omega$, i.e.,
	\[
		\mcal{L}_{\underbar{X}} \omega = 0
	\] for any $X \in \frak{t}$. By Cartan's magic formula, we have
	\[
		\mcal{L}_{\underbar{X}} \omega = d \circ i_{\underbar{X}} \omega + i_{\underbar{X}} \circ d \omega = d \circ i_{\underbar{X}} \omega.
	\]		
	Thus the action is symplectic if and only if $i_{\underbar{X}} \omega$ is a closed 1-form on $M$. If $i_{\underbar{X}} \omega$ is exact, then we say that the action is \textit{Hamiltonian}.
	In particular, any symplectic circle action is locally Hamiltonian by the classical Poincar\'{e} lemma.
	
	When the $S^1$-action is Hamiltonian, there exists a smooth function $H : M \rightarrow \R$, called a {\em moment map}, such that 
	\[
		i_{\underbar{X}} \omega = dH.
	\]
	It immediately follows that $p \in M$ is a fixed point of the action if and only if $p$ is a critical point of $H$.  	
	The following theorem describes a local behavior of the action near each fixed point of the action. 
	
	\begin{theorem}\label{theorem_equivariant_darboux}(Equivariant Darboux theorem)
		Let $(M,\omega)$ be a $2n$-dimensional symplectic manifold equipped with a Hamiltonian circle action and $H : M \rightarrow \R$ be a moment map. 
		For each fixed point $p \in M^{S^1}$, there is an $S^1$-invariant complex coordinate chart $(\mcal{U}_p, z_1, \cdots, z_n)$ with weights 
		$(\lambda_1, \cdots, \lambda_n) \in \Z^n$ such that
		\begin{enumerate}
			\item $\omega|_{\mcal{U}_p} = \frac{1}{2i} \sum_i dz_i \wedge d\bar{z_i},$ and
                                \item for any $t \in S^1$, the action can be expressed by $$t \cdot (z_1, \cdots, z_n) = (t^{\lambda_1}z_1, \cdots, t^{\lambda_n}z_n)$$ 
                                and the moment map can be written by
                                \[
                                	H(z_1, \cdots, z_n) = H(p) + \frac{1}{2}\sum_i \lambda_i |z_i|^2.
			\]
		\end{enumerate}
	\end{theorem}
	
	Using Theorem \ref{theorem_equivariant_darboux}, we obtain the following.
	
	\begin{corollary}\label{corollary_properties_moment_map}\cite[Chapter 4]{Au}
		Let $(M,\omega)$ be a closed symplectic manifold equipped with a Hamiltonian circle action with a moment map $H : M \rightarrow \R$. Then $H$
                     satisfies the followings.
		\begin{enumerate}
			\item $H$ is a Morse-Bott function.
			\item Let $p \in M^{S^1}$ be a fixed point of the action and let $(\mcal{U}_p, z_1, \cdots, z_n)$ be an equivariant Darboux chart near $p$ with weights 
			$(\lambda_1, \cdots, \lambda_n) \in \Z^n$. Then we have $$ \mathrm{ind}(p) = 2n_p $$ where $\mathrm{ind}(p)$ is the Morse-Bott index of $p$ and $n_p$ is the number of
			 negative $\lambda_i$'s. Moreover, twice the number of zeros in $\{\lambda_1, \cdots, \lambda_n \}$ is a real dimension of the fixed component
			  containing $p$.
			\item Any fixed component is a symplectic submanifold of $(M,\omega)$.
			\item Each level set of $H$ is connected. In particular, an extremal fixed component is connected.
		\end{enumerate}
	\end{corollary}

\subsection{Symplectic reduction} 
\label{ssecSymplecticReduction}

	Let $r \in \R$ be a regular value of $H$. Then the level set $H^{-1}(r)$ does not have any fixed point so that $H^{-1}(r)$ is a fixed point-free $S^1$-manifold of dimension $2n-1$. 
	The quotient space $M_r = H^{-1}(r) / S^1$ is an orbifold of dimension $2n-2$ with cyclic quotient singularities. Since the restriction of $\omega$ on $H^{-1}(r)$ satisfies
	\begin{itemize}
		\item $i_{\underbar{X}} \omega = -dH = 0$ on $H^{-1}(r)$, and
		\item $\mcal{L}_{\underbar{X}} \omega = i_{\underbar{X}} \circ d\omega + d \circ i_{\underbar{X}} \omega = 0$.
	\end{itemize} 
	Thus we can push-forward $\omega$ to $M_r$ via the quotient map $$\pi_r : H^{-1}(r) \rightarrow M_r$$ and so that we obtain a symplectic form $\omega_r$ on $M_r$.
	We call $(M_r, \omega_r)$ the \textit{symplectic reduction at $r$}.

\subsection{Duistermaat-Heckman theorem} 
\label{ssecDuistermaatHeckmanTheorem}

	Let $J$ be an $S^1$-invariant $\omega$-compatible almost complex structure on $M$ so that $g_J(\cdot, \cdot) := \omega(J \cdot, \cdot)$ is an $S^1$-invariant
	 Riemannian metric on $M$. Note that 
	the following equality
	\[
		 g( JX, Y) = \omega(-X, Y) = -\omega(X,Y) = dH(Y)
	\]
	implies that $JX$ is the gradient vector field of $H$ with respect to $g_J$.
	
	Let $(a,b) \subset \R$ be an open interval which does not contain any critical value of $H$.
	For any $r, s \in (a,b)$ with $r < s$, we may identify $H^{-1}(r)$ with $H^{-1}(s)$ via the diffeomorphism $\phi_{r, s} : H^{-1}(r) \rightarrow H^{-1}(s)$ which sends a point 
	$z \in H^{-1}(r)$ to a point in $H^{-1}(s)$ along the flow of the gradient vector field $JX$. Thus one gets a diffeomorphism
	\begin{equation}
		\begin{array}{ccc}
			\phi : H^{-1}([r,s]) &\stackrel{\cong}\rightarrow& H^{-1}(s) \times [r,s]\\[0.5em]
			p &\mapsto& (\phi_{H(p), s}(p), H(p)).
		\end{array}
	\end{equation}
	By pulling back $\omega$ to $H^{-1}(s) \times [r,s]$ via $\phi^{-1}$, we have an $S^1$-equivariant symplectomorphism
	\[
		 (H^{-1}([r,s]), \omega|_{(H^{-1}([r,s])}) \cong (H^{-1}(r) \times [r,s], (\phi^{-1})^*\omega) 
	\] with a moment map
	$(\phi^{-1})^*H : H^{-1}(r) \times [r,s] \rightarrow [r,s]$ which is simply a projection on the second factor. Therefore, we may identify $M_s$ with $M_r$ via $\phi$. 
	This identification allows us to think of reduced symplectic forms $\{\omega_t ~|~ t \in (a,b) \}$ as a one-parameter family of symplectic forms on $M_r$.

	\begin{theorem}\label{theorem_DH}\cite{DH}
		Let $\omega_s$ and $\omega_r$ be the reduced symplectic forms on $M_s$ and $M_r$, respectively. By identifying $M_s$ with $M_r$ as described above, we have
		$$ [\omega_s] - [\omega_r] = (r-s)e $$
		where $e \in H^2(M_r;\Q)$ is the Euler class of the $S^1$-fibration $\pi_r : H^{-1}(r) \rightarrow M_r.$
	\end{theorem}

	Note that if the action is semifree, then the reduced space 
	at a regular value becomes a smooth manifold and the fibration $\pi_r$ in Theorem \ref{theorem_DH} becomes a genuine $S^1$-bundle so that the Euler class $e \in H^2(M_r; \Z)$ is integral. 

\section{Equivariant cohomology}
\label{secEquivariantCohomology}

    In this section, we recall some well-known facts about equivariant cohomology of Hamiltonian $S^1$-manifolds. 
	Throughout this section, we take cohomology with coefficients in $\R$, unless stated otherwise.

	Let $(M,\omega)$ be a $2n$-dimensional closed symplectic manifold equipped with a Hamiltonian circle action. The equivariant cohomology $H^*_{S^1}(M)$ is defined by
	\[
		H^*_{S^1}(M) = H^*(M \times_{S^1} ES^1) 
	\] where $ES^1$ is a contractible space on which $S^1$ acts freely. In particular, the equivariant cohomology of
     a point $p$ is given by 
     \[
     		H^*_{S^1}(p) = H^*(p \times_{S^1} ES^1) = H^*(BS^1) 
	\]
	where $BS^1 = ES^1 / S^1$ is the classifying space of $S^1$. Note that $BS^1$ can be constructed as an inductive limit of the sequence of Hopf fibrations
	\begin{equation}
		\begin{array}{ccccccccc}
			S^3          & \hookrightarrow & S^5        & \hookrightarrow & \cdots & S^{2n+1} & \cdots & \hookrightarrow & ES^1 \sim S^{\infty} \\
		\downarrow   &                 & \downarrow &                 & \cdots & \downarrow & \cdots &                 & \downarrow \\
		   \C P^1       & \hookrightarrow & \C P^2     &\hookrightarrow  & \cdots & \C P^n      & \cdots & \hookrightarrow &BS^1 \sim \C P^{\infty}
		\end{array}
	\end{equation}
	so that  $H^*(BS^1) \cong \R[x]$ where $x$ is an element of degree two such that $\langle x, [\C P^1] \rangle = 1$.

\subsection{Equivariant formality}
\label{ssecEquivariantFormality}
	
	One remarkable fact on the equivariant cohomology of a Hamiltonian $S^1$-manifold is that the space is {\em equivariantly formal}.
	Before we state the equivariant formality of $(M,\omega)$, recall that $H^*_{S^1}(M)$ has a natural $H^*(BS^1)$-module structure as follows.
	The projection map $M \times ES^1 \rightarrow ES^1$ on the second factor is $S^1$-equivariant and it induces the projection map
	\[	
		\pi : M \times_{S^1} ES^1 \rightarrow BS^1 
	\] which makes $M \times_{S^1} ES^1$ into an $M$-bundle over $BS^1$ :
	\begin{equation}\label{equation_Mbundle}
		\begin{array}{ccc}
			M \times_{S^1} ES^1 & \stackrel{f} \hookleftarrow & M \\[0.3em]
			\pi \downarrow          &                             &   \\[0.3em]
			BS^1                   &                             &
		\end{array}
	\end{equation}
	where $f$ is an inclusion of $M$ as a fiber. Then $H^*(BS^1)$-module structure on $H^*_{S^1}(M)$ is given by the map $\pi^*$ such that 
	\[
		y \cdot \alpha = \pi^*(y)\cup \alpha 
	\] for $y \in H^*(BS^1)$ and $\alpha \in H^*_{S^1}(M)$. In particular, we have the following sequence of ring homomorphisms
	\[
		 H^*(BS^1) \stackrel{\pi^*} \rightarrow H^*_{S^1}(M) \stackrel{f^*} \rightarrow H^*(M). 
	\]

	\begin{theorem}\label{theorem_equivariant_formality}\cite{Ki}
		Let $(M,\omega)$ be a closed symplectic manifold equipped with a Hamiltonian circle action. Then $M$ is equivariatly formal, that is,
		$H^*_{S^1}(M)$ is a free $H^*(BS^1)$-module so that $$H^*_{S^1}(M) \cong H^*(M) \otimes H^*(BS^1).$$
		Equivalently, the map $f^*$ is surjective with kernel $x \cdot H^*_{S^1}(M)$ where $\cdot$ denotes the scalar multiplication of $H^*(BS^1)$-module structure on $H^*_{S^1}(M)$.
	\end{theorem}
	
\subsection{Localization theorem}
\label{ssecLocalizationTheorem} 

	Let $\alpha \in H^*_{S^1}(M)$ be any element of degree $k$. Theorem \ref{theorem_equivariant_formality} implies that $\alpha$ can be uniquely
	expressed as 
	\begin{equation}\label{equation_expression}
		 \alpha = \alpha_k \otimes 1 + \alpha_{k-2} \otimes x + \alpha_{k-4} \otimes x^2 + \cdots 
	\end{equation}
	where $\alpha_i \in H^i(M)$ for each $i = k, k-2, \cdots$. We then obtain $f^*(\alpha) = \alpha_k$ where $f$ is given in \eqref{equation_Mbundle}.

	\begin{definition}
		An \textit{integration along the fiber $M$} is an $H^*(BS^1)$-module homomorphism $\int_M : H^*_{S^1}(M) \rightarrow H^*(BS^1)$ defined by
		\[
			\int_M \alpha = \langle \alpha_k, [M] \rangle \cdot 1 + \langle \alpha_{k-2}, [M] \rangle \cdot x + \cdots 
		\]
		for every $ \alpha = \alpha_k \otimes 1 + \alpha_{k-2} \otimes x + \alpha_{k-4} \otimes x^2 + \cdots \in H^k_{S^1}(M)$ 
		where $[M] \in H_{2n}(M; \Z)$ denotes the fundamental homology class of $M$.
	\end{definition}

	Note that $\langle \alpha_i, [M] \rangle$ is zero for every $i < \dim M = 2n$ and $\alpha_i = 0$ for every $i > \deg \alpha$ by dimensional reason. 
	Therefore we have the following corollary.

	\begin{corollary}\label{corollary : localization degree 2n}
		Let $\alpha \in H^{2n}_{S^1}(M)$ be given in \eqref{equation_expression}.
		Then 
		\[
			\int_M \alpha = \langle \alpha_{2n}, [M] \rangle = \langle f^*(\alpha), [M] \rangle.
		\]
		Also if $\mathrm{deg} ~\alpha < 2n$, then we have 
		\[
			\int_M \alpha = 0.
		\]	
	\end{corollary}

	Now, let $M^{S^1}$ be the fixed point set of the $S^1$-action on $M$ and $F \subset M^{S^1}$ a fixed component. Then the inclusion $i_F : F \hookrightarrow M$ 
	induces a ring homomorphism $$i_F^* : H^*_{S^1}(M) \rightarrow H^*_{S^1}(F) \cong H^*(F) \otimes H^*(BS^1).$$
	For any $\alpha \in H^*_{S^1}(M)$, we call the image $i_F^*(\alpha)$ \textit{the restriction of $\alpha$ to $F$} and denote by 
	\[
		\alpha|_F := i_F^*(\alpha).
	\] 
	The following theorem due to Atiyah-Bott \cite{AB} and Berline-Vergne \cite{BV} states that 
	$\int_M \alpha$ can be computed in terms of the fixed point set.

	\begin{theorem}[ABBV localization]\label{theorem_localization}
		For any $ \alpha \in H^*_{S^1}(M)$, we have
		\[
			\int_M \alpha = \sum_{F \subset M^{S^1}} \int_F \frac{\alpha|_F}{e^{S^1}(F)}
		\]
		where $e^{S^1}(F)$ is the equivariant Euler class of the normal bundle $\nu_F$ of $F$ in $M$. That is, $e^{S^1}(F)$ is the Euler class of the bundle 
		\[
			\nu_F \times_{S^1} ES^1 \rightarrow F \times BS^1.
		\] induced from the projection $\nu_F \times ES^1 \rightarrow F \times ES^1$.
	\end{theorem}

	One more important property of a Hamiltonian $S^1$-action is that any two equivariant cohomology classes are distinguished by their images of the restriction to the fixed point set.

          \begin{theorem}\cite{Ki}(Kirwan's  injectivity theorem)\label{theorem_Kirwan_injectivity}
		Let $(M,\omega)$ be a closed Hamiltonian $S^1$-manifold and $i : M^{S^1} \hookrightarrow M$ be the inclusion of the fixed point set.
		Then the induced ring homomorphism $i^* : H^*_{S^1}(M) \rightarrow H^*_{S^1}(M^{S^1})$ is injective.
	\end{theorem}

\subsection{Equivariant Chern classes}
\label{ssecEquivariantChernClasses} 
                                           
	\begin{definition}
		Let $\pi : E \rightarrow B$ be an $S^1$-equivariant complex vector bundle over a topological space $B$. 
		Then the \textit{$i$-th equivariant Chern class $c_i^{S^1}(E)$} is defined as an $i$-th Chern class of the complex vector bundle $\widetilde{\pi}$
		\[\xymatrix{E \times ES^1 \ar[r]^{/S^1} \ar[d]^{\pi} & E \times_{S^1} ES^1 \ar[d]^{\widetilde{\pi}} \\ B \times ES^1 \ar[r]^{/S^1} & B \times_{S^1} ES^1}\]
		For an almost complex $S^1$-manifold $(X,J)$, we denote by $c_i^{S^1}(X) := c_i^{S^1}(TX,J)$ the $i$-th equivariant Chern class of 
		the complex tangent bundle $(TX,J)$.
	\end{definition}

	For a closed symplectic manifold $(M,\omega)$ equipped with a Hamiltonian circle action, there exists an $S^1$-invariant 
	$\omega$-compatible almost complex structure $J$ on $M$. 
	Moreover, the space of such almost complex structures is contractible so that the Chern classes of $(M,J)$ do not depend on the choice of $J$.
	The following proposition gives an explicit formula for the restriction of the first equivariant Chern class of $(M, J)$ on each fixed component.

	\begin{proposition}\label{proposition_equivariant_Chern_class}
		Let $(M,\omega)$ be a $2n$-dimensional closed symplectic manifold equipped with a Hamiltonian circle action.
		Let $F \subset M^{S^1}$ be any fixed component with weights $\{ w_1(F), \cdots, w_n(F) \} \in \Z^n$ of the tangential $S^1$-representation at $F$. 
		Then the restriction $c_1^{S^1}(M)|_F \in H^*_{S^1}(F) \cong H^*(F) \otimes H^*(BS^1)$
		is given by 
		\[
			c_1^{S^1}(M)|_F = c_1(M)|_F \otimes 1 + 1 \otimes (\sum_i w_i)x.
		\]
	\end{proposition}
                                           
	\begin{proof}
		Let $TF$ be the tangent bundle of $F$ and $\nu_F$ be the normal bundle of $F$, respectively, and consider the following bundle map : 
		\[\xymatrix{TM|_F \times ES^1 \ar[r]^{/S^1} \ar[d]^{\pi} & TM|_F \times_{S^1} ES^1 \ar[d]^{\widetilde{\pi}} \\ F \times ES^1 \ar[r]^{/S^1} & F \times_{S^1} ES^1}\]
		By definition, we have $c_1^{S^1}(M)|_F = c_1(\widetilde{\pi})$. Since $TM|_F = TF \oplus \nu_F$, 
		$\widetilde{\pi}$ is decomposed into the Whitney sum $\widetilde{\pi}_1 \oplus \widetilde{\pi}_2$ where $\widetilde{\pi}_1$ and $\widetilde{\pi}_2$ are given by
		\[\xymatrix{TF \times_{S^1} ES^1 \ar[d]^{\widetilde{\pi}_1}  & \nu_F \times_{S^1} ES^1 \ar[d]^{\widetilde{\pi}_2} \\ F \times BS^1  & F \times BS^1}\]
		Since $F$ is fixed by the $S^1$-action, the induced $S^1$-action on $TF$ is trivial, and therefore we get 
		\[
			c_1(\widetilde{\pi}_1) = c_1(TF) \otimes 1.
		\]
		Note that the restriction of $\widetilde{\pi}_2$ on $F \times \{\mathrm{pt}\} \subset F \times BS^1$ 
		is $\nu_F$. Also, the restriction of $\widetilde{\pi}_2$ on $\{ \mathrm{pt} \} \times BS^1$ is $\C^n \times_{S^1} ES^1$. More precisely, we have 
		\[\xymatrix{\C^n \times ES^1 \ar[d] \ar[r]^{/S^1} & \C^n \times_{S^1} ES^1 \ar[d]^{\widetilde{\pi}_2|_{\{\mathrm{pt}\} \times BS^1}} \\  ES^1 \ar[r]^{/S^1} & BS^1}\]
		and the total Chern class of the restricted bundle $\widetilde{\pi}_2 : \C^n \times_{S^1} ES^1 \rightarrow BS^1$ is $(1+w_1x)(1+w_2x)\cdots(1+w_nx)$. 
		Thus it follows that 
		\begin{displaymath}
			\begin{array}{ll}
				c_1(\widetilde{\pi}) = c_1(\widetilde{\pi}_1) + c_1(\widetilde{\pi}_2) & = c_1(TF) \otimes 1 + c_1(\nu(F)) \otimes 1 + 1 \otimes (\sum_i w_i)x\\[1em]
				& =c_1(M)|_F \otimes 1 + (\sum_i w_i)x
			\end{array}
		\end{displaymath}
		and the result follows.
	\end{proof}

\subsection{Equivariant symplectic classes} 
\label{ssecEquivariantSymplecticClasses}

	Let $H : M \rightarrow \R$ be a moment map for the $S^1$-action on $(M,\omega)$.  
	Let $\omega_H := \omega + d(H \cdot \theta)$ be a two-form on $M \times ES^1$ where $\theta$ is a connection form of the principal $S^1$-bundle 
	$ES^1 \rightarrow BS^1$. 
	By the $S^1$-invariance of $\omega$, $\theta$, and $H$, we obtain 
	\[
		\mcal{L}_{\underbar{X}} \omega_H = i_{\underbar{X}} \omega_H = 0
	\] 
	where $\underbar{X}$ denotes the vector field generated by the diagonal $S^1$-action on $M \times ES^1$. 
	Thus we may push-forward $\omega_H$ to $M \times_{S^1} ES^1$ 
	and denote the push-forward of $\omega_H$ by $\widetilde{\omega}_H$, which we call the \textit{equivariant symplectic form with respect to $H$}.
	Moreover, the corresponding cohomology class $[\widetilde{\omega}_H] \in H^2_{S^1}(M)$ is called the \textit{equivariant symplectic class with respect to $H$}.
	Note that the restriction of $\widetilde{\omega}_H$ on each fiber $M$ is precisely $\omega$.

	\begin{proposition}\label{proposition_equivariant_symplectic_class}
		Let $F \in M^{S^1}$ be a fixed component. Then we have
		\[
			[\widetilde{\omega}_H]|_F = [\omega]|_F \otimes 1 - H(F) \otimes x.
		\]
	\end{proposition}
	
	\begin{proof}
		Consider the push-forward of 
		$\widetilde{\omega}_H|_F = (\omega + dH \wedge \theta + H \cdot d\theta)|_{F \times ES^1}$ to $F \times BS^1$. 
		Since the restriction $dH|_{F \times ES^1}$ vanishes, we have $[\widetilde{\omega}_H]|_F = [\omega]|_F \otimes 1+ H(F) \otimes \widetilde{[d\theta]}|_{BS^1}$ where 
		$\widetilde{[d\theta]}$ is the push-forward of $[d\theta]$ to $F \times BS^1$. Since the push-forward of $d\theta$ is a curvature form which represents the first Chern class of 
		$ES^1 \rightarrow BS^1$, we have $\widetilde{[d\theta]} = -x$. Therefore, we have $[\widetilde{\omega}_H]|_F = [\omega]|_F \otimes 1 - H(F) \otimes x$.
	\end{proof}

\section{Monotone semifree Hamiltonian $S^1$-manifolds}
\label{secMonotoneSemifreeHamiltonianS1Manifolds}
	
	A \textit{monotone symplectic manifold} is a symplectic manifold $(M,\omega)$ such that $c_1(M) = \lambda \cdot [\omega]$ for some positive real number 
	$\lambda \in \R^+$ called the {\em monotonicity constant}. 
	In this section, we serve crucial ingredients for the classification of closed monotone semifree Hamiltonian $S^1$-manifolds.			
	
	Throughout this section, we assume that $(M,\omega)$ is a closed monotone symplectic manifold equipped with a 
	semifree Hamiltonian circle action
	such that $c_1(M) = [\omega]$ with a moment map $H : M \rightarrow \R$. 

\subsection{Topology of reduced spaces} 
\label{ssecTopologyOfReducedSpaces}

	Let $r \in \R$ be a regular value of $H$. The ``semifreeness'' implies that the level set $H^{-1}(r)$ is a free $S^1$-manifold of dimension $2n-1$. 
	Then the symplectic reduction $(M_r,\omega_r)$ becomes a closed symplectic manifold of dimension $2n-2$.
 
	Let $I \subset \R$ be an open interval consisting of regular values of $H$. Recall from Section \ref{ssecDuistermaatHeckmanTheorem} that 
	any two reduced spaces $M_r$ and $M_s$ ($r, s \in I$) 
	can be identified via the map $\phi_{r,s} / S^1 : M_r \rightarrow M_s$ where
	\[
		\phi_{r,s}  : H^{-1}(r) \rightarrow H^{-1}(s)
	\] is an $S^1$-equivariant diffeomorphism induced from the gradient vector field of $H$.
	In particular, $M_r$ and $M_s$ are diffeomorphic. 
	
	The topology of a reduced space $M_t$ may change when $t$ crosses a critical value of $H$. 
	More precisely, let $c \in \R$ be a critical value of $H$ and $p \in H^{-1}(c)$ a fixed point of index is $2k$. 
	By the equivariant Darboux theorem \ref{theorem_equivariant_darboux}, there is an $S^1$-equivariant complex coordinate chart $(\mcal{U}_p, z_1, \cdots, z_n)$ near $p$ 
	such that
	\begin{itemize}
        		\item $\omega|_{\mcal{U}_p} = \frac{1}{2i} \sum_i dz_i \wedge d\bar{z_i},$ and
	        	\item the action can be expressed by 
        		\[
        			t \cdot (z_1, \cdots, z_n) = (t^{-1}z_1, \cdots, t^{-1}z_k, z_{k+1}, \cdots, z_{k+l}, tz_{k+l+1}, \cdots, tz_n), \quad t \in S^1, 
		\] 
		\item the moment map is given by 
		\[
			H(z_1, \cdots, z_n) = \underbrace{H(p)}_{= c} - \frac{1}{2}(\sum_{i=1}^k |z_i|^2 - \sum_{i=k+l+1}^n |z_i|^2)
		\]
	    \end{itemize}
	where $2l$ is a dimension of the fixed component $Z_p$ containing $p$. 
	Thus $\mcal{U}_p \cap H^{-1}(c)$ locally looks like the solution space of the equation 
	\[
		\sum_{i=1}^k |z_i|^2 - \sum_{i=k+l+1}^n |z_i|^2 = 0.
	\]
	For a sufficiently small parameter $\epsilon > 0$, observe that
	\[
		\mcal{U}_p \cap H^{-1}(c-\epsilon) \cong \{(z_1,\cdots,z_n) \in \mcal{U}_p ~| ~\sum_{i=1}^k |z_i|^2 = \sum_{i=k+l+1}^n |z_i|^2 + 2\epsilon \}.
	\]
	In particular, $\mcal{U}_p \cap H^{-1}(c-\epsilon) \cong S^{2k-1} \times \C^l \times \C^{n-k-l} $ and it contains 
	\begin{equation}\label{equation_c_epsilon}
		\{ (z_1,\cdots,z_n) \in \mcal{U}_p \cap H^{-1}(c-\epsilon) ~| ~ z_{k+l+1} = \cdots = z_n = 0 \} = \mcal{S}_{2\epsilon} \times \C^l \times 0 \in \C^k \times \C^l \times \C^{n - k - l}
	\end{equation}
	where $\mcal{S}_{2\epsilon} = \{ (z_1, \cdots, z_n) \in \mcal{U}_p ~| ~z_{k+1} = \cdots = z_n = 0 \}$ is the $(2k-1)$-dimensional sphere of radius $\sqrt{2\epsilon}$ 
	centered at the origin in $\mcal{U}_p$.
	
	With respect to an $S^1$-invariant metric $g_J$,
	the set in \eqref{equation_c_epsilon} is the unstable submanifold of $Z_p$. (See Section \ref{ssecDuistermaatHeckmanTheorem}.)
	The gradient flow of $H$ induces a surjective continuous map 
	\[
		\phi_{c-\epsilon, c} : H^{-1}(c-\epsilon) \cap \mcal{U}_p \rightarrow H^{-1}(c) \cap \mcal{U}_p
	\] and this map sends $\mcal{S}_\epsilon \times \C^l \times 0$ to $0 \times \C^l \times 0$. We similarly apply this argument to other fixed components in $H^{-1}(c)$ so that 
	we obtain a surjective map 
	\[
		\phi_{c-\epsilon, c} /S^1 : M_{c-\epsilon} \rightarrow M_c
	\] When $k=1$, i.e., $Z_p$ is of index two, 
	then $\mcal{S}_\epsilon \times_{S^1} \C^l \cong \C^l$ so that $\phi_{c-\epsilon, c} / S^1$ is bijective on $\mcal{U}_p \cap H^{-1}(c-\epsilon)$. 
	Similarly, if $Z_p$ is of co-index two (i.e., $k=n-l-1$), 
	we have 
	\[
		\mcal{U}_p \cap H^{-1}(c-\epsilon) \cong S^{2n-2l-3} \times \C^l \times \C
	\] so that $\phi_{c-\epsilon, c} / S^1$ maps $\left(S^{2n-2l-3} \times_{S^1} \C \right) \times \C^l$ to 
	$\C^{n-l-1} \times \C^l$, that is, $\phi_{c-\epsilon, c} / S^1$ is the blow-down map along $\C^l$.  
	Therefore we have the following.

	\begin{proposition}\label{proposition_topology_reduced_space}\cite{McD2}\cite{GS}
		Let $(M,\omega)$ be a closed semifree Hamiltonian $S^1$-manifold with a moment map $H : M \rightarrow \R$ and $c \in \R$ a critical value of $H$. 
		If $Z_c := H^{-1}(c) \cap M^{S^1}$ consists of index two (co-index two, resp.) fixed points, then $M_c = H^{-1}(c) / S^1$ is smooth and is diffeomorphic to $M_{c-\epsilon}$
		($M_{c+\epsilon}$, resp.). 
		Also, $M_{c+\epsilon}$ is the blow-up (blow-down, resp.) of $M_c$ along $Z_c$.
	\end{proposition}

	More generally, Guillemin and Sternberg \cite[Theorem 11.1]{GS} described how the topology of a reduced space varies when crossing a critical value of $H$.  
	Here we introduce their result briefly, even though we apply it to a very special case.
	Let $Z \subset M^{S^1}$ be a critical submanifold (fixed component) of $H$ at level $c \in \R$ and its signature is $(2p, 2q)$. 
	For a sufficiently small $\epsilon >0$, if we perform an $S^1$-equivariant symplectic blow-up of $(M,\omega)$ along $Z$ an $\epsilon$-amount, then we get a new Hamiltonian $S^1$-manifold 
	$(\widetilde{M}, \widetilde{\omega})$ and it has two fixed components $\widetilde{Z}_{c-\epsilon}$ and $\widetilde{Z}_{c+\epsilon}$ at level $c-\epsilon$ and $c+\epsilon$, respectively.
	Moreover, $\widetilde{Z}_{c-\epsilon}$ (respectively $\widetilde{Z}_{c+\epsilon}$) is of signature $(2,2q)$ (respectively signature $(2, 2p)$). Since the blow-up changes nothing outside 
	$H^{-1}([c-\epsilon, c+\epsilon])$, 
	we see that the reduced space $M_{c-\eta}$ 
	(respectively $M_{c+\eta}$) is diffeomorphic to $\widetilde{M}_{c-\epsilon-\eta}$ (respectively $\widetilde{M}_{c+\epsilon+\eta}$) for a sufficiently small $\eta >0$.
	Therefore, $M_{c+\eta}$ is obtained by ``blowing down along $\widetilde{Z}_{c+\epsilon}$'' of a ``blow-up'' of $M_{c-\eta}$ along $\widetilde{Z}_{c - \epsilon}$.
	 (This relation is so-called a {\em birational equivalence of reduced spaces}.
	See \cite{GS} for more details.)
	Note that Proposition \ref{proposition_topology_reduced_space} is the special case ($p=2$, or $q=2$) of \cite[Theorem 11.1]{GS}.

\subsection{Variation of Euler classes} 
\label{ssecVariationOfEulerClasses}
	
	Recall that the Duistermaat-Heckman's theorem \ref{theorem_DH} says that 
	\[
		[\omega_r] - [\omega_s] = (s-r)e, \quad r,s \in I
	\]
	where $I$ is an interval consisting of regular values of $H$ and $e \in H^2(M_r; \Z)$ denotes the Euler class of the principal $S^1$-bundle 
	$\pi_r : H^{-1}(r) \rightarrow M_r$. 
	
	As the topology of a reduced space changes, the topology of a principal bundle $\pi_r : H^{-1}(r) \rightarrow M_r$ changes. 
	In \cite{GS}, Guillemin and Sternberg also provide a variation formula of the Euler class of the principal $S^1$-bundle $\pi_r : H^{-1}(r) \rightarrow M_r$
	when $r$ passes through a critical value of index (or co-index) two.  

	\begin{lemma}\cite[Theorem 13.2]{GS}\label{lemma_Euler_class}
		Suppose that $Z_c = M^{S^1} \cap H^{-1}(c)$ consists of fixed components $Z_1, \cdots, Z_k$ each of which is of index two.   
		Let $e^-$ and $e^+$ be the Euler classes of principal $S^1$-bundles $\pi_{c-\epsilon} : H^{-1}(c-\epsilon) \rightarrow M_{c-\epsilon}$ and 
		$\pi_{c+\epsilon} : H^{-1}(c+\epsilon) \rightarrow M_{c+\epsilon}$, respectively.
		Then 
		\[
			e^+ = \phi^*(e^-) + E \in H^2(M_{c+\epsilon}; \Z)
		\] where $\phi : M_{c+\epsilon} \rightarrow M_{c-\epsilon}$ is the blow-down map and $E$ is the Poincar\'{e} dual of the exceptional divisor of $\phi$.
	\end{lemma}
	
	The {\em Dustermaat-Heckman function}, denoted by ${\bf DH}$, is a function defined on $\mathrm{Im}~H$ and it assigns the symplectic area of the reduced space $M_r$, 
	i.e.,
	\[
		\begin{array}{ccccl}\vs{0.1cm}
			{\bf DH} & : & \mathrm{Im}~H & \rightarrow & \R \\ \vs{0.1cm}
					&	& r  & \mapsto & \ds {\bf DH}(r) := \int_{M_r} \omega_r^{n-1}. \\
		\end{array}
	\]
	It follows from Theorem \cite{DH} that the Duistermaat-Heckman function is a piecewise polynomial function, that is, if $(b,c) \subset \mathrm{Im}~H$ is an open interval consisting of 
	regular values of $H$, then the restriction of ${\bf DH}$ onto $(b,c)$ is a polynomial in one variable.

%
%

\subsection{Equivariant monotone symplectic form}
\label{ssecEquivariantMonotoneSymplecticForm}

	The monotonicity of $(M,\omega)$ guarantees the existence of 
	so-called the {\em equivariant monotone symplectic form}.

	\begin{proposition}\label{proposition_normalized_moment_map}
		Suppose that $(M,\omega)$ is a monotone closed Hamiltonian $S^1$-manifold such that $c_1(TM) = [\omega] \in H^2(M)$.
		Then there exists a unique moment map $H : M \rightarrow \R$ such that 
		\begin{equation}\label{equation_balanced}
			[\widetilde{\omega}_H] = c_1^{S^1}(TM) \in H^2_{S^1}(M)
		\end{equation}
		where $\widetilde{\omega}_H$ is the equivariant symplectic form defined in Section \ref{ssecEquivariantSymplecticClasses}.
	\end{proposition}

	\begin{proof}
		Recall that the equivariant formality of $(M,\omega)$ yells that
		\[
			 f^* : H^*_{S^1}(M) \rightarrow H^*(M)
		\]
		is a surjective ring homomorphism by Theorem \ref{theorem_equivariant_formality}
		whose kernel is given by
		\[
			\ker{f^*} = x \cdot H^*_{S^1}(M)
		\] 
		where $x \in H^2(BS^1)$ is the generator of $H^*(BS^1)$.
           
           Choose any moment map $\bar{H} : M \rightarrow \R$. Since $i^*([\widetilde{\omega}_{\bar{H}}]) = [\omega]$ and $i^*(c_1^{S^1}(TM)) = c_1(TM) = [\omega]$, the difference 
           $[\widetilde{\omega}_{\bar{H}}] - c_1^{S^1}(TM) \in H_{S^1}^2(M)$ is in $\ker f^*$ so that 
           we have
		\[
			[\widetilde{\omega}_{\bar{H}}] + x\cdot a = c_1^{S^1}(TM) 
		\]
		for some $a \in \R$ by Theorem \ref{theorem_equivariant_formality}. Set $H := \bar{H} - a$ as a new moment map.
		From the definition of an equivariant symplectic form in Section \ref{ssecEquivariantSymplecticClasses}, we get 
		\[
			[\widetilde{\omega}_H] - [\widetilde{\omega}_{\bar{H}}] = [\widetilde{\omega}_{\bar{H}-a}] - [\widetilde{\omega}_{\bar{H}}] = -a[d\theta] = x \cdot a,
		\]
		and therefore we obtain$[\widetilde{\omega}_H] = c_1^{S^1}(TM)$. 
	\end{proof}

	\begin{definition}\label{definition_balanced}
		The moment map $H$ in Proposition \ref{proposition_normalized_moment_map} is called the {\em balanced moment map}. 
	\end{definition}

	\begin{corollary}\label{corollary_sum_weights_moment_value}
		Assume that $c_1(M) = [\omega]$ and let $H : M \rightarrow \R$ be the balanced moment map. For each fixed component $Z \subset M^{S^1}$, 
		we have $H(Z) = -\Sigma(Z)$ where $\Sigma(Z)$ denotes the sum of all weights of the $S^1$-action at $Z$.
	\end{corollary}

	\begin{proof}
		Recall that Proposition \ref{proposition_equivariant_Chern_class} and Proposition \ref{proposition_equivariant_symplectic_class} imply that 
		\[
			[\widetilde{\omega}_H]|_Z = [\omega]|_Z - x \cdot H(Z), \quad \text{and} \quad c_1^{S^1}(TM)|_Z = c_1(TM)|_Z + x \cdot \Sigma(Z)  
		\]
		for each fixed component $Z \subset M^{S^1}$.
		Since $[\omega] = c_1(TM)$ and $H$ is balanced by our assumption, we have $[\widetilde{\omega}_H] = c_1^{S^1}(TM)$
		by Proposition \ref{proposition_normalized_moment_map}. Thus the result follows by Kirwan's injectivity theorem \ref{theorem_Kirwan_injectivity}.
        \end{proof}

	\begin{remark}\label{remark_balanced}
		Note that a moment map $H$ is called {\em normalized} if 
		\[
			\int_M H\omega^n = 0.
		\]
		It is worth mentioning that two notions `{\em normalized}' and `{\em balanced}' are different in general. 
		Indeed, we can easily check the difference between two notions in the case where $M$ is a monotone blow-up of $\p^1 \times \p^1$.
	\end{remark}
	
%
%
	
\subsection{Monotonicity of symplectic reduction}
\label{ssecMonotonicityOfSymplecticReduction}

		If $H$ is balanced, then the reduced space $H^{-1}(0) / S^1$ inherits the monotone reduced symplectic form.
		To show this, consider the embedding $H^{-1}(0) \hookrightarrow M$, which is obviously $S^1$-equivariant. 
		Consider
		\[
			\kappa : H^*_{S^1}(M) \rightarrow H^*_{S^1}(H^{-1}(0)) \cong H^*(M_0) 
		\]
		an induced ring homomorphism called the \textit{Kirwan map}.
        
		\begin{proposition}\label{proposition_monotonicity_preserved_under_reduction}
			Let $(M,\omega)$ be a semifree Hamiltonian $S^1$-manifold with $c_1(TM) = [\omega]$ and $H$ be the balanced moment map.
			If 0 is a regular value of $H$, then $(M_0, \omega_0)$ is a monotone symplectic manifold with $[\omega_0] = c_1(TM_0)$
		\end{proposition}

		\begin{proof}
			Observe that $\kappa$ takes $c_1^{S^1}(TM)$ to $c_1(TM_0)$ and $[\widetilde{\omega}_H]$ to 
			$[\omega_0]$. Since $[\widetilde{\omega}_H] = c_1^{S^1}(TM)$ by Proposition \ref{proposition_normalized_moment_map}, we have 
			$c_1(TM_0) = [\omega_0]$.
		\end{proof}

		\begin{remark}\label{remark_monotone_reduction_criticalvalue}
			Note that Proposition \ref{proposition_monotonicity_preserved_under_reduction} holds even for the case where $0$ is a critical value of $H$ under the assumption 
			that the symplectic reduction at level $0$ is well-defined. This is the case where $M^{S^1} \cap H^{-1}(0)$ consists of index two or co-index two fixed components, respectively,
			See Proposition \ref{proposition_topology_reduced_space}.
		\end{remark}

		It is an immediate consequence from Proposition \ref{proposition_monotonicity_preserved_under_reduction} that if $\dim M = 6$, then $M_0$ should be diffeomorphic to 
		either $\p^2$, $\p^1 \times \p^1$, or $\p^2 \# k \overline{\p}^2$ for $1 \leq k \leq 8$ by the classification (by Ohta and Ono \cite{OO2}) of closed monotone symplectic four manifolds. 
		
		We end up this section by the following lemma, 
		which give a list of all cohomology classes,
		called {\em exceptional classes}, in $H^2(M_0 ; \Z)$ each of which is represented by a symplectically embedded 2-sphere with the self-intersection number $-1$.
		
		\begin{lemma}\label{lemma_list_exceptional}\cite[Section 2]{McD2}
			Suppose that $M_0$ is the $k$-times simultaneous symplectic blow-up of $\p^2$ with the exceptional divisors $C_1, \cdots, C_k$.
			Denote by $E_i := \mathrm{PD}(C_i) \in H^2(M_0; \Z)$.  
			Then all possible exceptional classes are listed as follows (modulo permutations of indices) : 
			\[
				\begin{array}{l}
					E_1, u - E_{12},  \quad 2u - E_{12345}, \quad 3u - 2E_1 - E_{234567}, \quad 4u - 2E_{123} - E_{45678}  \\ \vs{0.1cm}
					5u - 2E_{123456}  - E_{78},  \quad 6u - 3E_1 - 2E_{2345678}  \\
				\end{array}
			\]
			Here, we denote by $E_{j \cdots n} := \sum_{i=j}^n E_i = $. Furthermore, elements involving $E_i$ do not appear in $M_0$ when $k < i$. 
		\end{lemma}

\section{Fixed point data}
\label{secFixedPointData}
	
	Consider a $2n$-dimensional closed Hamiltonian $S^1$-manifold $(M,\omega)$ with a moment map $H : M \rightarrow I \subset \R$. Assume that the critical values of $H$ are given by 
	\[
		\min H = c_1 < \cdots < c_k = \max H.
	\]	
	One can decompose $M$ into a union of $2n$-dimensional Hamiltonian $S^1$-manifolds $\{ (N_j, \omega_j) \}_{1 \leq j \leq 2k-1}$ with boundary where 
	\[
		N^{2j-1} = H^{-1}(\underbrace{[c_j - \epsilon, c_j + \epsilon]}_{ =: I_{2j-1}}), \quad N^{2j} = H^{-1}(\underbrace{[c_j + \epsilon, c_{j+1} - \epsilon]}_{=: I_{2j}})
	\]
	 and $\epsilon > 0$ is chosen to be sufficiently small so that $I_{2j-1}$ contains exactly one critical value $c_j$ of $H$ for each $j$.
	 We call those $N_j$'s {\em slices}. 
	 
	 \begin{definition}\label{definition_regular_slice}\cite[Definition 2.3]{G}
		A {\em regular slice} $(N,\sigma,K, I)$ is a free Hamiltonian $S^1$-manifold $(N, \sigma)$ with boundary and 
		$K : N \rightarrow I$ is a surjective proper moment map where $I = [a,b]$ is a closed interval. 
	\end{definition}

	By definition, a regular slice does not contain a fixed point and the image of a moment map consists of regular values. 

	\begin{definition}\label{definition_critical_slice}
		A {\em critical slice} $(N, \sigma, K, I)$ is a semifree Hamiltonian $S^1$-manifold $(N, \sigma)$ with boundary together with a surjective proper moment map 
		$K : N \rightarrow I = [a,b]$ such that 
		there exists exactly one critical value $c \in [a,b]$ satisfying one the followings : 
		\begin{itemize}
			\item (interior slice) $c \in (a,b)$, 
			\item (maximal slice) $c = b$ and $K^{-1}(c)$ is a critical submanifold, 
			\item (minimal slice) $c = a$ and $K^{-1}(c)$ is a critical submanifold. 
		\end{itemize}
		An interior critical slice is called {\em simple} if every fixed component in $K^{-1}(c)$ has the same Morse-Bott index. 
	\end{definition}
	
	One can define an isomorphism of slices as follows : two slices $(N_1,\sigma_1,K_1, I_1)$ and $(N_2,\sigma_2,K_2, I_2)$ are said to be 
	{\em isomorphic} if there exists an $S^1$-equivariant symplectomorphism $\phi : (N_1, \sigma_1) \rightarrow (N_2, \sigma_2)$ satisfying 
	\[
		\xymatrix{N_1 \ar[r]^{\phi} \ar[d]_{K_1} & N_2 \ar[d]^{K_2} \\ I_1 \ar[r]^{ + k} & I_2}
	\]
	where $+k$ denotes the translation map by some constant $k \in \R$. 
	We note that the notion of ``slices'' are already introduced by Li \cite{Li3} for constructing a closed Hamiltonian $S^1$-manifold by gluing slices. (She call a slice a {\em local piece} in \cite{Li3}.)

	\begin{lemma}\cite[Lemma 13]{Li3}\cite[Lemma 1.2]{McD2}\label{lemma_gluing}
		Two slices $(N_1, \sigma_1, K_1, [a,b])$ and $(N_2, \sigma_2, K_2, [b,c])$ can be glued along $K_i^{-1}(b)$ if there exists a diffeomorphism 
		\[
			\phi : (N_1)_b \rightarrow (N_2)_b, \quad \quad (N_i)_b := K_i^{-1}(b) / S^1
		\]
		such that 
		\begin{itemize}
			\item $\phi^* (\sigma_2)_b = (\sigma_1)_b$, and 
			\item $\phi^* (e_2)_b = (e_1)_b$ 
		\end{itemize}
		where $(\sigma_i)_b$ and $(e_i)_b$ denote the reduced symplectic form on $(N_i)_b$ and the Euler class of the principal $S^1$-bundle 
		$K_i^{-1}(b) \rightarrow (N_i)_b$, respectively.
	\end{lemma}
	
	Thus if we have a collection of slices (containing maximal and minimal critical slices) which satisfy the compatibility conditions given in Lemma \ref{lemma_gluing}, then 
	we can construct a closed Hamiltonian $S^1$-manifold. It is worth mentioning that the resulting closed manifold may not be unique, i.e., it might depend on the choice of 
	gluing maps. 
	Gonzalez \cite{G} used slices to classify semifree Hamiltonian $S^1$-manifolds in terms of so-called {\em fixed point data}. 
	Roughly speaking, he considered which conditions on a fixed point data of a given Hamiltonian $S^1$-manifold $(M,\omega)$ determine $(M,\omega)$ uniquely up to $S^1$-equivariant 
	symplectomorphism.

	Now, we focus on the case where $(M,\omega)$ is a six-dimensional closed monotone symplectic manifold 
	equipped with an effective semifree Hamiltonian $S^1$-action with the balanced moment map $H : M \rightarrow \R$.
	We further assume that, by scaling $\omega$ if necessary, $c_1(TM) = [\omega]$.
	
	\begin{definition}\cite[Definition 1.2]{G}\label{definition_fixed_point_data} 
		Let $(M,\omega)$ be a six-dimensional closed semifree Hamiltonian $S^1$-manifold equipped with a moment map $H : M \rightarrow I$ such that all critical level sets are simple
		in the sense of Definition \ref{definition_critical_slice}. 
		A {\em fixed point data} of $(M,\omega, H)$, denoted by $\frak{F}(M, \omega, H)$, is a collection 
		\[
			 \frak{F} (M, \omega, H) := \left\{(M_{c}, \omega_c, Z_c^1, Z_c^2, \cdots,  Z_c^{k_c}, e(P_{c}^{\pm})) ~|~c \in \mathrm{Crit} ~H \right\}
		\]
		which consists of the information below.
		\begin{itemize}
			\item $(M_c, \omega_c)$\footnote{$M_c$ is smooth manifold under  the assumption that the action is semifree and the dimension of $M$ is six.
				See Proposition \ref{proposition_topology_reduced_space}.} is the reduced symplectic manifold at level $c$.
			\item $k_c$ is the number of fixed components at level $c$. 
			\item Each $Z_c^i$ is a connected fixed component and hence is a symplectic submanifold of $(M_c, \omega_c)$ via the embedding
				\[
					Z_c^i \hookrightarrow H^{-1}(c) \rightarrow H^{-1}(c) / S^1 = M_c.
				\]
				(This information contains a normal bundle of $Z_c^i$ in $M_c$.)
			\item The Euler class $e(P_c^{\pm})$ of principal $S^1$-bundles $H^{-1}(c \pm \epsilon) \rightarrow M_{c \pm \epsilon}$.
		\end{itemize}		
	\end{definition}
	
	Gonzalez proved that the fixed point data determines $(M,\omega)$ uniquely under the assumption that every reduced symplectic form is {\em symplectically rigid}. 
	Following \cite[Definition 2.13]{McD2} or \cite[Definition 1.4]{G}, a manifold $B$ is said to be {\em symplectically rigid} if 
	\begin{itemize}
		\item (uniqueness) any two cohomologous symplectic forms are diffeomorphic, 
		\item (deformation implies isotopy) every path $\omega_t$ ($t \in [0,1]$) of symplectic forms such that $[\omega_0] = [\omega_1]$ can be homotoped through families of symplectic forms 
		with the fixed endpoints $\omega_0$ and $\omega_1$ to an isotopy, that is, a path $\omega_t'$ such that $[\omega_t']$ is constant in $H^2(B)$. 
		\item For every symplectic form $\omega$ on $B$, the group $\text{Symp}(B,\omega)$ of symplectomorphisms that act trivially on $H_*(B;\Z)$ is path-connected.
	\end{itemize}
	Using this terminology, together with Definition \ref{definition_fixed_point_data}, Gonzalez proved the following. 
	
	\begin{theorem}\cite[Theorem 1.5]{G}\label{theorem_Gonzalez_5}
		Let $(M,\omega)$ be a six-dimensional closed semifree Hamiltonian $S^1$-manifold such that every critical level is simple. 
		Suppose further that every reduced space is symplectically rigid.
		Then $(M,\omega)$ is determined by its fixed point data up to $S^1$-equivariant symplectomorphism.
	\end{theorem}
	
	\begin{remark}\label{remark_Gonzalez_5}
		 Note that Theorem \ref{theorem_Gonzalez_5} is a six-dimensional version of the original statement of the Gonzalez Theorem \cite[Theorem 1.5]{G}
		 so that we may drop ``(co)-index two'' condition in his original statement because every non-extremal fixed component has index two or co-index two in a six-dimensional case.
	 \end{remark}

	Now, we introduce the notion ``{\em topological fixed point data}'', which is a topological analogue of a fixed point data, as follows. 
				
	\begin{definition}\label{definition_topological_fixed_point_data}
		Let $(M,\omega)$ be a six-dimensional closed semifree Hamiltonian $S^1$-manifold equipped with a moment map $H : M \rightarrow I$ such that all critical level sets are simple.
		A {\em topological fixed point data} of $(M,\omega, H)$, denoted by $\frak{F}_{\text{top}}(M, \omega, H)$, is defined as a collection 
		\[
			 \frak{F}_{\text{top}}(M, \omega, H) := \left\{(M_{c}, [\omega_c], [Z_c^1], [Z_c^2], \cdots, [Z_c^{k_c}], e(P_c^{\pm}) ) ~|~c \in \mathrm{Crit} ~H \right\}
		\]
		where 
		\begin{itemize}
			\item $(M_c, \omega_c)$ is the reduced symplectic manifold at level $c$, 
			\item $k_c$ is the number of fixed components at level $c$, 
			\item each $Z_c^i$ is a connected fixed component lying on the level $c$ and $[Z_c^i] \in H^*(M_c)$ denotes the Poincar\'{e} dual class of the image of the embedding
				\[
					Z_c^i \hookrightarrow H^{-1}(c) \rightarrow H^{-1}(c) / S^1 = M_c.
				\]
			\item the Euler class $e(P_c^{\pm})$ of principal $S^1$-bundles $H^{-1}(c \pm \epsilon) \rightarrow M_{e \pm \epsilon}$.
		\end{itemize}		
	\end{definition}
	
	The following lemma allows us 
	to compute the data $e(P_c^{\pm})$ in terms of $\frak{F}_{\mathrm{top}}$.
	
	\begin{lemma}\label{lemma_Euler_condition}
		If $(M,\omega)$ is a six-dimensional closed monotone semifree Hamiltonian $S^1$-manifold, then 
		the Euler classes $\{e(P_c^{\pm}) ~|~ c \in \mathrm{Crit} ~H\}$ is completely determined by other topological fixed point data
		\[
			\left\{(M_{c}, [\omega_c], [Z_c^1], [Z_c^2], \cdots, [Z_c^{k_c}],  ) ~|~c \in \mathrm{Crit} ~H \right\}
		\] 
		and $e(P_{\min}^+)$, the Euler class of the $S^1$-bundle near the minimum of $H$.
	\end{lemma}
	
	\begin{proof}
		The proof directly follows from Proposition \ref{proposition_topology_reduced_space} and Lemma \ref{lemma_Euler_class}.
	\end{proof}
	
	Our aim is to classify all such manifolds up to $S^1$-equivariant symplectomorphism and show that each manifold is indeed algebraic Fano. 
	(See Theorem \ref{theorem_main}.)
	The rest of this paper consists of two parts : 
	\begin{itemize}
		\item {\bf First part  : classification of all topological fixed point data.} Throughout Section 
		\ref{secClassificationOfTopologicalFixedPointDataDimZMin}, \ref{secClassificationOfTopologicalFixedPointDataDimZMinDimZMax2}, and
		\ref{secClassificationOfTopologicalFixedPointDataDimZMinGeq2AndDimZMax4}
		we give a complete list of possible topological fixed point data that $(M,\omega)$ might have. We also show that there exists a smooth Fano 3-fold (in the Mori-Mukai list \cite{MM})
		with a semifree holomorphic $\C^*$-action having a (any) given topological fixed point data in our list. 
		\item {\bf Second part : uniqueness.} Based on our classification result, we will show that a topological fixed point data determines a fixed point data uniquely. Moreover, 
		all conditions in Theorem \ref{theorem_Gonzalez} are satisfied, and hence a topological fixed point data determines a manifold uniquely. Consequently, every $(M,\omega)$ 
		is $S^1$-equivariantly symplectomorphic to one of smooth Fano 3-folds described in the first part. 
	\end{itemize}

	We finalize this section with the following lemma which shows that a possible topological type of a fixed component is very restrictive. 
	We denote by $Z_{\min}$ and $Z_{\max}$ the minimal and the maximal fixed components of the action, respectively.

	\begin{lemma}\label{lemma_possible_critical_values} 
		All possible critical values of $H$ are $\pm 3, \pm 2, \pm 1$, and $0$. Moreover, any connected component $Z$ of $M^{S^1}$ satisfies one of the followings : 
		\begin{table}[H]
			\begin{tabular}{|c|c|c|c|}
			\hline
			    $H(Z)$ & $\dim Z$ & $\mathrm{ind}(Z)$ & $\mathrm{Remark}$ \\ \hline 
			    $3$ &  $0$ & $6$ & $Z = Z_{\max} = \mathrm{point}$ \\ \hline
			    $2$ &  $2$ & $4$ & $Z = Z_{\max} \cong S^2$ \\ \hline
			    $1$ &  $4$ & $2$ & $Z = Z_{\max}$ \\ \hline
			    $1$ &  $0$ & $4$ & $Z = \mathrm{pt}$ \\ \hline
			    $0$ &  $2$ & $2$ & \\ \hline
			    $-1$ &  $0$ & $2$ & $Z = \mathrm{pt}$ \\ \hline
			    $-1$ &  $4$ & $0$ & $Z = Z_{\min}$ \\ \hline
			    $-2$ &  $2$ & $0$ & $Z = Z_{\min} \cong S^2$ \\ \hline
			    $-3$ &  $0$ & $0$ & $Z = Z_{\min} = \mathrm{point}$ \\ \hline
			\end{tabular}
			\vspace{0.2cm}
			\caption{\label{table_fixed} List of possible fixed components}
		\end{table}

	\end{lemma}

	\begin{proof}
		Let $z$ be any point in the fixed component $Z \subset M^{S^1}$.
		Since the action is semifree, every weight of the $S^1$-representation on $T_z M$ is either $0$ or $\pm 1$. Thus all possible (unordered) weights at $z$ are 
		$(\pm 1, \pm1, \pm1)$, $(\pm 1, \pm1, 0)$, and $(\pm 1, 0, 0)$. Thus the first statement follows from Corollary \ref{corollary_sum_weights_moment_value}.
		
		For the second statement, it is enough to consider the case where $H(Z) \geq 0$ due to the symmetry of the table \ref{table_fixed}.
		Note that the zero-weight subspace of $T_z M$ is exactly the tangent space $T_z Z$ whose dimension equals twice the multiplicity of the zero weight on $T_z M$.
		If $H(Z) =3$, then the weights at $z$ is $(-1, -1, -1)$. Thus $\dim Z = 0$ (i.e., $Z = z$.) 
		Moreover, since twice the number of negative weights at $Z$ is equal to the Morse index of $Z$ 
		by Corollary \ref{corollary_properties_moment_map}, we have $\mathrm{ind}(Z) = 6$.
		Therefore, $H(Z)$ is the maximum value of $H$. 
		 
		We can complete the table \ref{table_fixed} in a similar way. The only non-trivial part of the lemma is that $Z \cong S^2$ when $H(Z) = 2$ (and hence $\dim Z = 2$.) 
		To show this, recall that $M_0$ is a monotone symplectic 4 manifold, diffeomorphic to a del Pezzo surface, by Proposition 
		\ref{proposition_monotonicity_preserved_under_reduction}. Since any del Pezzo surface simply connected, we have 
		\[
			\pi_1(M) \cong \pi_1(M_0) \cong \pi_1(Z_{\max}) \cong \{ 0 \}
		\] 
		by the Theorem \cite[Theorem 0.1]{Li1} of Li. Therefore we have $Z_{\max} = S^2$. 
	\end{proof}

	\begin{notation}\label{notation} From now on, we use the following notation. Let $c$ be a critical value of $H$. 
	\begin{itemize}
		\item $\mathrm{Crit}~H$ : set of critical values of $H$.
		\item $\mathrm{Crit}~ \mathring{H}$ : set of non-extremal critical values of $H$.
		\item $P_c^{\pm}$  : the principal bundle $\pi_{c \pm \epsilon} : H^{-1}(c \pm \epsilon) \rightarrow M_{c \pm \epsilon}$ where $\epsilon > 0$ is sufficiently small.
		\item $Z_c$ : fixed point set lying on the level set $H^{-1}(c)$. That is, $Z_c = M^{S^1} \cap H^{-1}(c)$.
		\item $\R[\lambda]$ : the cohomology ring of $H^*(BS^1;\R)$, where $-\lambda$ is the Euler class of the universal Hopf bundle $ES^1 \rightarrow BS^1.$
		\item $\R[u] / \langle u^3 \rangle$ : the cohomology ring of $H^*(\C P^2 ; \R)$ where $u$ is the Poincar\'{e} dual to a line.
		\item $P_M(t)$ : the Poincar\'{e} polynomial of $M$.
	\end{itemize}
	\end{notation}

\section{Classification of Topological fixed point data : $\dim Z_{\min} = 0$}
\label{secClassificationOfTopologicalFixedPointDataDimZMin}

In this section, we classify all topological fixed point data in the case where $\dim Z_{\min} = 0$, i.e., $Z_{\min}$ is an isolated point.
We divide into three cases by the dimension of $Z_{\max}$. 

\subsection{Case I : $\dim Z_{\max} = 0$}
\label{ssecCaseIDimZMax}

Let $(M,\omega)$ be a six-dimensional closed monotone semifree Hamiltonian $S^1$-manifold with the balanced moment map $H$ where $H(Z_{\min}) = -3$ and $H(Z_{\max}) = 3$.
In this section, we classify all possible topological fixed point data of $(M,\omega)$ as well as we provide algebraic Fano examples for each cases. 

By Lemma \ref{lemma_possible_critical_values}, the only possible non-extremal critical values are $\{\pm1, 0\}$
and each non-extremal fixed component $Z$ satisfies either 
\[
	\begin{cases}
		\text{$Z$ = pt} \hspace{1cm} \text{if $H(Z) = \pm 1$, \quad or} \\
		\text{$\dim Z = 2$} \quad \text{if $H(Z) = 0$.}
	\end{cases}
\]
Moreover, since $H$ is a perfect Morse-Bott function, we can easily see that 
\[
	|Z_1| = |Z_{-1}|
\]
by the Poincar\'{e} duality.

By the equivariant Darboux theorem, the $S^1$-action near the minimum (maximum, resp.) is locally identified with the standard semifree $S^1$-action 
\[
	 t \cdot (z_1, z_2, z_3) = (tz_1, tz_2, tz_3) \quad \left(\text{$(t^{-1}z_1, t^{-1}z_2, t^{-1}z_3)$, resp.}\right)
\]
on $(\C^3, \omega_{\mathrm{std}})$. Also, the balanced moment map is written by
\[
	H(z_1, z_2, z_3) = \frac{1}{2}|z_1|^2 + \frac{1}{2}|z_2|^2 + \frac{1}{2}|z_3|^2 - 3 \quad \left( \frac{1}{2}|z_1|^2 + \frac{1}{2}|z_2|^2 + \frac{1}{2}|z_3|^2 +3, \text{resp}.\right)
\] 
Therefore, the level set $H^{-1}(-3 + \epsilon)$ near the minimum is homeomorphic to $S^5$ and hence the reduced space is
\[
	M_{-3 + \epsilon} = H^{-1}(-3 + \epsilon) / S^1 \cong S^5 / S^1 \cong \p^2.
\]
Note that the Euler class $e(P_{-3}^+)$ of the principal $S^1$-bundle $H^{-1}(-3+\epsilon) \rightarrow M_{-3 + \epsilon} \cong \p^2$ is $-u$ where $u$ is the generator of $H^*(\p^2 ; \Z)$.
Similarly, for the maximal fixed point $Z_{\max}$, we can apply the previous argument to show that the reduced space near the maximum is $M_{3 - \epsilon} \cong \p^2$ 
and the Euler class of $H^{-1}(3 - \epsilon) \rightarrow M_{3-\epsilon}$ is given by $e(P_3^{-}) = u$. \\

\subsubsection{${\mathrm{Crit} ~\mathring{H}} = \{0\}$}
\label{sssecMathrmCritMathringH}
~ \\

In this case, the reduced space $M_0$ is $\p^2$ with the reduced symplectic form $\omega_0$ where $[\omega_0] = c_1(TM_0) = 3u$. 

\begin{lemma}\label{lemma_connectedness_1_1}
	$Z_0$ is connected.
\end{lemma}

\begin{proof}
	Recall  that $Z_0$ can be regarded as a symplectic submanifold of the reduced space $(M_0, \omega_0)$. If $Z_0$ is the disjoint union of two disjoint set $Z_0^1$ and $Z_0^2$, 
	then $[Z_0^1] \cdot [Z_0^2] = 0$. On the other hand, if we let $\mathrm{PD} (Z_0^1) = au$ and $\mathrm{PD}(Z_0^2) = bu$, then 
	\[
		\int_{Z_0^1} \omega_0 = 3a > 0, \quad \text{and} \quad \quad \int_{Z_0^2} \omega_0 = 3b > 0
	\]
	In particular both $a$ and $b$ are non-zero, and therefore $[Z_0^1] \cdot [Z_0^2] = ab \neq 0$ which leads to a contradiction.
\end{proof}

\begin{lemma}\label{lemma_Z0_1_1}
	$\mathrm{PD}(Z_0) = 2u$ and $Z_0 \cong S^2$ 
\end{lemma}

\begin{proof}
	It follows from Lemma \ref{lemma_Euler_class} that
	\[
		e(P_{0}^+) = e(P_{0}^-) + \mathrm{PD}(Z_0)
	\]
	where $e(P_{0}^+)  = e(P_{3 - \epsilon}^+) = u$ and $e(P_{0}^-)  = e(P_{-3 + \epsilon}^+) = -u$, and hence $\mathrm{PD}(Z_0) = 2u$. On the other hand, 
	the adjunction formula\footnote{Any embedded symplectic surface $\Sigma$ in a closed symplectic four manifold $(X,\omega)$ can be made into an image of some embedded
	$J$-holomorphic curve for some $\omega$-compatible almost complex structure. Therefore, we may apply the adjunction formula to $(M, \Sigma, J)$.}  for the symplectic surface 
	$Z_0$ gives
	\[
		[Z_0] \cdot [Z_0] + 2 - 2g = \langle c_1(TM_0), Z_0 \rangle = 6, \quad \quad \text{$g$ : genus of $Z_0$.}
	\]	
	So, we get $g = 0$.
\end{proof}

Summing up, we have the following. 

\begin{theorem}\label{theorem_Sec6_1_1}
	Let $(M,\omega)$ be a six-dimensional closed monotone semifree Hamiltonian $S^1$-manifold such that $\mathrm{Crit} H = \{ 3, 0, -3\}$. Then the only possible 
	topological fixed point data is given by 	
		\begin{table}[H]
			\begin{tabular}{|c|c|c|c|c|c|}
				\hline
				    & $(M_0, [\omega_0])$ & $Z_{-3}$ &  $Z_0$ & $Z_3$ \\ \hline \hline
				    {\bf (I-1)} & $(\p^2, 3u)$ & $\mathrm{pt}$ &  $Z_0 \cong S^2$, $[Z_0] = 2u$ & $\mathrm{pt}$\\ \hline    
			\end{tabular}
		\end{table}
	\noindent In particular, we have $b_2(M) = 1$ and $b_{\mathrm{odd}}(M) = 0$.
\end{theorem}

\begin{example}[Fano variety of type {\bf (I-1)}]\cite[16th in the list in p. 215]{IP}\label{example_Sec6_1_1} 
    Let $Q$ be a smooth quadric in $\p^4$, also known as a co-adjoint orbit of $\mathrm{SO}(5)$, an example of algebraic Fano 3-fold with $b_2 = 1$.
    With respect to the $\mathrm{SO}(5)$-invariant K\"{a}hler form $\omega$ with $[\omega] = c_1(TQ)$, 
    the diagonal maximal torus $T$ of $\mathrm{SO}(5)$ acts on $Q$ in a Hamiltonian fashion and its moment map image in the dual Lie algebra $\frak{t}^*$ of $T$
    is described as follows. (See \cite{Li3}, \cite{McD2}, or \cite{Tol} for more details.)
    \begin{figure}[H]
		\scalebox{1}{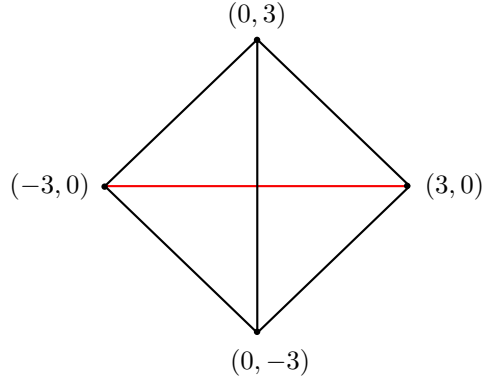}
		\caption{\label{figure_1_1} Moment map image of $Q$}
    \end{figure}
    \noindent 
    In this figure, each vertex (on the boundary of the image) corresponds to a fixed point and 
    each edge indicates an image of an invariant 2-sphere (called a 1-skeleton in \cite{GKM}). If we take a circle subgroup generated by $(0,1) \in \frak{t}$, then 
    the fixed point set is given by $\{ Z_{-3} = \mathrm{pt}, Z_0 \cong S^2, Z_3 = \mathrm{pt} \}$ where the image of $Z_0$ is colored by  red.
\end{example}
\vs{0.1cm}

\subsubsection{${\mathrm{Crit} \mathring{H}} = \{-1,1\}$}
\label{sssecMathrmCritMathringH11}
~\\

In this case, all fixed points in $Z_{-1}$ and $Z_1$ are isolated (see Table \ref{table_fixed}) and their Morse indices are two and four, respectively. Thus
the Poincar\'{e} polynomial $P_M$ of $M$ is given by 
\[
	P_M(t) = \sum b_i(M) t^i = 1 + |Z_{-1}|t^2 + |Z_{1}|t^4 + t^6.
\]
Let  $k = |Z_{-1}| = |Z_1| \in \Z_+$. For a sufficiently small $\epsilon >0$, the reduced space $M_{-1 + \epsilon}$
is diffeomorphic to the blow-up of $M_{-1-\epsilon} \cong M_{-3+\epsilon} \cong \p^2$ at $k$ generic points by Proposition \ref{proposition_topology_reduced_space}. 
Denote each classes of the exceptional divisors by 
\[
	E_1, \cdots, E_k \in H^2(M_{-1+\epsilon}; \Z).
\]
Since $e(P_{-1}^-) = -u$, Lemma \ref{lemma_Euler_class} implies
\[
	e(P_{-1}^+) = -u + E_1 + \cdots + E_k.
\]
Note that $M_0 \cong M_t$ for every $t \in (-1, 1)$ and therefore each reduced space $M_t$ can be identified with $M_0$ with the reduced symplectic form $\omega_t$
where $[\omega_t] \in H^2(M_0; \R)$. We also note that the set $\{u, E_1, \cdots, E_k\} \subset H^2(M_0; \Z)$ is an integral basis of $H^2(M_0; \Z)$ satisfying
\begin{equation}\label{equation_basis}
	\int_{M_0} u^2 = 1, \quad \int_{M_0} {E_i}^2 = -1, \quad \int_{M_0} u \cdot E_i = 0, \quad \text{and $\int_{M_0} E_i \cdot E_j = 0$}
\end{equation}
for every $0 \leq i,j \leq k$ with $i \neq j$. 

Now, let us compute the symplectic area of $(M_1, \omega_1)$. Using the Duistermaat-Heckman theorem \ref{theorem_DH}, we get 
\[
	\lim_{t \rightarrow 1^-} \int_{M_t} [\omega_t]^2 = \int_{M_0} \left([\omega_0] - e(P_{-1}^+) \right)^2 = \int_{M_0} \left( 4u - 2E_1 - \cdots - 2E_k \right)^2 = 16 - 4k.
\]	
On the other hand, the right limit of the symplectic area of $(M_t, \omega_t)$ at $t = 1$ is given by 
\[
	\lim_{t \rightarrow 1^+} \int_{M_t} [\omega_t]^2 = \lim_{t \rightarrow 1^+} \int_{\p^2} \left((3-t)u\right)^2 = \int_{\p^2} (2u)^2 = 4. 
\]
Since the Duistermaat-Heckman function is continuous, we have $k=3$ and therefore $\mathfrak{F}_{\mathrm{top}}(M)$ is given as follows.

\begin{theorem}\label{theorem_Sec6_1_2}
	Let $(M,\omega)$ be a six-dimensional closed monotone semifree Hamiltonian $S^1$-manifold such that $\mathrm{Crit} H = \{ 3, 1, -1, -3\}$. Then the only possible 
	topological fixed point data is given by 	
		\begin{table}[H]
			\begin{tabular}{|c|c|c|c|c|c|c|}
				\hline
				   & $(M_0, [\omega_0])$ & $Z_{-3}$ & $Z_{-1}$ & $Z_1$ & $Z_3$ \\ \hline \hline
				   {\bf (I-2)} & $(\p^2 \# 3 \overline{\p^2}, 3u - E_1 - E_2 - E_3)$ & {\em pt} & {\em 3 ~pts} & {\em 3 ~pts} & {\em pt} \\ \hline    
			\end{tabular}
		\end{table}
	\noindent In particular, we have $b_2(M) = 3$ and $b_{\mathrm{odd}}(M) = 0$.
\end{theorem}

\begin{example}[Fano variety of type {\bf (I-2)}]\label{example_Sec6_1_2}\cite[No. 27 in the list in Section 12.4]{IP}
	We denote by $\omega_{\mathrm{FS}}$ the Fubini-Study form on $\p^1$ which is {\em normalized}, i.e., $\int_{\p^1} \omega_{\mathrm{FS}} = ~1.$
	Consider $(M,\omega) = (\p^1 \times \p^1 \times \p^1, 2\omega_{\mathrm{FS}} \oplus 2\omega_{\mathrm{FS}} \oplus 2\omega_{\mathrm{FS}})$ with the standard 
	Hamiltonian $T^3$-action 
	\[
    		(t_1, t_2, t_3) \cdot ([x_1, y_1], [x_2, y_2], [x_3, y_3]) = ([t_1x_1, y_1], [t_2x_2, y_2] , [t_3x_3, y_3]), \quad (t_1, t_2, t_3) \in T^3
	\] 
	with a moment map $H : M \rightarrow \frak{t}^* \cong \R^3$ given by 
	\[
		H(([x_1, y_1], [x_2, y_2], [x_3, y_3])) = \left( \frac{2|x_1|^2}{|x_1|^2+|y_1|^2}, \frac{2|x_2|^2}{|x_2|^2+|y_2|^2}, \frac{2|x_3|^2}{|x_3|^2+|y_3|^2} \right) - (1,1,1)
	\]
	so that the image of the moment map is pictorially described as follows. \vs{0.1cm}
	
	\begin{figure}[H]
		\scalebox{1}{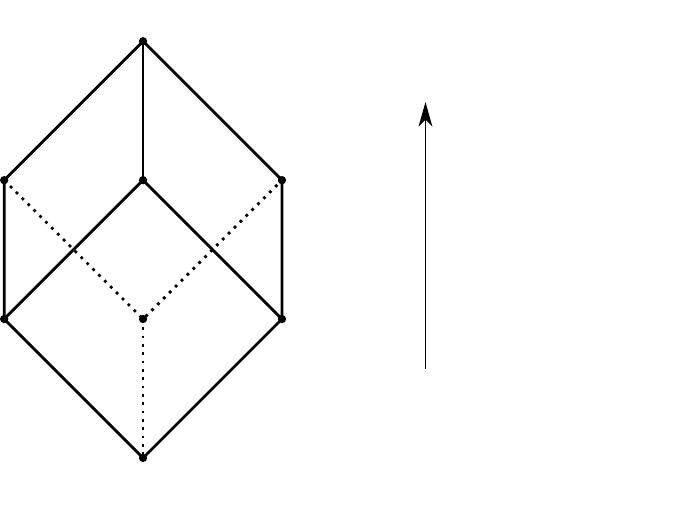}
		\caption{\label{figure_6_2}Moment graph of $\p^1 \times \p^1 \times \p^1$}
	\end{figure}
	\noindent 
	The diagonal subgroup $S^1$ of $T^3$ is generated by $\xi = (1,1,1)$ and the induced $S^1$-action has 
	the associated balanced moment map is given by $\mu = \langle H , \xi \rangle$.
	Then the $S^1$-action has the same topological data as in Theorem \ref{theorem_Sec6_1_2}.
\end{example}
\vs{0.1cm}

\subsubsection{${\mathrm{Crit} \mathring{H}} = \{-1,0,1\}$}
\label{sssecMathrmCritMathringH11}
~\\

Suppose that $|Z_{-1}| = |Z_1| = k \geq 1$ and that $Z_0$ has $r$ connected components. The Poincar\'{e} polynomial of $M$ is given by 
\[
	\begin{array}{ccl}
		P_M(t) & = & 1 + |Z_{-1}|t^2 + \left(P_{Z_0}(t) \right)t^2 + |Z_{1}|t^4 + t^6 \\
				& = & 1 + |Z_{-1}|t^2 + (r + st + rt^2) t^2 + |Z_{1}|t^4 + t^6 \\
				& = & 1 + (k+ r) t^2 + st^3 + (k + r) t^4 + t^6
	\end{array}
\]
where $s$ is the rank of $H^1(Z_0; \Z)$. 
In this case, the reduced space $M_{-1+\epsilon}$ is 
a $k$-times blow-up of $M_{-1-\epsilon} \cong \p^2$ with the exceptional classes $E_1, \cdots, E_k \in H^2(M_{-1+\epsilon}; \Z)$.

\begin{lemma}\label{lemma_1_3_k}
	Following the above notation, we have $k = 1$ and $\mathrm{Vol}(Z_0) = 4$. 
\end{lemma}

\begin{proof}
	First, we apply Theorem \ref{theorem_localization} to the equivariant first Chern class $c_1^{S^1}(TM)$ : 
	\begin{equation}\label{equation_k_1_vol_4}
		\begin{array}{ccl}\vs{0.1cm}
				   		0 & = & \ds \int_M c_1^{S^1}(TM) \\ \vs{0.1cm}
							& = & \ds \sum_{Z \subset M^{S^1}} \int_Z \frac{c_1^{S^1}(TM)|_Z}{e_Z^{S^1}} \\ \vs{0.1cm}
							& = & \ds  \frac{3\lambda}{\lambda^3} - \frac{k\lambda}{\lambda^3} + 
							\sum_{Z \subset Z_0} \int_{Z} \frac{c_1^{S^1}(TM)|_{Z}}{e_{Z}^{S^1}} - \frac{k\lambda}{\lambda^3} + \frac{3\lambda}{\lambda^3}. \\
		\end{array}
	\end{equation}
	Moreover, it follows from Corollary \ref{corollary_sum_weights_moment_value} and 
	Proposition \ref{proposition_monotonicity_preserved_under_reduction} that 
	\[
		\begin{cases}
			c_1^{S^1}(TM)|_{Z} = c_1(TM)|_{Z} + \lambda \cdot \underbrace{\sum (Z)}_{\text{sum of weights} = 0}  = c_1(TM)|_{Z} \quad \text{for $Z\subset Z_0$, ~and } \\
			[\omega_0] = c_1(TM_0) = 3u - E_1 - \cdots - E_k.  
		\end{cases}
	\]
	
	Let $q$ be the positive generator of $H^2(Z; \Z)$ (so that $q^2 = 0$). Since the action is semifree, 
	the equivariant first Chern classes of the positive and negative normal bundle of $Z$ in $M$ can be written by $\lambda + mq$ and $-\lambda + nq$ for some $m, n \in \Z$, respectively.
	Thus 
	\[
		\begin{array}{ccl}\vs{0.3cm}
			\ds \int_{Z} \frac{c_1^{S^1}(TM)|_{Z}}{e_{Z}^{S^1}} & = & \ds \int_{Z} \frac{c_1(TM)|_{Z}}{(\lambda + mq) (-\lambda + nq)} 
														= \int_{Z} \frac{c_1(TM)|_{Z}}{-\lambda^2 - (m-n)\lambda q}  \\ \vs{0.3cm}
												& = & \ds \int_{Z} \frac{c_1(TM)|_{Z}\cdot (\lambda - (m-n)q)}{-\lambda(\lambda + (m-n)q)(\lambda - (m-n)q)}  \\ \vs{0.3cm}
												& = & \ds \int_{Z} \frac{c_1(TM)|_{Z}\cdot \lambda}{-\lambda^3}  \\ \vs{0.3cm}
												& = &\ds -\frac{\langle c_1(TM), [Z] \rangle \lambda}{\lambda^3} \\ \vs{0.3cm}
												& = & \ds -\frac{\mathrm{Vol}(Z)\lambda}{\lambda^3} ~(\text{since $c_1(TM) = [\omega]$}).
		\end{array}
	\]
	From \eqref{equation_k_1_vol_4}, we get $6- 2k - \mathrm{Vol}(Z_0) = 0$ so that there are only two possibilities
	\[
		\left(k, \mathrm{Vol}(Z_0) \right) = \begin{cases} (1,4), \hs{0.2cm} \text{or} \\ (2,2) \end{cases}
	\]
	
	It remains to show that $(k, \mathrm{Vol}(Z_0)) \neq (2,2)$.
	Suppose that $k=2$. Then $M_0 \cong \p \# 2 \overline{\p}^2$ with two exceptional classes $E_1, E_2 \in H^2(M_0; \Z)$. 
	Let $\mathrm{PD}(Z_0) = au + bE_1 + cE_2 \in H^2(M_0; \Z)$. 
	Then Lemma \ref{lemma_Euler_class} implies that
	\[
		e(P_0^+) = e(P_0^-) + \mathrm{PD}(Z_0) = (a-1)u + (b+1)E_1 + (c+1)E_2.
	\]
	Since $c_1(TM_0) = [\omega_0] = 3u - E_1 - E_2$, the Duistermaat-Heckman theorem \ref{theorem_DH} implies that
	\[
		[\omega_t] = [\omega_0] - t e(P_0^+) = (3 - t(a-1))u + (-1 - t(b+1))E_1 + (-1 - t(c+1))E_2, \quad \quad t \in [0,1).
	\]
	By Proposition \ref{proposition_topology_reduced_space}, two symplectic blow-downs occur simultaneously on $M_1$.
	We denote by $C_1$ and $C_2$ the corresponding two exceptional divisors on $M_0$. 	
	Since the only possible exceptional classes in $H^2(M_0; \Z)$ is $E_1, E_2$, and $u-E_1-E_2$ by Lemma \ref{lemma_list_exceptional} 
	and $C_1$ and $C_2$ are disjoint, we have $\mathrm{PD}(C_1) = E_1$ and $\mathrm{PD}(C_2) = E_2$.
	As the symplectic areas of $C_1$ and $C_2$ go to zero as $t \rightarrow 1$, we get 
	\[
		\langle [\omega_1], [C_1] \rangle = -2-b = 0, ~~\langle [\omega_1], [C_2] \rangle = -2-c = 0, 
	\]
	i.e., $b = c = -2$. 
	
	To compute $a$, consider a symplectic volume of $M_1$. By the Duistermaat-Heckman theorem \ref{theorem_DH}, we have 
	\[
		\lim_{t\rightarrow 1^+} \int_{M_t} [\omega_t]^2 = \int_{\p^2} (2u)^2 = 4, \hs{0.5cm}
				\lim_{t\rightarrow 1^-} \int_{M_t} [\omega_t]^2 = \int_{M_0} (4-a)^2u^2 = (4-a)^2.
	\]
	Thus we obtain $4 = (4-a)^2$ and hence $a=2$ or $6$. However, if $a=6$, then the symplectic area of $M_t$ is given by
	\[
		\int_{M_t} [\omega_t]^2 = \int_{M_0} \left( (3-5t)^2 u^2 + (t - 1)^2 E_1^2 + (t - 1)^2 E_2^2 \right) = (5t - 3)^2 - 2(t-1)^2, \quad t \in (0,1)
	\]
	which is negative for some $t$ (e.g. $t = \frac{3}{5}$). Thus we get $a=2$ and
	$\mathrm{PD}(Z_0) = 2u - 2E_1 - 2E_2$, and therefore the symplectic area of $Z_0$ is given by $\langle [\omega_0], [Z_0] \rangle = \langle c_1(TM_0), [Z_0] \rangle = 2$. 
	
	Consequently, the number of connected components of $Z_0$ is at most two (since the symplectic area of each component should be a positive integer.)
	On the other hand, the adjunction formula 
	\begin{equation}\label{equation_adjunction}
		[Z_0] \cdot [Z_0] + \sum_i (2 - 2g_i)  = \langle c_1(TM_0), [Z_0] \rangle
	\end{equation}
	implies that $-4 + \sum_{i} (2-2g_i) = 2$ where the sum is taken over all connected components of $Z_0$ and $g_i$ denotes the genus of the component of $Z_0$ index by $i$.
	This equality implies that $Z_0$ should contain at least three spheres which contradicts that the number of component of $Z_0$ is at most two.
	This finishes the proof.
\end{proof}

\begin{theorem}\label{theorem_Sec6_1_3}
	Let $(M,\omega)$ be a six-dimensional closed monotone semifree Hamiltonian $S^1$-manifold such that $\mathrm{Crit} H = \{ 3, 1, 0, -1, -3\}$. 
	Then the topological fixed point data is given by 	
		\begin{table}[H]
			\begin{tabular}{|c|c|c|c|c|c|c|c|}
				\hline
				    & $(M_0, [\omega_0])$ & $Z_{-3}$ & $Z_{-1}$ & $Z_0$ & $Z_1$ & $Z_{3}$\\ \hline \hline
				    {\bf (I-3)} & $(\p^2 \# \overline{\p^2}, 3u - E_1)$ & {\em pt} & {\em pt} &  \makecell{ $Z_0 = Z_0^1 ~\dot \cup ~ Z_0^2$ \\
				    $Z_0^1 \cong Z_0^2 \cong S^2$ \\ $[Z_0^1] = [Z_0^2] = u - E_1$} & {\em pt} & {\em pt} \\ \hline    
			\end{tabular}
		\end{table}
	\noindent In particular, we have $b_2(M) = 3$, $b_{\mathrm{odd}}(M) = 0$, and $\langle c_1(TM)^3, [M] \rangle = 52$.
\end{theorem}

\begin{proof}
	Since $|Z_{-1}|=1$ by Lemma \ref{lemma_1_3_k}, we have $M_{-1 + \epsilon} \cong M_0 \cong \p^2 \# \overline{\p}^2$. 
	Let $\mathrm{PD}(Z_0) = au + bE_1 \in H^2(M_0)$. Since $e(P_{-1}^-) = -u$ and $e(P_{-1}^+) = -u + E_1$, we have 
    	\begin{equation}\label{equation_ep0}
    	   	e(P_0^+) = (a-1)u + (b+1)E_1 \quad \quad \text{(by Lemma \ref{lemma_Euler_class}.)}
	\end{equation}
	By the Duistermaat-Heckman theorem \ref{theorem_DH}, we get	   
	\[
		[\omega_t] = [\omega_0] - e(P_0^+) t = (3 - t(a-1))u + (-1 - t(b+1))E_1 \quad \text{for} \quad t \in [0,1)
	\]
	where $[\omega_0] = c_1(TM_0) = 3u - E_1$. 

	 On the level $t=1$, the symplectic blow-down occurs by Proposition \ref{proposition_topology_reduced_space}. We denote by $C$ the corresponding divisor where
	$\mathrm{PD}(C) = E_1$ by Lemma \ref{lemma_list_exceptional}. 
	Since the symplectic area of $C$ goes to zero as $t \rightarrow 1$, we get 
	\[
		0 = \langle [\omega_1], E_1 \rangle = -2-b \quad \Rightarrow \quad b = -2.
	\]

	To compute $a$, consider the equation 
           \[
        		(a-1)^2 - (b+1)^2 = \langle e(P_0^+)^2, [M_0] \rangle = \langle e(P_1^-)^2, [M_0] \rangle = \int_{M_0} (u-E_1)^2 = 0.
	\]
	where the first equality comes from \eqref{equation_ep0} and the last inequality is obtained from the fact that 
	\[
		e(P_1^-) = e(P_1^+) - E_1 = e(P_3^-) - E_1 = u - E_1 \quad \text{(by Lemma \ref{lemma_Euler_class}.)}
	\]
	This induces $(a-1)^2 - (b+1)^2 = (a-1)^2 - 1 = 0$ (since $b = -2$) 
	so that $a=0$ or $2$. 
	Moreover, since 
	\[
		\langle [\omega_0], [Z_0] \rangle = \langle 3u-E_1, [Z_0] \rangle = 3a + b = 3a - 2 > 0,
	\] we have $a=2$ and therefore $\mathrm{PD}(Z_0) = 2u - 2E_1$.

	Now, we apply the adjunction formula to $Z_0 = \sqcup Z_0^i$. Then 
	\[
		\underbrace{\langle c_1(TM_0) , [Z_0] \rangle}_{ = \langle [\omega_0], [Z_0] \rangle}
		 = \underbrace{\int_{M_0} (3u - E_1) \cdot (2u - 2E_1)}_{ = 4}  = \underbrace{\int_{M_0}(2u - 2E_1)^2}_{ = [Z_0] \cdot [Z_0] = 0} + 
		\sum_{i}\underbrace{(2-2g_i)}_{\langle c_1(TZ_0^i), [Z_0^i] \rangle} = \sum_{i}(2-2g_i), 
	\]
	where $g_i$ is the genus of $Z_0^i$. By direct computation, we may check that each $\mathrm{Z_0^i}$ is of the form $pu - pE_1$ for some $p \in \Z$ (since $Z_0^i$'s are disjoint)
	and we see that $Z_0$ is the disjoint union of two spheres $Z_0^1$ and $Z_0^2$ with 
	$\mathrm{PD}(Z_0^1) = \mathrm{PD}(Z_0^2) = u - E_1$. Using the perfectness of the moment map, it is straightforward that $b_2(M) = 3$ and $b_{\mathrm{odd}}(M) = 0$.

	For a computation of the Chern number, we use the localization theorem \ref{theorem_localization} and have that
	\begin{equation}\label{equation_CN_1_3}
		\begin{array}{ccl}\vs{0.1cm}
			\ds \int_M c_1^{S^1}(TM)^3 & = &  \ds \sum_{Z \subset M^{S^1}} \int_Z \frac{c_1^{S^1}(TM)|_Z}{e_Z^{S^1}} \\ \vs{0.1cm}
							& = & \ds  \frac{(3\lambda)^3}{\lambda^3} - \frac{(\lambda)^3}{\lambda^3} + 
							\sum_{Z \subset Z_0} \int_{Z} \frac{\left(c_1^{S^1}(TM)|_{Z}\right)^3}{e_{Z}^{S^1}} + \frac{(-\lambda)^3}{\lambda^3} - \frac{(-3\lambda)^3}{\lambda^3}. \\ \vs{0.1cm}	
		\end{array}
	\end{equation}
	Since $c_1^{S^1}(TM)|_Z = c_1(TM)|_Z$ by Proposition \ref{proposition_equivariant_Chern_class}, the term 
	$\left(c_1^{S^1}(TM)|_{Z}\right)^3$ vanishes so that 
	$\langle c_1(TM)^3, [M] \rangle = 52.$ Thus the result follows.

\end{proof}

\begin{example}[Fano variety of type {\bf (I-3)}]\label{example_Sec6_1_3}\cite[No. 31 in the list in Section 12.4]{IP}
	Let $M$ be a projective toric variety whose moment polytope 
	$\mcal{P}$ is given on the left of Figure \ref{figure_1_3}. Note that $\mcal{P}$ is a Delzant polytope so that our variety $M$ is smooth. 
	Moreover, it is easy to check that $\mcal{P}$ is reflexive (with the unique interior point $(1,1,1)$) which guarantees that $M$ is Fano. 
	
	Now, let $T$ be the 3-dimensional compact subtorus of $(\C^*)^3$ acting on $M$ and $S^1$ be the circle subgroup $S^1$ of $T$ generated by $(1,1,1) \in \frak{t}$. 
	Then the $S^1$-fixed point set consists of four points corresponding to vertices $(0,0,0), (0,0,2), (1,1,2),$ $(3,3,0)$ of $\mcal{P}$ and two spheres corresponding to
	edges $\overline{(0,1,2) (0,3,0)}$ and $\overline{(1,0,2) (3,0,0)}$ (colored by red in Figure \ref{figure_1_3}.) If we denote by $\mu : M \rightarrow \mcal{P}$ the moment map
	for the $T$-action, then the balanced moment map for the $S^1$-action is written by
	\[
		H(p) := \langle \mu(p), (1,1,1) \rangle - 3, \quad p \in M
	\]  
	Then it is straightforward to check that the topological fixed point data for the $S^1$-action is exactly the same as in Theorem \ref{theorem_Sec6_1_3}. 
	
	It is sometimes useful to use so-called a GKM-graph\footnote{See \cite{GKM}, \cite{GZ}, or \cite{CK1} for the precise definition of a GKM-graph and its properties.} 
	to describe higher dimensional Delzant polytopes. 
	In the context of toric variety, a GKM-graph is a ``well-projected'' image (onto a lower dimensional Euclidean space) of the one-skeleton of a Delzant polytope. 
	For example, the right one of Figure \ref{figure_1_3} is the projection image $\mcal{G}$ of $\mcal{P}$ onto the plane $x + y - 2z = 0$ with the coordinate system 
	whose axes are spanned by $(1,-1,0)$ and $(1,1,1)$, respectively. 
	(Note that 
	we can also think of $\mcal{G}$ as a moment map image of the Hamiltonian $T^2$-action generated by $(1,-1,0)$ and $(1,1,1)$ in $\frak{t}$. Then a moment map for the $S^1$-action
	is just a projection of $\mcal{G}$ onto the $y$-axis.)
	 
	\begin{figure}[h]
		\scalebox{1}{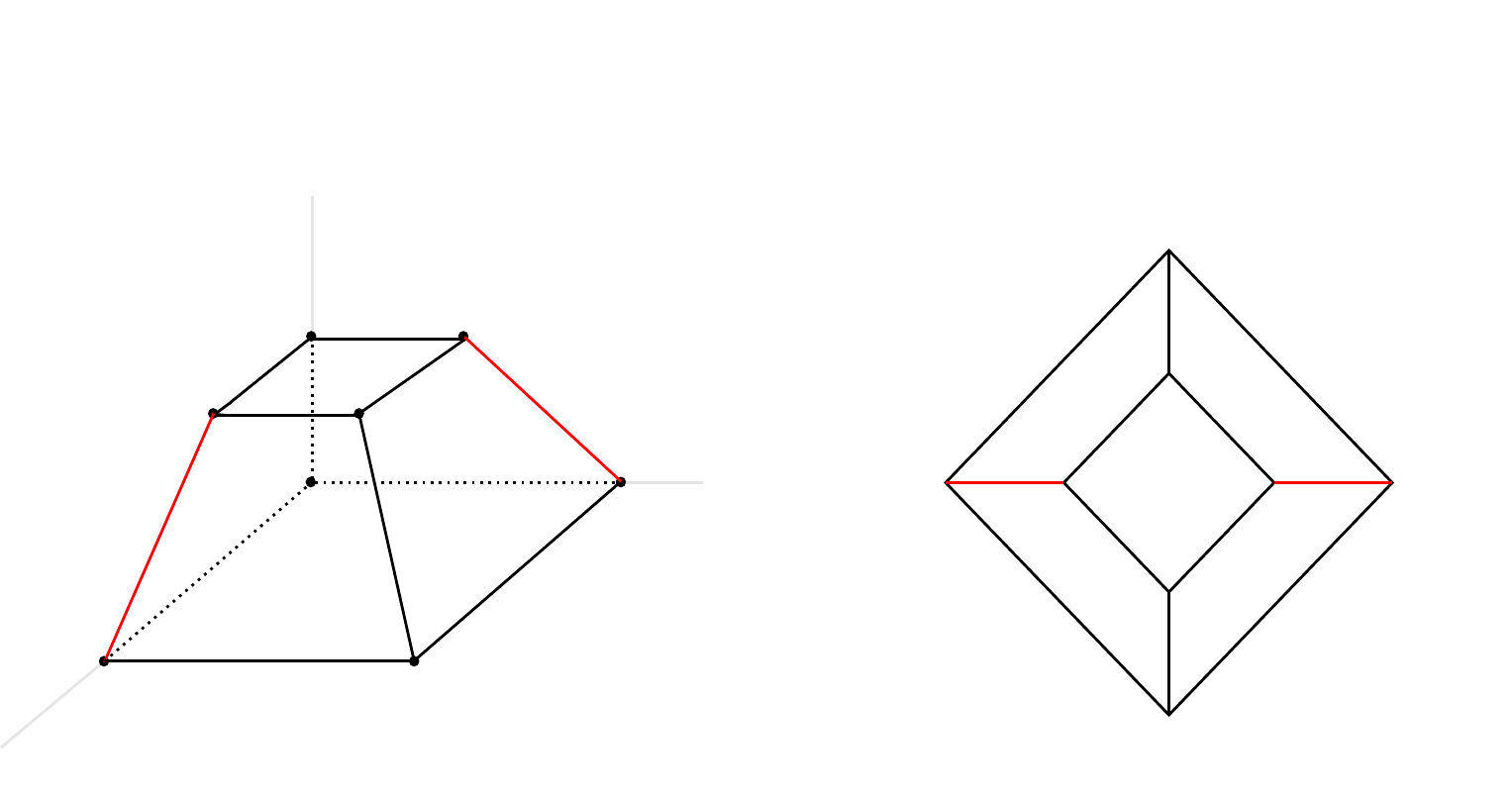}
		\caption{\label{figure_1_3} Moment graph of $\p(\mcal{O} \oplus \mcal{O}(1,1))$}
	\end{figure}

\end{example}

\subsection{Case II : $\dim Z_{\max} = 2$}
\label{ssecCaseIIDimZMax2}

In this section, we provide the classification of topological fixed point data for the case : $H(Z_{\min}) = -3$ and $H(Z_{\max}) = 2$, which is the case where
 $Z_{\max} \cong S^2$ and $Z_{\min} = \mathrm{point}$, see Lemma \ref{lemma_possible_critical_values}.

We start with two well-known facts about the number of index-two and four fixed points and the volume of the maximal fixed component.   
First, recall that $\mathrm{Crit}~ \mathring{H} \subseteq \{ 0, \pm 1\}$ and each non-extremal fixed component $Z$ satisfies
\[
	\begin{cases}
		\text{$Z$ = pt} \hspace{1cm} \text{if $H(Z) = \pm 1$, \quad or} \\
		\text{$\dim Z = 2$} \quad \text{if $H(Z) = 0$.}
	\end{cases}
\]
Since $H$ is perfect Morse-Bott, the Poincar\'{e} polynomial is given by 
\[
	P_M(t) = \sum b_i(M) t^i = 1 + |Z_{-1}|t^2 + \underbrace{P_{Z_0}(t)}_{ = r + st + rt^2} t^2 + |Z_{1}|t^4 + (1 + t^2) t^4
\]
where $r$ denotes the number of connected components of $Z_0$ and $s = \mathrm{rk} ~H_1(Z_0; \Z)$. In particular, 
the Poincar\'{e} duality implies that 
\begin{equation}\label{equation_plus_one}
	|Z_1| + 1= |Z_{-1}|,
\end{equation}
and therefore we get $Z_{-1} \neq \emptyset$. So, the set of interior critical values of $H$ is one of the followings : 
\[
	\mathrm{Crit}\mathring{H} = \{-1\}, \quad \{ -1,1\},\quad  \{ -1,0\}, \quad \text{or} \quad  \{-1,0,1\}.
\]

Second, we can compute the symplectic volume $\langle [\omega], [Z_{\max}] \rangle$ of $Z_{\max} \cong S^2$ as follows. 
Note that the reduced space near $Z_{\max}$ is an $S^2$-bundle over $S^2$ and it is well-known that there are two diffeomorphism types of 
$S^2$-bundles over $S^2$, namely a trivial bundle $S^2 \times S^2$ or a Hirzebruch surface denoted by $E_{S^2}$.

When $M_{2 - \epsilon}$ is a trivial bundle,  let $x$ and $y$ in $H^2(M_{2 - \epsilon} ;\Z)$ be the dual classes of the fiber $S^2$ and the base $S^2$ respectively so that 
\begin{equation}\label{equation_basis_trivial}
	\langle xy, [M_{2 - \epsilon}] \rangle = 1, \quad \langle y^2, [M_{2 - \epsilon}] \rangle = \langle x^2, [M_{2 - \epsilon}] \rangle = 0.
\end{equation}
Similarly, when $M_{2 - \epsilon} \cong E_{S^2}$, let $x$ and $y$ be the dual classes of the fiber $S^2$ and the base respectively which satisfy
\begin{equation}\label{equation_basis_twist}
	\langle xy, [M_{2 - \epsilon}] \rangle = 1, \quad \langle y^2, [M_{2 - \epsilon}] \rangle = -1, \quad \langle x^2, [M_{2 - \epsilon}] \rangle = 0.
\end{equation}
In either case, we have the following.

\begin{lemma}\cite[Lemma 6, 7]{Li2}\label{lemma_volume}
	Let $b_{\max}$ be the first Chern number of the normal bundle of $Z_{\max}$. Then 
	\[
		\langle e(P_2^-)^2, [M_{2-\epsilon}] \rangle = -b_{\max}.
	\]
	Also, $b_{\max}$ is even if and only if $M_{2-\epsilon} \cong S^2 \times S^2$. Moreover, if $b_{\max} = 2k$ or $2k+1$, then
	\[
		e(P_2^-) = -kx + y.
	\]
\end{lemma}

Using Lemma \ref{lemma_volume}, we obtain the following.

\begin{corollary}\label{corollary_volume}
	Let $(M,\omega)$ be a six-dimensional closed semifree Hamiltonian $S^1$-manifold. Suppose that $[\omega] = c_1(TM)$. If the maximal fixed component $Z_{\max}$ is diffeomorphic to 
	$S^2$ and $b_{\max} \in \Z$ is the first Chern number of the normal bundle of $Z_{\max}$, then 
	\[
		\int_{Z_{\max}} \omega = 2 + b_{\max} = 2 - \langle e(P_2^-)^2, [M_{2-\epsilon}] \rangle.
	\]
\end{corollary}

\begin{proof}
	The proof is straightforward from the fact that 
	\[
		\langle c_1(TM), [Z_{\max}] \rangle = 2 + b_{\max}
	\]
	Lemma \ref{lemma_volume}.
\end{proof}

Now, we are ready to classify topological fixed point data for the case where $Z_{\max} \cong S^2$ and $Z_{\min} = \mathrm{point}.$
As we mentioned above, there are four possibilities:
\[
	\mathrm{Crit}\mathring{H} = \{-1\}, \quad \{ -1,1\},\quad  \{ -1,0\}, \quad \text{or} \quad  \{-1,0,1\}.
\]

\subsubsection{${\mathrm{Crit} \mathring{H}} = \{-1\}$}
\label{ssecMathrmCritMathringH1}
~\\

In this case, $Z_{-1}$ consists of a single point since $M_0 \cong M_{2 - \epsilon}$ is an $S^2$-bundle over $S^2$, which is in particular 
diffeomorphic to $\p^2 \# \overline{\p}^2$ by Proposition \ref{proposition_topology_reduced_space}. 

\begin{theorem}\label{theorem_Sec6_2_1}
	There is no six-dimensional 
	closed monotone semifree Hamiltonian $S^1$-manifold such that $\mathrm{Crit} H = \{2, -1, -3\}$.
\end{theorem}

\begin{proof}
	Recall that $M_{2 - \epsilon}$ is an $S^2$-bundle over $Z_{\max} \cong S^2$.
	If we denote by $C$ a fiber of the bundle, then $[C] \cdot [C] = 0$ (by the local triviality of a fiber bundle), which implies that
	$\mathrm{PD}(C) = au \pm aE$ for some nonzero $a \in \Z$.
	 Furthermore, since 
	 \[
	 	[\omega_t] = [\omega_0] - te(P_{-1}^+) = 3u - E_1 - (-u + E_1)t = (3+t)u - (1 + t)E_1,  \quad t \in (-1, 2), 
	\] and $C$ is a vanishing cycle at $t=2$, we get
	\[
		\begin{array}{ccl} \vs{0.1cm}
			0 & = & \ds \lim_{t\rightarrow 2} \langle [\omega_t], [C] \rangle \\ \vs{0.1cm}
					& = & (5u - 3E)\cdot (au \pm aE) \\ \vs{0.1cm}
					& = & (5 \pm 3)a \neq 0
		\end{array}
	\]
	which leads to a contradiction. Therefore no such manifold exists.
\end{proof}
\vs{0.1cm}

\subsubsection{${\mathrm{Crit} \mathring{H}} = \{-1,1\}$}
\label{ssecMathrmCritMathringH11}
~\\

Suppose that $Z_1$ consists of $k$ points (so that $|Z_{-1}| = k+1$ by \eqref{equation_plus_one}.)

\begin{lemma}\label{lemma_2_2_number}
	$k=2$ is the only possible value of $k$.
\end{lemma}

\begin{proof}
	Note that the normal bundle of $Z_{\max}$ in $M$ splits into the direct sum of two complex  line bundles. We denote the first Chern classes of each line bundles by $d_1 q$ and $d_2 q$ in 
	$H^2(Z_{\max}; \Z)$, respectively.
	
	Applying the localization theorem \ref{theorem_localization} to $1$ and $c_1^{S^1}(TM)$, respectively, we get 
	\[
		\begin{array}{ccl} \vs{0.1cm}
			0 & = & \ds \int_M 1 \\ \vs{0.1cm}
				& = & \ds \frac{1}{\lambda^3} + (k+1) \cdot \frac{1}{-\lambda^3} + k \cdot \frac{1}{\lambda^3} + \int_{Z_{\max}} \frac{1}{(-\lambda + d_1 q)(-\lambda + d_2 q)} \\ \vs{0.1cm}
				& = & \ds \frac{1}{\lambda^3} \cdot \int_{Z_{\max}} (d_1 + d_2) q = \frac{d_1 + d_2}{\lambda^3}.
		\end{array}
	\]
	Thus we have $d_1 + d_2 = 0$. Moreover,
	\[
		\begin{array}{ccl} \vs{0.1cm}
			0 & = & \ds \int_M c_1^{S^1}(TM) \\ \vs{0.1cm}
				& = & \ds \frac{3\lambda}{\lambda^3} + (k+1) \cdot \frac{\lambda}{-\lambda^3} + k \cdot \frac{-\lambda}{\lambda^3} + \int_{Z_{\max}} 
				\frac{\overbrace{-2\lambda + (d_1 + d_2) q + 2q}^{ = -2\lambda + 2q}}{(\underbrace{-\lambda + d_1 q)(-\lambda + d_2 q)}_{ = \lambda^2}} \\ \vs{0.1cm}
				& = & \ds \frac{1}{\lambda^2} \cdot (3 - 2k - 1) + \int_{Z_{\max}} \frac{-2\lambda^2 + 2\lambda q}{\lambda^3} \\ \vs{0.1cm}
				& = & \ds \frac{1}{\lambda^2} \cdot (3 - 2k - 1 + 2).
		\end{array}
	\]
	So, we get $k = 2$.
\end{proof}

Then Lemma \ref{lemma_2_2_number} implies the following. 

\begin{theorem}\label{theorem_Sec6_2_2}
	There is no six-dimensional closed monotone semifree Hamiltonian $S^1$-manifold such that $\mathrm{Crit} H = \{2, 1, -1, -3\}$.
\end{theorem}

\begin{proof}
	Lemma \ref{lemma_2_2_number} says that $Z_{-1}$ consists of three points so that $M_0 \cong \p^2 \# 3 \overline{\p}^2$. 
	As $t$ approaches to $1$, two exceptional spheres, namely $C_1$ and $C_2$, are getting smaller in a symplectic sense and eventually vanish on $M_1$.
	In other words, two simultaneous blow-downs occur on the level $t = 1$. 
	
	On the other hand, the Duistermaat-Heckman theorem \ref{theorem_DH} says that
	\[
		\begin{array}{ccl}\vs{0.1cm}
			[\omega_t] = [\omega_0] - t e(P_0^+) & = & (3u - E_1 - E_2 - E_3) - t (-u + E_1 + E_2 + E_3) \quad t\in (-1, 1) \\ \vs{0.1cm}
									& = & (3 + t)u - (E_1 + E_2 + E_3)(1+t). 
		\end{array}
	\]
	Observe that 
	\[
			\ds \langle [\omega_t], u - E_1 - E_2 \rangle = \langle [\omega_t], u - E_2 - E_3 \rangle = \langle [\omega_t], u - E_1 - E_3 \rangle = 1 - t.
	\]
	That is, three disjoint exceptional divisors $\{u - E_i - E_j ~|~ 1 \leq i, j \leq 3, i \neq j\}$ vanish at $t = 1$, which leads to a contradiction.
\end{proof}

	\begin{figure}[h]
		\scalebox{1}{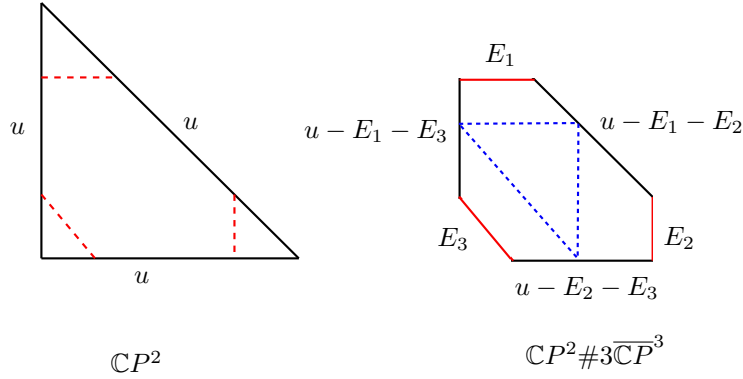}
		\caption{Blow-ups and blow-downs}
	\end{figure}

\subsubsection{${\mathrm{Crit} \mathring{H}} = \{-1,0\}$}
\label{ssecMathrmCritMathringH1}
~\\

	In this case, we have $|Z_{-1}| = 1$ by \eqref{equation_plus_one} and $M_0 \cong \p^2 \# \overline{\p}^2$.
	Regarding $Z_0$ as an embedded symplectic submanifold of $(M_0, \omega_0)$ via
	\[
		Z_0 \hookrightarrow H^{-1}(0) \stackrel{/S^1}\longrightarrow  M_0,
	\]
	let $\mathrm{PD}(Z_0) = au + bE_1 \in H^2(M_0; \Z)$ for some  $a, b \in \Z$.
	
	Note that $M_{2 - \epsilon}$ is a symplectic $S^2$-bundle over $Z_{\max}$ where we denote by $C$ a fiber of the bundle. Since $M_{2 - \epsilon} \cong M_0$
	by Proposition \ref{proposition_topology_reduced_space}, we can express 
	$\mathrm{PD}(C)$ as a linear combination of $u$ and $E_1$. 

	\begin{lemma}\label{lemma_2_3_1}
		$\mathrm{PD}(C) = u - E_1$. 
	\end{lemma}
	
	\begin{proof}
		Since $[C] \cdot [C] = 0$, we have $\mathrm{PD}(C) = pu \pm pE_1$ for some $p \neq 0$ in $\Z$. Also, the adjunction formula \eqref{equation_adjunction} implies that 
		\[
			 3p \pm p = \langle 3u - E_1, pu \pm pE_1 \rangle = \langle c_1(TM_0), [C] \rangle = [C] \cdot [C] + 2 = 2.
		\]
		So, we have $p=1$ and $\mathrm{PD}(C) = u - E_1$.
	\end{proof}	
	
	\begin{lemma}\label{lemma_2_3_2}
		All possible pairs of $(a,b)$ are $(0,1)$, $(1,0)$, or $(2, -1)$. In any case, we have $Z_0 \cong S^2$.
	\end{lemma}
	
	\begin{proof}
		We obtain three (in)equalities in $a$ and $b$ as follows. First, the Duistermaat-Heckman theorem \ref{theorem_DH} implies that 
		\[
			\begin{array}{ccl}\vs{0.1cm}
				[\omega_t] = [\omega_0] - t e(P_0^+) & = & (3u - E_1) - t(-u + E_1 + au + bE_1) \\ \vs{0.1cm}
										& = & (3 - at + t)u - (1 + bt + t)E_1, \quad t \in (0, 2).
			\end{array}
		\]
		Since the symplectic volume of $C$ tends to 0 as $t \rightarrow 2$, we have 
		\[
			\lim_{t\rightarrow 2} \langle [\omega_t], [C] \rangle = \lim_{t\rightarrow 2} \langle [\omega_t], u-E_1 \rangle = 5 - 2a -2b -3 = 0 \quad \Leftrightarrow \quad a + b = 1.
		\]
		Second, the condition $\langle [\omega_0], [Z_0] \rangle > 0$ implies that $3a+b>0$ (since $[\omega_0] = c_1(TM_0) = 3u - E_1$.)
		Third, consider any section $\sigma$ of the bundle $M_{2-\epsilon}$ over $Z_{\max} = Z_2$. Since the intersection of $\sigma$ and a fiber $C$
		equals one, we have 
		\[
			[\sigma] \cdot [C] = 1, 
		\]
		equivalently $\mathrm{PD}(\sigma) = u + d(u - E_1)$ for some $d \in \Z$. 
		In particular, we have 
		\[
			\begin{array}{ccl}\vs{0.1cm}
				\langle [\omega], [Z_{2}] \rangle & = & \lim_{t \rightarrow 2} \langle [\omega_t], [\sigma] \rangle \\ \vs{0.1cm}
									& = & \lim_{t\rightarrow 2} \langle [\omega_t] \cdot \mathrm{PD}(\sigma), [M_0] \rangle \\ \vs{0.1cm}
									& = & \lim_{t \rightarrow 2} \langle (5 - 2a) (u - E_1) \cdot (u + d(u-E_1)), [M_0] \rangle = 5- 2a > 0 
			\end{array}
		\]
		Combining the three (in)equalities $a+b = 1$, $3a + b > 0$, and $5 - 2a > 0$, we may conclude that $(a,b)$ is either $(0,1)$, $(1,0)$, or $(2,-1)$.
		(Note that this consequence is also obtained from Lemma \ref{lemma_volume})

		It remains to show that $Z_0 \cong S^2$. Using the adjunction formula \eqref{equation_adjunction}, we have 
		\[
			\underbrace{\langle 3u - E_1, au + bE_1 \rangle}_{= ~\text{symplectic area of $Z_0$ on $M_0$} ~= ~3a + b} = a^2 - b^2 + \sum_i 2 - 2g_i.
		\]
		where each connected component of $Z_0$ is indexed by $i$ and $g_i$ denotes the genus of each component. 
		For $(a,b) = (0,1)$, since $\mathrm{Vol}(Z_0) = 3a + b = 1$ (and so $2 = \sum_i (2 - 2g_i)$), 
		it is easy to see that $Z_0$ is connected and its genus is equal to zero. 
		
		For $(a,b) = (1,0)$, we have $\mathrm{Vol}(Z_0) = 3$ and $2 = \sum_i (2 - 2g_i)$. This implies that there is at least one sphere component denoted by $Z_0^1$. Let
		$Z_0^2$ be the complement of $Z_0^1$ in $Z_0$ so that $Z_0^1$ and $Z_0^2$ are disjoint. If we let $\mathrm{PD}(Z_0^1) = a_1u + b_1E_1$ 
		and $\mathrm{PD}(Z_0^2) = a_2u + b_2E_1$, respectively, then 
		\begin{equation}\label{equation_2_3_Z0}
			a_1a_2 - b_1b_2 = 0, \quad \underbrace{a_1 + a_2 = 1, \quad b_1 + b_2 = 0}_{\mathrm{PD}(Z_0^1) + \mathrm{PD}(Z_0^2) = \mathrm{PD}(Z_0)}, 
			\quad \underbrace{\begin{cases} 3a_1 + b_1 = 1 ~\text{or}~ 2 \\ 3a_2 + b_2 = 2 ~\text{or}~ 1 \end{cases}}_{\mathrm{Vol}(Z_0^1) + \mathrm{Vol}(Z_0^2) = 3}
		\end{equation}
		Also, the adjunction formula \eqref{equation_adjunction} for $Z_0^1$ implies that 
		\begin{equation}\label{equation_2_3_adjunction}
			(\text{1 or 2} =) ~3a_1 + b_1 = a_1^2 - b_1^2 + 2.
		\end{equation}
		When $3a_1 + b_1 = 1$, then $a_1 = 0$ and $b_1 = 1$ which implies that $a_2 = 1$ and $b_2 = -1$. Then it contradicts the first equation in 
		\eqref{equation_2_3_Z0}. Similarly, if $3a_1 + b_1 = 2$, then $a_1 = \pm b_1$ which implies that $a_1 = 1$ and $b_1 = -1$, and so $a_2 = 0$ and $b_2 = 1$.
		This also violates the first equation in \eqref{equation_2_3_Z0}. Consequently, $Z_0$ is connected and $Z_0 \cong S^2$.
		
		We can show that $Z_0 \cong S^2$ when $(a,b) = (2,-1)$ in a similar way. In this case, $\mathrm{Vol}(Z_0) = 5$ and $5 = 3 + \sum 2 - 2g_i$ by the adjunction 
		formula. Thus we can take $Z_0^1 \cong S^2$ and $Z_0^2$ as in the ``$(a,b) = (1,0)$''-case. Then 
		\begin{equation}\label{equation_2_3_Z0_2}
			a_1a_2 - b_1b_2 = 0, \quad \underbrace{a_1 + a_2 = 2, \quad b_1 + b_2 = -1}_{\mathrm{PD}(Z_0^1) + \mathrm{PD}(Z_0^2) = \mathrm{PD}(Z_0)}, 
			\quad \underbrace{\begin{cases} 3a_1 + b_1 = 1 ~\text{or}~ 2 
			~\text{or}~ 3 ~\text{or}~ 4\\ 3a_2 + b_2 = 4 ~\text{or}~ 3 ~\text{or}~ 2 ~\text{or}~ 1\end{cases}}_{\mathrm{Vol}(Z_0^1) + \mathrm{Vol}(Z_0^2) = 5}
		\end{equation}
		and 
		\begin{equation}\label{equation_2_3_adjunction}
			(\text{1 or 2 or 3 or 4} =) ~3a_1 + b_1 = a_1^2 - b_1^2 + 2.
		\end{equation}
		Then we have 
		\begin{table}[H]
			\begin{tabular}{|c|c|c|}
				\hline
				   $3a_1 + b_1$ & $(a_1, b_1)$ & $(a_2, b_2)$\\ \hline \hline
				    1 & $(0, 1)$ & $(2, -2)$ \\ \hline
				    2 & $(1, -1)$ & $(1, 0)$ \\ \hline
				    3 & $(1, 0)$ & $(1, -1)$ \\ \hline
				    4 & $\times$ & $\times$ \\ \hline
			\end{tabular}
		\end{table}
		\noindent
		and check that the first equation of \eqref{equation_2_3_Z0_2} fails in any case. Thus $Z_0$ is connected and $Z_0 = Z_0^1 \cong S^2$ and this completes the proof.
	\end{proof}

	Consequently, we obtain the following.

	\begin{theorem}\label{theorem_Sec6_2_3}
		Let $(M,\omega)$ be a six-dimensional closed monotone semifree Hamiltonian $S^1$-manifold such that $\mathrm{Crit} H = \{ 2, 0, -1, -3\}$. 
		Then the topological fixed point data is given by
		\begin{table}[H]
			\begin{tabular}{|c|c|c|c|c|c|c|}
				\hline
				    & $(M_0, [\omega_0])$ & $Z_{-3}$ & $Z_{-1}$ & $Z_0$ & $Z_{2}$\\ \hline \hline
				   {\bf (II-3)} &  $(\p^2 \# \overline{\p^2}, 3u - E_1)$ & {\em pt} & {\em pt} &  \makecell{ $Z_0 = S^2$ \\
				    $[Z_0] = au + bE_1$} & $S^2$ \\ \hline    
			\end{tabular}
		\end{table}
		\noindent where $(a,b) = (0,1), (1,0),$ or $(2, -1)$. Moreover, $b_2(M) = 2$ and 
		the Chern number $\langle c_1(TM)^3, [M] \rangle$ is given by
		\[
			\langle c_1(TM)^3, [M] \rangle = 
			\begin{cases}
				\text{$62$ \quad if $(a,b) = (0,1)$ \quad {\bf (II-3.1)}} \\ \text{$54$ \quad if $(a,b) = (1,0)$ \quad {\bf (II-3.2)}} \\ \text{$46$ \quad if $(a,b) = (2,-1)$ \quad {\bf (II-3.3)}}
			\end{cases}
		\]	

	\end{theorem}

\begin{proof}
	We have already computed the topological fixed point data in Lemma \ref{lemma_2_3_1} where the fact $|Z_{-1}| = 1$ follows from \eqref{equation_plus_one}. 
	Also, the perfectness (as a Morse-Bott function) of the moment map $H$ implies that $b_2(M) = 2$. Thus we only need to compute the Chern number in each case. 
	
	If we let $b_+$ and $b_-$ be the first Chern numbers of the positive and negative normal (line) bundles $\xi_+$ and $\xi_-$ of $Z_0$, respectively, then the equivariant first 
	Chern class of the normal bundle of $Z_0$ in $M$
	is 
	\[
		(-\lambda + b_-q) + (\lambda + b_+ q) = (b_- + b_+)q = ([Z_0]\cdot [Z_0]) q
	\] because the normal bundle of $Z_0$ in $M_0$ is isomorphic to $\xi_+ \otimes \xi_-$, see \cite[Proof of Lemma 5]{McD1}. 
	So, we have $c_1(TM)|_{Z_0} = (b_- + b_+ + 2)q$. 
	Applying the localization theorem \ref{theorem_localization} to $c_1^{S^1}(TM)^3$, we get 
	\begin{equation}\label{equation_CN_2_3}
		\begin{array}{ccl}\vs{0.1cm}
			\ds \int_M c_1^{S^1}(TM)^3 & = &  \ds \sum_{Z \subset M^{S^1}} \int_{Z} 
								\frac{\overbrace{c_1^{S^1}(TM)^3|_Z}^{ = \left( c_1^{S^1}(TM)|_{Z}\right)^3}}{e_Z^{S^1}} \\ \vs{0.1cm}
							& = & \ds  \frac{(3\lambda)^3}{\lambda^3} + \frac{(\lambda)^3}{-\lambda^3} + 
							\int_{Z_0} \frac{\overbrace{\left(c_1^{S^1}(TM)|_{Z_0}\right)^3}^{ = \left(c_1(TM)|_{Z_0}\right)^3 = 0}}{e_{Z_0}^{S^1}} + 
							\int_{Z_2} \frac{(-2\lambda + (d_1 + d_2 + 2)q)^3}{(-\lambda +d_1q)(-\lambda + d_2q)} \\ \vs{0.1cm}
							& = & 26 - 8(d_1 + d_2) + 12(d_1 + d_2 + 2) = 50 + 4(d_1 + d_2).
		\end{array}
	\end{equation}
	On the other hand, applying the localization theorem to $c_1^{S^1}(TM)$, we have 
	\begin{equation}\label{equation_2_3_localization}
		\begin{array}{ccl}\vs{0.1cm}
			0 & = & \ds \int_M c_1^{S^1}(TM) \\ \vs{0.1cm}
				& = &  \ds \sum_{Z \subset M^{S^1}} \int_{Z} \frac{c_1^{S^1}(TM)|_Z}{e_Z^{S^1}} \\ \vs{0.1cm}
						& = & 	\ds  \frac{3\lambda}{\lambda^3} + \frac{\lambda}{-\lambda^3} + 
							\int_{Z_0} \frac{\overbrace{c_1^{S^1}(TM)|_{Z_0}}^{ = ([Z_0]\cdot [Z_0] + 2)q = (a^2 - b^2 + 2)q}}{e_{Z_0}^{S^1}} + 
							\int_{Z_2} \frac{-2\lambda + (d_1 + d_2 + 2)q}{(-\lambda +d_1q)(-\lambda + d_2q)} \\ \vs{0.1cm}
							& = & \ds \frac{1}{\lambda^2} \cdot (2 - (a^2 - b^2 + 2) + (2 - d_1 - d_2)) \\ \vs{0.1cm}
		\end{array}		
	\end{equation}
	Therefore we get $d_1 + d_2 = 2 - a^2 + b^2$. Using \eqref{equation_CN_2_3} and \eqref{equation_2_3_localization}, we can confirm that the Chern numbers
	for $(a,b) = (0,1), (1,0),$ and $(2,-1)$ are the same as given in the theorem. This completes the proof.
\end{proof}

\begin{example}[Fano variety of type {\bf (II-3)}]\label{example_Sec6_2_3}
	We provide algebraic Fano examples for each topological fixed point data given in Theorem \ref{theorem_Sec6_2_3} as follows.
	\begin{enumerate}
		\item {\bf Case (II-3.1) \cite[No. 36 in the list in Section 12.3]{IP}} 
		: For $(a,b) = (0,1)$, let $M = \pp(\mcal{O} \oplus \mcal{O}(2))$. This is a toric variety with a moment map $\mu : M \rightarrow \mcal{P}$ 
		where the moment polytope $\mcal{P}$ 
		(with respect to the normalized monotone K\"{a}hler form)
		is described by 
		\begin{figure}[H]
			\scalebox{0.8}{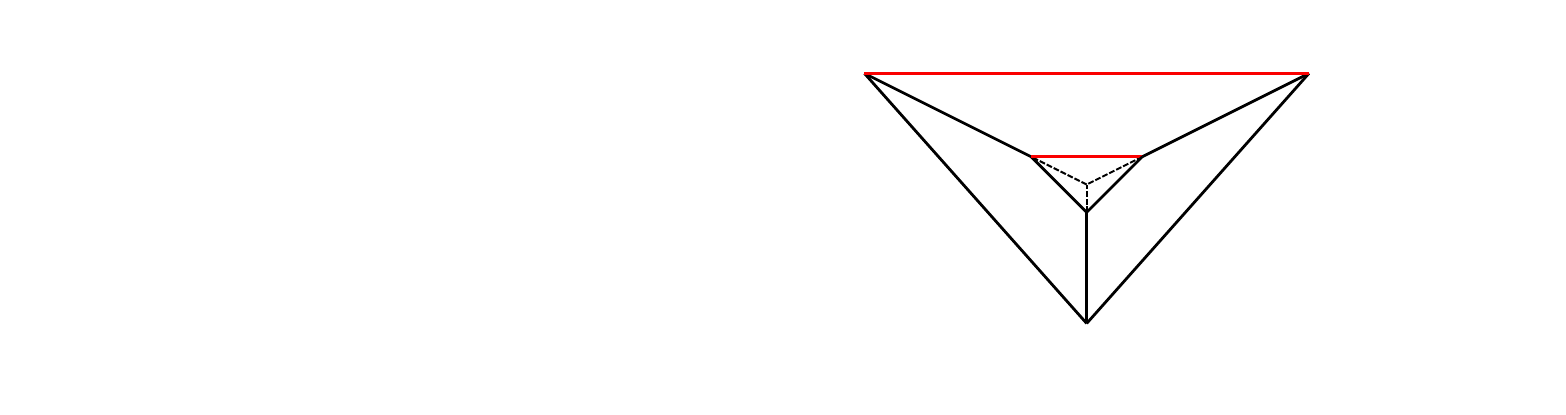}
			\caption{\label{figure_2_3_1} Moment polytope of $M = \pp(\mcal{O} \oplus \mcal{O}(2))$}
		\end{figure}
		\noindent
		where the right one of Figure \ref{figure_2_3_1} is the image of $\mcal{P}$ under the projection $\pi : \R^3 \rightarrow \R^2$ given by 
		\[
			\begin{pmatrix} a \\ b \\ c \end{pmatrix} \mapsto \begin{pmatrix} 1 & -1 & 0 \\ 1 & 1 & 1 \\ \end{pmatrix} \cdot \begin{pmatrix} a \\ b \\ c \end{pmatrix} + \begin{pmatrix} 0 \\ -3 \end{pmatrix}
		\]
		Then $\xi = (1,1,1) \in \frak{t}$ generates a semifree Hamiltonian circle action on $M$ with the balanced moment map 
		$
			\mu_\xi = \pi_y \circ \pi \circ \mu
		$
		where $\pi_y : \R^2 \rightarrow \R$ is the projection onto the $y$-axis. (Note that the ``semifreeness'' can be confirmed by showing that 
		\[
			\langle u_v, (1,1,1) \rangle = \pm 1 ~\text{or} ~0
		\]
		for any vertex $v$ of $\mcal{P}$ and a primitive integral edge vector $u_v$ at $v$. 
		
		To check that $M$ has the same topological fixed point data as in Theorem \ref{theorem_Sec6_2_3} for $(a,b) = (0,1)$, observe that there are exactly four connected faces 
		corresponding to the fixed components $Z_{-3}, Z_{-1}, Z_0, Z_2$, respectively, for the $S^1$-action generated by $(1,1,1)$, namely
		\[
			(0,0,0), \quad (0,0,2), \quad \overline{(0,1,2) (1,0,2)}, \quad \overline{(0,5,0)(5,0,0)}.
		\]
		
		In fact, we can check other geometric data of $M$, such as the volume of fixed components, coincide with those in Theorem \ref{theorem_Sec6_2_3}. 
		Note that
		\[
			\lim_{t \rightarrow 2} ~[\omega_t] = [\omega_0] - 2e(P_0^+) = 5(u - E_1)
		\]
		by the Duistermaat-Heckman theorem \ref{theorem_DH}. This implies that the symplectic area of $Z_2$ is $5$ while $Z_0$ has the symplectic area $1$. 
		(This can be obtained from the fact (used in the proof of Lemma \ref{lemma_2_3_1})
		that any section class of the $S^2$-bundle $M_{2 - \epsilon}$ over $Z_2$ 
		is of the form $u + d(u - E_1)$ for some $d \in \Z$.) 
		This is the reason why the edge $\overline{(5,-2)~(5,2)}$ (corresponding to $Z_2$) is five-times as long as $\overline{(-1,0)~(1,0)}$ (corresponding to $Z_0$).
		Furthermore, the Chern number also agrees, i.e., $\langle c_1^3 , [M] \rangle = 62$. \\
		\vs{0.1cm}
		
		\item {\bf Case (II-3.2) \cite[No. 34 in the list in Section 12.3]{IP}} : For $(a,b) = (1,0)$, consider the toric variety $M = \p^1 \times \p^2$ with the moment map $\mu : M \rightarrow \mcal{P}$ where 
		$\mcal{P}$ is given as follows. 
		\begin{figure}[H]
			\scalebox{0.8}{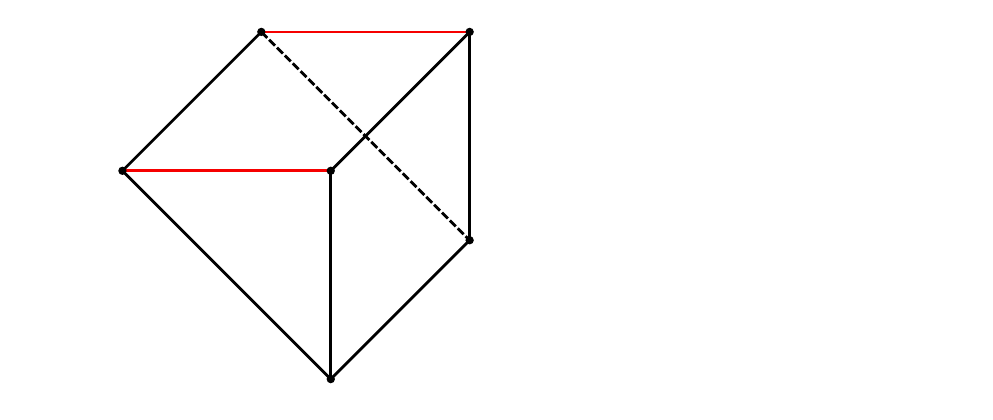}
			\caption{\label{figure_2_3_2} Moment polytope of $M = \p^1 \times \p^2$}
		\end{figure}
		\noindent
		Take $\xi = (0,-1,1) \in \frak{t}$. Then the balanced moment map for the action generated by $\xi$ is factored as $\mu_\xi = \pi_y \circ \pi \circ \mu$ where
		\[
			\pi : \R^3 \rightarrow \R^2, \quad \quad \pi(v) = \begin{pmatrix} 1 & -1 & 0 \\ 0 & -1 & 1 \end{pmatrix} \cdot v + (2,2).
		\]
		The $S^1$-action has four fixed components 
		\[
			\mu^{-1}(0,2,-3), \quad \mu^{-1}(0,0,-3), \quad \mu^{-1}(\overline{(-3,2,0) (0,2,0)}), \quad \mu^{-1}(\overline{(-3,0,0) (0,0,0)}).
		\]
		which correspond to $Z_{-3}, Z_{-1}, Z_0,$ and $Z_2$, respectively.
		
		If we want to check whether the symplectic areas of $Z_2$ and $Z_0$ coincide with those given in Theorem \ref{theorem_Sec6_2_3}, 
		recall that the Duistermaat-Heckman theorem \ref{theorem_DH} says that
		\[
			\lim_{t \rightarrow 2} ~[\omega_t] = [\omega_0] - 2e(P_0^+) = 3(u - E_1). 
		\]
		This implies that the symplectic volume of $Z_2$ (which coincides with the length of $\overline{(-3,0,0) ~(0,0,0)}$, is three and it coincides with the symplectic volume of $Z_0$, 
		the length of $\overline{(-3,2,0) ~(0,2,0)}$ in Figure \ref{figure_2_3_2}.
		Also, we may easily check that $\langle c_1^3 , [M] \rangle = 54$. \\
		\vs{0.1cm}
		
		\item {\bf Case (II-3.3) \cite[No. 31 in the list in Section 12.3]{IP}} : 
		For $(a,b) = (2,-1)$, consider the smooth quadric $Q \subset \p^4$ with a moment map $\mu : Q \rightarrow \mcal{P}$ for the $T^2$-action 
		described in Example \ref{example_Sec6_1_1}. By taking an equivariant blow-up along the rational curve corresponding 
		to the edge $\overline{(0,3) ~(3,0)}$ in Figure \ref{figure_1_1}, we get a new algebraic variety denoted by $\widetilde{Q}$ whose moment polytope $\widetilde{\mcal{P}}$
		is described in Figure \ref{figure_2_3_3}. 
		
		\begin{figure}[H]
			\scalebox{1}{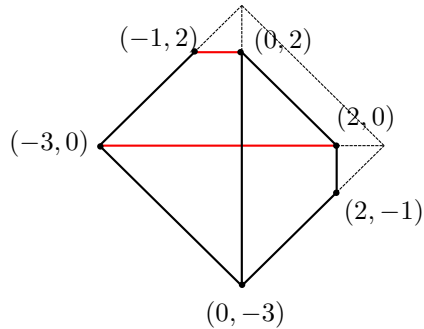}
			\caption{\label{figure_2_3_3} Moment polytope of $\widetilde{Q}$}
		\end{figure}
		\noindent
		The $S^1$-action generated by $\xi = (0,1) \in \frak{t}$ is semifree and has four fixed components corresponding to 
		\[
			(0,-3), \quad (2,-1), \quad \overline{(-3,0)(2,0)}, \quad \overline{(-1,2)(0,2)}.
		\]
		We can easily check that the moment map $\mu_\xi$ for the $S^1$-action is the projection of $\mcal{P}$ onto $y$-axis, and the four fixed components 
		has values $-3, -1, 0, 2$, respectively.
		So, $\widetilde{Q}$ together with the balanced moment map $\mu_\xi$ has the same topological fixed point data 
		given in Theorem \ref{theorem_Sec6_2_3}.  

		To compare the symplectic areas of $Z_0$ and $Z_2$ with those in Theorem \ref{theorem_Sec6_2_3}, 
		we use the Duistermaat-Heckman theorem \ref{theorem_DH} so that 
		\[
			\lim_{t \rightarrow 2} ~[\omega_t] = [\omega_0] - 2e(P_0^+) = (u - E_1). 
		\]
		Thus $\mathrm{Vol}(Z_0) = 1$ and it coincides with the length of $\overline{(-1,2) ~(0,2)}$. 
		Also, we obtain 
		\[
			\mathrm{Vol}(Z_0) = \int_{M_0} (3u - E_1) \cdot (2u - E_1) = 5, 
		\] which is equal to the length of $\overline{(-3,0) ~(2,0)}$. 
		Moreover, we may easily check that $\langle c_1^3 , [\widetilde{Q}] \rangle = 46$. \\
		\vs{0.1cm}
	\end{enumerate}
\end{example}

\subsubsection{${\mathrm{Crit} \mathring{H}} = \{-1,0,1\}$}
\label{ssecMathrmCritMathringH11}
~\\

	Assume that $Z_{-1}$ consists of $k$ points (so that $|Z_1| = k-1$ by \eqref{equation_plus_one}) for some $k \geq 2$. 

	\begin{lemma}\label{lemma_2_4_k}
		The only possible values of $k$ are $2$ and $3$.
	\end{lemma}
	
	\begin{proof}
		Applying the localization theorem \ref{theorem_localization} to $c_1^{S^1}(TM)$, we have 
	\begin{equation}\label{equation_2_4_localization}
		\begin{array}{ccl}\vs{0.1cm}
			0 & = & \ds \int_M c_1^{S^1}(TM) \\ \vs{0.1cm}
				& = &  \ds \sum_{Z \subset M^{S^1}} \int_{Z} \frac{c_1^{S^1}(TM)|_Z}{e_Z^{S^1}} \\ \vs{0.1cm}
						& = & 	\ds  \frac{3\lambda}{\lambda^3} + \frac{k\lambda}{-\lambda^3} + 
							\sum_{Z \subset Z_0} \int_{Z} \frac{\overbrace{c_1^{S^1}(TM)|_{Z}}^{ = (c_1(TM)|_Z)q = \mathrm{Vol}(Z)q}}{e_{Z}^{S^1}} + 
							\frac{-(k-1)\lambda}{\lambda^3} + 
							\int_{Z_2} \frac{-2\lambda + (c_1(TM)|_{Z_2})q}{(-\lambda +a_1q)(-\lambda + a_2q)} \\ \vs{0.1cm}
							& = & \ds \frac{1}{\lambda^2} \left( 4 - 2k - \mathrm{Vol}(Z_0) - \mathrm{Vol}(Z_2) + 4\right) \quad \quad 
							(\because a_1 + a_2 + 2 = \int_{Z_2} c_1(TM)|_{Z_2} = \mathrm{Vol}(Z_2))  \\
		\end{array}		
	\end{equation}
	where $a_1q$ and $a_2q$ denote the first Chern classes of the complex line bundles $\xi_1$ and $\xi_2$ over $Z_2$, respectively, such that $\xi_1 \oplus \xi_2$ is isomorphic to the normal bundle 
	over $Z_2$. Therefore,
	\[
		2k + \mathrm{Vol}(Z_0) + \mathrm{Vol}(Z_{\max}) = 8. 
	\]Since $\mathrm{Vol}(Z_0)$ and $\mathrm{Vol}(Z_2)$ are positive integers and $k \geq 2$, 
	the only possible value of $k$ is $2$ or $3$.
	\end{proof}

\begin{theorem}\label{theorem_Sec6_2_4}
		Let $(M,\omega)$ be a six-dimensional closed monotone semifree Hamiltonian $S^1$-manifold such that $\mathrm{Crit} H = \{ 2, 1, 0, -1, -3\}$. 
		Then, up to permutation of indices\footnote{Any permutation on $\{1,2,\cdots,k\}$ switches the ordering of exceptional divisors on $\p^2 \#~k\overline{\p^2}$. }, 
		there are two possible topological fixed point data given by 
			\begin{table}[H]
				\begin{tabular}{|c|c|c|c|c|c|c|c|}
					\hline
					   &  $(M_0, [\omega_0])$ & $Z_{-3}$ & $Z_{-1}$ & $Z_0$ & $Z_1$ & $Z_{2}$\\ \hline \hline
					   {\bf (II-4.1)} & $(\p^2 \# 2\overline{\p^2}, 3u - E_1 - E_2)$ & {\em pt} & {\em 2 pts} & \makecell{ $Z_0 = Z_0^1 ~\dot \cup ~ Z_0^2$ \\
					    $Z_0^1 \cong Z_0^2 \cong S^2$ \\ $[Z_0^1] = u - E_1$ \\ $[Z_0^2] = u - E_1 - E_2$} & {\em pt} & $S^2$ \\ \hline    
					   {\bf (II-4.2)} &  $(\p^2 \# 3\overline{\p^2}, 3u - E_1 - E_2 - E_3)$ & {\em pt} & {\em 3 pts} & \makecell{ $Z_0 = S^2$ \\		
					    			    $[Z_0] = u - E_2 - E_3$} & {\em 2 pts} & $S^2$ \\ \hline    
				\end{tabular}
			\end{table}
			\noindent
		In either case, we have $b_2(M) = 4$ and 
		\[
			\langle c_1(TM)^3, [M] \rangle = \begin{cases}
				44 & \text{\bf (II-4-1)} \\ 42 & \text{\bf (II-4-2)}
			\end{cases}
		\]
	\end{theorem}

\begin{proof}
	Thanks to Lemma \ref{lemma_2_4_k}, we know that $|Z_{-1}| = 2$ (and $|Z_1| = 1$) or $3$ (and $|Z_1| = 2$).  \vs{0.1cm}
	
	Suppose that $|Z_{-1}| = 2$. Then, we have $M_{-1 + \epsilon} \cong M_0 \cong \p^2 \# 2 \overline{\p}^2$. Denote by $\mathrm{PD}(Z_0) = au + bE_1 + cE_2$ for some $a,b,c \in \Z$. 
	Since $|Z_1| = 1$, by Proposition \ref{proposition_topology_reduced_space}, 
	exactly one symplectic blow-down occurs at $(M_1, \omega_1)$, i.e., there is a certain symplectic sphere $C$ with $[C] \cdot [C] = -1$ in $(M_0, \omega_0)$ vanishing at level $t = 1$.  
	So, 
		\begin{itemize}
			\item $\mathrm{PD}(C) = E_1, E_2,$ or $u - E_1 - E_2$ \quad (by Lemma \ref{lemma_list_exceptional}), 
			\item $\langle [\omega_0], [C] \rangle = \langle c_1(TM_0), [C] \rangle = [C]\cdot[C] + 2 = 1$ \quad (by the adjunction formula \eqref{equation_adjunction}),
			\item $\langle [\omega_1], [C] \rangle = 0$
		\end{itemize}
	where the last equation, letting $\mathrm{PD}(C) = xu + yE_1 + zE_2$, can be rephrased as  
	\begin{equation}\label{equation_2_4_blowdown}
			x(4-a) + y(b+2) + z(c+2) = 0.
	\end{equation}
	which follows from the fact that 
	\[
		\begin{array}{ccl}\vs{0.1cm}
			\lim_{t \rightarrow 1} ~[\omega_t] & = & [\omega_0] - e(P_0^+) \\ \vs{0.1cm}
									& = & = 3u - E_1 - E_2 - (-u + E_1 + E_2 + \mathrm{PD}(Z_0)) \\ \vs{0.1cm}
									& = & (4-a)u - (b+2)E_1 - (c+2)E_2
		\end{array}
	\]
	by the Duistermaat-Heckman theorem \ref{theorem_DH}. \vs{0.1cm}
	
	If $\mathrm{PD}(C) = u - E_1 - E_2$ (so that $x = 1, y = z = -1$), then $a + b + c = 0$ by \eqref{equation_2_4_blowdown}. Also, since $E_1$ and $E_2$ should not vanish on 
	the reduced space $M_1$, we have 
	\[
		\langle [\omega_1], E_1 \rangle = b+2 > 0 \quad \text{and} \quad \langle [\omega_1], E_2 \rangle = c+2 >0. 
	\]
	Moreover, as the symplectic area $\langle [\omega_t]^2, [M_t] \rangle$ is consistently positive for every $t \in (-3, 2)$,
	the coefficient of $u$ of $[\omega_t]$ should never vanish (by the mean value theorem), in particular, we have $4-a > 0$. 
	Furthermore, since $\langle [\omega_0], [Z_0] \rangle  = \langle c_1(TM), [Z_0] \rangle \geq 1$,  
	we also get $3a + b + c >0$. To sum up, we obtain 
	\begin{equation}\label{equation_II_4_inequality}
		a + b + c = 0, \quad a \leq 3,\quad  b \geq -1, \quad c \geq -1, \quad 3a + b + c \geq 1. 
	\end{equation}
	Solving \eqref{equation_II_4_inequality}, 
	we see that $\mathrm{PD}(Z_0) = 2u - E_1 - E_2$, $u - E_1$, or $u - E_2$. However, in either case, it satisfies 
	\[
		\lim_{t \rightarrow 2} \langle [\omega_t]^2, [M_t] \rangle = (5-2a)^2 - (2b+3)^2 - (2c+3)^2 < 0
	\] 
	which leads to a contradiction. Therefore, we have $\mathrm{PD}(C) \neq u - E_1 - E_2$. \vs{0.1cm}
	
	If $\pd(C) = E_1$ (or equally $\pd(C) = E_2$ up to permutation of indices), then we have $x = z = 0$ and $y = 1$ so that $b = -2$ by \eqref{equation_2_4_blowdown}. 
	Moreover, other exceptional classes $u - E_1 - E_2$ and $E_2$ should not vanish at $t = 1$ so that 
	$\langle [\omega_1], u - E_1 - E_2 \rangle = 4 - a - (c+2) >0$ and $\langle [\omega_1], E_2 \rangle = c+2 > 0$. So,
	\[
		a < 2 - c, \quad c \geq -1, \quad \text{and} \quad b = -2.
	\]
	Together with the condition $\langle [\omega_0], [Z_0] \rangle > 0$ (equivalently $3a + b + c > 0$), we can easily check that there are two possible cases;
	$\pd(Z_0) = 2u  - 2E_1 - E_2$ or $u - 2E_1$.  
	
	If $\pd(Z_0) = u - 2E_1$, then $Z_0$ has the symplectic area $\langle [\omega_0], [Z_0] \rangle = 1$ so that it is connected. 
	On the other hand, the adjunction formula \eqref{equation_adjunction} implies that 
	\[
		1 = \langle c_1(TM_0), [Z_0] \rangle = [Z_0] \cdot [Z_0] + 2 - 2g = -1 -2g
	\]
	which is impossible. Therefore, the only possible case is that $\pd(Z_0) = 2u - 2E_1 - E_2$. 
	So, the symplectic area of $Z_0$ is $\langle [\omega_0], [Z_0] \rangle = 3$, which means that $Z_0$ consists of at most three connected components. 
	In addition, the adjunction formula \eqref{equation_adjunction} says that 
	\[
		3 = \langle c_1(TM_0), [Z_0] \rangle = [Z_0] \cdot [Z_0] + \sum_i (2 - 2g_i) = -1 + \sum_i (2 - 2g_i)
	\]
	where the sum is taken over all connected components of $Z_0$. This implies that $Z_0$ should contain at least two spheres and we have two possibilities : 
	\begin{itemize}
		\item $Z_0 = S^2 \sqcup S^2$ (with symplectic areas $1$ and $2$, respectively), or 
		\item $Z_0 = S^2 \sqcup S^2 \sqcup T^2$ (with the same symplectic areas 1). 
	\end{itemize}
	In either case, we denote the two sphere components by $Z_0^1$ and $Z_0^2$ and assume the symplectic area of $Z_0^1$ is equal to one. 
	Then the adjunction formula \eqref{equation_adjunction} implies that $[Z_0^1] \cdot [Z_0^1] = -1$. Then, by Lemma \ref{lemma_list_exceptional}, 
	we have $\pd(Z_0^1) = u - E_1 - E_2$, $E_1$, or $E_2$. 
	
	First, if $Z_0 = Z_0^1 \sqcup Z_0^2$, then we have $\pd(Z_0^1) = u - E_1 - E_2$ and $\pd(Z_0^2) = u - E_1$ because, if $\pd(Z_0^1) = E_1$ 
	(or $E_2$ respectively), then 
	$\pd(Z_0^2)$ should be $2u - 3E_1 - E_2$ (or $2u - 2E_1 - 2E_2$ respectively), which implies that $[Z_0^1] \cdot [Z_0^2] \neq 0$ and this cannot happen since $Z_0^1$ and $Z_0^2$
	are disjoint.
	
	Second, if $Z_0 = Z_0^1 \sqcup Z_0^2 \sqcup Z_0^3$ with $Z_0^3 \cong T^2$, then the symplectic area of $Z_0^1$ and $Z_0^2$ are all equal to one so that 
	$[Z_0^1] \cdot [Z_0^1] = [Z_0^2] \cdot [Z_0^2] = -1$ by
	the adjunction formula \eqref{equation_adjunction}. Again by Lemma \ref{lemma_list_exceptional}, $\pd(Z_0^1)$ and $\pd(Z_0^2)$ are one of $\{u - E_1 - E_2, E_1, E_2\}$, respectively. 
	Since $[Z_0^1] \cdot [Z_0^2] = 0$, the only possible case is that $\pd(Z_0^1) = E_1$ and $\pd(Z_0^2) = E_2$, or $\pd(Z_0^1) = E_2$ and $\pd(Z_0^2) = E_1$. However, 
	in either case, we have 
	$\pd(Z_0^3) = 2u - 3E_1 - 2E_2$, which is impossible because $[Z_0^1] \cdot [Z_0^3] \neq 0$. 
	
	Consequently, if $|Z_{-1}| = 2$, then the only possible case is where $Z_0$ consists of two spheres whose dual classes are $u - E_1 - E_2$ and $u - E_1$, respectively. 
	This proves the half of 
	Theorem \ref{theorem_Sec6_2_4}. \vs{0.3cm}
	
	Now, we consider the case where $|Z_{-1}| = 3$ (and $|Z_1| = 2$ by \eqref{equation_plus_one}). In this case, we have 
	\[
		M_{-1 + \epsilon} \cong M_0 \cong \p^2 \# 3 \p^3.
	\]
	On $(M_1, \omega_1)$, two blow-downs occur simultaneously and we denote the exceptional divisors by $C_1$ and $C_2$ (with $[C_1] \cdot [C_2] = 0$). 
	Let $\pd(Z_0) = au + bE_1 + cE_2 + dE_3$ for some $a,b,c,d \in \Z$. Note that, up to permutation of indices, there are three possible cases : 
	\begin{itemize}
		\item {\bf Case I :} $\pd(C_1) = E_1$ and $\pd(C_2) = E_2$, 
		\item {\bf Case II :} $\pd(C_1) = E_1$ and $\pd(C_2) = u - E_2 - E_3$, 
		\item {\bf Case III :} $\pd(C_1) = u - E_1 - E_2$ and $\pd(C_2) = u - E_1 - E_3$. 
	\end{itemize} 
	For each case, let us investigate the (in)equalities $\langle [\omega_1], [C_1] \rangle = \langle [\omega_1], [C_2] \rangle = 0$, $\langle [\omega_0], [Z_0] \rangle > 0$, and
	$\langle [\omega_t]^2, [M_t] \rangle > 0$ for $t < 2$. Note that 
	\[
		[\omega_1] = (4-a)u - (b+2)E_1 - (c+2)E_2 -(d+2)E_3
	\] by the Duistermaat-Heckman theorem \ref{theorem_DH}.
	\vs{0.3cm}
	
	\noindent {\bf Case I.} Suppose that $\pd(C_1) = E_1$ and $\pd(C_2) = E_2$. Then we have 
	\[
		\langle [\omega_1], [C_1] \rangle = \langle [\omega_1], [C_2] \rangle = b+2 = c+2 = 0, \quad (\Leftrightarrow \quad b = c = -2.)
	\]
	Also, by assumption, the classes $E_3$, $u - E_1 - E_2$, $u - E_1 - E_3$, and $u - E_2 - E_3$ do not vanish on $(M_1, \omega_1)$ and therefore
	we have 
	\[
		\langle [\omega_1], E_3 \rangle = d + 2 > 0, \quad \langle [\omega_1], u - E_1 - E_2 \rangle = -a + 4> 0, \quad 
	\] 
	and 
	\[
		\langle [\omega_1], u - E_1 - E_3 \rangle = -a - b - d > 0, \quad \langle [\omega_1], u - E_2 - E_3 \rangle = -a - c - d > 0
	\]
	(or equivalently $2 > a + d$).
	Also, since 
	$\langle [\omega_0], [Z_0] \rangle > 0$, and $\langle [\omega_t]^2, [M_t] \rangle > 0$ for $t < 2$, we have 
	\[
		3a +  b + c + d > 0, \quad 5 - 2a \geq 0.
	\]
	Consequently, we obtain
	\[
		a \leq 2, \quad b = c = -2, \quad d \geq -1, \quad a + d \leq 1, \quad 3a + d \geq 5.
	\]
	This has the only integral solution $a = 2, b = c = -2, d = -1$
	Thus $\pd(Z_0) = 2u - 2E_1 - 2E_2 - E_3$ and the symplectic area of $Z_0$ is 
	\[
		\langle [\omega_0], [Z_0] \rangle = 1,
	\]
	which implies that $Z_0$ is connected. On the other hand, the adjunction formula \eqref{equation_adjunction} says that 
	\[
		1 = [Z_0] \cdot [Z_0] + 2 - 2g = -3 - 2g, \quad g \geq 0
	\]
	which is impossible. 
	Therefore, $\left( \pd(C_1), \pd(C_2) \right) \neq (E_1, E_2)$. \vs{0.3cm}

	\noindent {\bf Case II.} Assume that $\pd(C_1) = E_1$ and $\pd(C_2) = u - E_2 - E_3$. Then we obtain 
	\[
		\langle [\omega_1], [C_1] \rangle = b+2 = 0, \quad \langle [\omega_1], [C_2] \rangle = (4-a) - (c+2) - (d+2) = 0 \quad (\Leftrightarrow \quad b = -2, ~a + c + d = 0).
	\]
	Also, since $E_2$, $E_3$, $u - E_1 - E_2$, and $u - E_1 - E_3$ do not vanish on $(M_1, \omega_1)$, 
	we get
	\[
		\langle [\omega_1], E_2 \rangle  = c + 2 > 0, \quad \langle [\omega_1], E_3 \rangle = d + 2 > 0
	\]
	and
	\[
		\langle [\omega_1], u - E_1 - E_2 \rangle = -a - b - c > 0, \quad \langle [\omega_1], u - E_1 - E_3 \rangle = -a - b - d > 0.
	\] 
	Furthermore, since 
	$\langle [\omega_0], [Z_0] \rangle > 0$, and $\langle [\omega_t]^2, [M_t] \rangle > 0$ for $t < 2$, we have 
	\[
		3a +  b + c + d > 0, \quad 5 - 2a \geq 0.
	\]
	Therefore, 
	\[
		a \leq 2, \quad b = -2, \quad c \geq -1, \quad d \geq -1, \quad a + c \leq 1, \quad a + d \leq 1, \quad 3a + c + d \geq 3.
	\]
	and it has a unique solution $a = 2, b = -2, c = d = -1$ so that $\pd(Z_0) = 2u - 2E_1 - E_2 - E_3$. 
	
	On the other hand, note that $e(P_0^+) = u - E_1$ which implies that $\langle e(P_2^-)^2, [M_{2-\epsilon}] \rangle = 2$ by Lemma \ref{lemma_Euler_class}.
	Applying Lemma \ref{lemma_volume}, the first Chern number of the normal bundle of $Z_{\max}$ is $b_{\max} = -2$.  Then, by Corollary \ref{corollary_volume} implies that 
	\[
		\int_{Z_{\max}} \omega = 2 + b_{\max} = 0
	\]
	which leads to a contradiction. Consequently, we see that $\left( \pd(C_1), \pd(C_2) \right) \neq (E_1, u - E_2 - E_3)$. \vs{0.3cm}
	

	\noindent {\bf Case III.} Assume that $\pd(C_1) = u - E_1 - E_2$ and $\pd(C_2) = u - E_1 - E_3$. 
	Then, by the fact that  $\langle [\omega_1], [C_1] \rangle = \langle [\omega_1], [C_2] \rangle = 0$, we obtain
	\[
		(4-a) - (b+2) - (c+2) = (4 - a) - (b+2) - (d+2) = 0, \quad (\Leftrightarrow \quad a + b + c = a + b + d = 0) .
	\]	
	Also, since the classes $E_1, E_2, E_3, u - E_2 - E_3$ 
	 do not vanish on $(M_1, \omega_1)$, we have 
	 \[
	 	\langle [\omega_1], E_1 \rangle = 2 + b > 0,  \quad 
	 	\langle [\omega_1], E_2 \rangle = 2 + c > 0,  \quad 
	 	\langle [\omega_1], E_3 \rangle = 2 + d > 0,
	\]
	and 
	\[
	 	\langle [\omega_1], u - E_2 - E_3 \rangle = (4-a) - (c+2) - (d+2) > 0 \quad (\Leftrightarrow \quad a + c + d < 0).
	\]
	Moreover, since $\langle [\omega_0], [Z_0] \rangle > 0$, and $\langle [\omega_t]^2, [M_t] \rangle > 0$ for $t < 2$, we have 
	\[
		\langle [\omega_0], [Z_0] \rangle = 3a + b + c + d > 0 \quad \text{and} \quad 	5 - 2a \geq 0.
	\]
	To sum up, we obtain
	\[
		a + b + c = a + b + d = 0,\quad b,c,d \geq -1,\quad a+c+d \leq -1, \quad 3a+b+c+d \geq 1, \quad a \leq 2
	\]
	and it has a unique solution $(a,b,c,d) = (1,0,-1,-1)$, i.e., $\mathrm{PD}(Z_0) = u - E_2 - E_3$. 
	It follows from the adjunction formula \eqref{equation_adjunction} that 
	\[
		1 = \langle c_1(TM), [Z_0] \rangle = [Z_0] \cdot [Z_0] + 2 - 2g = 1 - 2g \quad (\Leftrightarrow g = 0)
	\]
	and hence $Z_0 \cong S^2$. This completes the proof.
\end{proof}

\begin{example}[Fano variety of type {\bf (II-4)}]\label{example_Sec6_2_4} In this example, we provide Fano varieties equipped with semifree $\C^*$-actions 
having topological fixed point data described in Theorem \ref{theorem_Sec6_2_4}. We first denote by $S_d$ a del Pezzo surface of degree $d$, 
that is, the $(9-d)$-times blow up of $\p^2$ where $0 \leq d\leq 8$. \vs{0.1cm}

\begin{enumerate}
    \item {\bf Case (II-4.1) \cite[No. 11 in the list in Section 12.5]{IP}} :
    Let $M = \p^1 \times S_8$ equipped with the monotone K\"{a}hler form $\omega$ with $c_1(TM) = [\omega] \in H^2(M; \Z)$. Since $\p^1$ and $S_8$ are both 
    toric varieties, $M$ is also a toric variety with the induced Hamiltonian $T^3$-action whose moment map image is the product of a closed interval (of length 2)
    and a right trapezoid as in Figure \ref{figure_2_4_1}. Let $C$ be a smooth rational curve corresponding to the edge (a dotted edge in Figure \ref{figure_2_4_1}) connecting $(0,2,2)$ and $(1,2,2)$. 
    Let $\widetilde{M}$ be the monotone toric blow-up of $M$ along $C$ with a moment map $\mu$ where the moment map image $\mu(\widetilde{M})$ is described below.   \vs{0.3cm}
		\begin{figure}[H]
			\scalebox{0.8}{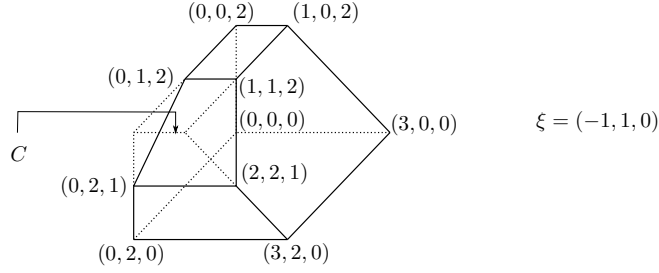} \vs{0.3cm}
			\caption{\label{figure_2_4_1} Blow-up of $\p^1 \times S_8$ along the sphere corresponding to the edge $C$.}
		\end{figure}
  Take the $S^1$-subgroup of $T^3$ generated by $\xi = (-1,1,0) \in \frak{t}$. The $S^1$-action is semifree since the dot product of each primitive edge vectors and $\xi$ is 
   either $0$ or $\pm 1$. Moreover, the fixed point set $\widetilde{M}^{S^1}$ can be expressed as \\
\begin{itemize}
   	\item $Z_{\min} = \mu^{-1}(3,0,0)) = \mathrm{pt},$ \vs{0.1cm}
   	\item $Z_{-1} = \mu^{-1}(3,2,0) ~\cup ~\mu^{-1}(1,0,2) = \mathrm{two~pts} ,$ \vs{0.2cm}
	\item $Z_0 = \underbrace{\mu^{-1}\left(\overline{(1,1,2) ~(2,2,1)}\right)}_{\text{volume} = 1} ~\cup ~\underbrace{\mu^{-1}\left(\overline{(0,0,0) ~(0,0,2)}\right)}_{\text{volume} = 2} ~\cong 
	~S^2 ~\dot \cup ~S^2,$ \vs{0.1cm}
	\item $Z_1 = \mu^{-1}(0,1,2) = \mathrm{pt}$, \vs{0.2cm}
	\item $Z_{\max} = Z_2 = \underbrace{\mu^{-1}\left(\overline{(0,2,0) ~(0,2,1)}\right)}_{\text{volume} = 1} ~\cong ~S^2$ \vs{0.1cm}
\end{itemize}
and this data coincides with the fixed point data {\bf (II-4.1)} in Theorem \ref{theorem_Sec6_2_4}.
\vs{0.5cm}

    \item {\bf Case (II-4.2) \cite[No. 10 in the list in Section 12.5]{IP}} :
    Let $M = S_7 \times \p^1$ where $S_7$ denotes the del Pezzo surface of degree $7$, i.e., two points blow up of $\p^2$. Then the toric structure on 
    $S_7$ and $\p^1$ inherits a toric structure on $M$ whose moment polytope is given in Figure \ref{figure_2_4_2}. 
    
		\begin{figure}[H]
			\scalebox{0.8}{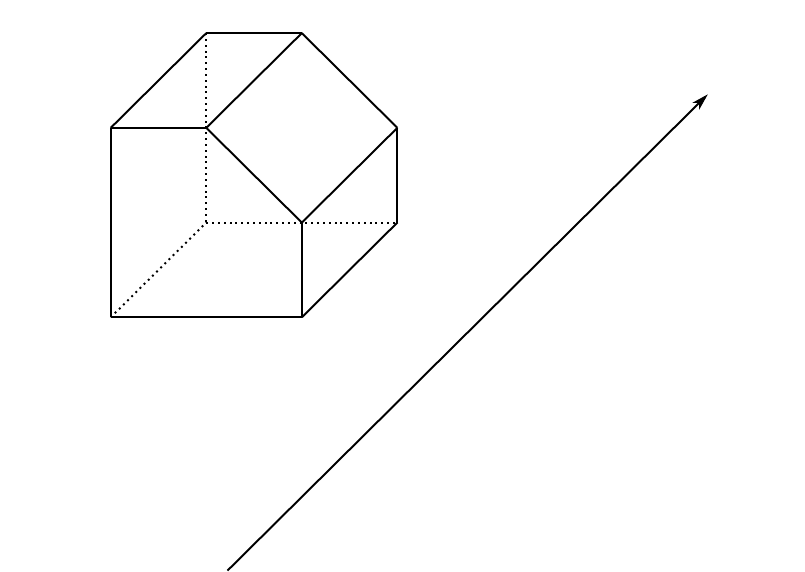} \vs{0.3cm}
			\caption{\label{figure_2_4_2} $\p^1 \times S_7$.}
		\end{figure}
		
	\noindent Now, let $S^1$ be the subgroup of $T^3$ generated by $\xi = (1, -1, 1) \in \frak{t}$. Then the $S^1$-action is semifree because every primitive edge vector 
	(except for those corresponding to fixed components of the $S^1$-action) is either one of $(\pm 1,0,0)$, $(0, \pm 1, 0)$, or $(0,0,\pm 1)$. Moreover, the fixed components for the action are given by \\ 
\begin{itemize}
   	\item $Z_{\min} = \mu^{-1}(0,2,0)) = \mathrm{pt}$, \vs{0.2cm}
   	\item $Z_{-1} = \mu^{-1}(2,2,0) ~\cup ~\mu^{-1}(0,2,2) ~\cup ~\mu^{-1}(0,0,0) = \mathrm{three~pts} ,$ \vs{0.3cm}
	\item $Z_0 = \underbrace{\mu^{-1}(\overline{(2,2,1) ~(1,2,2)})}_{\text{volume} = 1} ~\cong 
	~S^2,$ \vs{0.1cm}
	\item $Z_1 = \mu^{-1}(0,0,2) ~\cup ~\mu^{-1}(2,0,0) = \mathrm{two ~pts},$ \vs{0.2cm}
	\item $Z_{\max} = Z_2 = \underbrace{\mu^{-1}(\overline{(1,0,2) ~(2,0,1)})}_{\text{volume} = 1} ~\cong ~S^2$ \vs{0.2cm}
\end{itemize}
and are exactly the same as {\bf (II-4.2)} in Theorem \ref{theorem_Sec6_2_4}.
\end{enumerate}
\end{example}
\vs{0.1cm}

\subsection{Case III : $\dim Z_{\max} = 4$}
\label{ssecCaseIIIDimZMax4}

In this section, we give the complete classification of topological fixed point data in the case where 
\[
	H(Z_{\min}) = -3 \quad \text{and} \quad H(Z_{\max}) = 1, 
\]
or equivalently, $Z_{\min} = \mathrm{point}$ and $\dim Z_{\max} = 4$. 
We also provide algebraic Fano examples for each cases and describe them in terms of moment polytopes for certain 
Hamiltonian torus actions as in Section \ref{ssecCaseIDimZMax} and \ref{ssecCaseIIDimZMax2}.

Lemma \ref{lemma_possible_critical_values} implies all possible non-extremal critical values of the balanced moment map $H$ are $- 1$ or $0$. Also, since $Z_{\max}$ is four dimensional, 
we have $M_0 \cong Z_{\max}$ by Proposition \ref{proposition_topology_reduced_space} and therefore 
$Z_{\max} \cong M_0 \cong M_{-1+\epsilon}$ is $|Z_{-1}|$-times blow-up of $\p^2$. \\

\subsubsection{${\mathrm{Crit} \mathring{H}} = \emptyset$}
~\\

Assume that there is no non-extremal fixed point. Then, 
\[
	M_{\max} \cong M_0 \cong M_{-3 + \epsilon} \cong \p^2. 
\]
In addition, since $e(P^+_{-3}) = -u \in H^2(\p^2 ; \Z)$, the cohomology class of the reduced symplectic form on each $M_t$ is given by 
\[
	[\omega_t] = (t + 3) u, \quad -3 \leq t \leq 1.
\]
by the Duistermaat-Heckman theorem \ref{theorem_DH}. 

\begin{theorem}\label{theorem_Sec6_3_1}
		Let $(M,\omega)$ be a six-dimensional closed monotone Hamiltonian $S^1$-manifold such that $\mathrm{Crit} H = \{ 1,-3\}$. 
		Then there is a unique possible topological fixed point data given by 
			\begin{table}[H]
				\begin{tabular}{|c|c|c|c|}
					\hline
					   &  $(M_0, [\omega_0])$ & $Z_{-3}$ & $Z_1$ \\ \hline \hline
					   {\bf (III-1)} & $(\p^2, 3u)$ & {\em pt} & $\p^2$ \\ \hline
				\end{tabular}
			\end{table}
			\noindent
		Moreover, we have $b_2(M) = 1$ and 
		$
			\langle c_1(TM)^3, [M] \rangle = 64.
		$
\end{theorem}

\begin{proof}
	We only need to prove that $\langle c_1(TM)^3, [M] \rangle = 64.$ Using the localization theorem \ref{theorem_localization}, we obtain
	\[
		\begin{array}{ccl}\vs{0.1cm}
			\ds \int_M c_1^{S^1}(TM)^3 & = & \ds  \sum_{Z \subset M^{S^1}} \int_Z \frac{\left(c_1^{S^1}(TM)|_Z\right)^3}{e_Z^{S^1}} \\ \vs{0.2cm}
								& = & \ds  \frac{(3\lambda)^3}{\lambda^3} + \int_{Z_{\max}} \frac{\left(3u + (-\lambda - e(P^-_1)\right)^3 }{-\lambda - e(P^-_{1})} \quad \quad (e(P^-_1) = -u) \\ \vs{0.1cm}
								& = & 27 + \ds \int_{Z_{\max}} \frac{(4u - \lambda)^3}{u - \lambda} = 27 + \ds \int_{Z_{\max}} (48u^2 - 12u^2 + u^2) = 64.
		\end{array}
	\]
\end{proof}

\begin{example}[Fano variety of type {\bf (III-1)}]\cite[17th in the list in p. 215]{IP}\label{example_Sec6_3_1}
	Let $M = \p^3$ equipped with the Fubini-Study form $\omega$ with $[\omega ] = 4u = c_1(TM)$. 
	Then $(\p^3, \omega)$ admits a toric structure given by 
	\[
		(t_1, t_2, t_3) \cdot [x,y,z,w] = [t_1x, t_2y, t_3z, w], \quad (t_1, t_2, t_3) \in T^3, [x,y,z,w] \in \p^3
	\]
	and the corresponding moment polytope is the 3-simplex in $\R^3$ as in Figure \ref{figure_3_1}.
	
		\begin{figure}[H]
			\scalebox{0.8}{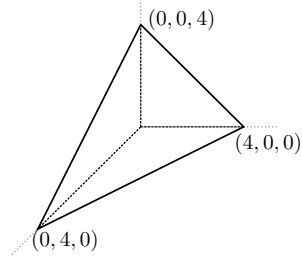} \vs{0.3cm}
			\caption{\label{figure_3_1} Moment polytope for $\p^3$.}
		\end{figure}
		
	\noindent If we take a subgroup $S^1 \subset T^3$ generated by $\xi = (1,1,1)$, then the induced action is expressed as
	\begin{equation}\label{equation_standard_action}
		t \cdot [x,y,z,w] = [tx,ty,tz,w], \quad t\in S^1, [x,y,z,w] \in \p^3
	\end{equation}
	which is semifree with two fixed components
	\[
		Z_{\min} = \mu^{-1}(0,0,0) = \mathrm{pt} \quad \text{and} \quad Z_{\max} = \mu^{-1}(\Delta) \cong \p^2
	\]
	where $\Delta$ is the triangle whose vertices are $(4,0,0)$, $(0,4,0)$, and $(0,0,4)$. This fixed point data coincides with {\bf (III-1)} in Theorem \ref{theorem_Sec6_3_1}.
\end{example}
\vs{0.3cm}

\subsubsection{${\mathrm{Crit} \mathring{H}} = \{ -1\}$}
~\\

Let $k = |Z_{-1}|$ be the number of index two fixed points. Then $M_{-1 + \epsilon} \cong M_0 \cong Z_{\max}$ is the $k$-times blow-up of $\p^2$. 
We denote by $E_1, \cdots, E_k$ the corresponding exceptional divisors. 
By Theorem \ref{theorem_DH}, we have 
\[
	[\omega_1] = 4u - 2E_1 - \cdots - 2E_k.
\]
\begin{lemma}\label{lemma_3_2}
	The only possible case is $k = 1$.
\end{lemma}
\begin{proof}
	Suppose that $k > 1$ and consider a symplectic sphere $C$ in the class $u - E_1 - E_2$. Then 
	\[
		\langle [\omega_1], [C] \rangle = 0
	\]
	so that $C$ vanishes, i.e., the symplectic blow-down occurs on $Z_{\max}$ and this contradicts that $M_{1 - \epsilon} \cong Z_{\max}$.
\end{proof}

\begin{theorem}\label{theorem_Sec6_3_2}
		Let $(M,\omega)$ be a six-dimensional closed monotone semifree Hamiltonian $S^1$-manifold such that $\mathrm{Crit} H = \{ 1,-1,-3\}$. 
		Then there is a unique possible topological fixed point data given by 
			\begin{table}[H]
				\begin{tabular}{|c|c|c|c|c|}
					\hline
					   &  $(M_0, [\omega_0])$ & $Z_{-3}$ & $Z_{-1}$ & $Z_1$ \\ \hline \hline
					   {\bf (III-2)} & $(\p^2, 3u)$ & {\em pt} & {\em pt} & $\p^2 \# \overline{\p^2}$ \\ \hline
				\end{tabular}
			\end{table}
			\noindent
		Moreover, we have $b_2(M) = 2$ and 
		$
			\langle c_1(TM)^3, [M] \rangle = 56.
		$
\end{theorem}

\begin{proof}
	By Lemma \ref{lemma_3_2}, we only need to prove that 
		$
			\langle c_1(TM)^3, [M] \rangle = 56.
		$
	Using Theorem \ref{theorem_localization}, we obtain
	\[
		\begin{array}{ccl}\vs{0.1cm}
			\ds \int_M c_1^{S^1}(TM)^3 & = & \ds  \sum_{Z \subset M^{S^1}} \int_Z \frac{\left(c_1^{S^1}(TM)|_Z\right)^3}{e_Z^{S^1}} \\ \vs{0.2cm}
								& = & \ds  \frac{(3\lambda)^3}{\lambda^3} + \frac{\lambda^3}{-\lambda^3} + \int_{Z_{\max}} \frac{\left((3u-E_1) + (-\lambda - e(P^-_1)\right)^3) }{-\lambda - e(P^-_{1})} \quad \quad (e(P^-_1) = -u + E_1) \\ \vs{0.1cm}
								& = & 26 + \ds \int_{Z_{\max}} \frac{(4u - 2E_1 - \lambda)^3}{u - E_1 - \lambda} \\ \vs{0.1cm}
								& = & 26 + \ds \int_{Z_{\max}} \frac{(4u - 2E_1 - \lambda)^3 \cdot ((u-E_1)^2 + (u-E_1) \lambda + \lambda^2)}{-\lambda^3} \\ \vs{0.1cm}
								& = & 26 + 30 = 56. 
		\end{array}
	\]
\end{proof}

\begin{example}[Fano variety of type {\bf (III-2)}]\cite[No. 35 in the list in Section 12.3]{IP}\label{example_Sec6_3_2}
	Let $M$ be the one-point blow-up of $\p^3$ equipped with the monotone K\"{a}hler form $\omega$ with $[\omega ] = c_1(TM)$. 
	(Following Mori-Mukai's notation, we give a special name on $M$ by $V_7$. See \cite{MM}.)
	If we consider a toric structure on $M$ regarding $M$ as the toric blow-up of $\p^3$, the corresponding moment polytope 
	is given in Figure \ref{figure_3_2}.
	
		\begin{figure}[H]
			\scalebox{0.8}{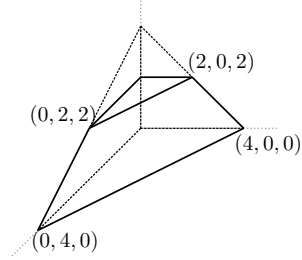} \vs{0.3cm}
			\caption{\label{figure_3_2} Moment polytope for the blow-up of $\p^3$.}
		\end{figure}
		
	\noindent Let $S^1$ be the subgroup of $T^3$ generated by $\xi = (0, -1, 0)$. Then, one can easily see that the action is semifree and the fixed point set of the $S^1$-action consists of 
	\[
		Z_{\min} = \mu^{-1}(0,4,0) = \mathrm{pt}, \quad Z_{-1} = \mu^{-1}(0,2,2) = \mathrm{pt}, \quad\text{and} \quad Z_{\max} = \mu^{-1}(\Delta) \cong \p^2
	\]
	where $\Delta$ is the trapezoid whose vertices are $(4,0,0)$, $(0,0,0)$, $(0,0,2)$, and $(2,0,2)$. So, the $S^1$-action on $M$, together with the balanced moment map
	$\langle \mu, \xi \rangle + 1$, has the same topological fixed point data as in Theorem \ref{theorem_Sec6_3_2}.
\end{example}
\vs{0.3cm}

\subsubsection{${\mathrm{Crit} \mathring{H}} = \{0\}$}
~\\

In this case, we have $M_0 \cong \p^2 \cong Z_{\max}$. Let $\mathrm{PD}(Z_0) = au \in H^2(M_0 ; \Z)$ for some integer $a > 0$. 

\begin{lemma}\label{lemma_3_3}
	$Z_0$ is connected. Also, we have $a=1,2,$ or $3$. 
\end{lemma}

\begin{proof}
	Suppose not. If $Z_0^1$ and $Z_0^2$ are disjoint components in $Z_0$,  then $[Z_0^1] = a_1u$ and $[Z_0^2] = a_2u$ for some positive integers $a_1, a_2 \in \Z$. 
	Then $[Z_0^1] \cdot [Z_0^2] = a_1 a_2 \neq 0$ which contradicts that $Z_0^1$ and $Z_0^2$ are disjoint.
	
	For the second statement, using the Duistermaat-Heckman theorem \ref{theorem_DH}. it follows that
	\[
		[\omega_1] = (4-a)u. 
	\]
	Since $4-a>0$, the only possible values of $a$ are $1,2,$ and$3$. 
\end{proof}

\begin{theorem}\label{theorem_3_3}
		Let $(M,\omega)$ be a six-dimensional closed monotone semifree Hamiltonian $S^1$-manifold such that $\mathrm{Crit} H = \{ 1,0,-3\}$. 
		Then the  topological fixed point data is one of the followings :
			\begin{table}[H]
				\begin{tabular}{|c|c|c|c|c|}
					\hline
					   &  $(M_0, [\omega_0])$ & $Z_{-3}$ & $Z_0$ & $Z_1$ \\ \hline \hline
					   {\bf (III-3.1)} & $(\p^2, 3u)$ & {\em pt} & $Z_0 \cong S^2, [Z_0] = u$ & $\p^2$ \\ \hline
					   {\bf (III-3.2)} & $(\p^2, 3u)$ & {\em pt} & $Z_0 \cong S^2, [Z_0] = 2u$ & $\p^2$ \\ \hline
					   {\bf (III-3.3)} & $(\p^2, 3u)$ & {\em pt} & $Z_0 \cong T^2, [Z_0] = 3u$ & $\p^2$ \\ \hline
				\end{tabular}
			\end{table}
			\noindent
		In any case, we have $b_2(M) = 2$ and 
		\[
			\langle c_1(TM)^3, [M] \rangle = 
			\begin{cases}
				54 & \text{\bf (III-3.1)} \\ 
				46 & \text{\bf (III-3.2)} \\
				40 & \text{\bf (III-3.3)}
			\end{cases}
		\]	
\end{theorem}

\begin{proof}
	Since $Z_0$ is connected by Lemma \ref{lemma_3_3}, we apply the adjunction formula \eqref{equation_adjunction} to $Z_0$ so that 
	\[
		 3a = \langle c_1(M_0), [Z_0] \rangle = [Z_0] \cdot [Z_0] + 2 - 2g = a^2 + 2 - 2g, \quad g =~ \text{genus of $Z_0$}.
	\]
	So, we have $g=0$ if $a = 1$ or $2$ and $g = 1$ if $a=3$. This proves the first statement of Theorem \ref{theorem_3_3}. The second assertion ``$b_2(M) = 2$'' easily follows from the fact that 
	the moment map $H$ is perfect Morse-Bott so that the Poincar\'{e} polynomial of $M$ is given by 
	\[
		\begin{array}{ccl}\vs{0.1cm}
			P_t(M) & = & t^0 + t^2(P_t(Z_0) + P_t(Z_{\max})) \\ \vs{0.1cm}
				& = & 1 + t^2(1 + (2-2g)t + t^2 + 1 + t^2 + t^4) \\ \vs{0.1cm}
				& = & 1 + 2t^2 + (2-2g)t^3 + 2t^4 + t^6, \quad \quad (Z_0, Z_{\max} : \text{index~two})
		\end{array}
	\]
	For the final assertion (for Chern numbers), we apply the localization theorem \ref{theorem_localization} : 
	\[
		\begin{array}{ccl}\vs{0.1cm}
			\ds \int_M c_1^{S^1}(TM)^3 & = & \ds  \sum_{Z \subset M^{S^1}} \int_Z \frac{\left(c_1^{S^1}(TM)|_Z\right)^3}{e_Z^{S^1}} \\ \vs{0.2cm}
								& = & \ds  \frac{(3\lambda)^3}{\lambda^3} 
								+ \int_{Z_0} \frac{\left(\mathrm{Vol}(Z_0) q\right)^3) }{(\lambda + b^+q)(-\lambda + b^-q)} 
								+ \int_{Z_{\max}} \frac{\left(3u + (-\lambda - e(P^-_1)\right)^3}{-\lambda - e(P^-_{1})} \\ \vs{0.1cm}
								& = & 27 + \ds \int_{Z_{\max}} \frac{((4-a)u - \lambda)^3}{- \lambda - (a-1)u} \quad \quad (e(P_1^-) = -u + au = (a-1)u)\\ \vs{0.1cm}
								& = & 27 + a^2 - 11a + 37 \\ \vs{0.1cm}
								& = & \begin{cases}
									54 & \text{if $a = 1$} \\
									46 & \text{if $a = 2$} \\
									40 & \text{if $a = 3$} 
								\end{cases}
		\end{array}
	\]
	and this completes the proof.
\end{proof}

\begin{example}[Fano variety of type {\bf (III-3)}]\label{example_Sec6_3_3}
	For each {\bf (III-1.a)} (${\bf a}=1,2,3$), we present a Fano variety $X_{\text{\bf a}}$ equipped with a semifree Hamiltonian $\C^*$-action whose topological fixed point data coincides with {\bf (III-1.a)}
	in Theorem \ref{theorem_3_3}.
	(We will see that each manifold $X_{\text{\bf a}}$ can be obtained by {\em an $S^1$-equivariant blow-up} of $\p^3$ along some smooth curve $Q_{\text{\bf a}}$.)
	
	Following Example \ref{example_Sec6_3_1}, we consider $\p^3$ with the Fubini-Study form $\omega$ with $[\omega] = c_1(T\p^3) = 4u$. 
	Also, we consider the $S^1$-action induced from the standard $T^3$-action (given by \eqref{equation_standard_action})
	generated by $\xi = (1,1,1)$ so that the fixed point set of the $S^1$-action is given by 
	\[
		\mu^{-1}(0,0,0) = \mathrm{pt}, \quad \mu^{-1}(\Delta) = \{ [x,y,z,0] \} \subset \p^3 \}.	
	\]
	Let $Q_{\text{\bf a}}$ be the smooth curve in $\mu^{-1}(\Delta)$ defined by $\{ [x,y,z,0] ~|~ x^{\text{\bf a}} + y^{\text{\bf a}} + z^{\text{\bf a}} = 0 \}$. 
	Note that the adjuction formula \eqref{equation_adjunction} implies that 
	\[
		Q_{\text{\bf a}} \cong \begin{cases}
			\p^1 & \text{\bf a} = 1, 2 \\
			T^2 &  \text{\bf a} = 3. \\
		\end{cases}
	\]
	If we perform an $S^1$-equivariant symlectic blow up $\p^3$ along $Q_{\text {\bf a}}$, then we obtain a symplectic manifold $M_{\text{\bf a}}$ with an induced Hamiltonian $S^1$-action.
	It is worth mentioning that 
	\begin{itemize}
		\item $M_{\text{\bf a}}$ is Fano as \cite[No. 33,30,28 in the list in Section 12.3]{IP}, 
		\item the induced action is semifree in the following reason : for a fixed point $z_0 \in Q_{\text{\bf a}}$, let $\mcal{U}$ be an $S^1$-equivariant open neighborhood 
		of $z_0$ with a local complex coordinates $(z_1, z_2, w)$ such that 
			\begin{itemize}
				\item $(z_1, z_2, 0)$ is a local coordinate system of $\p^2 \cong \mu^{-1}(\Delta)$ near $z_0 = (0,0,0)$, 
				\item $(z_1,0,0)$ is a local coordinate system of $Q_{\text{\bf a}}$ near $z_0$
			\end{itemize} 
		where the action can be expressed as 
		\[
			t(z_1, z_2, w) = (z_1, z_2, t^{-1}w), \quad t \in S^1.
		\]
		Then, an $S^1$-equivariant blow-up of $M_{\text{\bf a}}$ along $Q_{\text{\bf a}}$ is locally described as a blow-up 
		of $\mcal{U}$ along a submanifold $\{(z_1, 0,0) \} \cong \C \subset \mcal{U}$ : 
		\[
			\widetilde{\mcal{U}} = \{ \left(z_1, \left( [z_2, w], \lambda z_2, \lambda w \right) \right) \in \C \times (\p^1 \times \C^2)  ~|~ \lambda \in \C  \}
		\]
		where the induced $S^1$-action is given by 
		\[
			t \cdot (z_1, ([z_2, w], \lambda z_2, \lambda w)) =  (z_1, ([z_2, t^{-1} w], \lambda z_2, t^{-1} \lambda w)), \quad t \in S^1.
		\]
		It can be easily verified that the $S^1$-action on $\widetilde{\mcal{U}}$ is semifree (since there is no point having a finite non-trivial stabilizer). Moreover, there are two fixed components
		\[
			\{(z_1, ([1,0], \lambda, 0)) \} \cong \C^2 \quad \text{and} \quad \{(z_1, ([0,1], 0, 0) \} \cong \C
		\]
		where the first one corresponds to an open subset of $Z_{\max}$ ($= Z_1$) and the latter corresponds to an open subset of $Z_0$.
	\end{itemize}
	
		\begin{figure}[H]
			\scalebox{1}{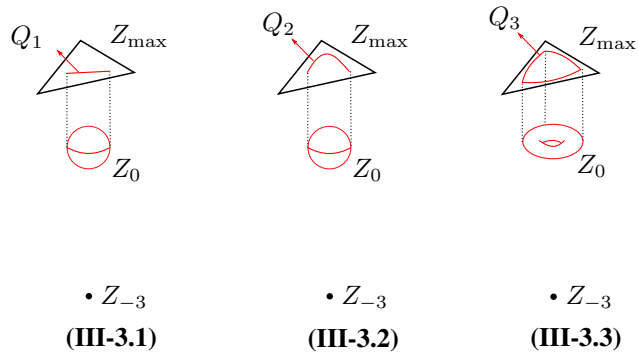} \vs{0.3cm}
			\caption{\label{figure_3_3} $S^1$-equivariant blow-ups of $\p^3$ along $Q_{\text{\bf a}}$ for ${\text{\bf a}} = 1,2,3$.}
		\end{figure}

	Note that we may also choose $Q = \{ [x,y,0,0] \}$ as a $T^3$-invariant rational curve of degree one in $\mu^{-1}(\Delta)$. Then the toric blow-up of $\p^3$ along 
	$Q$ inherits a toric structure and the induced $S^1$-action also has a topological fixed point data that coincides with {\bf (III-3.1)}. See Figure \ref{figure_3_3_toric}.

		\begin{figure}[H]
			\scalebox{1}{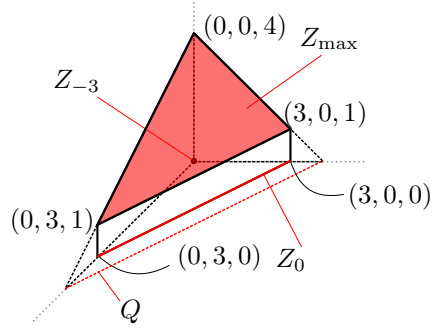} \vs{0.3cm}
			\caption{\label{figure_3_3_toric} Toric blow-up of $\p^3$ along $Q$.}
		\end{figure}	
	
\end{example}

\subsubsection{${\mathrm{Crit} \mathring{H}} = \{ -1,0\}$}
~\\

Let $|Z_{-1}| = k \in \Z_+$ be the number of fixed points of index two and 
\[
	\mathrm{PD}(Z_0) = au + b_1E_1 + \cdots + b_kE_k, \quad a, b_1, \cdots, b_k \in \Z.
\]
From Lemma \ref{lemma_Euler_class}, we obtain
\[
	e(P_1^-) = -u + \sum_{i=1}^k E_i + \mathrm{PD}(Z_0) = (a-1)u + \sum_{i=1}^k (b_i + 1)E_i
\]
Also, the Duistermaat-Heckman theorem \ref{theorem_DH} implies that 
\begin{equation}\label{equation_3_4_omega}
	[\omega_1] = (3u - \sum_{i=1}^k E_i) - e(P_1^-) = (4-a)u - \sum_{i=1}^k (2 + b_i) E_i
\end{equation}

\begin{lemma}\label{lemma_3_4_coef}
	The following inequalities hold : 
	\[
		a < 4, \quad b_i > -2, \quad a + b_i + b_j < 0, \quad (4-a)^2 - \sum_{i=1}^k (2 + b_i)^2 > 0, \quad 3a + \sum_{i=1}^k b_i > 0 
	\]
	for  $i,j = 1,\cdots, k$ and $i \neq j$.
\end{lemma}

\begin{proof}	
	By the Duistermaat-Heckman theorem \ref{theorem_DH}, we obtain 
	\[
		[\omega_t] = (3 + (1-a)t ) u - \sum_{k=1}^k (1 + (b_i + 1)t) E_i, \quad 0 \leq t \leq 1. 
	\] 
	Note that $3+(1-a)t$ should never vanish on $[0,1]$ since $\langle [\omega_t]^2, [M_0] \rangle > 0$ for every $t \in [0,1]$. Thus we get $a > 4$. For the second and third inequalities, 
	we consider exceptional classes $E_i$ or $u - E_i - E_j$ with $i \neq j$. As the ``(-1)-curve theorem'' by Li and Liu \cite[Theorem A]{LL} guarantees the existence of a symplectic sphere
	representing $E_i$ or $u - E_i - E_j$, the symplectic volume of each class should be positive, that is, 
	\[
		\langle [\omega_t], E_i \rangle = 2 + b_i > 0 \quad \text{and} \quad \langle [\omega_t], u - E_i - E_j \rangle = -a - b_i - b_j > 0.
	\]
	The last two inequalities immediately follow from the fact that $\langle [\omega_1]^2, [M_1] \rangle > 0$ and $\langle [\omega_0], [Z_0] \rangle > 0$.
\end{proof}

\begin{lemma}\label{lemma_3_4_a}
	We have $a \geq 0$. Moreover if $k > 1$, then $a \leq 1.$
\end{lemma}
	
\begin{proof}
	Suppose that $a < 0$. Then there is a connected component, say $Z_0^1$, of $Z_0$ such that the coefficient, say $a_1$, of $u$ in $\mathrm{PD}(Z_0^1)$ is negative. 
	If $[Z_0^1] \cdot [Z_0^1] < 0$, then by the adjunction formula \eqref{equation_adjunction}
	\[
		\langle c_1(TM_0), [Z_0^1] \rangle = [Z_0^1] \cdot [Z_0^1] + 2 - 2g
	\]
	implies that $\langle c_1(TM_0), [Z_0^1] \rangle = 1$, $[Z_0^1] \cdot [Z_0^1] = -1$, and $g = 0$ since $\langle c_1(TM_0), [Z_0^1] \rangle = \langle [\omega_0], [Z_0^1] \rangle$ is a positive integer.
	It means that $Z_0^1$ is an exceptional sphere so that, by Lemma \ref{lemma_list_exceptional}, $\mathrm{PD}(Z_0^1)$ cannot have a negative coefficient of $u$. So, we have 
	$[Z_0^1] \cdot [Z_0^1] \geq 0$. 
	
	Now, let $\mathrm{PD}(Z_0^1) = a_1 u + b_1^1 E_1 + \cdots + b_k^1 E_k$. 
	By the last inequality of Lemma \ref{lemma_3_4_coef}, there exists some $b_i^1 > 0$, which implies that $[Z_0^1] \cdot E_i = - b_i^1 < 0$. 
	This situation exactly fits into the case of Li's Theorem \cite[Corollary 3.10]{Li} which states that any symplectic surface with non-negative self-intersection number 
	should intersect the exceptional class $E_i$ non-negatively.
	(See also \cite[Theorem 5.1]{W}.) Consequently, $[Z_0^1] \cdot E_i = b_i$ cannot be negative and 	this leads to a contradiction. So, we have $a \geq 0$.
	
	For the second statement, it follows from Lemma \ref{lemma_3_4_coef} that
	\[
		b_i \geq -1~\text{for every $i$} \quad \Rightarrow \quad a - 2 \leq a + b_i + b_j \leq -1 \quad \Rightarrow \quad a \leq 1.
	\]
\end{proof}

\begin{lemma}\label{lemma_3_4_k}
	The only possible values of $k$ are $1$ and $2$. 
\end{lemma}

\begin{proof}
	Assume to the contrary that $k > 2$. Then, 
	\begin{itemize}
		\item $b_i \leq 0$ for every $i$ (by Lemma \ref{lemma_3_4_a} and the second and third inequalities of Lemma \ref{lemma_3_4_coef},)
		\item $a > 0$ (by the last inequality of Lemma \ref{lemma_3_4_coef} and $b_i \leq 0$.)
	\end{itemize}
	From the second part of Lemma \ref{lemma_3_4_a}, we have $a=1$. Moreover  $b_i = -1$ for every $i$ 
	(since $a+ b_i + b_j < 0$ and $-1 \leq b_i \leq 0$.) Then, by the last inequality of Lemma \ref{lemma_3_4_coef}, 
	\[
		3 - k = 3a + \sum_{i=1}^k b_i \geq 1,
	\]
	which implies that $k \leq 2$. 
\end{proof}

\begin{theorem}\label{theorem_Sec6_3_4}
		Let $(M,\omega)$ be a six-dimensional closed monotone semifree Hamiltonian $S^1$-manifold such that $\mathrm{Crit} H = \{ 1,0,-1,-3\}$. 
		Then the  topological fixed point data is one of the followings :
			\begin{table}[H]
				\begin{tabular}{|c|c|c|c|c|c|}
					\hline
					   &  $(M_0, [\omega_0])$ & $Z_{-3}$ & $Z_{-1}$ & $Z_0$ & $Z_1$ \\ \hline \hline
					   {\bf (III-4.1)} & $(\p^2 \# \overline{\p^2}, 3u - E_1)$ & {\em pt} & {\em pt} & $Z_0 \cong S^2, [Z_0] = E_1$ & $\p^2 \# \overline{\p^2}$ \\ \hline
					   {\bf (III-4.2)} & $(\p^2 \# \overline{\p^2}, 3u - E_1)$ & {\em pt} & {\em pt} & $Z_0 \cong S^2, [Z_0] = u - E_1$ & $\p^2 \# \overline{\p^2}$ \\ \hline
					   {\bf (III-4.3)} & $(\p^2 \# \overline{\p^2}, 3u - E_1)$ & {\em pt} & {\em pt} & $Z_0 \cong S^2, [Z_0] = u$ & $\p^2 \# \overline{\p^2}$ \\ \hline
					   {\bf (III-4.4)} & $(\p^2 \# \overline{\p^2}, 3u - E_1)$ & {\em pt} & {\em pt} & $Z_0 \cong S^2, [Z_0] = 2u - E_1$ & $\p^2 \# \overline{\p^2}$ \\ \hline
					   {\bf (III-4.5)} & $(\p^2 \# 2\overline{\p^2}, 3u - E_1-E_2)$ & {\em pt} & {\em 2 pts} & $Z_0 \cong S^2, [Z_0] = u - E_1 - E_2$ & $\p^2 \# 2\overline{\p^2}$ \\ \hline
				\end{tabular}
			\end{table}
			\noindent
		Also, we have 
		\[
			b_2(M) = \begin{cases}
				3 & \text{\bf (III-4.1$\sim$4)} \\ 
				4 & \text{\bf (III-4.5)} \\ 
			\end{cases}
			\quad \text{and} \quad 
			\langle c_1(TM)^3, [M] \rangle = 
			\begin{cases}
				50 & \text{\bf (III-4.1)} \\ 
				50 & \text{\bf (III-4.2)} \\
				46 & \text{\bf (III-4.3)} \\
				42 & \text{\bf (III-4.4)} \\
				46 & \text{\bf (III-4.5)} \\
			\end{cases}
		\]	
\end{theorem}

\begin{proof}
	We divide the proof into two cases : $k = 1$ and $k=2$. \vs{0.2cm}
	\begin{itemize}
		\item {\bf Case I : ${\bf k = 1}$.} Recall that Lemma \ref{lemma_3_4_coef}, together with Lemma \ref{lemma_3_4_a}, says that we have 
		\[
			0 \leq a \leq 3, \quad b_1 \geq -1, \quad (4-a)^2 - (2+b_1)^2 \geq 1, \quad 3a + b_1 \geq 1.
		\]
		Thus the list of all possible pairs $(a, b_1)$ is as below : 
		\[
			(a, b_1) = (0,1), (1,-1), (1,0), (2,-1), 
		\]	
		or equivalently, $\mathrm{PD}(Z_0) = E_1, u - E_1, u, 2u - E_1$ as listed in the table of Theorem \ref{theorem_Sec6_3_4}. 
		Moreover, $Z_0$ is connected in any case in the following reasons. \vs{0.1cm}
		\begin{itemize}
			\item If $\mathrm{PD}(Z_0) = E_1$, then $\langle [\omega_0], [Z_0] \rangle = 1$.
			\item When $\mathrm{PD}(Z_0) = u - E_1$ or $\mathrm{PD}(Z_0) = u$, suppose that $Z_0$ is disconnected. Then 
			there exists a connected component of $Z_0$, say $Z_0^1$, such that $\mathrm{PD}(Z_0^1) = pE_1$ for some $p \in \Z_+$ 
			(because of Lemma \ref{lemma_3_4_a} and the last inequality in Lemma \ref{lemma_3_4_coef}). Moreover, since $Z_0^1$ does not intersect other components, we have 
			\[
				pE_1 \cdot (\mathrm{PD}(Z_0) - pE_1) = 0 
			\]  
			which is impossible unless $p = 0$. \vs{0.1cm}
			\item When $\mathrm{PD}(Z_0) = 2u - E_1$, assume that $Z_0$ is disconnected. 
			Then we can easily see that $Z_0$ should consist of exactly two components, namely $Z_0^1$ and $Z_0^2$,
			and $\mathrm{PD}(Z_0^i) = u + p_iE_1$ for some $p_i \in \Z$ where $i=1,2$. 
			(Otherwise there is a component whose dual class is of the form $pE_1$ for some $p \in \Z_+$ and this is impossible since 
			$pE_1 \cdot (\mathrm{PD}(Z_0) - pE_1) \neq 0$.) Moreover, since $[Z_0^1] \cdot [Z_0^2] = 0$, we have $p_1p_2 = -1$. In other words, we have 
			$p_1 = 1$ and $p_2 = -1$ (by rearranging the order of $Z_0^1$ and $Z_0^2$ if necessary). However, there cannot exist a symplectic surface representing class $u + E_1$ 
			by \cite[Corollary 3.10]{Li} since it has non-negative (actually zero) 
			self-intersection number and intersects the stable class $E_1$ negatively. Therefore $Z_0$ is connected. \vs{0.1cm}
		\end{itemize}
	
		Now, we apply the adjunction formula \eqref{equation_adjunction} to each case
		\[
			\langle c_1(M_0), [Z_0] \rangle = [Z_0] \cdot [Z_0] + 2 - 2g, \quad c_1(M_0) = 3u - E_1
		\]
		where $g$ is the genus of $Z_0$. Then,\vs{0.1cm}
		\begin{itemize}
			\item If $\mathrm{PD}(Z_0) = E_1$, then $1 = -1 + 2 - 2g$. 
			\item If $\mathrm{PD}(Z_0) = u - E_1$, then $2 = 0 + 2 - 2g$. 
			\item If $\mathrm{PD}(Z_0) = u$, then $3 = 1 + 2 - 2g$.
			\item If $\mathrm{PD}(Z_0) = 2u  - E_1$, then $5 = 3 + 2 - 2g$. \vs{0.1cm}
		\end{itemize}
		In either case, we have $g=0$ and hence $Z_0 \cong S^2$. This proves the first claim of Theorem \ref{theorem_Sec6_3_4}. The claim $b_2(M) = 3$ can be obtained directly 
		by computing the Poincar\'{e} polynomial of $M$ in terms of fixed components. 
		
		It remains to compute the Chern numbers for each case. Applying the localization theorem \ref{theorem_localization}, we get
		\[
			\begin{array}{ccl}\vs{0.1cm}
				\ds \int_M c_1^{S^1}(TM)^3 & = & \ds  \sum_{Z \subset M^{S^1}} \int_Z \frac{\left(c_1^{S^1}(TM)|_Z\right)^3}{e_Z^{S^1}} \\ \vs{0.2cm}
									& = & \ds  \frac{(3\lambda)^3}{\lambda^3} + \frac{\lambda^3}{-\lambda^3} 
									+ \int_{Z_0} \frac{\left(\mathrm{Vol}(Z_0)q\right)^3 }{(\lambda + b^+q)(-\lambda + b^-q)} 
									+ \int_{Z_{\max}} \frac{\left((3u - E_1 + (-\lambda - e(P^-_1)\right)^3) }{-\lambda - e(P^-_{1})} \\ \vs{0.1cm}
									& = & 26 + \ds \int_{Z_{\max}} \frac{\left((3u - E_1 + (-\lambda - e(P^-_1)\right)^3) }{-\lambda - e(P^-_{1})} \quad \quad (e(P_1^-) = (a-1)u + (b_1 + 1)E_1)\\ \vs{0.1cm}
									& = & 26 + 24 -3(3(a-1) + (b_1+1)) + (a-1)^2 - (b_1+1)^2 \\ \vs{0.1cm}
									& = & \begin{cases}
										50 & \text{if $(a, b_1) = (0,1)$} \\
										50 & \text{if $(a, b_1) = (1,-1)$} \\
										46 & \text{if $(a, b_1) = (1,0)$} \\
										42 & \text{if $(a, b_1) = (2,-1)$} \\
									\end{cases}
			\end{array}
		\]
		\vs{0.2cm} 
		
		\item {\bf Case II : ${\bf k = 2}$.} Note that we have $0 \leq a \leq 1$ by Lemma \ref{lemma_3_4_a}. If $a = 0$, then two inequalities 
		\[
			a + b_1 + b_2 < 0 \quad \text{and} \quad 3a + b_1 + b_2 > 0
		\]
		in Lemma \ref{lemma_3_4_coef} contradict each other, that is, we have $a = 1$ (and hence $b_1 = b_2 = -1$.) Therefore, the only possible triple
		$(a,b_1,b_2)$ is $(1,-1,-1)$, or equivalently, $\mathrm{PD}(Z_0) = u -E_1 - E_2$. The symplectic area of $Z_0$ is then $1$ so that $Z_0$ is connected. 
		Also the adjunction formula \eqref{equation_adjunction} implies that $Z_0 \cong S^2$. Furthermore, we have 
		\[
			\begin{array}{ccl}\vs{0.1cm}
				\ds \int_M c_1^{S^1}(TM)^3 & = & \ds  \sum_{Z \subset M^{S^1}} \int_Z \frac{\left(c_1^{S^1}(TM)|_Z\right)^3}{e_Z^{S^1}} \\ \vs{0.2cm}
									& = & \ds  \frac{(3\lambda)^3}{\lambda^3} + \frac{2\lambda^3}{-\lambda^3} 
									+ \int_{Z_0} \frac{\left(\mathrm{Vol}(Z_0)q\right)^3 }{(\lambda + b^+q)(-\lambda + b^-q)} 
									+ \int_{Z_{\max}} \frac{\left((3u - E_1 - E_2 + (-\lambda - e(P^-_1)\right)^3) }{-\lambda - e(P^-_{1})} \\ \vs{0.1cm}
									& = & 25 + \ds \int_{Z_{\max}} \frac{\left((3u - E_1 - E_2 + (-\lambda - e(P^-_1)\right)^3) }{-\lambda - e(P^-_{1})} \quad \quad (e(P_1^-) = 0) 
									\\ \vs{0.1cm}
									& = & 25 + \ds \int_{Z_{\max}} 3(3u - E_1 - E_2)^2 = 46.
			\end{array}
		\]
		The statement $b_2(M) = 4$ follows from the perfectness of a moment map (as a Morse-Bott function). 
	\end{itemize}
	
\end{proof}

\begin{example}[Fano varieties of type {\bf (III-4)}]\label{example_Sec6_3_4}
	We follow Mori-Mukai's notation in \cite{MM}. 
	Let $V_7$ be the one-point toric blow-up of $\p^3$ so that the corresponding moment polytope (with respect to $\omega$ with $[\omega] = c_1(TV_7)$) is given in Figure \ref{figure_3_4}.
	See also Example \ref{example_Sec6_3_2} (case {\bf (III-2)}). 
	
		\begin{figure}[H]
			\scalebox{0.8}{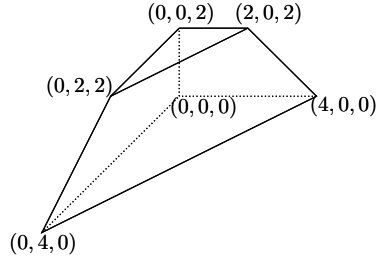} \vs{0.3cm}
			\caption{\label{figure_3_4} $V_7$ : Toric blow-up of $\p^3$ at a fixed point corresponding to the vertex $(0,0,4)$.}
		\end{figure}	
	\noindent		
	We will see that Fano varieties of type {\bf (III-4.1$\sim$3)} can be obtained as (toric) blow-ups of $V_7$ along some rational curves.
	We also construct a Fano variety of type {\bf (III-4.4)} as a blow-up of $V_7$ which is not toric.  
	Moreover, an example of a Fano variety of type {\bf (III-4.5)} can be given as the blow-up of $Y$ along two rational curves where $Y$ is the toric blow-up of $\p^3$ along a torus invariant line,  
	see Figure \ref{figure_3_4_5}.

	\begin{enumerate}
		\item {\bf Case (III-4.1) \cite[No. 29 in the list in Section 12.4]{IP}} : 
		Let $M$ be the toric blow-up of $V_7$ along the torus invariant sphere which is the preimage of the edge ${\bf e}_1$ connecting 
		$(0,2,2)$ and $(2,0,2)$ in Figure \ref{figure_3_4_1}, i.e., the blow-up of a line lying on the exceptional divisor of $V_7$.
		Then the corresponding moment polytope can be illustrated as in Figure \ref{figure_3_4_1}. Let $S^1$ be the subgroup of $T^3$ generated 
		by $\xi = (1,1,1)$. Then we can easily check (by calculating the inner products of each primitive edge vectors and $\xi$) that the induced $S^1$-action is semifree. Also, 
		with respect to the balanced moment map $H = \langle \mu, \xi \rangle - 3$, the fixed point set for the $S^1$-action consists of 
		\[
			\begin{array}{ll}
				Z_{-3} = \mu^{-1}((0,0,0)) = \mathrm{pt}, &  Z_{-1} = \mu^{-1}((0,0,2)) = \mathrm{pt}, \\
				Z_{0} = \mu^{-1}(e) \cong S^2,  & Z_1 = \mu^{-1}(\Delta)
			\end{array}
		\]
		where $e$ is the edge connecting $(0,1,2)$ and $(1,0,2)$ and $\Delta$ is the trapezoid whose vertex set is given by 
		$\{ (3,0,1), (0,3,1), (0,4,0), (4,0,0) \}$ in Figure \ref{figure_3_4_1} (on the right.)
		
		\begin{figure}[H]
			\scalebox{0.9}{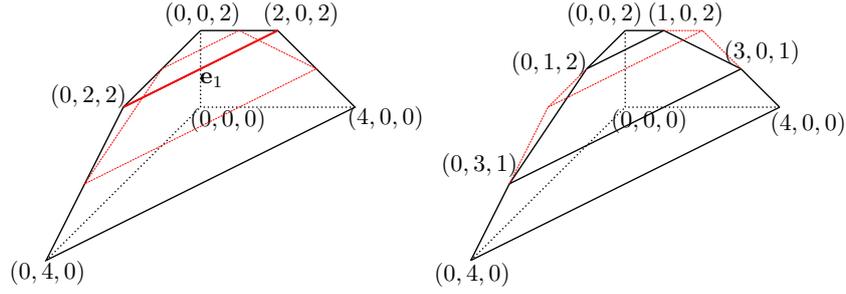} \vs{0.3cm}
			\caption{\label{figure_3_4_1} {\bf (III-4.1)} Toric blow-up of $V_7$ along the sphere corresponding to the edge ${\bf e}_1$.}
		\end{figure}			
		
		\item {\bf Case (III-4.2) \cite[No. 30 in the list in Section 12.4]{IP}} : Let $M$ be the toric blow-up of $V_7$ along the sphere corresponding to the edge ${\bf e}_2$ 
		in Figure \ref{figure_3_4_2} (on the left) where the corresponding moment polytope is described on the right. 
		(In other words, $M$ is the blow-up of $V_7$ along a line passing through the exceptional divisor of $V_7$.)
		Let $S^1$ be the subgroup of $T^3$ generated 
		by $\xi = (-1,0,0)$. The induced $S^1$-action is semifree and has the balanced moment map $H = \langle \mu, \xi \rangle +1$. Also the fixed point set for the $S^1$-action is given by
		\[
			\begin{array}{ll}
				Z_{-3} = \mu^{-1}((4,0,0)) = \mathrm{pt}, &  Z_{-1} = \mu^{-1}((2,0,2)) = \mathrm{pt}, \\
				Z_{0} = \mu^{-1}(e) \cong S^2,  & Z_1 = \mu^{-1}(\Delta)
			\end{array}
		\]
		where $e$ is the edge connecting $(1,0,2)$ and $(1,0,0)$ and $\Delta$ is the trapezoid whose vertex set is given by 
		$\{ (0,4,0), (0,1,0), (0,1,2), (0,2,2) \}$ in Figure \ref{figure_3_4_2} on the right. 

		\begin{figure}[H]
			\scalebox{0.9}{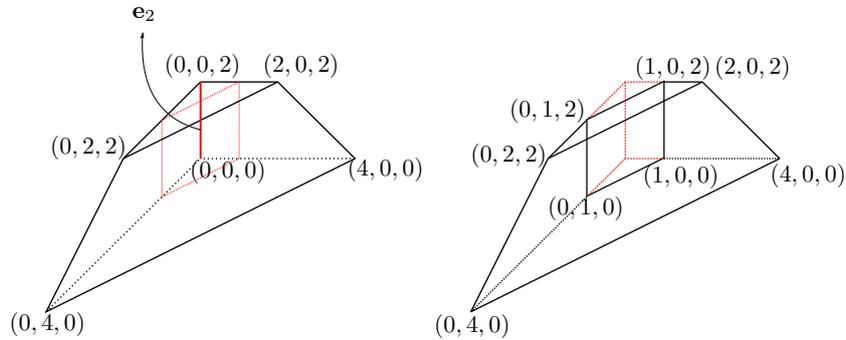} \vs{0.3cm}
			\caption{\label{figure_3_4_2} {\bf (III-4.2)} Toric blow-up of $V_7$ along the sphere corresponding to the edge ${\bf e}_2$.}
		\end{figure}			
		
		\item {\bf Case (III-4.3) \cite[No. 26 in the list in Section 12.4]{IP}} : Now, let $M$ be the toric blow-up of $V_7$ along the sphere corresponding to the edge ${\bf e}_3$ 
		in Figure \ref{figure_3_4_3}, that is, $M$ is the blow-up of $V_7$ along a line not intersecting the exceptional divisor of $V_7$.
		The corresponding moment polytope is described in Figure \ref{figure_3_4_3} on the right. 
		Let $S^1$ be the subgroup of $T^3$ generated 
		by $\xi = (1,1,1)$. Then we can easily check (by looking up Figure \ref{figure_3_4_3}) that the induced $S^1$-action is semifree and has the balanced moment map $H = \langle \mu, \xi \rangle -3$. 
		Also the fixed point set for the $S^1$-action is listed as 
		\[
			\begin{array}{ll}
				Z_{-3} = \mu^{-1}((0,0,0)) = \mathrm{pt}, &  Z_{-1} = \mu^{-1}((0,0,2)) = \mathrm{pt}, \\
				Z_{0} = \mu^{-1}(e) \cong S^2,  & Z_1 = \mu^{-1}(\Delta)
			\end{array}
		\]
		where $e$ is the edge connecting $(3,0,0)$ and $(0,3,0)$ and $\Delta$ is the trapezoid whose vertex set is given by 
		$\{ (0,3,1), (3,0,1), (2,0,2), (0,2,2) \}$ in Figure \ref{figure_3_4_3}. 
		
		\begin{figure}[H]
			\scalebox{0.9}{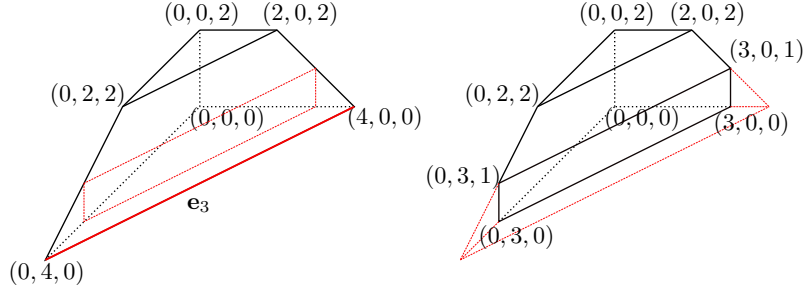} \vs{0.3cm}
			\caption{\label{figure_3_4_3} {\bf (III-4.3)} Toric blow-up of $V_7$ along the sphere corresponding to the edge ${\bf e}_3$.}
		\end{figure}			

		\item {\bf Case (III-4.4) \cite[No. 23 in the list in Section 12.4]{IP}} : 
		Consider $V_7$ as a toric variety and let $S^1$ be the subgroup of $T^3$ generated by $\xi = (-1,0,0)$. Then the induced $S^1$-action on $V_7$ is semifree and 
		the fixed point set is given by
		\[
			Z_{-3} = \mathrm{pt}, \quad Z_{-1} = \mathrm{pt}, \quad Z_{1} = \mu^{-1}(\Delta) \cong \p^2 \# \overline{\p^2}.
		\]
		Let $Q \subset Z_1 \subset V_7$ be the proper transformation of a conic in the hyperplane $\p^2 \subset \p^3$ passing through the blown-up point, i.e., the center of $V_7$.
		As a symplectic submanifold of $V_7$, one can describe $Q$ as a smoothing of two spheres in $Z_1$ representing $u$ and $u-E$ respectively. 
		Then, similar to the case of {\bf (III-3)} in Example \ref{example_Sec6_3_3}, we perform an $S^1$-equivariant blow up of $V_7$ along $Q$ and we denote the resulting manifold 
		by $M$.
		(Note that the procedure of the blowing-up construction is exactly the same as described in Example \ref{example_Sec6_3_3}.)
		As appeared in \cite[Table 3. no.23]{MM}, $M$ is a smooth Fano variety and the induced $S^1$-action on $M$ has a fixed point set which coincides with {\bf (III-4.4)}. 
		See Figure \ref{figure_3_4_non}.
		
		\begin{figure}[H]
			\scalebox{1}{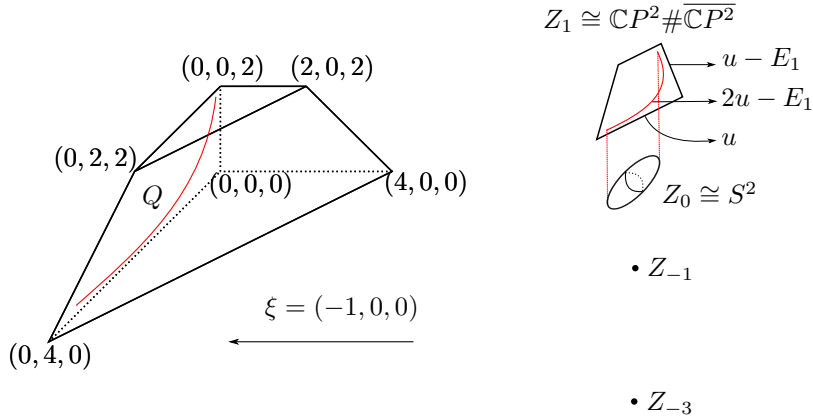} \vs{0.3cm}
			\caption{\label{figure_3_4_non} {\bf (III-4.4)} (Non-toric) blow-up of $V_7$ along a conic.}
		\end{figure}			

		\item {\bf Case (III-4.5) \cite[No. 12 in the list in Section 12.5]{IP}} : 
		In this case, we consider $Y$, the toric blow-up of $\p^3$ along a torus invariant line whose moment polytope is 
		described in Figure \ref{figure_3_4_5} on the left. Then, we let 
		$M$ be the toric blow-up of $Y$ along two disjoint spheres corresponding to the edges ${\bf C}_1$ and ${\bf C}_2$. 
		The corresponding moment polytope is given in Figure \ref{figure_3_4_5} on the right. 
		Let $S^1$ be the subgroup of $T^3$ generated 
		by $\xi = (-1,0,0)$. Then it follows that the induced $S^1$-action is semifree and has the balanced moment map $H = \langle \mu, \xi \rangle +1$. 
		Moreover, the fixed point set for the $S^1$-action is given by
		\[
			\begin{array}{ll}
				Z_{-3} = \mu^{-1}((4,0,0)) = \mathrm{pt}, &  Z_{-1} = \mu^{-1}((2,0,2)) ~\cup ~ \mu^{-1}((2,0,0)) = \mathrm{2 ~pts}, \\
				Z_{0} = \mu^{-1}(e) \cong S^2,  & Z_1 = \mu^{-1}(\Delta)
			\end{array}
		\]
		where $e$ is the edge connecting $(1,0,1)$ and $(1,0,2)$ and $\Delta$ is the five gon whose vertex set is given by 
		$\{ (0,4,0), (0,2,0), (0,3,1), (0,1,2), (0,2,2) \}$ in Figure \ref{figure_3_4_5}. 
		
		\begin{figure}[H]
			\scalebox{0.9}{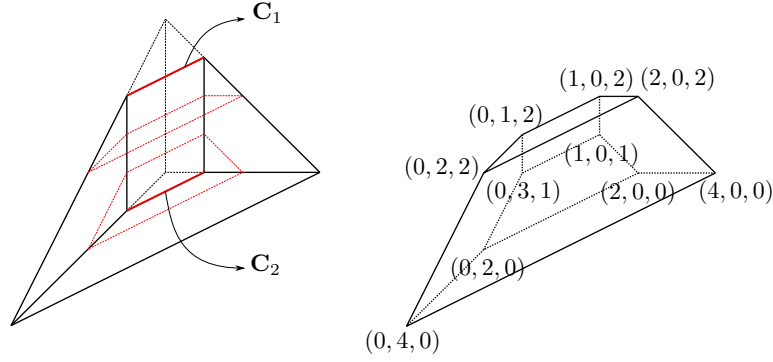} \vs{0.3cm}
			\caption{\label{figure_3_4_5} {\bf (III-4.5)} Toric blow-up of $Y$ along two spheres corresponding to the edge ${\bf C}_1$ and ${\bf C}_2$.}
		\end{figure}					
		
	\end{enumerate}		
				
	\begin{table}[h]
		\begin{tabular}{|c|c|c|c|c|c|c|c|c|c|c|}
			\hline
			    & $(M_0, [\omega_0])$ & $Z_{-3}$ & $Z_{-1}$ &  $Z_0$ & $Z_1$ & $Z_2$ & $Z_3$ & $c_1^3$ \\ \hline \hline
			    {\bf (I-1)} & $(\p^2, 3u)$ & pt &  & $Z_0 \cong S^2$, $[Z_0] = 2u$  &  & & pt &54 \\ \hline    
			   {\bf (I-2)} & $(\p^2 \# 3 \overline{\p^2}, 3u - E_1 - E_2 - E_3)$ &pt & 3 ~pts & & 3 ~pts & & pt & 48 \\ \hline    
			   
			   {\bf (I-3)} & $(\p^2 \# \overline{\p^2}, 3u - E_1)$ & pt & pt &  \makecell{ $Z_0 = Z_0^1 ~\dot \cup ~ Z_0^2$ \\
			    $Z_0^1 \cong Z_0^2 \cong S^2$ \\ $[Z_0^1] = [Z_0^2] = u - E_1$} & pt & & pt & 52\\ \hline    
			    
			   {\bf (II-3.1)} &  $(\p^2 \# \overline{\p^2}, 3u - E_1)$ & pt & pt &  \makecell{ $Z_0 = S^2$ \\
			    $[Z_0] = E_1$} & & $S^2$ & & 62\\ \hline    				    
			   {\bf (II-3.2)} &  $(\p^2 \# \overline{\p^2}, 3u - E_1)$ & pt & pt &  \makecell{ $Z_0 = S^2$ \\
			    $[Z_0] = u $} & & $S^2$ & & 54\\ \hline    				    
			   {\bf (II-3.3)} &  $(\p^2 \# \overline{\p^2}, 3u - E_1)$ & pt & pt &  \makecell{ $Z_0 = S^2$ \\
			    $[Z_0] = 2u - E_1$} & & $S^2$ & & 46 \\ \hline    			
			    
			   {\bf (II-4.1)} & $(\p^2 \# 2\overline{\p^2}, 3u - E_1 - E_2)$ & pt & 2 pts & \makecell{ $Z_0 = Z_0^1 ~\dot \cup ~ Z_0^2$ \\
				    $Z_0^1 \cong Z_0^2 \cong S^2$ \\ $[Z_0^1] = u - E_1$ \\ $[Z_0^2] = u - E_1 - E_2$} & pt & $S^2$ & & 44\\ \hline    
			   {\bf (II-4.2)} &  $(\p^2 \# 3\overline{\p^2}, 3u - E_1 - E_2 - E_3)$ & pt & 3 pts & \makecell{ $Z_0 = S^2$ \\		
		    		    $[Z_0] = u - E_2 - E_3$} & 2pts & $S^2$ & &42\\ \hline    	    
		    		    
			   {\bf (III-1)} & $(\p^2, 3u)$ & pt & & & $\p^2$ & &  & 64\\ \hline
			   {\bf (III-2)} & $(\p^2, 3u)$ & pt & pt & & $\p^2 \# \overline{\p^2}$ & & &56 \\ \hline
			   {\bf (III-3.1)} & $(\p^2, 3u)$ & pt & & $Z_0 \cong S^2, [Z_0] = u$ & $\p^2$ & & &54\\ \hline
			   {\bf (III-3.2)} & $(\p^2, 3u)$ & pt & & $Z_0 \cong S^2, [Z_0] = 2u$ & $\p^2$ & & & 46\\ \hline
			   {\bf (III-3.3)} & $(\p^2, 3u)$ & pt & & $Z_0 \cong T^2, [Z_0] = 3u$ & $\p^2$ & & & 40\\ \hline
			   
			   {\bf (III-4.1)} & $(\p^2 \# \overline{\p^2}, 3u - E_1)$ & pt & pt  & $Z_0 \cong S^2, [Z_0] = E_1$ & $\p^2 \# \overline{\p^2}$ & & & 50\\ \hline
			   {\bf (III-4.2)} & $(\p^2 \# \overline{\p^2}, 3u - E_1)$ & pt & pt  & \makecell{ $Z_0 \cong S^2$,\\ $ [Z_0] = u - E_1$} & $\p^2 \# \overline{\p^2}$ & & & 50\\ \hline
			   {\bf (III-4.3)} & $(\p^2 \# \overline{\p^2}, 3u - E_1)$ & pt & pt & $Z_0 \cong S^2, [Z_0] = u$ & $\p^2 \# \overline{\p^2}$ & & &46 \\ \hline
			   {\bf (III-4.4)} & $(\p^2 \# \overline{\p^2}, 3u - E_1)$ & pt & pt & \makecell{$Z_0 \cong S^2$, \\ $[Z_0] = 2u - E_1$} & $\p^2 \# \overline{\p^2}$ & & &42 \\ \hline
			   {\bf (III-4.5)} & $(\p^2 \# 2\overline{\p^2}, 3u - E_1-E_2)$ & pt & 2 pts & \makecell{$Z_0 \cong S^2$, \\ $[Z_0] = u - E_1 - E_2$} & $\p^2 \# 2\overline{\p^2}$ & & &46\\ \hline
		\end{tabular}
		\vs{0.3cm}
		\caption {List of topological fixed point data for $Z_{\min} = \mathrm{pt}$} \label{table_list} 
	\end{table}

\end{example}

\section{Classification of Topological fixed point data : $\dim Z_{\min} = \dim Z_{\max} = 2$}
\label{secClassificationOfTopologicalFixedPointDataDimZMinDimZMax2}

In this section, we classify all topological fixed point data in case of $\dim Z_{\min} = \dim Z_{\max} = 2$. That is, $Z_{\min}$ and $Z_{\max}$ are both two dimensional surfaces diffeomorphic to 
$S^2$ by \cite[Theorem 0.1]{Li1}.
We also give the corresponding example of a Fano variety with an explicit holomorphic $S^1$(or $\C^*$)-action for each TFD. 
We divide into four cases by the set $\mathrm{Crit} ~\mathring{H}$ of interior critical values of the balanced moment map $H$.

Up to the orientation of $M$, we may assume that 
\begin{equation}\label{equation_assumption}
	b_{\min} \leq b_{\max}
\end{equation}
where $b_{\min}$ and $b_{\max}$ denote the first Chern numbers of the normal bundles of $Z_{\min}$ and $Z_{\max}$, respectively.

\subsection{Case I : $\mathrm{Crit} ~\mathring{H} = \emptyset$}
\label{ssecCaseIMathrmCritMathringHEmptyset}
	
\begin{theorem}\label{theorem_I_1}
	Let $(M,\omega)$ be a six-dimensional closed monotone semifree Hamiltonian $S^1$-manifold with $c_1(TM) = [\omega]$.
	Suppose that $\mathrm{Crit} H = \{ 2, -2\}$. Then the only possible 
	topological fixed point data is given by 	
		\begin{table}[H]
			\begin{tabular}{|c|c|c|c|c|c|c|c|}
				\hline
				    & $(M_0, [\omega_0])$ & $e(P_{-2}^+)$ &$Z_{-2}$  & $Z_2$ & $b_2$ & $c_1^3$ \\ \hline \hline
				    {\bf (I-1)} & $(S^2 \times S^2, 2x + 2y)$ & $x-y$  &$S^2$ &  $S^2$ & $1$ & $64$\\ \hline    
			\end{tabular}
			\vs{0.5cm}			
			\caption{\label{table_I_1} Topological fixed point data for $\mathrm{Crit} H = \{-2, 2\}$}			
		\end{table}
			\vs{-0.7cm}
\end{theorem}
	
\begin{proof}
	Note that $M_{-2 + \epsilon} \cong M_0 \cong M_{2 - \epsilon}$ is an $S^2$-bundle over $S^2$, which is diffeomorphic to either $S^2 \times S^2$ or a Hirzebruch surface $E_{S^2}$.
	
	We first assume that $M_0 \cong S^2 \times S^2$. Then $b_{\min} = 2k$ for some $k \in \Z$ and $e(P_{-2}^+) = kx - y$ by Lemma \ref{lemma_volume}.
	Also we have $b_{\min} = 2k \geq -1$, i.e., $k \geq 0$, by Corollary \ref{corollary_volume}. Using the monotonicity of the reduced space (Proposition
	\ref{proposition_monotonicity_preserved_under_reduction}) and the Duistermaat-Heckman theorem \ref{theorem_DH}, 
	we obtain
	\[
		[\omega_t] = 2x + 2y - t(kx - y) = (2 - kt)x + (2 + t)y, \quad \quad t \in (-2,2).
	\]
	As $\lim_{t \rightarrow 2} \int_{M_t} [\omega_t]^2 = 0$, we get $k=1$ and so $b_{\min} = 2$. Moreover there is a natural identification 
	$H^{-1}(-2 + \epsilon) \cong H^{-1}(2 - \epsilon)$ by the gradient flow of $H$ so that we obtain $e(P_2^-) = e(P_{-2}^+)$ and 
	\[
		\langle e(P_{-2}^+)^2, [M_{-2 + \epsilon}] \rangle  = \langle e(P_{-2}^+)^2, [M_{2 - \epsilon}] \rangle = -2.
	\]	
	Therefore $b_{\max} = 2$ by Lemma \ref{lemma_volume}.
	
	Let $u$ be the positive generator of $H^2(Z_{\min};\Z) = H^2(Z_{\max};\Z)$.
	Then the first Chern number can be obtained by applying the 
	localization theorem \ref{theorem_localization} : 
	\[
		\begin{array}{ccl}\vs{0.1cm}
			\ds \int_M c_1^{S^1}(TM)^3 & = &  \ds  
							\int_{Z_{\min}} \frac{\left(c_1^{S^1}(TM)|_{Z_{\min}}\right)^3}{e_{Z_{\min}}^{S^1}} + 
							\int_{Z_{\max}} \frac{\left(c_1^{S^1}(TM)|_{Z_{\max}}\right)^3}{e_{Z_{\max}}^{S^1}} \\ \vs{0.1cm}
							& = &  \ds  
							\int_{Z_{\min}} \frac{\left( (2+b_{\min}) u + 2\lambda \right)^3}{b_{\min} u \lambda + \lambda^2} + 
							\int_{Z_{\max}} \frac{\left( (2+b_{\max}) u - 2\lambda \right)^3}{-b_{\max} u \lambda + \lambda^2} \\ \vs{0.1cm}
							& = &  \ds  
							\int_{Z_{\min}} (\lambda - 2u)(48u\lambda^2 + 8\lambda^3) + 
							\int_{Z_{\max}} (\lambda + 2u)(48u\lambda^2 - 8\lambda^3) = 32 + 32 = 64. 
		\end{array}
	\]
	
	Now we consider the case where $M_0 \cong E_{S^2}$.
	In this case, we have $b_{\mathrm{min}} = 2k + 1$ for some $k \in \Z$ by Lemma \ref{lemma_volume} and the Duistermaat-Heckman theorem \ref{theorem_DH} implies  
	\[
		[\omega_t] = (3x + 2y) - t (kx - y) = (3 - kt)x + (2 + t)y, \quad t \in (-2, 2).
	\]
	Again, since $\lim_{t \rightarrow 2} \int_{M_t} [\omega_t]^2 = 8(3-2k) - 16 = 8 - 16k = 0$, we have $k = \frac{1}{2}$ which contradicts that $k \in \Z$. Consequently, 
	$M_0$ cannot be diffeomorphic to $E_{S^2}$. This finishes the proof.
\end{proof}	
	
\begin{example}[Fano variety of type {\bf (I-1)}]\cite[17th in Section 12.2]{IP}\label{example_I_1} 
	Let $X = \p^3$ with the symplectic form $4 \omega_{\mathrm{FS}}$ (so that $c_1(TX) = [4\omega_{\mathrm{FS}}]$).
	Consider the Hamiltonian $S^1$-action on $(X, 4\omega_{\mathrm{FS}})$
	given by 
	\[
		t \cdot [z_0, z_1, z_2, z_3] = [tz_0, tz_1, z_2, z_3], \quad \quad t \in S^1
	\]
	where the balanced moment map for the action is given by
	\[
		H([z_0, z_1, z_2, z_3]) = \frac{4|z_0|^2 + 4|z_1|^2}{|z_0|^2 + |z_1|^2 + |z_2|^2 + |z_3|^2} - 2.
	\]
	Then the fixed point set (whose image is the union of the red lines in Figure \ref{figure_II_1}) is given by $\{ Z_{-2} \cong Z_2 \cong S^2 \}$ and this coincides with the one given in Theorem
	 \ref{theorem_I_1}.
	(See also \cite[Table 1-(4)]{Li2}.)
	
	\begin{figure}[H]
		\scalebox{1}{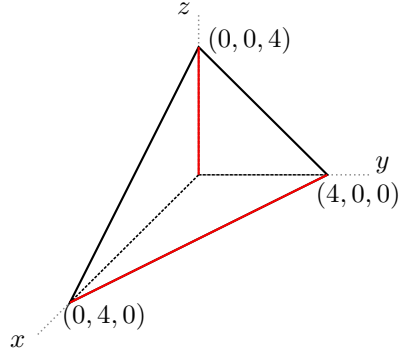}
		\caption{\label{figure_I_1} Toric moment map on $\p^3$}
	\end{figure}
\end{example}

\subsection{Case II : $\mathrm{Crit} ~\mathring{H} = \{0\}$}
\label{ssecCaseIIMathrmCritMathringH}

	In this case, we have $M_{-2 + \epsilon} \cong M_0 \cong M_{2 - \epsilon}$ by Proposition \ref{proposition_topology_reduced_space}.
	Thus we divide the proof into two cases as follows:
\begin{itemize}
    \item $M_0 \cong S^2 \times S^2.$
    \item $M_0 \cong E_{S^2}.$
\end{itemize} 
We begin with the case $M_0 \cong S^2 \times S^2.$

\begin{theorem}\label{theorem_II_1}
	Let $(M,\omega)$ be a six-dimensional closed monotone semifree Hamiltonian $S^1$-manifold with $c_1(TM) = [\omega]$. Suppose that $\mathrm{Crit} H = \{ 2, 0, -2\}$ and
	$M_0 \cong S^2 \times S^2$. Then, up to orientation of $M$, the list of all possible topological fixed point data is given by
		\begin{table}[H]
			\begin{tabular}{|c|c|c|c|c|c|c|c|c|}
				\hline
				    & $(M_0, [\omega_0])$ & $e(P_{-2}^+)$ &$Z_{-2}$  & $Z_0$ & $Z_2$ & $b_2(M)$ & $c_1^3(M)$ \\ \hline \hline
				    {\bf (II-1.1)} & $(S^2 \times S^2, 2x + 2y)$ & $-y$  &$S^2$ & $Z_0 \cong S^2, ~\mathrm{PD}(Z_0) = x+y$ & $S^2$ & $2$ &$48$ \\ \hline
				    {\bf (II-1.2)} & $(S^2 \times S^2, 2x + 2y)$ & $-y$  &$S^2$ & $Z_0 \cong S^2, ~\mathrm{PD}(Z_0) = x$ & $S^2$ & $2$ & $56$\\ \hline    
				    {\bf (II-1.3)} & $(S^2 \times S^2, 2x + 2y)$ & $-y$  &$S^2$ &
				     \makecell{ $Z_0 = Z_0^1 ~\dot \cup ~ Z_0^2$ \\
				    $Z_0^1 \cong Z_0^2 \cong S^2$ \\ $\mathrm{PD}(Z_0^1) = \mathrm{PD}(Z_0^2) = y$}   & $S^2$ & $3$ & $48$\\ \hline
			\end{tabular}
			\vs{0.5cm}			
			\caption{\label{table_II_1} Topological fixed point data for $\mathrm{Crit} H = \{-2, -0, 2\}$, $M_0 \cong S^2 \times S^2$}
		\end{table}
			\vs{-0.7cm}
\end{theorem}

\begin{proof}
	Denote by $\mathrm{PD}(Z_0) = ax + by \in H^2(M_0; \Z)$ for some $a,b \in \Z$.
	By Lemma \ref{lemma_volume}, we have $b_{\mathrm{min}} = 2k$ and $e(P_{-2}^+) = kx - y$ for some integer $k \in \Z$.
	Also the Duistermaat-Heckman theorem \ref{theorem_DH} implies that
	\[
		[\omega_2] = [\omega_0] - 2(kx - y + \mathrm{PD}(Z_0)) = 2(1-a-k)x + (4-2b)y. 
	\]	
	As $\lim_{t \rightarrow 2} \int_{M_t} [\omega_t]^2 = 0$, we see that
	\begin{enumerate}
		\item $1-a-k=0$ and $4-2b > 0$, or
		\item $b=2$ and $1-a-k > 0$
	\end{enumerate}
	where the above two strict inequalities follow from the fact that $\int_{M_t} [\omega_t]^2 > 0$ for every $0 \leq t < 2$. Moreover, we have 
	\begin{equation}\label{equation_area_vol}
		\langle c_1(TM_0), [Z_0] \rangle = \langle [\omega_0], [Z_0] \rangle = 2a + 2b > 0, \quad \quad 
		\mathrm{Vol}(Z_{-2}) = b_{\min} + 2 = 2k + 2 > 0 \hs{0.2cm} (\Leftrightarrow ~k \geq 0) 
	\end{equation}	
	by Corollary \ref{corollary_volume}. \\
	
	\noindent 
	{\bf Case (1).} ~If $a + k = 1$ and $b \leq 1$, then the integer solutions $(a,b,k)$ for \eqref{equation_area_vol} are
	\[
		(1,1,0), (1,0,0), (0,1,1).
	\]
	
	\noindent
	{\bf Case (2).} ~If $a + k \leq 0$ and $b = 2$, then the integer solutions for $(a,b,k)$ are 
	\[
		(0,2,0), (-1,2,0), (-1,2,1).
	\]
	However, we may rule out the last two solutions in {\bf Case (2)} using the adjuction formula 
	\begin{equation}
		[Z_0]\cdot [Z_0]  + \sum (2 - 2g_i) = \langle c_1(TM_0), [Z_0] \rangle
	\end{equation}
	where the sum is taken over all fixed components of $Z_0$. If $(a,b,k)$ is $(-1,2,0)$ or $(-1,2,1)$, we have 
	\[
		-4 + \sum (2 - 2g_i) = \langle 2x + 2y, [Z_0] \rangle = 2, \quad \quad \mathrm{PD}(Z_0) = -x + 2y
	\]
	which implies that there are at least three components each of which is homeomorphic to a sphere. Meanwhile, since 
	$\langle c_1(TM_0), [Z_0] \rangle = 2$ is the symplectic area of $Z_0$, there should be at most two components in $Z_0$ and this leads to a contradiction.
	Summing up, we have 
	\begin{equation}\label{equation_2_1_bminbmax}
		\begin{array}{lll}
			(a,b,k) = (1,1,0) ~(b_{\min} = 0, b_{\max} = 0), & & (a,b,k) = (1,0,0) ~(b_{\min} = 0, b_{\max} = 2) \\
			(a,b,k) = (0,1,1) ~(b_{\min} = 2, b_{\max} = 0), & & (a,b,k) = (0,2,0) ~(b_{\min} = 0, b_{\max} = 0) \\
		\end{array}
	\end{equation}
	where $b_{\min} = 2k$ and $b_{\max}$ are obtained by Lemma \ref{lemma_volume}.
	Since we only need to classify TFD's satisfying $b_{\min} \leq b_{\max}$ by \eqref{equation_assumption}, the case $(a,b,k) = (0,1,1)$ can be ruled out.
	
	Notice that the symplectic area of each component of $Z_0$ is even (since $[\omega_0] = 2x + 2y$). Applying \eqref{equation_adjunction} to each 
	solutions in \eqref{equation_2_1_bminbmax},
	we deduce that
	\begin{equation}\label{equation_II_1}
		\begin{array}{lllll}
			\text{\bf (II-1.1)} : (a,b,k) = (1,1,0) & \Rightarrow & 2 + \sum (2-2g_i) = 4 & \Rightarrow & \text{$Z_0$ has at most two components,}\\
			\text{\bf (II-1.2)} : (a,b,k) = (1,0,0) & \Rightarrow & 0 + \sum (2-2g_i) = 2 & \Rightarrow & Z_0 \cong S^2,\\
			\text{\bf (II-1.3)} : (a,b,k) = (0,2,0) & \Rightarrow & 0 + \sum (2-2g_i) = 4 & \Rightarrow & \text{$Z_0$ has exactly two components.}\\
		\end{array}
	\end{equation}
	For the last case, it is easy to check that each two components are spheres (with area $2$) whose Poincar\'{e} dual classes are both $y$. 
	
	For the first case {\bf (II-1.1)}, if $Z_0$ consists of two components, say $Z_0^1$ and $Z_0^2$, then the adjunction formula implies that
	$Z_0^1 \cong S^2$ and $Z_0^2 \cong T^2$ with 
	\[
		[Z_0^1] \cdot [Z_0^1] = 0, \quad [Z_0^2] \cdot [Z_0^2] = 2, \quad 
		[Z_0^1] \cdot [Z_0^2] = 0, \quad \mathrm{PD}(Z_0^1) + \mathrm{PD}(Z_0^2) = x+y.  
	\]
	The first and the third equalities imply that $(\mathrm{PD}(Z_0^1), \mathrm{PD}(Z_0^2)) = (ax, bx)$ 
	or $(ay, by)$ for some $a,b \in \Z$, but in either case, the second (as well as fourth) equality does not hold.
	Therefore, $Z_0$ is connected and homeomorphic to $S^2$.
	
	To calculate the Chern number for each fixed point data, we apply the localization theorem \ref{theorem_localization} : 
	\[
		\begin{array}{ccl}\vs{0.3cm}
			\ds \int_M c_1^{S^1}(TM)^3 & = &  \ds  
							\int_{Z_{\min}} \frac{\left(c_1^{S^1}(TM)|_{Z_{\min}}\right)^3}{e_{Z_{\min}}^{S^1}} + 
							\int_{Z_{\max}} \frac{\left(c_1^{S^1}(TM)|_{Z_{\max}}\right)^3}{e_{Z_{\max}}^{S^1}} + 
							\int_{Z_0} \frac{\overbrace{\left(c_1^{S^1}(TM)|_{Z_0}\right)^3}^{= 0}}{e_{Z_0}^{S^1}} \\ \vs{0.2cm}
							& = &  \ds  
							\int_{Z_{\min}} \frac{\left( (2+b_{\min}) u + 2\lambda \right)^3}{b_{\min} u\lambda + \lambda^2} + 
							\int_{Z_{\max}} \frac{\left( (2+b_{\max}) u - 2\lambda \right)^3}{-b_{\max} u\lambda + \lambda^2} \\ \vs{0.1cm}
							& = &  \ds  
							\int_{Z_{\min}} (\lambda - b_{\min}u)(12(2+b_{\min}) u\lambda^2 + 8\lambda^3) + 
							\int_{Z_{\max}} (\lambda + b_{\max}u)(12(2+ b_{\max}) u\lambda^2 - 8\lambda^3) \\ \vs{0.1cm}
							& = &\ds 24 + 4b_{\min} + 24 + 4b_{\max}.
		\end{array}
	\]
	By \eqref{equation_2_1_bminbmax}, this completes the proof. See Table \ref{table_II_1} and compare it with \eqref{equation_II_1}.
\end{proof}

\begin{remark}\label{remark_localization_surface}
	We use the following formulas frequently for calculating the Chern numbers : 
	\[
		\left(c_1^{S^1}(TM)|_{Z_0} \right)^3 = 0, \quad \int_{Z_{\min}} \frac{\left(c_1^{S^1}(TM)|_{Z_{\min}}\right)^3}{e_{Z_{\min}}^{S^1}} = 24 + 4b_{\min}, \quad 
		\int_{Z_{\max}} \frac{\left(c_1^{S^1}(TM)|_{Z_{\max}}\right)^3}{e_{Z_{\max}}^{S^1}} = 24 + 4b_{\max}.
	\]
\end{remark}
\vs{0.3cm}

\begin{example}[Fano varieties of type {\bf (II-1)}]\label{example_II_1} 
	We denote by $T^k$ a $k$-dimensional compact torus, $\frak{t}$ the Lie algebra of $T$, and $\frak{t}^*$ the dual of $\frak{t}$. 
	We provide algebraic Fano examples for each topological fixed point data given in Theorem \ref{theorem_II_1} as follows. \vs{0.3cm}

	\begin{figure}[H]
		\scalebox{1}{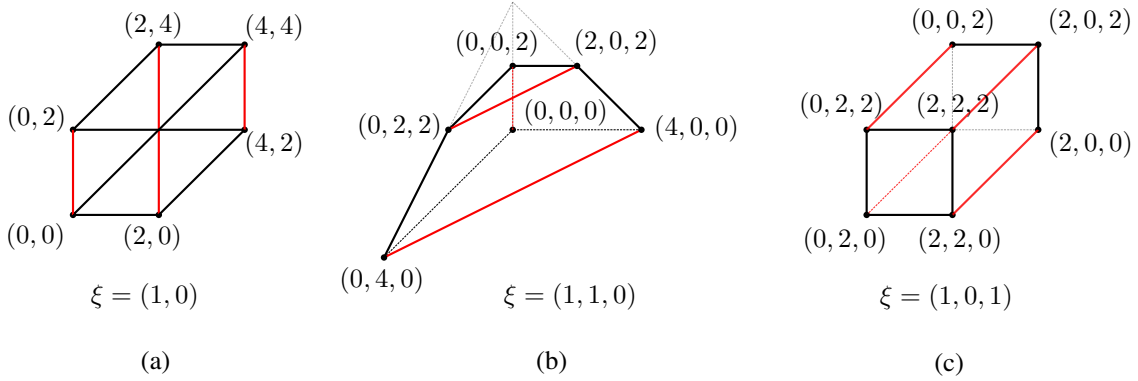}
		\caption{\label{figure_II_1} Fano varieties of type {\bf (II-1)}}
	\end{figure}          	
	
	\begin{enumerate}
		\item {\bf Case (II-1.1)} \cite[32nd in Section 12.3]{IP} : Let $W= \mcal{F}(3)$ be the complete flag variety of $\C^3$, or equivalently, a smooth divisor of bidegree $(1,1)$ in $\p^2 \times \p^2$
		(via the Pl\"{u}cker embedding).
		One can also think of $M$ as a co-adjoint orbit 
		of $U(3)$. It is well-known that $M$ admits a unique $U(3)$-invariant monotone K\"{a}hler form $\omega$ (called a {\em Kirillov-Kostant-Souriau form}) such that 
		$c_1(TW) = [\omega]$. A maximal torus $T^2$ of $U(3)$ acts on $(W,\omega)$ in a Hamiltonian fashion with a moment map 
		\[
			\mu : W \rightarrow \frak{t}^*
		\]
		such that the moment map image can be described by Figure \ref{figure_II_1} (a), where edges corresponds to $T$-invariant spheres (called 1-skeleton in \cite{GKM}).   
		If we take a circle subgroup $S^1$ generated by $\xi = (1,0) \in \frak{t} \cong \R^2$, then the action is semifree and the balanced moment map is given by 
		\[
			\mu_\xi = \langle \mu, \xi \rangle - 2
		\] 
		The fixed point set for the $S^1$-action consists of three spheres 
		corresponding to the edges (colored by red in Figure \ref{figure_II_1} (a))
		\[
			e_1 = \overline{(0,0) ~(0,2)}, \quad e_2 = \overline{(2,0) ~(2,4)}, \quad e_3 = \overline{(4,2) ~(4,4)}
		\]
		The symplectic areas of the minimum $Z_{-2} = \mu^{-1}(e_1)$ and the maximum
		$Z_2 = \mu^{-1}(e_3)$ are both equal to $2 = 2 + b_{\min} = 2 + b_{\max}$ by Corollary \ref{corollary_volume} and hence 
		$b_{\min} = b_{\max} = 0$. Thus $W_{-2 + \epsilon} \cong S^2 \times S^2$ by Lemma \ref{lemma_volume}. Therefore, the corresponding fixed point data
		should coincide with {\bf (II-1.1)} in Table \ref{table_II_1}.
		\vs{0.2cm}
		
		\item {\bf Case (II-1.2)} \cite[35th in Section 12.3]{IP} : Let $M = V_7$, the toric blow-up of $\p^3$ at a point. Then the moment polytope is given by 
		Figure \ref{figure_II_1} (b) where we denote the moment map by $\mu$. 
		If we take a circle subgroup generated by $\xi = (1,1,0) \in \frak{t}$, then 
		we can easily check that the $S^1$-action is semifree and the balanced moment map is given by 
		$\mu_\xi := \langle \mu, \xi \rangle - 2$.
		Moreover, the fixed components $Z_{-2}$, $Z_0$, and $Z_2$ are three spheres whose moment map images are the edges 
		(colored by red in Figure \ref{figure_II_1} (b))
		\[
			e_1 = \overline{(0,0,0) ~(0,0,2)}, \quad e_2 = \overline{(0,2,2) ~(2,0,2)}, \quad e_3 = \overline{(0,4,0) ~(4,0,0)}. 
		\]
		In this case, we have $Z_{-2} = \mu^{-1}(e_1)$ and $Z_2 = \mu^{-1}(e_3)$ with the symplectic areas 2 and 4, respectively.
		By Corollary \ref{corollary_volume}, we have $b_{\min} = 0$ and $b_{\max} = 2$ and so $M_{-2 + \epsilon} \cong S^2 \times S^2$
		by Lemma \ref{lemma_volume}. Also, one can easily check that the fixed point data for the $S^1$-action equals {\bf (II-1.2)} in Table \ref{table_II_1}
		(see also \eqref{equation_2_1_bminbmax}).
		
		\vs{0.2cm}
		\item {\bf Case (II-1.3)} \cite[27th in Section 12.4]{IP} : Let $M = \p^1 \times \p^1 \times \p^1$ with the monotone K\"{a}hler form $\omega = 2\omega_{\mathrm{FS}} 
		\oplus 2\omega_{\mathrm{FS}} \oplus 2\omega_{\mathrm{FS}}$ so that $c_1(TM) = [\omega]$. Then the standard Hamiltonian $T^3$-action admits a 
		moment map whose image is a  cube with side length 2, see Figure \ref{figure_II_1} (c). Take a circle subgroup $S^1$ of $T^3$ generated by $\xi = (1,0,1)$. 
		Then the induced $S^1$-action becomes semifree with the balanced moment map is given by $\mu_\xi = \langle \mu, \xi \rangle - 2$. It is easy to see that 
		there are four fixed components homeomorphic to spheres and their moment map images are 
		\[
			e_1 = \overline{(0,2,0) ~(0,0,0)}, \quad e_2 = \overline{(0,2,2) ~(0,0,2)}, \quad e_3 = \overline{(2,2,0) ~(2,0,0)}, \quad e_4 = \overline{(2,2,2) ~(2,0,2)}  
		\]
		colored by red in Figure \ref{figure_II_1} (c).
		Since $Z_{-2} = \mu^{-1}(e_1)$ and $Z_2 = \mu^{-1}(e_4)$ both have the symplectic area 2, we have $b_{\min} = b_{\max} = 0$ and this fixed point data coincides with
		{\bf (II-1.3)} in Table \ref{table_II_1}. 
	\end{enumerate} 
\end{example}

Next we deal with the case of $M_0 \cong E_{S^2}$. 

\begin{theorem}\label{theorem_II_2}
	Let $(M,\omega)$ be a six-dimensional closed monotone semifree Hamiltonian $S^1$-manifold with $c_1(TM) = [\omega]$. Suppose that $\mathrm{Crit} H = \{ 2, 0, -2\}$ 
	and $M_0 \cong E_{S^2}$. Then, up to orientation of $M$, the list of all possible topological fixed point data is given by
		\begin{table}[H]
			\begin{tabular}{|c|c|c|c|c|c|c|c|c|}
				\hline
				    & $(M_0, [\omega_0])$ & $e(P_{-2}^+)$ &$Z_{-2}$  & $Z_0$ & $Z_2$ & $b_2(M)$ & $c_1^3(M)$ \\ \hline \hline
				    {\bf (II-2.1)} & $(E_{S^2}, 3x + 2y)$ & $-x -y$  &$S^2$ & 
				    	\makecell{ $Z_0 = Z_0^1 ~\dot \cup ~ Z_0^2$ \\
					    $Z_0^1 \cong Z_0^2 \cong S^2$ \\ $\mathrm{PD}(Z_0^1) = y$, $\mathrm{PD}(Z_0^2) = x+y$}  & $S^2$ & $3$ & $48$\\ \hline    
				    {\bf (II-2.2)} & $(E_{S^2}, 3x + 2y)$ & $-x-y$  &$S^2$ & $Z_0 \cong S^2, ~\mathrm{PD}(Z_0) = 2x+2y$ & $S^2$ & $2$ &$40$ \\ \hline
			\end{tabular}
			\vs{0.5cm}			
			\caption{\label{table_II_2} Topological fixed point data for $\mathrm{Crit} H = \{-2, -0, 2\}$, $M_0 \cong E_{S^2}$}
		\end{table}
			\vs{-0.7cm}
\end{theorem}

\begin{proof}
	The idea of the proof is essentially similar to the proof of Theorem \ref{theorem_II_1}.
	
	In this case, Lemma \ref{lemma_volume} implies that $b_{\mathrm{min}} = 2k+1$ and $e(P_{-2}^+) = kx - y$ for some integer $k \in \Z$.
	If we denote by $\mathrm{PD}(Z_0) = ax + by \in H^2(M_0; \Z)$ for some $a,b \in \Z$, then it follows that 
          \[
	          \langle c_1(TM_0), [Z_0] \rangle > 0, \quad  \mathrm{Vol}(Z_{-2}) = 2k + 3 > 0 \quad \quad \Rightarrow \quad \quad 2a+b \geq 1, \quad k \geq -1.
          \]  by Corollary \ref{corollary_volume}.
	Also, by the Duistermaat-Heckman theorem \ref{theorem_DH}, we obtain
	\[
		[\omega_2] = [\omega_0] - 2(kx - y + \mathrm{PD}(Z_0)) = (3-2a-2k)x + (4-2b)y.
	\]
          Since $\lim_{t \rightarrow 2} \int_{M_t} [\omega_t]^2 = 0$, we have 
          \[
          		2(3-2a-2k)(4-2b) - (4-2b)^2 = 0 \quad \Rightarrow \quad b=2 \hs{0.3cm} \text{or} \hs{0.3cm} 1+b = 2a + 2k 
          \]
          Note that in  the latter case, $b$ becomes odd and this implies that 
          \begin{equation}\label{equation_2_2_bmax}
	          	\langle e(P_2^-)^2, [M_0] \rangle = \langle ((a+k)x + (b-1)y)^2, [M_0] \rangle = 2(a+k)(b-1) - (b-1)^2  \equiv 0 ~\mod 2
	\end{equation}
	which contradicts that $- b_{\max} = \langle e(P_2^-)^2, [M_0] \rangle$ is odd by Lemma \ref{lemma_volume}
	(since $M_{2-\epsilon} \cong M_0 \cong E_{S^2}$). 
	Consequently, we get 
	\begin{equation}\label{equation_2_2}
		b=2, \quad a \geq 0, \quad k \geq -1, \quad a + k \leq 1 ~( \Leftrightarrow ~b_{\max} + 2 = \mathrm{vol}(Z_{\max}) \geq 1). 
	\end{equation}
          Therefore, all possible solutions $(k,a,b)$ to \eqref{equation_2_2} are given by 
          \[
          		(-1,0,2), (-1,1,2), (-1, 2, 2), (0,0,2), (0,1,2),(1,0,2).
          \]          
          Applying the adjunction formula, we may rule out some solutions : if $a=0$, then $\mathrm{PD}(Z_0) = 2y$ so that 
          we have $[Z_0] \cdot [Z_0] = -4$ and $\langle c_1(TM_0), [Z_0] \rangle = 2$ and hence there are at most two connected component in $Z_0$. 
          On the other hand, the adjunction formula \eqref{equation_adjunction} implies that 
          \[
		\underbrace{[Z_0] \cdot [Z_0]}_{= ~-4} + \sum (2 - 2g_i) = \langle c_1(TM_0), [Z_0] \rangle = 2
          \]
	and therefore there should be at least three spheres, which contradicts that $Z_0$ consists of at most two connected components.
	Also, if $(k,a,b)=(0,1,2)$, then the formula \eqref{equation_2_2_bmax} induces that 
	 $b_{\min} = 1$ and $b_{\max} = -1$ (in particular $b_{\min} > b_{\max}$) and hence we may rule out this case by \eqref{equation_assumption}.
	To sum up, we have only two possible cases  : \vs{0.2cm}
	
	\noindent
	{\bf (II-2.1)} : $(k,a,b)=(-1,1,2)$. In this case, $[Z_0] \cdot [Z_0] = 0$ and $\langle c_1(TM_0), [Z_0] \rangle = 4$, $b_{\min} = -1$ and $b_{\max} = 1$. Thus the adjunction formula 
	implies that there are at least two sphere components $Z_0^1$ and $Z_0^2$ of $Z_0$ where the followings are satisfied : 
		\begin{itemize}
			\item $1 \leq \langle [\omega_0], [Z^i_0] \rangle \leq 3$. 
			\item $2 \leq \langle [\omega_0], [Z_0^1] + [Z_0^2]  \rangle \leq 4$.
			\item $[Z_0^1] \cdot [Z_0^2] =0$.
		\end{itemize}
	Let $\mathrm{PD}(Z_0^1) = p x + q y$. If $\langle [\omega_0], [Z_0^1] \rangle = 2p + q =  1$, then $[Z_0^1] \cdot [Z_0^1] = 2pq - q^2 = -1$ by the adjunction formula so that 
	we have $(p,q) = (0,1)$. Similarly, if $\langle [\omega_0], [Z_0^1] \rangle = 2p + q =  2$, then we have $2pq - q^2 = 0$ and hence 
	\[
		q = 0 ~(p = 1) \quad \text{or} \quad q = 2p ~(4p = 2). 
	\]
	So, we have $(p,q) = (1,0)$. This computation is similarly applied for $Z_0^2$.
	
	Note that if $\langle [\omega_0], [Z_0^i] \rangle \leq 2$ for every $i = 1,2$, since $[Z_0^1] \cdot [Z_0^2] = 0$, the only possible case is 
	$\langle [\omega_0], [Z_0^i] \rangle = 2$ for each $i$. However this cannot happen since $\mathrm{PD}(Z_0^1) + \mathrm{PD}(Z_0^2) = 2x \neq x + 2y$.
	Thus the only possibility is that $\langle [\omega_0], [Z_0^1] \rangle = 1$ and $\langle [\omega_0], [Z_0^2] \rangle = 3$.
	Therefore, we obtain $\mathrm{PD}(Z_0^1) = y$, $\mathrm{PD}(Z_0^2) = x+y$, and $Z_0^1 \cong Z_0^2 \cong S^2$. See Table \ref{table_II_2}: {\bf (II-2.1)}.
	\vs{0.3cm}
			
	\noindent
	{\bf (II-2.2)} : $(k,a,b)=(-1,2,2)$. In this case, we have $[Z_0] \cdot [Z_0] = 4$ and $\langle c_1(TM_0), [Z_0] \rangle = 6$, $b_{\min} = -1$ and $b_{\max} = -1$. 
	By the adjunction formula, there exists a component $Z_0^1 \cong S^2$ of $Z_0$ where we denote by $\mathrm{PD}(Z_0^1) = px + qy$.  Then we have 
	\[
		[Z_0^1] \cdot ([Z_0] - [Z_0^1]) = \langle (px + qy) \cdot ((2-p)x + (2-q)y), [M_0] \rangle = 0 \quad \Leftrightarrow \quad -2pq + 2p + q^2 = 0.
	\]
	Also, since 
	\[
		V := \mathrm{vol}(Z_0^1) = [Z_0^1] \cdot [Z_0^1] + 2 = \langle (px+qy)^2, [M_0] \rangle + 2 = 2pq - q^2 + 2, 
	\] we get $2p + 2 - V = 0$ and hence $V$ is even (with $2 \leq V \leq 6$.)
	
	If $V=2$, then $p = q= 0$ which is impossible. Also if $V=4$, then $p = 1$ and $q^2 - 2q + 2 = 0$ whose solution cannot be an integer. Therefore we have $V=6$
	so that $Z_0$ is connected and homemorphic to $S^2$. See Table \ref{table_II_2}: {\bf (II-2.2)}.	
	
	The Chern number computations in Table \ref{table_II_2} easily follow from Remark \ref{remark_localization_surface}. 
	
\end{proof}

\begin{example}[Fano varieties of type {\bf (II-2)}]\label{example_II_2} 

	We illustrate algebraic Fano varieties with holomorphic Hamiltonian torus actions having each topological fixed point data given in Theorem \ref{theorem_II_2}.\vs{0.3cm}
	
	\begin{figure}[h]
		\scalebox{1}{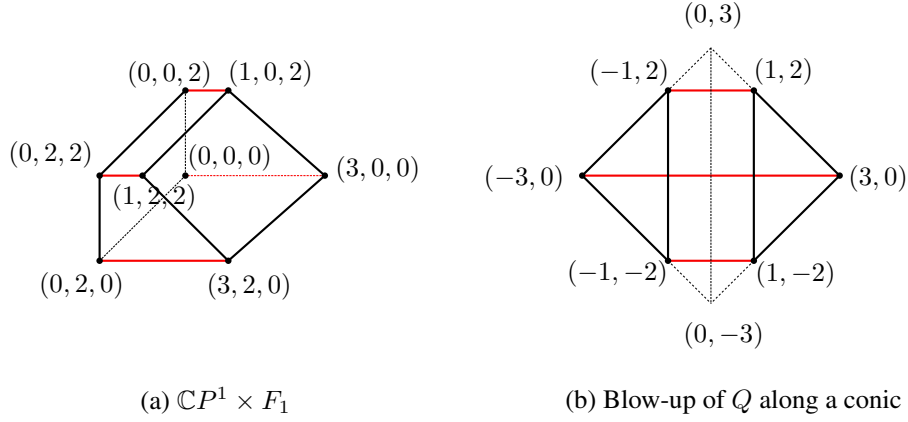}
		\caption{\label{figure_II_2} Fano varieties of type {\bf (II-2)}}
	\end{figure}          	
		
	\begin{enumerate}
		\item {\bf Case (II-2.1)} \cite[28th in Section 12.4]{IP} : Let $M = \p^1 \times F_1$ where $F_1 = \p(\mcal{O} \oplus \mcal{O}(1))$ is the Hirzebruch surface.
		Equip $M$ with the toric K\"{a}hler form $\omega$ such that $c_1(TM) = [\omega]$ so that the moment map $\mu : M \rightarrow \frak{t}^*$ 
		has the image FIgure \ref{figure_II_2} (a). 
		If we take a circle subgroup $S^1$ generated by $\xi = (0,1,-1) \in \frak{t}$, then one can check that the action is semifree and the balanced moment map is given by 
		\[
			\mu_\xi = \langle \mu, \xi \rangle.
		\] 
		The fixed point set for the $S^1$-action has four connected components each of which are all spheres and have the moment map images 
		(colored by red in Figure \ref{figure_II_2} (a))
		\[
			e_1 = \overline{(0,0,2) ~(1,0,2)}, \quad e_2 = \overline{(0,2,2) ~(1,2,2)}, \quad e_3 = \overline{(0,0,0) ~(3,0,0)}, \quad e_4 = \overline{(0,2,0) ~(3,2,0)}.
		\]
		The symplectic areas of the minimum $Z_{-2} = \mu^{-1}(e_1)$ and the maximum
		$Z_2 = \mu^{-1}(e_4)$ are 1 and 3, respectively, so that $b_{\min} = -1$ and $b_{\max} = 1$ by Corollary \ref{corollary_volume}. 
		Thus $M_{-2 + \epsilon} \cong E_{S^2}$ by Lemma \ref{lemma_volume} and the corresponding fixed point data
		coincides with {\bf (II-2.1)} in Table \ref{table_II_2}.
		\vs{0.2cm}
		
		\item {\bf Case (II-2.2)} \cite[29th in Section 12.3]{IP} : Let $M$ be a smooth quadric in $\p^4$. As a co-adjoint orbit of $SO(5)$, $M$ admits a $SO(5)$-invariant 
		monotone K\"{a}hler form $\omega$ such that $c_1(TM) = [\omega]$. With respect to the maximal torus $T^2$-action on $(M,\omega)$, we get a moment map
		$\mu : M \rightarrow \frak{t}^*$ whose image is a square with four vertices $(0, \pm 3)$, $(\pm 3, 0)$ (see Figure \ref{figure_II_2} (b)). 
		Let $C$ be the $T^2$-invariant sphere $\mu^{-1}(\overline{(0,-3) ~(0,3)})$ and define 
		\[
			\widetilde{M} := ~\text{$T^2$-equivariant (or GKM) blow-up of $M$ along $C$}
		\]
		where the {\em $T^2$-equivariant blowing up} can be done via the following two steps:\vs{0.2cm}
		\begin{itemize}
			\item Take a $T^2$-equivariant neighborhood $\mcal{U}$ of $C$, isomorphic to some $T^2$-equivariant $\C^2$-bundle over $\p^1$, and extend the $T^2$-action 
			to (any effective Hamiltonian) $T^3$-action so that we get a toric model. 
			\item Take the toric blow-up of $\mcal{U}$ along the zero section, i.e., $C$, and restrict the toric action to the original $T^2$-action. 
		\end{itemize}
		\vs{0.2cm}
		The resulting moment map image is given in Figure \ref{figure_II_2} (b). 
		
		Now, we take a circle subgroup generated by $\xi = (0,1) \in \frak{t}$. One can directly check that 
		the $S^1$-action is semifree and the balanced moment map is given by 
		$\mu_\xi := \langle \mu, \xi \rangle - 2$.
		Moreover, the fixed components $Z_{-2}$, $Z_0$, and $Z_2$ are given by 
		\[
			Z_{-2} = \mu^{-1}(e_1), \quad Z_{-2} = \mu^{-1}(e_2), \quad Z_{-2} = \mu^{-1}(e_3)
		\]
		where 
		\[
			e_1 = \overline{(-1,-2) ~(1,-2)}, \quad e_2 = \overline{(-3,0) ~(3,0)}, \quad e_3 = \overline{(-1,2) ~(1,2)}
		\]
		(colored by red in Figure \ref{figure_II_2} (b). In particular, we have $\mathrm{vol}(Z_{-2}) = \mathrm{vol}(Z_{-2}) = 1$ so that  
		$b_{\min} = b_{\max} = -1$. By Lemma \ref{lemma_volume}, we have 
		 $M_{-2 + \epsilon} \cong S^2 \times S^2$. So, the fixed point data for the $S^1$-action coincides with {\bf (II-2.2)} in Table \ref{table_II_2}. \\
		
	\end{enumerate} 
	
\end{example}

\subsection{Case III : $\mathrm{Crit} ~\mathring{H} = \{-1, 1\}$}
\label{ssecCaseIIIMathrmCritMathringH11}

	Let $m = |Z_{-1}|$ be the number of isolated fixed points of index two. By the Poincar\'{e} duality, we have $|Z_1| = m$. Applying the localization theorem to 
	$1 \in H^0_{S^1}(M)$ and $c_1^{S^1}(TM) \in H^2_{S^1}(M)$, 
we obtain
	\begin{equation}\label{equation_3_localization_0}
		\begin{array}{ccl}\vs{0.2cm}
			0 = \ds \int_M 1 & = & \ds  \int_{Z_{\min}} \frac{1}{e_{Z_{\min}}^{S^1}} + m \cdot \frac{1}{-\lambda^3} + m \cdot \frac{1}{\lambda^3} + 
							\int_{Z_{\max}} \frac{1}{e_{Z_{\max}}^{S^1}}  \\ \vs{0.2cm}
							& = &  \ds  
							\int_{Z_{\min}} \frac{1}{b_{\min} u\lambda + \lambda^2} +
							\int_{Z_{\max}} \frac{1}{-b_{\max} u\lambda + \lambda^2} \\ \vs{0.2cm}
							& = &\ds \frac{- b_{\min} + b_{\max}}{\lambda^3} 
		\end{array}
	\end{equation}
	and
	\begin{equation}\label{equation_3_localization_c1}
		\begin{array}{ccl}\vs{0.2cm}
			0 =  \ds \int_M c_1^{S^1}(TM) & = & \ds  \int_{Z_{\min}} \frac{c_1^{S^1}(TM)|_{Z_{\min}}}{e_{Z_{\min}}^{S^1}} + m \cdot \frac{\lambda}{-\lambda^3} + 
							m \cdot \frac{-\lambda}{\lambda^3} + 
							\int_{Z_{\max}} \frac{c_1^{S^1}(TM)|_{Z_{\max}}}{e_{Z_{\max}}^{S^1}}  \\ \vs{0.2cm}
							& = &  \ds  
							\int_{Z_{\min}} \frac{2\lambda + (b_{\min}+2)u}{b_{\min} u\lambda + \lambda^2} -2m \cdot  \frac{\lambda}{\lambda^3} + 
							\int_{Z_{\max}} \frac{- 2\lambda + (b_{\max} + 2)u}{-b_{\max} u\lambda + \lambda^2} \\ \vs{0.2cm}
							& = &\ds \frac{- b_{\min} - b_{\max} -2m +4}{\lambda^2}.
		\end{array}
	\end{equation}
	From \eqref{equation_3_localization_0} and \eqref{equation_3_localization_c1}, we get $b_{\max} = b_{\min}$ and $b_{\min} + m = 2$. Moreover, Corollary \ref{corollary_volume}
	implies that $b_{\min} \geq -1$ and therefore we have three possible cases : 
	\[
		(b_{\min}, m) = (1,1), (0,2), (-1, 3).
	\]
	Therefore we obtain the following.
	
	\begin{theorem}\label{theorem_III}
	Let $(M,\omega)$ be a six-dimensional closed monotone semifree Hamiltonian $S^1$-manifold with $c_1(TM) = [\omega]$. Suppose that $\mathrm{Crit} H = \{ 2, 1, -1, -2\}$. 
	Then the list of all possible topological fixed point data is given by
		\begin{table}[H]
			\begin{tabular}{|c|c|c|c|c|c|c|c|c|}
				\hline
				    & $(M_0, [\omega_0])$ & $e(P_{-2}^+)$ &$Z_{-2}$  & $Z_{-1}$ & $Z_1$ & $Z_2$ & $b_2(M)$ & $c_1^3(M)$ \\ \hline \hline
				    {\bf (III.1)} & $(E_{S^2} \# ~\overline{\p^2}, 3x + 2y - E_1)$ & $-y$  &$S^2$ & 
				    	{\em pt} & {\em pt} & $S^2$ & $2$ & $54$\\ \hline    
				    {\bf (III.2)} & $(S^2 \times S^2 \# ~2\overline{\p^2}, 2x + 2y - E_1 - E_2)$ & $-y$  &$S^2$ & 
				    	{\em 2 pts} & {\em 2 pts}  & $S^2$ & $3$ & $44$\\ \hline    
				    {\bf (III.3)} & $(E_{S^2} \# ~\overline{\p^2}, 3x + 2y - E_1)$ & $-x-y$  &$S^2$ & {\em 3 ~pts} & {\em 3 ~pts} & $S^2$ & $4$ &$34$ \\ \hline
			\end{tabular}
			\vs{0.5cm}
			\caption{\label{table_III} Topological fixed point data for $\mathrm{Crit} H = \{-2, -1,1, 2\}$.}
		\end{table}
			\vs{-0.7cm}
	\end{theorem}

	\begin{proof}
		The formula follows from Lemma \ref{lemma_volume} that $e(P_{-2}^+) = kx - y$ with $b_{\min} = 2k$ (if $M_{-2+\epsilon} \cong S^2 \times S^2$) or $2k+1$
		(if $M_{-2+\epsilon} \cong E_{S^2}$). 
		Also the Chern numbers in each cases can be easily obtained from the localization theorem \ref{theorem_localization}
		using
		Remark \ref{remark_localization_surface}. 
	\end{proof}	

\begin{example}[Fano varieties of type {\bf (III)}]\label{example_III} 
We provide examples of algebraic Fano varieties with holomorphic Hamiltonian $\C^*$-actions with topological fixed point data given in Theorem \ref{theorem_III} as follows.

	\begin{figure}[h]
		\scalebox{1}{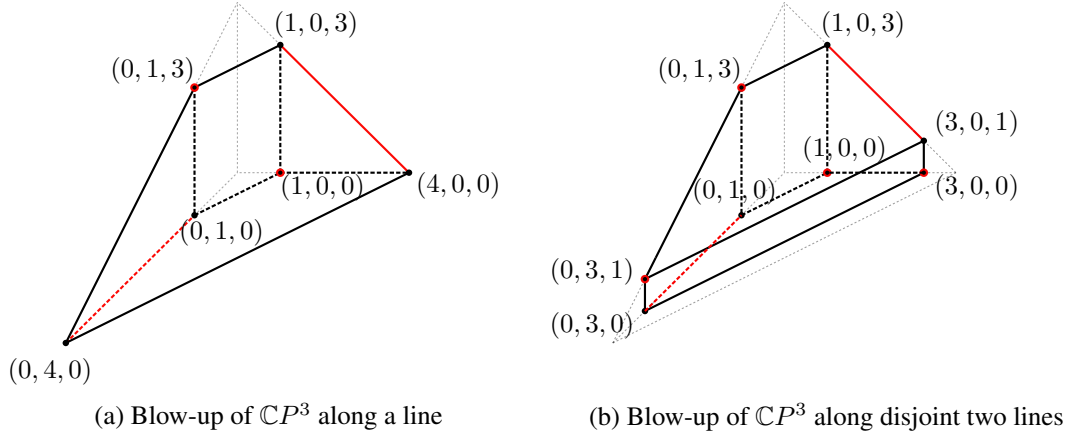}
		\caption{\label{figure_III} Fano varieties of type {\bf (III)}}
	\end{figure}          	

	\begin{enumerate}
	            \item {\bf Case (III.1)} \cite[33rd in Section 12.3]{IP} : Let $M$ be the toric blow-up of $\p^3$ along a $T^3$-invariant line. With respect to the $T^3$-invariant 
	            normalized monotone K\"{a}hler form, we get a moment map $\mu : M \rightarrow \frak{t}^*$ whose image is given by Figure \ref{figure_III} (a). If we take a circle 
	            subgroup $S^1$ generated by $\xi = (1,0,1) \in \frak{t}$, then the action is semifree with the balanced moment map $\mu_\xi = \langle \mu, \xi \rangle - 2$ 
	            and the fixed point set consists of 
	            \[
	            	Z_{-2} = \mu^{-1}(e_1), \quad Z_{-1} = \mu^{-1}(1,0,0), \quad Z_1 = \mu^{-1}(0,1,3), \quad \mu^{-1}(e_2)
	            \]
	            where $e_1 = \overline{(0,1,0) ~(0,4,0)}$ and $e_2 = \overline{(1,0,3) ~(4,0,0)}$. Note that $\mathrm{Vol}(Z_{-2}) = \mathrm{Vol}(Z_2) = 3$ and so 
	            $b_{\min} = b_{\max} = 1$ by Corollary \ref{corollary_volume}. Thus the fixed point data for the $S^1$-action coincides with 
	            Table \ref{table_III} {\bf (III.1)}. \vs{0.2cm}
	            
           	 \item {\bf Case (III.2)} \cite[25th in Section 12.4]{IP} : Let $M$ be the toric blow-up of $\p^3$ along two disjoint $T^3$-invariant lines. Then the image 
           	 of a moment map $\mu : M \rightarrow \frak{t}^*$
           	 (with respect to the normalized $T^3$-invariant K\"{a}hler form) is described as in Figure \ref{figure_III} (b). One can easily check that the circle action generated by $\xi = (1,0,1)
           	 \in \frak{t}$ is semifree and the balanced moment map is given by $\mu_\xi = \langle \mu, \xi \rangle - 2$. The fixed components are 
           	 \[
           	 	Z_{-2} = \mu^{-1}(e_1), \quad Z_{-1} = \{ (0,3,1), (1,0,0) \}, \quad Z_1 = \{ (0,1,3), (3,0,0)\}, \quad Z_2 = \mu^{-1}(e_2)
           	 \]
           	 where $e_1 = \overline{(0,3,0) ~(0,1,0)}$ and $e_2 = \overline{(1,0,3) ~(3,0,1)}$. As the symplectic volumes of $Z_{-2}$ and $Z_2$ are both 2, we have 
           	 $b_{\min} = b_{\max} = 0$ by Corollary \ref{corollary_volume} and so the fixed point data of the action is the same as Table \ref{table_III} {\bf (III.2)}.

           	 \item {\bf Case (III.3)} \cite[6th in Section 12.5]{IP} : Consider $M = \p^1 \times \p^1 \times \p^1$ equipped with the normalized monotone K\"{a}hler form $\omega$ on $M$
           	 with the standard $\omega$-compatible integrable complex structure $J$ on $M$. 
           	 Consider the standard $T^3$-action on $(M,\omega)$ with a moment map given by 
           	 \[
           	 	\mu([x_0, x_1], [y_0, y_1], [z_0, z_1]) = \left( \frac{2x_0|^2}{|x_0|^2 + |x_1|^2}, \frac{2|y_0|^2}{|y_0|^2 + |y_1|^2}, \frac{2|z_0|^2}{|z_0|^2 + |z_1|^2} \right).
           	 \]
           	 For the diagonal circle subgroup
           	 \[
           	 	S^1 = \{(t,t,t) ~|~ t \in S^1 \} \subset T^3, 
           	 \]
           	 generated by $\xi = (1,1,1) \in \frak{t}$, 
           	 the induced $S^1$-action on $(M,\omega, J)$ is semifree with the balanced moment map $\mu_\xi = \langle \mu, \xi \rangle - 3$. See Figure \ref{figure_6_2} in Example \ref{example_Sec6_1_2}.
           	 
           	 Now, we take the $S^1$-invariant diagonal sphere $D = \{ \left([z_0, z_1], [z_0, z_1], [z_0, z_1] \right) ~|~ [z_0, z_1] \in \p^1 \}$ in $M$, which is obviously a K\"{a}hler 
           	 submanifold of $(M,\omega,J)$. One can obtain an equivariant blowing-up $(\widetilde{M}, \widetilde{\omega}, \widetilde{J})$ 
           	 of $(M,\omega,J)$ along $D$ as follows. (The construction seems to be well-known to experts.) \vs{0.2cm}
           	 \begin{itemize}
           	 	\item Let $\mcal{U}$ be a sufficiently small $T^3$-invariant neighborhood of $D$ such that $\mcal{U}$ equipped with the induced K\"{a}hler structure is 
           	 	$S^1$-equivariantly isomorphic to 
           	 	some neighborhood of the zero section of $E_D := \mcal{O}(k_1) \oplus \mcal{O}(k_2)$ for some $k_1, k_2 \in \Z$ where 
           	 	\begin{itemize}
           	 		\item $E_D$ is equipped with the K\"{a}hler structure whose restriction on each fiber of $E_D$ equals the standard symplectic form on $\C \oplus \C$, \vs{0.1cm} 
           	 		\item $E_D$ admits an $S^1$-action compatible with the bundle structure such that the normal bundle $\nu_D$ of $D$ in $M$ is 
           	 		$S^1$-equivariantly isomorphic to $E_D$. \vs{0.1cm}
           	 	\end{itemize}  
			Note that each $\mcal{O}(k_i)$ has a fiberwise circle action so that $E_D$ has a fiberwise $T^2$-action. Together with the $S^1$-action given, $E_D$ becomes 
			a (non-complete) toric variety and a zero section becomes $T^3$-invariant. \vs{0.1cm}
						           	 	  
			\item Equip $\mcal{U}$ the toric structure (called a {\em local toric structure near $D$}) 
			induced by the $T^3$-action on $E_D$. Then one can obtain a toric blow-up of $\mcal{U}$ along $D$ so that we obtain a new K\"{a}hler manifold, say  
			$(\widetilde{M}, \widetilde{\omega}, \widetilde{J})$. We finally restrict the $T^3$-action to the $S^1$-subgroup of $T^3$. 
			
			\begin{figure}[h]
				\scalebox{1}{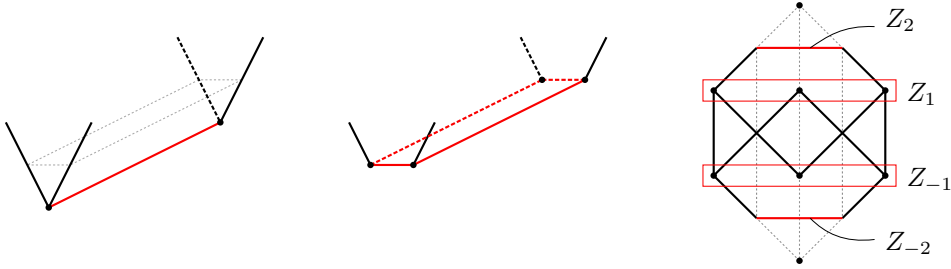}
				\caption{\label{figure_III_blowup} Blow up along an $S^1$-invariant sphere}
			\end{figure}          	
           	 \end{itemize}

	It is not hard to see that the induced $S^1$-action on $\widetilde{M}$ is semifree. Also, new fixed components which appear on $\widetilde{M}$ 
	instead of two isolated fixed points on $D$ in $M$ are two 2-spheres and hence the fixed point data coincides with Table \ref{table_III} {\bf (III.3)}
	(see Figure \ref{figure_III_blowup}). \\
           	 
          \end{enumerate}		
\end{example}

\subsection{Case IV : $\mathrm{Crit} ~\mathring{H} = \{-1, 0, 1\}$}	
\label{ssecCaseIVMathrmCritMathringH11}

	Let $m = |Z_{-1}|= |Z_1| > 0$ be the number of fixed points of index two. 

\begin{lemma}\label{lemma_number_indextwo}
	We have $m=1$ or $2$.
\end{lemma}

\begin{proof}
	Applying the localization theorem to $c_1^{S^1}(TM)$, we obtain
	\[
		\begin{array}{ccl}\vs{0.2cm}
			0 & = & \ds \int_M c_1^{S^1}(TM) \\  \vs{0.2cm}
					& = &  \ds  
							\int_{Z_{\min}} \frac{c_1^{S^1}(TM)|_{Z_{\min}}}{e_{Z_{\min}}^{S^1}} + m \cdot \frac{\lambda}{-\lambda^3} + m \cdot \frac{-\lambda}{\lambda^3} + 
							\int_{Z_0} \frac{c_1^{S^1}(TM)|_{Z_0}}{e_{Z_0}^{S^1}} + \int_{Z_{\max}} \frac{c_1^{S^1}(TM)|_{Z_{\max}}}{e_{Z_{\max}}^{S^1}}  \\ \vs{0.2cm}
							& = &  \ds  
							\int_{Z_{\min}} \frac{2\lambda + (b_{\min}+2)u}{b_{\min} u\lambda + \lambda^2} -2m \cdot  \frac{\lambda}{\lambda^3}
							+ \int_{Z_0} \frac{\overbrace{c_1(TM)|_{Z_0}}^{=~\mathrm{Vol}(Z_0)}}{(b^- - b^+) u\lambda - \lambda^2} + 
							\int_{Z_{\max}} \frac{- 2\lambda + (b_{\max} + 2)u}{-b_{\max} u\lambda + \lambda^2} \\ \vs{0.2cm}
							& = &\ds \frac{- b_{\min} - b_{\max} -2m +4 - \mathrm{Vol}(Z_0)}{\lambda^2}
		\end{array}
	\]
	where $b^+$ and $b^-$ denote the Chern numbers of the positive and negative normal bundle of $Z_0$ in $M$, respectively. 
	So, we have 
	\begin{equation}\label{equation_m}
		b_{\min} + b_{\max} + 2m + \mathrm{Vol}(Z_0) = 4.
	\end{equation} 
	Moreover, since 
	$b_{\min}, ~b_{\max} \geq -1$ by Corollary \ref{corollary_volume}, 
	we have $m \leq 2.$ 
\end{proof}

By Lemma \ref{lemma_number_indextwo}, we may divide into two cases: $m=1$ and $m=2$. 
Indeed, it follows directly from \eqref{equation_m} that there are 13 solutions for $(m, \mathrm{Vol}(Z_0), b_{\min}, b_{\max})$: 

	\begin{equation}\label{equation_8_solutions}
		m=2, ~\begin{cases}
			\underline{\mathrm{Vol}(Z_0) = 2, (b_{\min}, b_{\max}) = (-1,-1)}\\
			\underline{\mathrm{Vol}(Z_0) = 1, (b_{\min}, b_{\max}) = (-1,0)}\\
			\mathrm{Vol}(Z_0) = 1, (b_{\min}, b_{\max}) = (0,-1)\\
		\end{cases}
		m=1, ~\begin{cases}
			\underline{\mathrm{Vol}(Z_0) = 4, (b_{\min}, b_{\max}) = (-1,-1)}\\
			\underline{\mathrm{Vol}(Z_0) = 3, (b_{\min}, b_{\max}) = (-1,0)}\\
			\mathrm{Vol}(Z_0) = 3, (b_{\min}, b_{\max}) = (0,-1)\\
			\underline{\mathrm{Vol}(Z_0) = 2, (b_{\min}, b_{\max}) = (-1,1)}\\
			\underline{\mathrm{Vol}(Z_0) = 2, (b_{\min}, b_{\max}) = (0,0)}\\
			\mathrm{Vol}(Z_0) = 2, (b_{\min}, b_{\max}) = (1,-1)\\						
			\underline{\mathrm{Vol}(Z_0) = 1, (b_{\min}, b_{\max}) = (-1,2)}\\						
			\underline{\mathrm{Vol}(Z_0) = 1, (b_{\min}, b_{\max}) = (0,1)}\\						
			\mathrm{Vol}(Z_0) = 1, (b_{\min}, b_{\max}) = (1,0)\\						
			\mathrm{Vol}(Z_0) = 1, (b_{\min}, b_{\max}) = (2,-1)\\												
		\end{cases}
	\end{equation}
	By the assumption \eqref{equation_assumption}, we may rule out the case of ``$b_{\min} > b_{\max}$'', and therefore we only need to deal with 8 solutions 
	(underlined in \eqref{equation_8_solutions})
	with $b_{\min} \leq b_{\max}$
	and obtain the following. 

	\begin{theorem}\label{theorem_IV_1}
		Let $(M,\omega)$ be a six-dimensional closed monotone semifree Hamiltonian $S^1$-manifold with $c_1(TM) = [\omega]$. Suppose that $\mathrm{Crit} H = \{ 2, 1, 0, -1, -2\}$. 
		If the number of fixed points of index two equals two, up to orientation of $M$, the list of all possible topological fixed point data is given in the Table \ref{table_IV_1}.
		\begin{table}[h]
			\begin{tabular}{|c|c|c|c|c|c|c|c|c|c|}
				\hline
				    & $(M_0, [\omega_0])$ & $e(P_{-2}^+)$ &$Z_{-2}$  & $Z_{-1}$ & $Z_0$ & $Z_1$ & $Z_2$ & $b_2(M)$ & $c_1^3(M)$ \\ \hline \hline
				    {\bf (IV-1-1.1)} & \makecell{$(E_{S^2} \# ~2\overline{\p^2},$ \\$3x + 2y - E_1-E_2)$} & $-x-y$  &$S^2$ & 
				    	{\em 2 pts} &
				    		\makecell{ $Z_0 = Z_0^1 ~\dot \cup ~ Z_0^2$ \\ $Z_0^1 \cong Z_0^2 \cong S^2$ \\ 
				    		$\mathrm{PD}(Z_0^1) = x+y-E_1 - E_2$ \\ $\mathrm{PD}(Z_0^2) = x - E_1$}
					     & {\em 2 pts} & $S^2$ & $5$ & $36$\\ \hline    
				    {\bf (IV-1-1.2)} & \makecell{$(E_{S^2} \# ~2\overline{\p^2},$ \\$3x + 2y - E_1-E_2)$} & $-x-y$  &$S^2$ & 
				    	{\em 2 pts} &
				    		\makecell{ $Z_0 = Z_0^1 ~\dot \cup ~ Z_0^2$ \\ $Z_0^1 \cong Z_0^2 \cong S^2$ \\ $\mathrm{PD}(Z_0^1) = y$ \\ 
				    		$\mathrm{PD}(Z_0^2) = x+y-E_1 - E_2$}
					     & {\em 2 pts} & $S^2$ & $5$ & $36$\\ \hline    					     
				    {\bf (IV-1-1.3)} & \makecell{$(E_{S^2} \# ~2\overline{\p^2},$ \\$3x + 2y - E_1-E_2)$} & $-x-y$  &$S^2$ & 
				    	{\em 2 pts} &
				    		\makecell{ $Z_0 \cong S^2$  \\ $\mathrm{PD}(Z_0) = x+y-E_1$}
					     & {\em 2 pts} & $S^2$ & $4$ & $36$\\ \hline    
				    {\bf (IV-1-2)} & \makecell{$(E_{S^2} \# ~2\overline{\p^2},$ \\$3x + 2y - E_1-E_2)$} & $-x-y$  &$S^2$ & 
				    	{\em 2 pts} & 
				    		\makecell{ $Z_0 \cong S^2$  \\ $\mathrm{PD}(Z_0) = x - E_1$}
				    	& {\em 2 pts}  & $S^2$ & $4$ & $40$\\ \hline 
			\end{tabular}		
			\vs{0.5cm}			
			\caption{\label{table_IV_1} Topological fixed point data for $\mathrm{Crit} H = \{-2, -1,0,1, 2\}$ with $|Z_{-1}| = 2$.}
		\end{table}				   
	\end{theorem}
		
	\begin{proof}
		
		As in \eqref{equation_8_solutions}, $b_{\min} = -1$ so that $M_{-2 + \epsilon} \cong E_{S^2}$ by Lemma \ref{lemma_volume}, and therefore $M_0$ is a 
		two points blow-up of $E_{S^2}$ where we denote the dual classes of the exceptional divisors by $E_1$ and $E_2$. 
		Also, we have $e(P_{-2}^+) = kx - y = -x -y$ as 
		$b_{\min} = 2k+1 = -1$.
		
		Let $\mathrm{PD}(Z_0) = ax + by + cE_1 + dE_2 \in H^2(M_0;\Z)$ for some $a,b,c,d \in \Z$. By the Duistermaat-Heckman theorem \ref{theorem_DH}, we have 
		\[
			\begin{array}{ccl}
				[\omega_1] = [\omega_0] - e(P_0^+) & = & (3x + 2y - E_1 - E_2) - (-x - y + E_1 + E_2 + \mathrm{PD}(Z_0)) \\ 
				& = &(4-a)x + (3-b)y - (2+c)E_1 - (2+d)E_2
			\end{array}
		\] 
		where $[\omega_0] = c_1(TM_0) = 3x + 2y - E_1 - E_2$ and $e(P_0^+) = -x - y + E_1 + E_2 + \mathrm{PD}(Z_0)$ by Lemma \ref{lemma_Euler_class}.
		Observe that exactly two blow-downs occur simultaneously on $M_1$. We denote by $C_1, C_2$ the vanishing cycles so that they satisfy
		\begin{equation}\label{equation_vanishing_IV_1}
			\langle [\omega_1], C_1 \rangle = \langle [\omega_1], C_2 \rangle = 0, \quad C_1 \cdot C_2 = 0.
		\end{equation}
		By Lemma \ref{lemma_list_exceptional}, 
		the list of all possible $(\mathrm{PD}(C_1), \mathrm{PD}(C_2))$ (up to permutation on $\{E_1, E_2\}$) is given by
		\[
			 (E_1, E_2), \quad (E_1, E_3), \quad (E_3, u - E_1 - E_2), \quad (E_1, u - E_2- E_3),
			\quad (u-E_1-E_2, u-E_1- E_3), \quad (u-E_1-E_3, u-E_2- E_3)
		\]
		with the identification $u = x+y$ and $E_3 = y$. Equivalently in terms of $\{x,y,E_1, E_2\}$, all possible candidates for $(\mathrm{PD}(C_1), \mathrm{PD}(C_2))$ are
		\[
			 (E_1, E_2), \quad (E_1, y), \quad (y, x+y  - E_1 - E_2), \quad (E_1, x - E_2),
			\quad (x+y -E_1-E_2, x - E_1), \quad (x -E_1, x -E_2).
		\] 		
		We divide the proof into two cases; {\bf (IV-1-1)}: $(b_{\min}, b_{\max}) = (-1,-1)$ and {\bf (IV-1-2)}: $(b_{\min}, b_{\max}) = (-1,0)$ as listed 
		in \eqref{equation_8_solutions}.
		\vs{0.3cm}

		\noindent
		{\bf (IV-1-1) : $m = 2, \mathrm{Vol}(Z_0) = 2, (b_{\min}, b_{\max}) = (-1,-1)$} \vs{0.3cm}
		
		\noindent 
		Note that there are at most two connected components of $Z_0$. From $b_{\max} = -1$, it follows that 
		\begin{equation}\label{equation_IV_1_1}
			\mathrm{Vol}(Z_0) = 2a+b+c+d = 2, \quad \langle e(P_2^-)^2, [M_{2-\epsilon}] \rangle = 1 ~\text{so that $\langle e(P_0^+)^2, [M_0] \rangle = -1$}
		\end{equation}
		by Lemma \ref{lemma_volume} and Lemma \ref{lemma_Euler_class}. \vs{0.3cm}
			
		\noindent
		{\bf Case (1) : $(\mathrm{PD}(C_1), \mathrm{PD}(C_2)) = (E_1, E_2)$}  \vs{0.3cm}
		
		\noindent
		In this case, we have $c=d=-2$ by \eqref{equation_vanishing_IV_1}. 
		Also \eqref{equation_IV_1_1} implies that $2a + b = 6$ and
		\[
			\langle ((a-1)x + (b-1)y + (c+1)E_1 + (d+1)E_2)^2, [M_0] \rangle = 2(a-1)(b-1) - (b-1)^2 - 2 = -1.
		\]
		So, we get $a=2, ~b=2, ~c=d=-2.$, i.e., $\mathrm{PD}(Z_0) = 2x + 2y - 2E_1 - 2E_2$ and hence $Z_0 \cdot Z_0 = -4.$ 
		Because the number of connected components of $Z_0$ is at most two, there is no such manifold by the adjunction formula \eqref{equation_adjunction} : 
		\[
			[Z_0] \cdot [Z_0] + \sum (2 - 2g_i) = \langle c_1(TM_0), [Z_0] \rangle = 2
		\]
		where the sum is taken over all connected components of $Z_0$
		\vs{0.1cm}

		\noindent		
		{\bf Case (2) : $(\mathrm{PD}(C_1), \mathrm{PD}(C_2)) = (E_1, y)$}  \vs{0.1cm}
		
		\noindent 
		By \eqref{equation_vanishing_IV_1}, we obtain $c = -2$ and $a = b + 1$. Also from \eqref{equation_IV_1_1}, we get
		\[
			b =  1 \hs{0.2cm}(a = 2) \quad \text{and} \quad d = -1, 
		\]
		that is, $\mathrm{PD}(Z_0) = 2x + y - 2E_1 - E_2$ and $[Z_0] \cdot [Z_0] = -2$. The adjunction formula \eqref{equation_adjunction} says that 
		\[
			[Z_0] \cdot [Z_0] + \sum (2 - 2g_i) = \langle c_1(TM_0), [Z_0] \rangle = 2
		\]
		and this implies that $Z_0$ consists of two spheres $Z_0^1$ and $Z_0^2$ (since $Z_0$ consists at most two components) with 
		\begin{equation}\label{equation_IV_1_1_1}
			\mathrm{PD}(Z_0^1) = x + y - E_1 - E_2 \quad \mathrm{PD}(Z_0^2) = x - E_1
		\end{equation}
		up to permutation on $\{E_1, E_2\}$. (Note that this computation can be easily obtained from the fact that each $[Z_0^i]$ is an exceptional class 
		so  that one can apply Lemma \ref{lemma_list_exceptional}.) See Table \ref{table_IV_1} : {\bf (IV-1-1.1)}. 
		
		For the Chern number computation, we apply the localization theorem \ref{theorem_localization} and get
		\begin{equation}\label{equation_Chern_IV_1_1}
			\begin{array}{ccl}\vs{0.3cm}
				\ds \int_M c_1^{S^1}(TM)^3 & = &  \ds  
							\int_{Z_{\min}} \frac{\left(c_1^{S^1}(TM)|_{Z_{\min}}\right)^3}{e_{Z_{\min}}^{S^1}} + 2 \frac{\overbrace{\lambda^3}^{Z_{-1} ~\text{term}}}
							{-\lambda^3}
							+ \int_{Z_0} \frac{\overbrace{\left(c_1^{S^1}(TM)|_{Z_0}\right)^3}^{= 0}}{e_{Z_0}^{S^1}}
							 + 2 \frac{\overbrace{-\lambda^3}^{Z_1 ~\text{term}}}{\lambda^3} + \int_{Z_{\max}} \frac{\left(c_1^{S^1}(TM)|_{Z_{\max}}\right)^3}
							 {e_{Z_{\max}}^{S^1}} \\ \vs{0.2cm}
							& = &  (24 + 4b_{\min}) + (24 + 4b_{\max}) - 4 = 36
			\end{array}			
		\end{equation}
		by Remark \ref{remark_localization_surface}.
		\vs{0.3cm}
		
		\noindent		
		{\bf Case (3) : $(\mathrm{PD}(C_1), \mathrm{PD}(C_2)) = (y, x+y-E_1-E_2)$} \vs{0.1cm}

		\noindent
		From \eqref{equation_vanishing_IV_1} and \eqref{equation_IV_1_1}, we have 
		\[
			a = b+1, \quad a+c+d = 0, \quad 2a + b + c + d = 2  \quad (\Leftrightarrow 3b + c + d = 0).
		\]
		This implies that $a = 3b$ so that $b = \frac{1}{2}$ and it leads to a contradiction. Thus no such manifold exists.
		\vs{0.3cm}
		
		\noindent		
		{\bf Case (4) : $(\mathrm{PD}(C_1), \mathrm{PD}(C_2)) = (E_1, x-E_2)$} \vs{0.1cm}

		\noindent
		We similarly obtain 
		\[
			c = -2, \quad b+d = 1, \quad 2a + b + c + d = 2 \quad (\Leftrightarrow 2a + c = 1).
		\]
		Then we see that $a = \frac{3}{2}$, which is not an integer. Therefore no such manifold exists. 
		
		\vs{0.3cm}
		\noindent		
		{\bf Case (5) : $(\mathrm{PD}(C_1), \mathrm{PD}(C_2)) = (x+y-E_1 - E_2, x-E_1)$} \vs{0.1cm}

		\noindent
		In this case, we have 
		\[
			a+c+d = 0, \quad b+c = 1, \quad 2a+b+c+d=2 \quad (\Leftrightarrow 2a + d = 1), 
		\]
		and 
		\[
			\langle e(P_0^+)^2, [M_0] \rangle = 2(a-1)(b-1) - (b-1)^2 - (c+1)^2 - (d+1)^2 = -1.
		\]
		Those equations have the unique solution $(a,b,c,d) = (1,1,0,-1)$ so that $\mathrm{PD}(Z_0) = x + y - E_2$ and hence
		$[Z_0] \cdot [Z_0] = 0$. We can also check that  $Z_0$ is connected (by Lemma \ref{lemma_list_exceptional})
		and therefore $Z_0 \cong S^2$ by the adjunction formula \eqref{equation_adjunction}. 
		The Chern number can be obtained in exactly the same way as in \eqref{equation_Chern_IV_1_1}. See Table \ref{table_IV_1} : {\bf (IV-1-1.3)}. 
		(The connectedness of $Z_0$ is proved as follows : if $Z_0^1$ and $Z_0^2$ are connected components of $Z_0$, then 
			\begin{itemize}
				\item $\mathrm{Vol}(Z_0^1) = \mathrm{Vol}(Z_0^2) = 1$, and 
				\item $[Z_0^1] \cdot [Z_0^1] = -1$ and $[Z_0^2] \cdot [Z_0^2] = 1$ since 
				\[
					[Z_0^i] \cdot [Z_0^i] + 2 - 2g_i = 1\quad \text{and} \quad [Z_0^1] \cdot [Z_0^1] + [Z_0^2] \cdot [Z_0^2] = 0.
				\]
			\end{itemize}
		Then $Z_0^1 \cong S^2$ by the adjunction formula \eqref{equation_adjunction} and $\mathrm{PD}(Z_0^1)$ should be on the list in Lemma \ref{lemma_list_exceptional}.
		However, it contradicts that $\mathrm{PD}(Z_0^1) \cdot (x + y -E_2 - \mathrm{PD}(Z_0^1)) = 0$. Therefore $Z_0$ has to be connected.)			
		\vs{0.3cm}
		
		\noindent		
		{\bf Case (6) : $(\mathrm{PD}(C_1), \mathrm{PD}(C_2)) = (x - E_1, x - E_2)$} \vs{0.1cm}		
		
		\noindent 
		Again by \eqref{equation_vanishing_IV_1} and \eqref{equation_IV_1_1} , we get
		\[
			b+c = 1, \quad b+d = 1, \quad 2a + b + c + d = 2 \quad (\Leftrightarrow 2a + d = 2a + c = 1), 
		\]
		and 
		\[
			\langle e(P_0^+)^2, [M_0] \rangle = 2(a-1)(b-1) - (b-1)^2 - (c+1)^2 - (d+1)^2 = -1.
		\]
		Then we get the unique solution $(a,b,c,d) = (1,2,-1,-1)$ so that $\mathrm{PD}(Z_0) = x + 2y - E_1 - E_2$. Moreover, since $[Z_0] \cdot [Z_0] = -2$, the adjunction formula
		 \eqref{equation_adjunction}
		implies that $Z_0$ consists of two spheres $Z_0^1$ and $Z_0^2$ such that $[Z_0^1] \cdot [Z_0^1] = [Z_0^2] \cdot [Z_0^2] = -1$. 
		Applying Lemma \ref{lemma_list_exceptional}, we finally obtain
		\[
			\mathrm{PD}(Z_0^1) = y \quad \text{and} \quad \mathrm{PD}(Z_0^2) = x + y - E_1 - E_2.
		\]
		See Table \ref{table_IV_1} : {\bf (IV-1-1.2)}. \vs{0.5cm}
		
		\noindent
		{\bf (IV-1-2) : $m = 2, \mathrm{Vol}(Z_0) = 1, (b_{\min}, b_{\max}) = (-1,0)$} \vs{0.3cm}		
		
		\noindent 
		In this case, $Z_0$ is connected by the assumption $\mathrm{Vol}(Z_0) = 1$.  
		Together with the condition $b_{\max} = 0$, we have
		\begin{equation}\label{equation_IV_1_2}
			\mathrm{Vol}(Z_0) = 2a+b+c= 1, \quad \langle e(P_2^-)^2, [M_{2-\epsilon}] \rangle = 0 ~\text{so that $\langle e(P_0^+)^2, [M_0] \rangle = -2$}
		\end{equation}
		by Lemma \ref{lemma_volume}. The latter equation can be re-written as 
		\begin{equation}\label{equation_IV_1_2_Euler}
			2(a-1)(b-1) - (b-1)^2 - (c+1)^2 - (d+1)^2 = -2.
		\end{equation}
		Using \eqref{equation_vanishing_IV_1}, \eqref{equation_IV_1_2}, \eqref{equation_IV_1_2_Euler}, we analyze each cases as follows:
		\vs{0.3cm}

		\noindent
		{\bf Case (1) : $(\mathrm{PD}(C_1), \mathrm{PD}(C_2)) = (E_1, E_2).$}  \vs{0.1cm}		
		\[
			c=-2, \quad d=-2, \quad 2a+b+c+d = 1 \quad (\Leftrightarrow 2a + b = 5), \quad 2(a-1)(b-1) - (b-1)^2 = 0
		\]
		so that $(a,b,c,d) = (2,1,-2,-2)$, i.e., $\mathrm{PD}(Z_0) = 2x + y  - 2E_1 - 2E_2$. However, since $[Z_0] \cdot [Z_0] = -5$, no such manifold exists by the adjunction 
		formula \eqref{equation_adjunction}.
		
		\vs{0.3cm}
		\noindent
		{\bf Case (2) : $(\mathrm{PD}(C_1), \mathrm{PD}(C_2)) = (E_1, y)$.}  \vs{0.1cm}
		\[
			c = -2, \quad a = b+1, \quad 2a+b+c+d = 1 \quad (\Leftrightarrow 3b + d = 1), \quad \underbrace{2(a-1)(b-1) - (b-1)^2 - (d+1)^2}_{= ~-8b^2 + 12b - 5} = -1
		\]
		so that $(a,b,c,d) = (2,1,-2,-2)$, i.e., $\mathrm{PD}(Z_0) = 2x + y  - 2E_1 - 2E_2$. Again by the adjunction 
		formula \eqref{equation_adjunction}, no such manifold exists.

		\vs{0.3cm}		
		\noindent
		{\bf Case (3) : $(\mathrm{PD}(C_1), \mathrm{PD}(C_2)) = (y, x+y-E_1 - E_2)$.}  \vs{0.1cm}	
		\[
			a = b+1, \quad a+c+d = 0, \quad 2a+b+c+d = 1 \quad (\Leftrightarrow a + b = 1 ~\Leftrightarrow ~b=0, a=1), \quad (c+1)^2 + (d+1)^2 = 1
		\]
		so that $(a,b,c,d) = (1,0,-1,0)$ or $(1,0,0,-1)$ (where they are equal up to permutation on $\{E_1, E_2\}$.) 
		In this case, we have $Z_0 \cong S^2$ by \eqref{equation_adjunction}. See Table \ref{table_IV_1} : {\bf (IV-1-2)}.

		\vs{0.3cm}		
		\noindent
		{\bf Case (4) : $(\mathrm{PD}(C_1), \mathrm{PD}(C_2)) = (E_1, x - E_2)$.}  \vs{0.1cm}		
		\[
			c=-2, \quad b+d = 1, \quad 2a+b+c+d = 1 \quad (\Leftrightarrow 2a + c = 0 \Leftrightarrow a=1), \quad (b-1)^2 + (d+1)^2 = 1
		\]
		so that $(a,b,c,d) = (1,1,-2,0)$ or $(1,2,-2,-1)$. In either case, $[Z_0] \cdot [Z_0] < -1$ so that it violates the adjunction 
		formula \eqref{equation_adjunction}.
		Therefore no such manifold exists.

		\vs{0.3cm}		
		\noindent
		{\bf Case (5) : $(\mathrm{PD}(C_1), \mathrm{PD}(C_2)) = (x+y - E_1 - E_2, x - E_1)$.}  \vs{0.1cm}		
		\[
			a+c+d = 0, \quad b+c = 1, \quad \underbrace{2a+b+c+d = 1}_{\Leftrightarrow ~a+b = 1, ~2a+d=0}, \quad \underbrace{-2b(b-1) - (b-1)^2 - (2-b)^2 - (2b-1)^2}_{= ~-8b^2 + 12b - 6} = -2
		\]
		and we obtain $(a,b,c,d) = (0,1,0,0)$, i.e., $\mathrm{PD}(Z_0) = y$. However, we can check that a cycle representing $x - E_2$ vanishes on $M_1$ which leads to a contradiction.
		Therefore no such manifold exists.
				
		\vs{0.3cm}		
		\noindent
		{\bf Case (6) : $(\mathrm{PD}(C_1), \mathrm{PD}(C_2)) = (x - E_1, x - E_2)$.}  \vs{0.1cm}										
		\[
			b+c=1, \quad b+d = 1, \quad \underbrace{2a+b+c+d = 1}_{\Leftrightarrow 2a + d =0, ~2a + c = 0}, \quad \underbrace{2(a-1)(b-1) - (b-1)^2 - (c+1)^2 - (d+1)^2}_{4a(a-1) -
			4a^2 - (1-2a)^2 - 
			(1-2a)^2} = -2. 
		\]
		So, $(a,b,c,d) = (0,1,0,0)$. 
		Similar to {\bf Case (5)}, a cycle representing $x+y - E_1 -E_2$ vanishes on $M_1$, and therefore no such manifold exists. \vs{0.5cm}

	\end{proof}		
		
\begin{example}[Fano varieties of type {\bf (IV-1)}]\label{example_IV_1} 
	In this example, we illustrate algebraic Fano varieties with holomorphic Hamiltonian $S^1$-action with topological fixed point data given in Theorem \ref{theorem_IV_1}. 

	\begin{enumerate}
	            \item {\bf Case (IV-1-1.1)} \cite[2nd in Section 12.6]{IP} : Let $Y$ be the toric blow-up of $\p^3$ along two disjoint $T^3$-invariant lines where the moment map image
	            			is described in Figure \ref{figure_IV_1_1_1} (see also Figure \ref{figure_III} (b)).
	            			We take two disjoint lines $C_1$ and $C_2$ corresponding to the edges 
	            			\[
	            				e_1 = \overline{(0,1,3) ~(1,0,3)} \quad e_2 = \overline{(0,1,0) ~(1,0,0)}
	            			\]
	            			respectively. Let $M$ be the monotone toric blow-up of $Y$ along $C_1$ and $C_2$ so that the resulting moment polytope 
	            			(with respect to a moment map $\mu : M \rightarrow \R^3$)
	            			is illustrated 
	            			on the right of Figure \ref{figure_IV_1_1_1}. Now, we take the circle subgroup of $T^3$ generated by
	            			$\xi = (1,0,1)$. It is straightforward (by calculating the inner product of $\xi$ and each primitive edge vector) that the action is semifree
	            			and the balanced moment map is given by 
	            			\[
	            				\mu_\xi = \langle \mu, \xi \rangle -2.
	            			\] 
	            			Moreover, the fixed point set consists of 
	            			\begin{itemize}
	            				\item $Z_{-2} = \mu^{-1}(\overline{(0,2,0) ~(0,3,0)})$
	            				\item $Z_{-1} = \mu^{-1}(0,3,1) \cup \mu^{-1}(0,1,1)$
	            				\item $Z_{0} = \mu^{-1}(\overline{(0,2,2) ~(0,1,2)}) \cup \mu^{-1}(\overline{(1,0,1) ~(2,0,0)})$
	            				\item $Z_1 = \mu^{-1}(1,0,2) \cup \mu^{-1}(3,0,0)$
	            				\item $Z_2 = \mu^{-1}(\overline{(2,0,2) ~(3,0,1)})$
	            			\end{itemize}
					Furthremore, the symplectic areas of $Z_{-2}, Z_0^1, Z_0^2,$  and $Z_2$ are all 1 (see \eqref{equation_IV_1_1_1})
					and hence $b_{\min} = b_{\max} = -1$. Thus the fixed point data of $M$ coincides with the one in Table \ref{table_IV_1} {\bf (IV-1-1.1)}.
	            	            
			\begin{figure}[H]
				\scalebox{1}{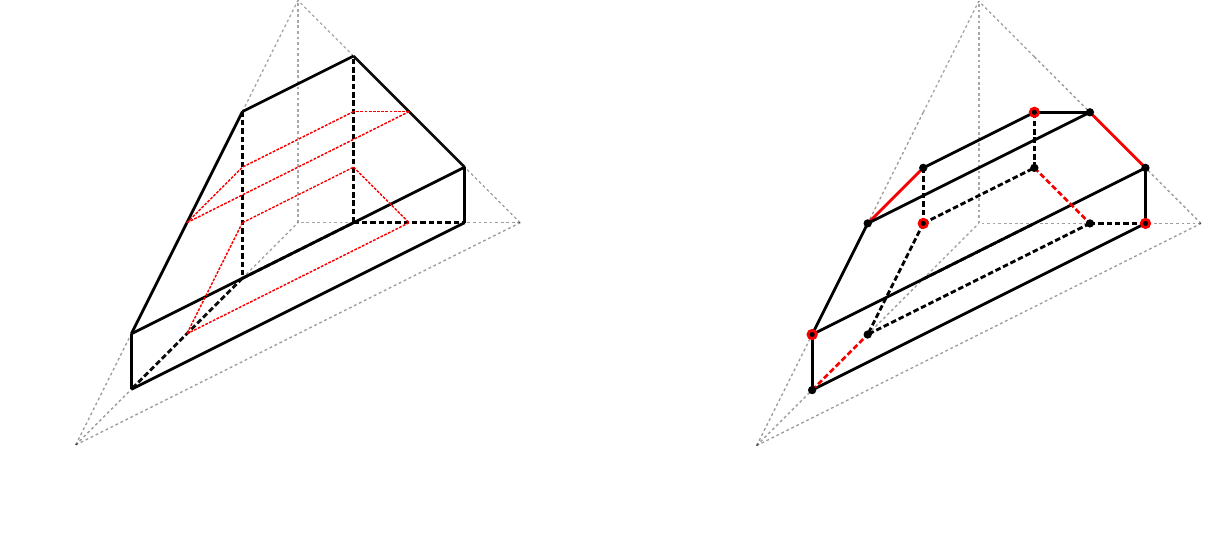}
				\caption{\label{figure_IV_1_1_1} Blow up of $Y$ along two lines $C_1$ and $C_2$ lying on the same exceptional components}
			\end{figure}          	
	            	            
           	 \item {\bf Case (IV-1-1.2)} \cite[3rd in Section 12.6]{IP} : Let $M = \p^1 \times X_3$ where $X_k$ denotes the blow-up of $\p^2$ at $k$ generic points.
           					 In particular we assume that $X_3$ is the toric blow-up of $\p^3$ equipped with the standard toric structure. 
           	 
				           	 Equip $M$ with the monotone toric K\"{a}hler form $\omega$ such that $c_1(TM) = [\omega]$ so that the moment map $\mu : M \rightarrow \R^3$ 
				           	 has the image given in Figure \ref{figure_IV_1_1_2}. Take $\xi = (0,-1,1)$. Then the $S^1$-action generated by $\xi$ is semifree and the 
				           	 balanced moment map is given by $\mu_\xi = \langle \mu, \xi \rangle$. The fixed point set consists of 
	            				\begin{itemize}
	            					\item $Z_{-2} = \mu^{-1}(\overline{(0,2,0) ~(1,2,0)})$
		            				\item $Z_{-1} = \mu^{-1}(0,1,0) \cup \mu^{-1}(2,1,0)$
		            				\item $Z_{0} = \mu^{-1}(\overline{(0,2,2) ~(1,2,2)}) \cup \mu^{-1}(\overline{(1,0,0) ~(2,0,0)})$
	           	 				\item $Z_1 = \mu^{-1}(0,1,2) \cup \mu^{-1}(2,1,2)$
	            					\item $Z_2 = \mu^{-1}(\overline{(2,0,2) ~(1,0,2)})$
	            				\end{itemize}
						It is not hard to check that the fixed point data of $M$ coincides with the one in Table \ref{table_IV_1} {\bf (IV-1-1.2)}.
						
			\begin{figure}[H]
				\scalebox{1}{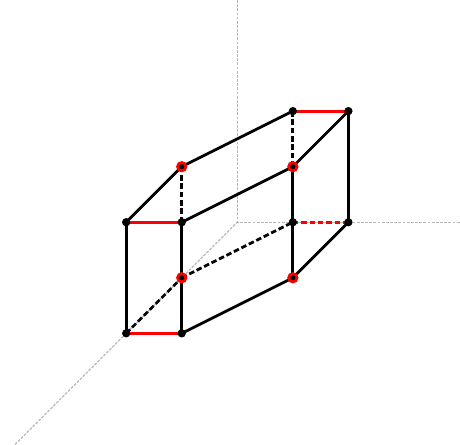}
				\caption{\label{figure_IV_1_1_2} $\p^1 \times X_3$}
			\end{figure}          	

           	 \item {\bf Case (IV-1-1.3)} \cite[7th in Section 12.5]{IP} : 
           	 Let $(W, \omega)$ be the monotone complete flag variety 
           	 given in Example \ref{example_II_1} (1) equipped with the Hamiltonian $T^2$-action where the moment polytope
           	 is described on the left of Figure \ref{figure_IV_1_1_3}. 
           	 
           	 Consider two edges $A$ and $B$ indicated in Figure \ref{figure_IV_1_1_3} and 
           	 denote by $C_A$ and $C_B$ the corresponding $T^2$-invariant spheres, respectively. (Note that $C_A$ and $C_B$ are curves of bidegree $(1,0)$ and $(0,1)$
           	 with respect to the Pl\"{u}cker embedding $W \subset \p^2 \times \p^2$.) 
      		 Using local toric structures on the normal bundles of $C_A$ and $C_B$, respectively, we may take $T^2$-equivariant blow up of $W$ along $C_A$ and $C_B$
      		 and denote the resulting manifold by $M$ and the image of the moment map $\mu : M \rightarrow \R^2$
      		  is given on the right of Figure \ref{figure_IV_1_1_3} (with respect to the monotone 
		 K\"{a}hler form). 
      		            	 
			\begin{figure}[H]
				\scalebox{1}{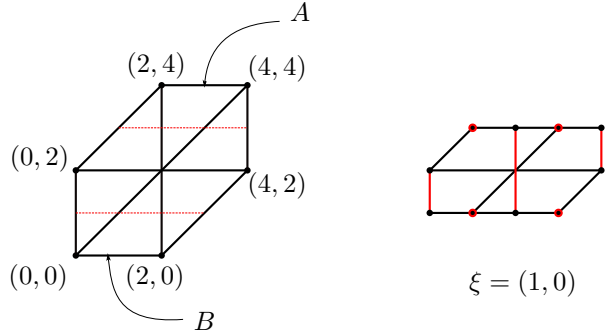}
				\caption{\label{figure_IV_1_1_3} Blow up of $W$ along two disjoint curves of bidegree $(1,0)$ and $(0,1)$. }
			\end{figure}          	

      		 Take the circle subgroup $S^1$ generated by $\xi = (1,0)$. Then the $S^1$-action is semifree and the balanced moment map is given by 
      		 $\mu_\xi = \langle \mu, \xi \rangle - 2$. The fixed point set consists of       		 

		\begin{itemize}
			\item $Z_{-2} = \mu^{-1}(\overline{(0,1) ~(0,2)})$
			\item $Z_{-1} = \mu^{-1}(1,1) \cup \mu^{-1}(1,3)$
			\item $Z_{0} = \mu^{-1}(\overline{(2,1) ~(2,3)})$
			\item $Z_1 = \mu^{-1}(3,1) \cup \mu^{-1}(3,3)$
			\item $Z_2 = \mu^{-1}(\overline{(4,2) ~(4,3)})$
		\end{itemize}
		and we can easily check that this should coincide with {\bf (IV-1-1.3)} in Table \ref{table_IV_1}.
           	(Note that the symplectic area of $Z_{-2}$ and $Z_2$ are both 1 so that $b_{\min} = b_{\max} = -1$.) \vs{0.5cm}

           	 \item {\bf Case (IV-1-2)} \cite[9th in Section 12.5]{IP} : Let Y be the toric blow-up of $\p^3$ along two disjoint $T^3$-invariant lines where the moment map 
           	 image is given on the left of Figure \ref{figure_IV_1_2} (see also Figure \ref{figure_III} (b)). Let $M$ be a toric blow up of $Y$ along a $T$-invariant exceptional line 
           	 (corresponding to the edge $A$ in Figure \ref{figure_IV_1_2}). With respect to the $T^3$-invariant monotone K\"{a}hler form, the image of a moment map $\mu$ is described 
           	 on the right of Figure \ref{figure_IV_1_2}. 
           	 
			\begin{figure}[H]
				\scalebox{1}{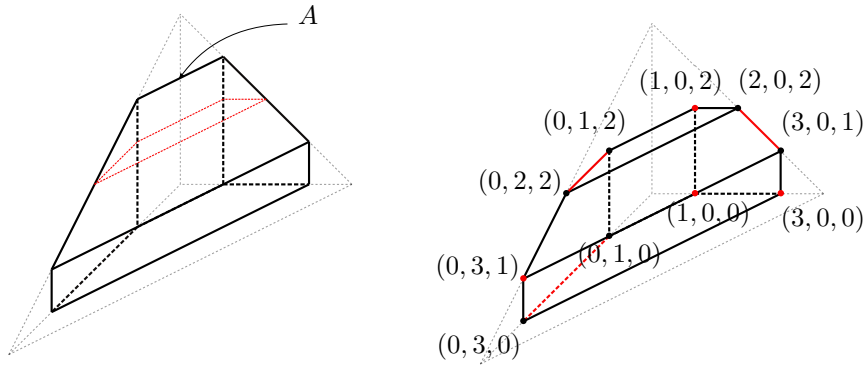}
				\caption{\label{figure_IV_1_2} Blow up of $Y$ along an exceptional line on $Y$. }
			\end{figure}          	
			\vs{-0.5cm}
		\noindent
		Take the circle subgroup $S^1$ of $T^3$ generated by $\xi = (-1, 0, -1)$. Then it is easy to check that the $S^1$-action is semifree and has the balanced moment map
		given by $\mu_\xi = \langle \mu, \xi \rangle + 2$. Also, the fixed point set consists of 

		\begin{itemize}
	            				\item $Z_{-2} = \mu^{-1}(\overline{(2,0,2) ~(3,0,1)})$
	            				\item $Z_{-1} = \mu^{-1}(1,0,2) \cup \mu^{-1}(3,0,0)$
	            				\item $Z_{0} = \mu^{-1}(\overline{(0,1,2) ~(0,2,2)})$
	            				\item $Z_1 = \mu^{-1}(0,3,1) \cup \mu^{-1}(1,0,0)$
	            				\item $Z_2 = \mu^{-1}(\overline{(0,1,0) ~(0,3,0)})$
		\end{itemize}
		where $\mathrm{Area}(Z_{-2}) = \mathrm{Area}(Z_{0}) = 1$ and $\mathrm{Area}(Z_{2}) = 2$. Thus one can see that the fixed point data of $M$ 
		coincides with {\bf (IV-1-2)} in Table \ref{table_IV_1}. 
          \end{enumerate}		
          
\end{example}
		
	Next, we deal with the case where $m = |Z_{-1}| = 1$.	
		
	\begin{theorem}\label{theorem_IV_2}
		Let $(M,\omega)$ be a six-dimensional closed monotone semifree Hamiltonian $S^1$-manifold with $c_1(TM) = [\omega]$. Suppose that $\mathrm{Crit} H = \{ 2, -1, 0, 1, -2\}$. 
		If the number of fixed points of index two equals one, up to orientation of $M$, the list of all possible topological fixed point data is given in the Table \ref{table_IV_2}.
		\begin{table}[h]
			\begin{tabular}{|c|c|c|c|c|c|c|c|c|c|}
				\hline
				    & $(M_0, [\omega_0])$ & $e(P_{-2}^+)$ &$Z_{-2}$  & $Z_{-1}$ & $Z_0$ & $Z_1$ & $Z_2$ & $b_2(M)$ & $c_1^3(M)$ \\ \hline \hline
				    {\bf (IV-2-1.1)} & \makecell{$(E_{S^2} \# ~\overline{\p^2},$ \\$3x + 2y - E_1)$} & $-x-y$  &$S^2$ & {\em pt} &
				    		\makecell{ $Z_0 \cong S^2$  \\ $\mathrm{PD}(Z_0) = 2x + y - E_1$}				    
				     &{\em pt} & $S^2$ & $3$ &$38$ \\ \hline
				    {\bf (IV-2-1.2)} & \makecell{$(E_{S^2} \# ~\overline{\p^2},$ \\$3x + 2y - E_1)$} & $-x-y$  &$S^2$ & {\em pt} &
				    		\makecell{ $Z_0 = Z_0^1 ~\dot \cup ~ Z_0^2$ \\ $Z_0^1 \cong Z_0^2 \cong S^2$ \\ 
				    		$\mathrm{PD}(Z_0^1) = \mathrm{PD}(Z_0^2) = x + y - E_1$}
				     &{\em pt} & $S^2$ & $4$ &$38$ \\ \hline

				    {\bf (IV-2-2.1)} & \makecell{$(E_{S^2} \# ~\overline{\p^2},$ \\$3x + 2y - E_1)$} & $-x-y$  &$S^2$ & {\em pt} & 
				    		\makecell{ $Z_0 \cong S^2$  \\ $\mathrm{PD}(Z_0) = x + y$}						    
				    &{\em pt} & $S^2$ & $3$ &$42$ \\ \hline
				    {\bf (IV-2-2.2)} & \makecell{$(E_{S^2} \# ~\overline{\p^2},$ \\$3x + 2y - E_1)$} & $-x-y$  &$S^2$ & {\em pt} & 
				    		\makecell{ $Z_0 = Z_0^1 ~\dot \cup ~ Z_0^2$ \\ $Z_0^1 \cong Z_0^2 \cong S^2$ \\ 
				    		$\mathrm{PD}(Z_0^1) = y$ \\ $\mathrm{PD}(Z_0^2)= x + y - E_1$}				    				    
				    &{\em pt} & $S^2$ & $4$ &$42$ \\ \hline				    
				    {\bf (IV-2-3)} & \makecell{$(E_{S^2} \# ~\overline{\p^2},$ \\$3x + 2y - E_1)$} & $-x-y$  &$S^2$ & {\em pt} &
				    		\makecell{ $Z_0 \cong S^2$  \\ $\mathrm{PD}(Z_0) = x$}				    
				     &{\em pt} & $S^2$ & $3$ &$46$ \\ \hline
				    {\bf (IV-2-4)} & \makecell{$(E_{S^2} \# ~\overline{\p^2},$ \\$3x + 2y - E_1)$} & $-x-y$  &$S^2$ & {\em pt} &
				    		\makecell{ $Z_0 \cong S^2$ 
				    		 \\ $\mathrm{PD}(Z_0) = E_1$}			    
				     &{\em pt} & $S^2$ & $3$ &$50$ \\ \hline

				    {\bf (IV-2-5)} & \makecell{$(S^2 \times S^2  \# ~\overline{\p^2},$ \\$2x + 2y - E_1)$} & $-y$  &$S^2$ & {\em pt} & 
				    		\makecell{ $Z_0 = Z_0^1 \dot \cup Z_0^2$ \\ $Z_0^1 \cong Z_0^2 \cong S^2$ \\ 
				    		$\mathrm{PD}(Z_0^1) = x - E_1$ \\  $\mathrm{PD}(Z_0^2) = y - E_1$  \\ }				    
				    &{\em pt} & $S^2$ & $4$ &$46$ \\ \hline
				    {\bf (IV-2-6)} & \makecell{$(S^2 \times S^2  \# ~\overline{\p^2},$ \\$2x + 2y - E_1)$} & $-y$  &$S^2$ & {\em pt} &
				    		\makecell{ $Z_0 \cong S^2$  \\ $\mathrm{PD}(Z_0) = x - E_1$}					    
				     &{\em pt} & $S^2$ & $3$ &$50$ \\ \hline
			\end{tabular}		
			\vs{0.5cm}			
			\caption{\label{table_IV_2} Topological fixed point data for $\mathrm{Crit} H = \{-2, -1,0,1, 2\}$ with $|Z_{-1}| = 1$.}
		\end{table}				   
	\end{theorem}
	
	\begin{proof}
	
		By \eqref{equation_8_solutions}, $b_{\min}$ is either  $-1$ or $0$.  For each cases, we have 
		\begin{equation}\label{equation_bmin_IV_2}
			\begin{cases}
				M_{-2 + \epsilon} \cong E_{S^2}, \quad c_1(TM_0) = [\omega_0] = 3x + 2y - E_1, \quad e(P_{-2}^+) = kx - y = -x -y & \text{if $b_{\min} = -1$} \\ \vs{0.1cm}
				M_{-2 + \epsilon} \cong S^2 \times S^2, \quad c_1(TM_0) = [\omega_0] = 2x + 2y - E_1, \quad e(P_{-2}^+) = kx - y = -y & \text{if $b_{\min} = 0$} 
			\end{cases}
		\end{equation}
		by Lemma \ref{lemma_volume}, where $M_0$ is a one point blow-up of $M_{-2 + \epsilon}$ and $E_1$ is the dual class of the 
		exceptional divisor on $M_0$.
		
		Let $\mathrm{PD}(Z_0) = ax + by + cE_1$ for some $a,b,c \in \Z$. By the Duistermaat-Heckman theorem \ref{theorem_DH}, we have 
		\[
				[\omega_1] = [\omega_0] - e(P_0^+) = \begin{cases}
						(4-a)x + (3-b)y - (2+c)E_1 & \text{if $b_{\min} = -1$}\\ \vs{0.1cm}
						(2-a)x + (3-b)y - (2+c)E_1 & \text{if $b_{\min} = 0$}.				
				\end{cases} 
		\] 
		Moreover, only one blow-down occurs on $M_1$ with the vanishing cycle $C$ so that 
		\begin{equation}\label{equation_vanishing_IV_2}
			\langle [\omega_1], C \rangle = 0. 
		\end{equation}
		By Lemma \ref{lemma_list_exceptional}, 
		the list of all possible $\mathrm{PD}(C)$ is given by
		\[
			u - E_1 - E_2, \quad E_1, \quad E_2.
		\]
		In terms of $\{x,y,E_1\}$, the list is expressed by
		\begin{itemize}
			\item if $b_{\min} = -1$, then 
				\[
					x - E_1, \quad E_1, \quad y,
				\] 		
			\item if $b_{\min} = 0$, then 
				\[
					E_1, \quad x - E_1, \quad y - E_1.
				\]
		\end{itemize}
		
		Now we compute the fixed point data for the six cases (on the right of \eqref{equation_8_solutions}) as follows.
		(Note that the Chern number computation can be easily obtained from the localization theorem \ref{theorem_localization} and Remark \ref{remark_localization_surface}.)
		
		\vs{0.3cm}
				
		\noindent
		{\bf (IV-2-1) : $m = 1, \mathrm{Vol}(Z_0) = 4, (b_{\min}, b_{\max}) = (-1,-1)$} \vs{0.3cm}				
		
		\noindent
		Because $\mathrm{Vol}(Z_0) = 4$ and $b_{\max} = -1$, it follows that 
		\begin{equation}\label{equation_IV_2_1}
			\mathrm{Vol}(Z_0) = 2a+b+c= 4, \quad \langle e(P_2^-)^2, [M_{2-\epsilon}] \rangle = 1 ~\text{so that $\langle e(P_0^+)^2, [M_0] \rangle = 0$}
		\end{equation}
		by Lemma \ref{lemma_volume}. The latter equation can be re-written as
		\[
			2(a-1)(b-1) - (b-1)^2 - (c+1)^2 = 0 \quad \quad \text{since \quad $e(P_0^+) = (a-1)x + (b-1)y + (c+1)E_1$.}
		\]
		\vs{0.1cm}
		
		\noindent
		{\bf Case (1) :} $\mathrm{PD}(C) = x - E_1$. \vs{0.1cm}
		
		\noindent
		Since $b+c = 1$ by \eqref{equation_vanishing_IV_2}, we have $2a=3$ by \eqref{equation_IV_2_1}. Thus no such manifold exists. \vs{0.3cm}
		
		\noindent
		{\bf Case (2) :} $\mathrm{PD}(C) = E_1$.  \vs{0.1cm}

		\noindent
		In this case, we have $c = -2$ by \eqref{equation_vanishing_IV_2}. Then \eqref{equation_IV_2_1} implies that 
		\[
			2a + b = 6, \quad 2(a-1)(b-1) - (b-1)^2 = (b-1)(2a - b -1) = 1
		\]
		which has the unique integeral solution $(a,b,c) = (2,2,-2)$. So, $\mathrm{PD}(Z_0) = 2x + 2y - 2E_1$ and $[Z_0] \cdot [Z_0] = 0$.
		Then the adjunction formula \eqref{equation_adjunction} implies that 
		\[
			[Z_0] \cdot [Z_0] + \sum (2 - 2g_i) = 4 \quad \text{(sum is taken over connected components of $Z_0$)}. 
		\]
		Thus there are at least two spheres, namely $Z_0^1$ and $Z_0^2$. Moreover, they satisfy (again by \eqref{equation_adjunction})
		\[
			[Z_0^1] \cdot [Z_0^1] \geq -1 \quad \text{and } \quad [Z_0^2] \cdot [Z_0^2] \geq -1.
		\]
		Note that if $[Z_0^i] \cdot [Z_0^i] = -1$ for some $i$, then we can easily check that $([Z_0] - [Z_0^i]) \cdot [Z_0^i] \neq 0$ using Lemma \ref{lemma_list_exceptional}.  Therefore
		\[
			[Z_0^1] \cdot [Z_0^1] \geq 0  \quad \text{and } \quad [Z_0^2] \cdot [Z_0^2] \geq 0.
		\]		
		In particular, we have $\mathrm{Vol}(Z_0^i) = [Z_0^i] \cdot [Z_0^i] + 2 \geq 2$ so that the only possibility is that 
		\[
			[Z_0^i] \cdot [Z_0^i] = 0, \quad i=1,2,
		\]
		and we get  $\mathrm{PD} (Z_0^1) = \mathrm{PD}(Z_0^2) = x + y - E_1$. See Table \ref{table_IV_2} : {\bf (IV-2-1.2)}. 
		
		\vs{0.3cm}
		
		\noindent
		{\bf Case (3) :} $\mathrm{PD}(C) = y$.  \vs{0.1cm}
		
		\noindent
		From \eqref{equation_vanishing_IV_2}, we get $a = b + 1$. By \eqref{equation_IV_2_1}, 
		\[
			3b + c = 2, \quad 2b(b-1) - (b-1)^2 - (c+1)^2 = 0, 
		\]
		whose solution is $(a,b,c) = (2, 1, -1)$, that is, $\mathrm{PD}(Z_0) = 2x + y  - E_1$ (and so $[Z_0] \cdot [Z_0] = 2$). On the other hand, the adjunction formula 
		\[
			[Z_0] \cdot [Z_0] + \sum (2 - 2g_i) = 4 
		\]
		implies that 
		there exists a sphere component, say $Z_0^1$, of $Z_0$. If we denote by $\mathrm{PD}(Z_0^1) = \alpha x + \beta y + \gamma E_1$, it satisfies 
		\[
			 2\alpha\beta - \beta^2 - \gamma^2 + 2 = [Z_0^1] \cdot [Z_0^1] + 2 = \langle c_1(TM_0), [Z_0^1] \rangle = 2\alpha + \beta + \gamma.
		\]
		Also since $([Z_0] - [Z_0^1]) \cdot [Z_0^1] = 0$, we have
		\[
			\left( (2 - \alpha)x + (1 - \beta)y - (1+\gamma)E_1) \right) \cdot (\alpha x + \beta y + \gamma E_1) = -2\alpha\beta + \alpha + \beta + \gamma + \beta^2 + \gamma^2 = 0.
		\]
		Combining those two equations above, we get $\alpha = 2$ and 
		\[
			\beta^2 + \gamma^2 - 3\beta + \gamma + 2 = 0 \quad \Leftrightarrow \quad (\beta - \frac{3}{2})^2 + (\gamma + \frac{1}{2})^2 - \frac{1}{2} = 0.
		\]
		Therefore, $(\beta, \gamma) = (2, 0), (2, -1), (1, 0), (1, -1)$. In any case, $\mathrm{Vol}(Z_0^1) \geq 4$ which is impossible unless $Z_0^1 = Z_0$. 
		This implies that $Z_0$ is connected and is a sphere. See {\bf (IV-2-1.1)}. 
		
		 \vs{0.3cm}
			
		\noindent
		{\bf (IV-2-2) : $m = 1, \mathrm{Vol}(Z_0) = 3, (b_{\min}, b_{\max}) = (-1,0)$} \vs{0.3cm}				
		
		\noindent
		By Lemma \ref{lemma_volume}, it follows that 
		\begin{equation}\label{equation_IV_2_2}
			\mathrm{Vol}(Z_0) = 2a+b+c= 3, \quad \langle e(P_2^-)^2, [M_{2-\epsilon}] \rangle = 0 ~\text{so that $\langle e(P_0^+)^2, [M_0] \rangle = -1$}
		\end{equation}
		where the latter equation is equivalent to 
		\[
			2(a-1)(b-1) - (b-1)^2 - (c+1)^2 = -1.
		\]
		\vs{0.1cm}
		
		\noindent
		{\bf Case (1) :} $\mathrm{PD}(C) = x - E_1$. \vs{0.1cm}
		
		\noindent
		By \eqref{equation_vanishing_IV_2}, we have $b+c = 1$ so that $a = 1$ and $(b-1)^2 + (c+1)^2 = 1$ (and so $(b,c) = (1, 0)$ or $(2, -1)$). \vs{0.1cm}
		\begin{itemize}
			\item If $(a,b,c) = (1,1,0)$, then $\mathrm{PD}(Z_0) = x + y$ and $[Z_0] \cdot [Z_0] = 1$ so that there exists at least one sphere component, denote by $Z_0^1$,
			 in $Z_0$ by \eqref{equation_adjunction}.

				Suppose that $Z_0$ is not connected. Then $\mathrm{Vol}(Z_0^1) = 1$ or $2$. If $\mathrm{Vol}(Z_0^1) = 1$, then $[Z_0^1] \cdot [Z_0^1] = -1$
				by the adjunction formula, and hence $\mathrm{PD}(Z_0^1) = E_1, y, x - E_1$ by Lemma \ref{lemma_list_exceptional}.
				In either case, it follows that 
				\[
					[Z_0^1] \cdot ([Z_0] - [Z_0^1]) \neq 0
				\]
				which leads to a contradiction. So, $\mathrm{Vol}(Z_0^1) \neq 1$.
				
				On the other hand, if $\mathrm{Vol}(Z_0^1) = 2$, then $[Z_0^1] \cdot [Z_0^1] = 0$ by the adjunction formula. If we let 
				$\mathrm{PD}(Z_0^1) = \alpha x + \beta y + \gamma E_1$, then 
				\begin{itemize}
					\item $2\alpha\beta - \beta^2 - \gamma^2 = 0$, \quad ($\because ~[Z_0^1]\cdot [Z_0^1] = 0$), 
					\item $\alpha - 2\alpha\beta + \beta^2 + \gamma^2 = 0$, \quad ($\because ~[Z_0^1] \cdot ([Z_0] - [Z_0^1]) = 0$), 
					\item $2\alpha + \beta + \gamma = 2$ \quad ($\because ~\mathrm{Vol}(Z_0^1) = 2$)
				\end{itemize}
				whose (real) solution does not exist. Thus $Z_0$ is connected and we have $Z_0 \cong S^2$. See Table \ref{table_IV_2}: {\bf (IV-2-2.1)}.\vs{0.2cm}
			\item If $(a,b,c) = (1, 2, -1)$, i.e., $\mathrm{PD}(Z_0) = x + 2y - E_1$, then we have $[Z_0] \cdot [Z_0] = -1$ and there are at least two 
			sphere components $Z_0^1$ and $Z_0^2$ in $Z_0$ by \eqref{equation_adjunction}. Since $\mathrm{Vol}(Z_0^1) + \mathrm{Vol}(Z_0^2) \leq 3$,
			we may assume that $\mathrm{PD}(Z_0^1) = 1$ (so that $[Z_0^1] \cdot [Z_0^1] = -1$).
			Then we obtain $\mathrm{PD}(Z_0^1) = y$ by the fact that $([Z_0] - [Z_0^1]) \cdot [Z_0^1] = 0$ and Lemma \ref{lemma_list_exceptional}. So, 
			\[
				Z_0^1 \cong S^2 ~(\mathrm{PD}(Z_0^1) = y) \quad \text{and} \quad Z_0^2 \cong S^2 ~(\mathrm{PD}(Z_0^2) = x + y - E_1)
			\]
			See Table \ref{table_IV_2}: {\bf (IV-2-2.2)}. (Note that $\mathrm{Vol}(Z_0^2) \neq 1$ otherwise $\mathrm{PD}(Z_0^2)$ also should be $y$ which contradicts 
			that $[Z_0^1] \cdot [Z_0^2] = 0$.)
		\end{itemize}
		
		\vs{0.3cm}
		
		\noindent
		{\bf Case (2) :} $\mathrm{PD}(C) = E_1$.  \vs{0.1cm}

		\noindent
		Since $c=-2$ by \eqref{equation_vanishing_IV_2}, we have 
		\[
			2a + b = 5 \quad \text{and} \quad 2(a-1)(b-1) - (b-1)^2 = 0
		\]
		where it has a unique integral solution $(a,b,c) = (2,1,-2)$. However, since
		\[
			[\omega_1] \cdot y = (2x + 2y) \cdot y = 0, 
		\]
		the exceptional divisor representing $y$ vanishes on $M_1$, i.e., two simultaneous blow-downs occur on $M_1$. 
		Thus no such manifold exists.
		
		\vs{0.3cm}
		
		\noindent
		{\bf Case (3) :} $\mathrm{PD}(C) = y$.  \vs{0.1cm}
		
		\noindent
		Now we have $a = b+1$ and so 
		\[
			3b+c = 1 \quad \text{and} \quad 2b(b-1) - (b-1)^2 - (c+1)^2 = -1
		\]
		by \eqref{equation_IV_2_2}. This has a unique integral solution $(a,b,c) = (2,1,-2)$. This case is exactly 
		the same as in {\bf Case (2)} above and we have $[\omega_1] \cdot E_1 = 0$. 
		Then two simultaneous blow-downs occur on $M_1$ which is impossible. 
		Therefore there is no such manifold.

		 \vs{0.3cm}

		\noindent
		{\bf (IV-2-3) : $m = 1, \mathrm{Vol}(Z_0) = 2, (b_{\min}, b_{\max}) = (-1,1)$} \vs{0.3cm}				

		\noindent
		In  this case, we have 
		\begin{equation}\label{equation_IV_2_3}
			\mathrm{Vol}(Z_0) = 2a+b+c= 2, \quad \langle e(P_2^-)^2, [M_{2-\epsilon}] \rangle = -1 ~\text{so that $\langle e(P_0^+)^2, [M_0] \rangle = -2$}
		\end{equation}
		where the latter one is 
		\[
			2(a-1)(b-1) - (b-1)^2 - (c+1)^2 = -2.
		\]
		\vs{0.1cm}
		
		\noindent
		{\bf Case (1) :} $\mathrm{PD}(C) = x - E_1$. \vs{0.1cm}
		
		\noindent
		Using $b+c = 1$ by \eqref{equation_vanishing_IV_2}, we have $2a = 1$. Thus no such manifold exists.
		
		\vs{0.3cm}
		
		\noindent
		{\bf Case (2) :} $\mathrm{PD}(C) = E_1$.  \vs{0.1cm}

		\noindent
		Since $c = -2$, we have 
		\[
			2a + b = 4, \quad 2(a-1)(b-1) - (b-1)^2 = -1
		\]
		and the only possible solution is $(a,b) = (1,2)$, i.e., $\mathrm{PD}(Z_0) = x + 2y - 2E_1$. However, the adjunction formula \eqref{equation_adjunction}
		implies that 
		\[
			[Z_0] \cdot [Z_0] + \sum (2 - 2g_i) = -4 + \sum (2 - 2g_i) = 2, 
		\]
		i.e., there are three sphere components $Z_0^1, Z_0^2, Z_0^3$ and hence $\mathrm{Vol}(Z_0) \geq 3$ which leads to a contradiction.
		So, no such manifold exists.

		\vs{0.3cm}
		
		\noindent
		{\bf Case (3) :} $\mathrm{PD}(C) = y$.  \vs{0.1cm}
		
		\noindent
		In this case, $a = b+1$ so that 
		\[
			3b + c = 0, \quad 2b(b-1) - (b-1)^2 - (c+1)^2 = -2
		\]
		and it has a unique solution $(a,b,c) = (1,0,0)$. If $Z_0$ is not connected, then the adjunction formula implies that $Z_0$ consists of two spheres $Z_0^1$ and $Z_0^2$
		each of which has symplectic area $1$ (so that it is an exceptional sphere). On the other hand, by the fact that $[Z_0^1] \cdot [Z_0^2] = 0$ and Lemma \ref{lemma_list_exceptional}
		imply that the dual classes of $Z_0^1$ and $Z_0^2$ are 
		$y$ and $E_1$, respectively. Then $\mathrm{PD}(Z_0) = x \neq \mathrm{PD}(Z_0^1) + \mathrm{PD}(Z_0^2)$, and therefore $Z_0$ is connected and 
		\[
			Z_0 \cong S^2, \quad \mathrm{PD}(Z_0) = x.
		\]
		See Table \ref{table_IV_2}: {\bf (IV-2-3)}.
		 \vs{0.3cm}

		\noindent
		{\bf (IV-2-4) : $m = 1, \mathrm{Vol}(Z_0) = 1, (b_{\min}, b_{\max}) = (-1,2)$} \vs{0.3cm}				

		\noindent
		As $\mathrm{Vol}(Z_0) = 1$, $Z_0$ is connected. Also, 
		\begin{equation}\label{equation_IV_2_4}
			\mathrm{Vol}(Z_0) = 2a+b+c= 1, \quad \langle e(P_2^-)^2, [M_{2-\epsilon}] \rangle = -2 ~\text{so that $\langle e(P_0^+)^2, [M_0] \rangle = -3$}
		\end{equation}
		i.e., 
		\[
			2(a-1)(b-1) - (b-1)^2 - (c+1)^2 = -3.
		\]
		\vs{0.1cm}
		
		\noindent
		{\bf Case (1) :} $\mathrm{PD}(C) = x - E_1$. \vs{0.1cm}
		
		\noindent
		We have $b+c = 1$ so that $(a,b,c) = (0, 2,-1)$ or $(0,0,1)$. If $(a,b,c) = (0, 2,-1)$, then $\mathrm{PD}(Z_0) = 2y - E_1$ and $[Z_0] \cdot [Z_0] = -5$. 
		This is impossible by the adjunction formula since $Z_0$ is connected. So, no such manifold exists. On the other hand, if $(a,b,c) = (0,0,1)$, i.e., $\mathrm{PD}(Z_0) = E_1$, 
		then we have 
		\[
			Z_0 \cong S^2, \quad \mathrm{PD}(Z_0) = E_1. 
		\]
		See Table \ref{table_IV_2}: {\bf (IV-2-4)}.
		
		\vs{0.3cm}
		
		\noindent
		{\bf Case (2) :} $\mathrm{PD}(C) = E_1$.  \vs{0.1cm}

		\noindent
		Now, we have $c = -2$ and \eqref{equation_IV_2_4} implies that 
		\[
			2a + b = 3, \quad 2(a-1)(b-1) - (b-1)^2 = -2
		\]		
		which has no integral solution. Thus there is no such manifold.
		\vs{0.3cm}
		
		\noindent
		{\bf Case (3) :} $\mathrm{PD}(C) = y$.  \vs{0.1cm}
		
		\noindent
		From \eqref{equation_vanishing_IV_2}, we have $a = b+1$ and so
		\[
			3b + c = -1, \quad 2b(b-1) - (b-1)^2 - (c+1)^2 = -3
		\]
		where no integral solution exists. Thus no such manifold exists.
		 \vs{0.3cm}
		
		\noindent
		{\bf (IV-2-5) : $m = 1, \mathrm{Vol}(Z_0) = 2, (b_{\min}, b_{\max}) = (0,0)$} \vs{0.3cm}				

		\noindent
		Since $b_{\min} = 0$, we have $M_{-2 + \epsilon} \cong S^2 \times S^2$ and $e(P)_{-2}^+ = -y$ with $c_1(TM_0) = 2x + 2y - E_1$, see \eqref{equation_bmin_IV_2}.  
		Also, Lemma \ref{lemma_volume} implies that 
		\begin{equation}\label{equation_IV_2_5}
			\mathrm{Vol}(Z_0) = 2a+2b+c= 2, \quad \langle e(P_2^-)^2, [M_{2-\epsilon}] \rangle = 0 ~\text{so that $\langle e(P_0^+)^2, [M_0] \rangle = -1$}
		\end{equation}
		where the latter equation can be re-written by 
		\[
			2a(b-1) - (c+1)^2 = -1.
		\]
		Note that if $Z_0$ is connected, then $[Z_0] \cdot [Z_0] = 0$ by the adjunction formula.
		Also, if $Z_0$ is disconnected with two components $Z_0^1$ and $Z_0^2$ such that $\mathrm{Vol}(Z_0^1) = \mathrm{Vol}(Z_0^2) = 1$, then 
		the adjunction formula implies that $[Z_0^1] \cdot [Z_0^1] = [Z_0^2] \cdot [Z_0^2] = -1$. In particular, $[Z_0] \cdot [Z_0] = -2$.
		
		Recall that a possible dual class of the cycle $C$ vanishing at the reduced space $M_1$ is $x - E_1$, $E_1$, or $y - E_1$ by Lemma \ref{lemma_list_exceptional}.
		\vs{0.1cm}
		
		\noindent
		{\bf Case (1) :} $\mathrm{PD}(C) = x - E_1$. \vs{0.1cm}
		
		\noindent
		By \eqref{equation_vanishing_IV_2}, we have $b+c = 1$ so that 
		\[
			2a -c  = 0, \quad -2ac - (c+1)^2 = -1
		\]
		where it has a unique integral solution $(a,b,c) = (0,1,0)$. However in this case, a cycle representing $y - E_1$ is also vanishing on $M_1$. In other words, 
		two blow-downs occur simultaneously on $M_1$. So, no such manifold exists.
		\vs{0.3cm}
		
		\noindent
		{\bf Case (2) :} $\mathrm{PD}(C) = E_1$.

		\noindent
		In this case, we have $c = -2$ and hence 
		\[
			a+b = 2, \quad 2a(b-1) = 0
		\]
		where the solution is $(a,b,c) = (0,2,-2)$ or $(1,1,-2)$. If $(a,b,c) = (0,2,-2)$, then $[Z_0] \cdot [Z_0] = -4$ so that there are at least three spheres in $Z_0$ by the adjunction formula, 
		which is impossible since 
		$\mathrm{Vol}(Z_0) = 2$. Thus there is no such manifold. 
		
		If $(a,b,c) = (1,1,-2)$, then $[Z_0] \cdot [Z_0] = -2$ and so $Z_0$ consists of two spheres, say $Z_0^1$ and $Z_0^2$, each of which has self-intersection number $-1$ by the 
		adjunction formula. 
		By Lemma \ref{lemma_list_exceptional}, we get 
		\[
			Z_0^1 \cong Z_0^2 \cong S^2, \quad \mathrm{PD}(Z_0^1) = x - E_1, \quad \mathrm{PD}(Z_0^2) = y - E_1.
		\]
		See Table \ref{table_IV_2}: {\bf (IV-2-5)}.
		\vs{0.3cm}
		
		\noindent
		{\bf Case (3) :} $\mathrm{PD}(C) = y - E_1$.
		
		\noindent
		From \eqref{equation_vanishing_IV_2}, we have $a + c = 0$ and so 
		\[
			a + 2b = 2, \quad 2a(b-1) - (1-a)^2 = -1
		\]
		and it has the unique solution $(a,b,c) = (0, 1, 0)$. Similar to {\bf Case (1)}, a cycle representing $x - E_1$ also vanishes on $M_1$ so that two blow-downs occur simultaneously 
		on $M_1$. Therefore there is no such manifold.
		 \vs{0.3cm}
		
		\noindent
		{\bf (IV-2-6) : $m = 1, \mathrm{Vol}(Z_0) = 1, (b_{\min}, b_{\max}) = (0,1)$} \vs{0.3cm}	

		\noindent
		Note that $Z_0$ is connected and the condition $b_{\min} = 0$ implies that $e(P)_{-2}^+ = -y$ by Lemma \ref{lemma_volume}. 
		Moreover, $\mathrm{Vol}(Z_0) = 1$ and $b_{\max} = 1$ implies that 
		\begin{equation}\label{equation_IV_2_5}
			\mathrm{Vol}(Z_0) = 2a+2b+c= 1, \quad \langle e(P_2^-)^2, [M_{2-\epsilon}] \rangle = -1 ~\text{so that $\langle e(P_0^+)^2, [M_0] \rangle = -2$}
		\end{equation}
		where the latter one is equivalent to 
		\[
			2a(b-1) - (c+1)^2 = -2.
		\]
		\vs{0.1cm}
		
		\noindent
		{\bf Case (1) :} $\mathrm{PD}(C) = x - E_1$. \vs{0.1cm}
		
		\noindent
		Since $b+c = 1$, we have 
		\[
			2a + b = 0, \quad 2a(-2a-1) - (2 + 2a)^2 = -2
		\]
		so that $(a,b,c) = (-1,2,-1)$. That is, $\mathrm{PD}(Z_0) = -x + 2y - E_1$ and so $[Z_0] \cdot [Z_0] = -5$. This contradicts the fact that $Z_0$ is conencted by the adjunction formula.
		So, there is no such manifold.
		\vs{0.3cm}
		
		\noindent
		{\bf Case (2) :} $\mathrm{PD}(C) = E_1$.

		\noindent
		We have $c = -2$ by \eqref{equation_vanishing_IV_2} which implies that $2a + 2b = 3$. Thus no such manifold exists.
		\vs{0.3cm}
		
		\noindent
		{\bf Case (3) :} $\mathrm{PD}(C) = y - E_1$.
		
		\noindent
		In this case, we have $a + c = 0$ so that 
		\[
			a + 2b = 1, \quad 2a(b-1) - (1-a)^2 = -2.
		\]							
		It has a unique solution $(a,b,c) = (1,0,-1)$, i.e., 
		\[
			Z_0 \cong S^2, \quad \mathrm{PD}(Z_0) = x - E_1.
		\] 
		See Table \ref{table_IV_2}: {\bf (IV-2-6)}.
	\end{proof}

	\begin{example}[Fano variety of type {\bf (IV-2)}]\label{example_IV_2} In this example, we describe Fano varieties of type {\bf (IV-2)} listed in Theorem \ref{theorem_IV_2}.
		
		\begin{itemize}
	           	 \item {\bf (IV-2-1.1)} \cite[20th in Section 12.4]{IP}  : Recall that a smooth quadric in $\p^4$, isomorphic to a coadjoint orbit of $\mathrm{SO}(5)$, admits a 
	           	 maximal torus $T^2$ action whose moment map image is given on the left of Figure \ref{figure_IV_2_1_1} (see also Example \ref{example_Sec6_1_1}). Let $M$ be 
	           	 the blow-up of the smooth quadric along two disjoint $T^2$-invariant spheres with the induced $T^2$-action. Then the corresponding moment map can be 
	           	 described as on the right of Figure \ref{figure_IV_1_1_1}. 
	           	 
	           	 	\begin{figure}[h]
	           	 		\scalebox{1}{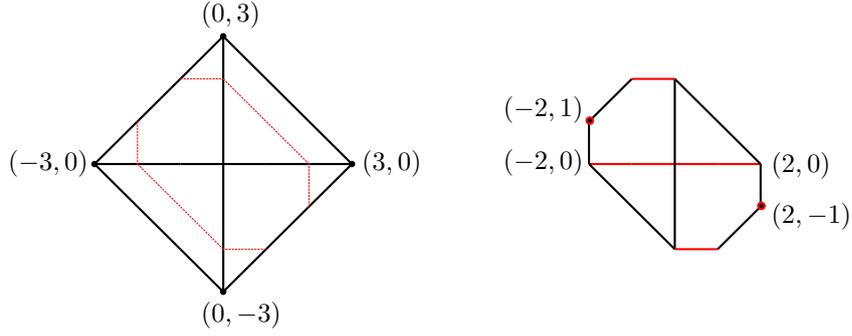}
	           	 		\caption{\label{figure_IV_2_1_1} Blow up of the smooth quadric along two disjoint lines}
	           	 	\end{figure}
	           	 	
			Now, we take the $S^1$-subgroup of $T^2$ generated by $\xi = (0,1) \in \frak{t}$. Then the fixed point set consists of 
			
	            				\begin{itemize}
	            					\item $Z_{-2} = S^2$ with  $\mu(Z_{-2}) = \overline{(0,-2) ~(1,-2)}$ and $\mathrm{Vol}(Z_{-2}) = 1$,
		            				\item $Z_{-1} = \mathrm{pt}$ with $\mu(Z_{-1}) = (2,-1)$,
		            				\item $Z_{0} = S^2$ with $\mu(Z_{0}) = \overline{(-2,0) ~(2,0)}$ and $\mathrm{Vol}(Z_{0}) = 4$,
	           	 				\item $Z_1 = \mathrm{pt}$ with $\mu(Z_1) = (-2,1)$,
	            					\item $Z_2 = S^2$ with $\mu(Z_2) = \overline{(-1,2) ~(0,2)})$ and $\mathrm{Vol}(Z_{2}) = 1$.
	            				\end{itemize}
			\vs{0.5cm}
			
	           	 \item {\bf (IV-2-1.2)} \cite[8th in Section 12.5]{IP} : Consider $X = \p^1 \times \p^1 \times \p^1$ equipped with $T^2$-action defined by 
	           	 \[
	           	 	(t_1, t_2) \cdot ([x_0 : x_1], [y_0 : y_1], [z_0 : z_1]) := ([t_1x_0 : x_1], [t_2y_0 : y_1], [t_2z_0 : z_1])
	           	 \]
	           	 with respect to the normalized monotone K\"{a}hler form on $X$. The moment map image is given in the middle of Figure \ref{figure_IV_2_1_2}. 
	           	 (Note that the red double line in the middle  indicates the image of the upper-left and lower-right red edges in the first of Figure \ref{figure_IV_2_1_2}.)
	           	 
	           	 Let $C$ be the $T$-invariant sphere given by 
	           	 \[
	           	 	C = \{ ([1:0], [y_0 : y_1],  [y_0 : y_1]) ~|~ [y_0:y_1] \in \p^1\}
			\]
	           	whose moment map image is indicated by the blue line in Figure \ref{figure_IV_2_1_2}. Then, let $M$ be the $T^2$-equivariant blow-up of $X$ whose moment map 
	           	is described in the third of Figure \ref{figure_IV_2_1_2}. The fixed point set consists of 
	            				\begin{itemize}
	            					\item $Z_{-2} = S^2$ with  $\mu(Z_{-2}) = \overline{(1,-2) ~(2,-2)}$ and $\mathrm{Vol}(Z_{-2}) = 1$,
		            				\item $Z_{-1} = \mathrm{pt}$ with $\mu(Z_{-1}) = (0,-1)$,
		            				\item $Z_{0} = S^2 ~\dot \cup~ S^2$ with $\mu(Z_{0}^1) =\mu(Z_{0}^2) = \overline{(0,0) ~(2,0)}$ and $\mathrm{Vol}(Z_{0}^1) = \mathrm{Vol}(Z_{0}^2) = 2$,
	           	 				\item $Z_1 = \mathrm{pt}$ with $\mu(Z_1) = (0,1)$,
	            					\item $Z_2 = S^2$ with $\mu(Z_2) = \overline{(1,2) ~(2,2)})$ and $\mathrm{Vol}(Z_{2}) = 1$.
	            				\end{itemize}
	           	 	\vs{0.3cm}
	           	 
	           	 	\begin{figure}[h]
	           	 		\scalebox{1}{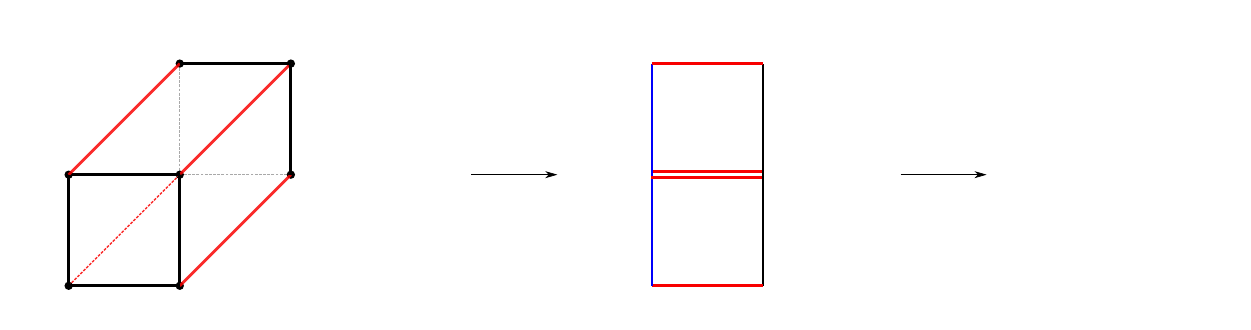}
	           	 		\caption{\label{figure_IV_2_1_2} Blow up of $\p^1 \times \p^1 \times \p^1$ along $C$}
	           	 	\end{figure}	           	 
	           	 
	           	 \item {\bf (IV-2-2.1)} \cite[24th in Section 12.4]{IP} : Consider the complete flag variety $\mcal{F}(3) \cong U(3) / T^3$ together with the induced $T^2$-action whose moment map 
	           	 image is given in the first of Figure \ref{figure_IV_2_2_1}. (See also Example \ref{example_II_1}.) Let $C$ be a $T$-invariant sphere (for instance, take a sphere whose moment map
	           	 image is $\overline{(0,0) ~(0,2)}$ as in Figure \ref{figure_IV_2_2_1}). Let $M$ be the $T^2$-equivariant blow-up of $\mcal{F}(3)$ along $C$. Then the moment map image 
	           	 for the induced $T^2$-action on $M$ can be depicted as in the second in Figure \ref{figure_IV_2_2_1}. 
		           The fixed point set consists of 
                           				\begin{itemize}
	            					\item $Z_{-2} = S^2$ with  $\mu(Z_{-2}) = \overline{(1,0) ~(2,0)}$ and $\mathrm{Vol}(Z_{-2}) = 1$,
		            				\item $Z_{-1} = \mathrm{pt}$ with $\mu(Z_{-1}) = (1,1)$,
		            				\item $Z_{0} = S^2$ with $\mu(Z_{0}) = \overline{(1,2) ~(4,2)}$ and $\mathrm{Vol}(Z_{0}) = 3$,
	           	 				\item $Z_1 = \mathrm{pt}$ with $\mu(Z_1) = (1,3)$,
	            					\item $Z_2 = S^2$ with $\mu(Z_2) = \overline{(2,4) ~(4,4)})$ and $\mathrm{Vol}(Z_{2}) = 2$.
	            				\end{itemize}
	           	 	           	 	\vs{0.3cm}
	           	 	           	 	
	           	 	\begin{figure}[H]
	           	 		\scalebox{1}{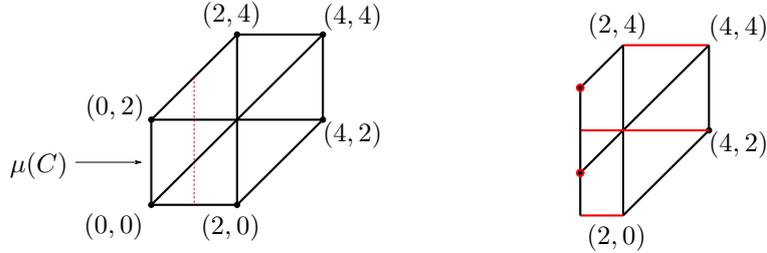}
	           	 		\caption{\label{figure_IV_2_2_1} Blow up of $\mcal{F}(3)$ along $C$}
	           	 	\end{figure}
	            
                          	\item {\bf (IV-2-2.2)} \cite[10th in Section 12.5]{IP} : Consider $\C P^1 \times ~X_2$ with the standard $T^3$-action, where $X_k$ is the $k$-times blow-up of $\p^2$. 
                          	The corresponding moment polytope is given in Figure \ref{figure_IV_2_2_2}. Take a circle subgroup of $T^3$ generated by $\xi = (-1,1,0)$. Then one can easily check
                          	that the $S^1$-action is semifree and the fixed point set consists of 
                           				\begin{itemize}
	            					\item $Z_{-2} = S^2$ with  $\mu(Z_{-2}) = \overline{(2,0,0) ~(2,0,1)}$ and $\mathrm{Vol}(Z_{-2}) = 1$,
		            				\item $Z_{-1} = \mathrm{pt}$ with $\mu(Z_{-1}) = (1,0,2)$,
		            				\item $Z_{0} = S^2$ with $\mu(Z_{0}) = \overline{(2,2,0) ~(2,2,1)}$ and $\mathrm{Vol}(Z_{0}) = 1$,
	           	 				\item $Z_1 = \mathrm{pt}$ with $\mu(Z_1) = (1,2,2)$,
	            					\item $Z_2 = S^2$ with $\mu(Z_2) = \overline{(0,2,0) ~(0,2,2)})$ and $\mathrm{Vol}(Z_{2}) = 2$.
	            				\end{itemize}
	           	 	           	 	\vs{0.3cm}
                          	
	           	 	\begin{figure}[H]
	           	 		\scalebox{1}{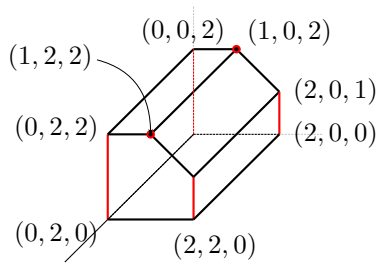}
	           	 		\caption{\label{figure_IV_2_2_2} $S^2 \times X_2$}
	           	 	\end{figure}
                          	                          	
	           	\item {\bf (IV-2-3)} \cite[26th in Section 12.4]{IP} : Consider $\p^3$ with the standard $T^3$-action and let $M$ be the $T^3$-equivariant blow-up of 
	           	$\p^3$ along a disjoint union of a fixed point and a $T^3$-invariant sphere. Then the moment map image of $M$ is described in Figure \ref{figure_IV_2_3}. 
	           	If we take a circle subgroup of $T^3$ generated by $\xi = (0,-1,-1)$, then the $S^1$-action becomes semifree and the fixed point set is give by 

	           	 	\begin{figure}[H]
	           	 		\scalebox{0.8}{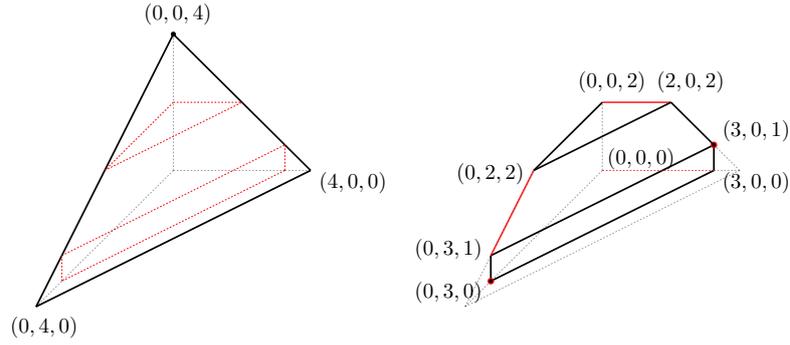}
	           	 		\caption{\label{figure_IV_2_3} Toric blow up of $\p^3$ along a fixed point and a $T^3$-invariant sphere}
	           	 	\end{figure}
	           	
                           				\begin{itemize}
	            					\item $Z_{-2} = S^2$ with  $\mu(Z_{-2}) = \overline{(0,2,2) ~(0,3,1)}$ and $\mathrm{Vol}(Z_{-2}) = 1$,
		            				\item $Z_{-1} = \mathrm{pt}$ with $\mu(Z_{-1}) = (0,3,0)$,
		            				\item $Z_{0} = S^2$ with $\mu(Z_{0}) = \overline{(0,0,2) ~(2,0,2)}$ and $\mathrm{Vol}(Z_{0}) = 2$,
	           	 				\item $Z_1 = \mathrm{pt}$ with $\mu(Z_1) = (3,0,1)$,
	            					\item $Z_2 = S^2$ with $\mu(Z_2) = \overline{(0,0,0) ~(3,0,0)})$ and $\mathrm{Vol}(Z_{2}) = 3$.
	            				\end{itemize}
	           	 	           	 	\vs{0.3cm}

	           	\item {\bf (IV-2-4)} \cite[29th in Section 12.4]{IP} : Consider $V_7$, the $T^3$-equivariant blow-up of $\p^3$ at a fixed point. (See also Example \ref{example_II_1} (2).) 
	           	Take $C$ be any $T^3$-invariant sphere lying on the exceptional divisor of the blow-up $V_7 \rightarrow \p^3$. Then the moment map image is given in 
	           	Figure \ref{figure_IV_2_4}. Take a circle subgroup generated by $\xi = (0,-1,-1)$. 
	           	 	\begin{figure}[H]
	           	 		\scalebox{0.7}{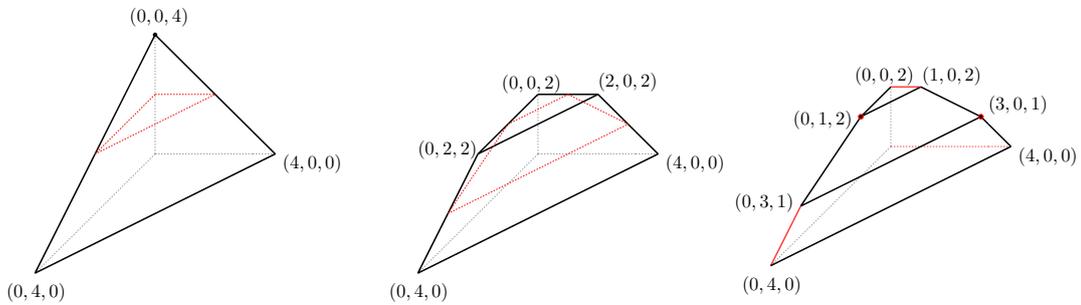}
	           	 		\caption{\label{figure_IV_2_4} Blow up of $V_7$ along a $T$-invariant sphere on the exceptional divisor}
	           	 	\end{figure}
	           	 	\noindent
			The $S^1$-action is semifree and the fixed point set consists of 
                           				\begin{itemize}
	            					\item $Z_{-2} = S^2$ with  $\mu(Z_{-2}) = \overline{(0,4,0) ~(0,3,1)}$ and $\mathrm{Vol}(Z_{-2}) = 1$,
		            				\item $Z_{-1} = \mathrm{pt}$ with $\mu(Z_{-1}) = (0,1,2)$,
		            				\item $Z_{0} = S^2$ with $\mu(Z_{0}) = \overline{(0,0,2) ~(1,0,2)}$ and $\mathrm{Vol}(Z_{0}) = 1$,
	           	 				\item $Z_1 = \mathrm{pt}$ with $\mu(Z_1) = (3,0,1)$,
	            					\item $Z_2 = S^2$ with $\mu(Z_2) = \overline{(0,0,0) ~(4,0,0)})$ and $\mathrm{Vol}(Z_{2}) = 4$.
	            				\end{itemize}
	           	 	           	 	\vs{0.3cm}
	           		           	
	           	\item {\bf (IV-2-5)} \cite[12th in Section 12.5]{IP} : We consider $Y$, the blow-up of $\p^3$ along a $T^3$-invariant line (see Example \ref{example_III}).
	           	Let $C_1$ and $C_2$ be two $T^3$-invariant disjoint lines lying on the exceptional divisor of $Y \rightarrow \p^3$. See Figure \ref{figure_IV_2_5} (a).
	           	Let $M$ be the $T^3$-equivariant blow-up of $Y$ along $C_1$ and $C_2$. Then the moment map image of the induced $T^3$-action is given by 
	           	Figure \ref{figure_IV_2_5}. 

	           	 	\begin{figure}[H]
	           	 		\scalebox{0.7}{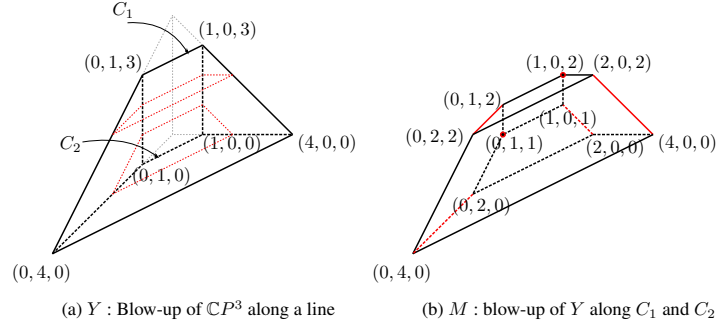}
	           	 		\caption{\label{figure_IV_2_5} Blow up of $Y$ along disjoint $T$-invariant two spheres on the exceptional divisor}
	           	 	\end{figure}
	           	
	           	Take an $S^1$ subgroup of $T^3$ generated by $\xi = (1,0,1)$. One can easily check that the $S^1$-action is semifree and the fixed point set is given by
                           				\begin{itemize}
	            					\item $Z_{-2} = S^2$ with  $\mu(Z_{-2}) = \overline{(0,4,0) ~(0,2,0)}$ and $\mathrm{Vol}(Z_{-2}) = 2$,
		            				\item $Z_{-1} = \mathrm{pt}$ with $\mu(Z_{-1}) = (0,1,1)$,
		            				\item $Z_{0} = S^2 ~\dot \cup ~ S^2$ with 
		            				\[
		            					\mu(Z_{0}^1) = \overline{(0,1,2) ~(0,2,2)},  \quad \mu(Z_{0}^2) = \overline{(1,0,1) ~(2,0,0)}, 
		            					\quad \quad \mathrm{Vol}(Z_{0}^1) = \mathrm{Vol}(Z_{0}^2) = 1,
							\]
	           	 				\item $Z_1 = \mathrm{pt}$ with $\mu(Z_1) = (1,0,2)$,
	            					\item $Z_2 = S^2$ with $\mu(Z_2) = \overline{(2,0,2) ~(4,0,0)})$ and $\mathrm{Vol}(Z_{2}) = 2$.
	            				\end{itemize}
	           	 	           	 	\vs{0.3cm}

	           	\item {\bf (IV-2-6)} \cite[30th in Section 12.4]{IP} : Consider the $T^3$-equivariant blow-up $V_7$ of $\p^3$ at a fixed point and let $M$ be the blow-up of $V_7$
	           	along a $T^3$-invariant sphere passing through the exceptional divisor of $V_7 \rightarrow \p^3$. Then the moment map image of $M$ with respect to the induced 
	           	action is given by Figure \ref{figure_IV_2_6}. 

	           	 	\begin{figure}[H]
	           	 		\scalebox{0.7}{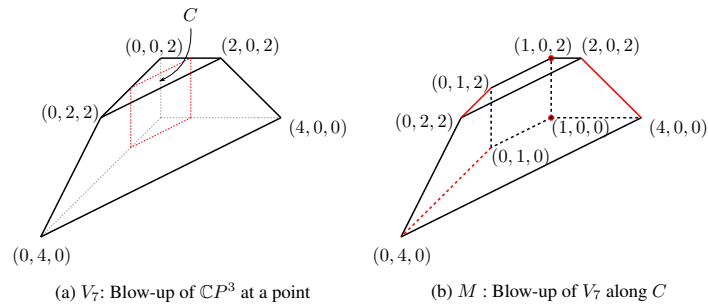}
	           	 		\caption{\label{figure_IV_2_6} Blow up of $V_7$ along a $T$-invariant sphere passing through the exceptional divisor}
	           	 	\end{figure}

			Take a circle subgroup of $T^3$ generated by $\xi = (-1,0,-1)$. Then the action is semifree and the fixed point set 
	           	consists of 
                           				\begin{itemize}
	            					\item $Z_{-2} = S^2$ with  $\mu(Z_{-2}) = \overline{(4,0,0) ~(2,0,2)}$ and $\mathrm{Vol}(Z_{-2}) = 2$,
		            				\item $Z_{-1} = \mathrm{pt}$ with $\mu(Z_{-1}) = (1,0,2)$,
		            				\item $Z_{0} = S^2$ with 
		            				$
		            					\mu(Z_{0}) = \overline{(0,1,2) ~(0,2,2)} 
							$ with $\mathrm{Vol}(Z_{0}^2) = 1$, 
	           	 				\item $Z_1 = \mathrm{pt}$ with $\mu(Z_1) = (1,0,0)$,
	            					\item $Z_2 = S^2$ with $\mu(Z_2) = \overline{(0,1,0) ~(0,4,0)})$ and $\mathrm{Vol}(Z_{2}) = 3$.
	            				\end{itemize}
	           	 	           	 	\vs{0.3cm}
	          \end{itemize}		
	\end{example}

	\begin{table}[H]
	\begin{adjustbox}{width=1\textwidth}
		\begin{tabular}{|c|c|c|c|c|c|c|c|c|c|}
			\hline
			         & $(M_0, [\omega_0])$           & $e(P_{-2+\epsilon})$ & $Z_{-2}$ & $Z_{-1}$ &  $Z_0$ & $Z_1$ & $Z_2$ & $b_2$ &  $c_1^3$ \\ \hline \hline
	          {\bf (I-1)}      & $(S^2 \times S^2, 2x + 2y)$ & $x-y$                         &  $S^2$    &               &              &            & $S^2$ & $1$ & $64$\\ \hline    
		{\bf (II-1.1)} & $(S^2 \times S^2, 2x + 2y)$ & $-y$                           &$S^2$ & &$Z_0 \cong S^2, ~\mathrm{PD}(Z_0) = x+y$ & & $S^2$ & $2$ &$48$ \\ \hline
		{\bf (II-1.2)} & $(S^2 \times S^2, 2x + 2y)$ & $-y$                           &$S^2$ & & $Z_0 \cong S^2, ~\mathrm{PD}(Z_0) = x$ & & $S^2$ & $2$ & $56$\\ \hline    
		{\bf (II-1.3)} & $(S^2 \times S^2, 2x + 2y)$ & $-y$                           &$S^2$ &  & \makecell{ $Z_0 = Z_0^1 ~\dot \cup ~ Z_0^2$ \\
						$Z_0^1 \cong Z_0^2 \cong S^2$ \\ $\mathrm{PD}(Z_0^1) = \mathrm{PD}(Z_0^2) = y$}   &  & $S^2$ & $3$ & $48$\\ \hline
		{\bf (II-2.1)} & $(E_{S^2}, 3x + 2y)$ & $-x -y$  &$S^2$ & &
				    	\makecell{ $Z_0 = Z_0^1 ~\dot \cup ~ Z_0^2$ \\
					    $Z_0^1 \cong Z_0^2 \cong S^2$ \\ $\mathrm{PD}(Z_0^1) = y$, $\mathrm{PD}(Z_0^2) = x+y$}  & & $S^2$ & $3$ & $48$\\ \hline    
		{\bf (II-2.2)} & $(E_{S^2}, 3x + 2y)$ & $-x-y$  &$S^2$ & & $Z_0 \cong S^2, ~\mathrm{PD}(Z_0) = 2x+2y$ &  & $S^2$ & $2$ &$40$ \\ \hline
		{\bf (III.1)} & \makecell{$(E_{S^2} \# ~\overline{\p^2},$ \\$3x + 2y - E_1)$} & $-y$  &$S^2$ & 
				    	{ pt} & &{pt} & $S^2$ & $2$ & $54$\\ \hline    
		{\bf (III.2)} & \makecell{$(S^2 \times S^2 \# ~2\overline{\p^2},$ \\ $2x + 2y - E_1 - E_2)$} & $-y$  &$S^2$ & 
				    	{2 pts} & &{2 pts}  & $S^2$ & $3$ & $44$\\ \hline    
		{\bf (III.3)} & \makecell{$(E_{S^2} \# ~\overline{\p^2},$ \\ $3x + 2y - E_1)$} & $-x-y$  &$S^2$ & {3 ~pts} & &{3 ~pts} & $S^2$ & $4$ &$34$ \\ \hline
		{\bf (IV-1-1.1)} & \makecell{$(E_{S^2} \# ~2\overline{\p^2},$ \\$3x + 2y - E_1-E_2)$} & $-x-y$  &$S^2$ & 
				    	{ 2 pts} &
				    		\makecell{ $Z_0 = Z_0^1 ~\dot \cup ~ Z_0^2$ \\ $Z_0^1 \cong Z_0^2 \cong S^2$ \\ 
				    		$\mathrm{PD}(Z_0^1) = x+y-E_1 - E_2$ \\ $\mathrm{PD}(Z_0^2) = x - E_1$}
					     & { 2 pts} & $S^2$ & $5$ & $36$\\ \hline    
		{\bf (IV-1-1.2)} & \makecell{$(E_{S^2} \# ~2\overline{\p^2},$ \\$3x + 2y - E_1-E_2)$} & $-x-y$  &$S^2$ & 
				    	{ 2 pts} &
				    		\makecell{ $Z_0 = Z_0^1 ~\dot \cup ~ Z_0^2$ \\ $Z_0^1 \cong Z_0^2 \cong S^2$ \\ $\mathrm{PD}(Z_0^1) = y$ \\ 
				    		$\mathrm{PD}(Z_0^2) = x+y-E_1 - E_2$}
					     & { 2 pts} & $S^2$ & $5$ & $36$\\ \hline    					     
		{\bf (IV-1-1.3)} & \makecell{$(E_{S^2} \# ~2\overline{\p^2},$ \\$3x + 2y - E_1-E_2)$} & $-x-y$  &$S^2$ & 
				    	{ 2 pts} &
				    		\makecell{ $Z_0 \cong S^2$  \\ $\mathrm{PD}(Z_0) = x+y-E_1$}
					     & { 2 pts} & $S^2$ & $4$ & $36$\\ \hline    
		{\bf (IV-1-2)} & \makecell{$(E_{S^2} \# ~2\overline{\p^2},$ \\$3x + 2y - E_1-E_2)$} & $-x-y$  &$S^2$ & 
				    	{ 2 pts} & 
				    		\makecell{ $Z_0 \cong S^2$  \\ $\mathrm{PD}(Z_0) = x - E_1$}
				    	& { 2 pts}  & $S^2$ & $4$ & $40$\\ \hline 
		{\bf (IV-2-1.1)} & \makecell{$(E_{S^2} \# ~\overline{\p^2},$ \\$3x + 2y - E_1)$} & $-x-y$  &$S^2$ & { pt} &
				    		\makecell{ $Z_0 \cong S^2$  \\ $\mathrm{PD}(Z_0) = 2x + y - E_1$}				    
				     &{ pt} & $S^2$ & $3$ &$38$ \\ \hline
		{\bf (IV-2-1.2)} & \makecell{$(E_{S^2} \# ~\overline{\p^2},$ \\$3x + 2y - E_1)$} & $-x-y$  &$S^2$ & { pt} &
				    		\makecell{ $Z_0 = Z_0^1 ~\dot \cup ~ Z_0^2$ \\ $Z_0^1 \cong Z_0^2 \cong S^2$ \\ 
				    		$\mathrm{PD}(Z_0^1) = \mathrm{PD}(Z_0^2) = x + y - E_1$}
				     &{ pt} & $S^2$ & $4$ &$38$ \\ \hline
		{\bf (IV-2-2.1)} & \makecell{$(E_{S^2} \# ~\overline{\p^2},$ \\$3x + 2y - E_1)$} & $-x-y$  &$S^2$ & { pt} & 
				    		\makecell{ $Z_0 \cong S^2$  \\ $\mathrm{PD}(Z_0) = x + y$}						    
				    &{ pt} & $S^2$ & $3$ &$42$ \\ \hline
		{\bf (IV-2-2.2)} & \makecell{$(E_{S^2} \# ~\overline{\p^2},$ \\$3x + 2y - E_1)$} & $-x-y$  &$S^2$ & { pt} & 
				    		\makecell{ $Z_0 = Z_0^1 ~\dot \cup ~ Z_0^2$ \\ $Z_0^1 \cong Z_0^2 \cong S^2$ \\ 
				    		$\mathrm{PD}(Z_0^1) = y$ \\ $\mathrm{PD}(Z_0^2)= x + y - E_1$}				    				    
				    &{ pt} & $S^2$ & $4$ &$42$ \\ \hline				    
		{\bf (IV-2-3)} & \makecell{$(E_{S^2} \# ~\overline{\p^2},$ \\$3x + 2y - E_1)$} & $-x-y$  &$S^2$ & { pt} &
				    		\makecell{ $Z_0 \cong S^2$  \\ $\mathrm{PD}(Z_0) = x$}				    
				     &{ pt} & $S^2$ & $3$ &$46$ \\ \hline
		{\bf (IV-2-4)} & \makecell{$(E_{S^2} \# ~\overline{\p^2},$ \\$3x + 2y - E_1)$} & $-x-y$  &$S^2$ & { pt} &
				    		\makecell{ $Z_0 \cong S^2$ 
				    		 \\ $\mathrm{PD}(Z_0) = E_1$}			    
				     &{ pt} & $S^2$ & $3$ &$50$ \\ \hline
		{\bf (IV-2-5)} & \makecell{$(S^2 \times S^2  \# ~\overline{\p^2},$ \\$2x + 2y - E_1)$} & $-y$  &$S^2$ & { pt} & 
				    		\makecell{ $Z_0 = Z_0^1 \dot \cup Z_0^2$ \\ $Z_0^1 \cong Z_0^2 \cong S^2$ \\ 
				    		$\mathrm{PD}(Z_0^1) = x - E_1$ \\  $\mathrm{PD}(Z_0^2) = y - E_1$  \\ }				    
				    &{ pt} & $S^2$ & $4$ &$46$ \\ \hline
		{\bf (IV-2-6)} & \makecell{$(S^2 \times S^2  \# ~\overline{\p^2},$ \\$2x + 2y - E_1)$} & $-y$  &$S^2$ & { pt} &
				    		\makecell{ $Z_0 \cong S^2$  \\ $\mathrm{PD}(Z_0) = x - E_1$}					    
				     &{ pt} & $S^2$ & $3$ &$50$ \\ \hline
		\end{tabular}
		\end{adjustbox}
		\vs{0.1cm}
		\caption {List of topological fixed point data for $\dim Z_{\min} = \dim Z_{\max} = 2$} \label{table_list_2} 
	\end{table}

\section{Classification of topological fixed point data : $\dim Z_{\min} \geq 2$ and $\dim Z_{\max} = 4$}
\label{secClassificationOfTopologicalFixedPointDataDimZMinGeq2AndDimZMax4}

In this section, we classify all topological fixed point data for the remaining cases where $\dim Z_{\min} \geq 2$ and $\dim Z_{\max} = 4$.

\subsection{Case I : $\dim Z_{\min} = 2$}
\label{ssecCaseIDimZMin2}

	We first note that $Z_{\min} \cong S^2$ by Lemma \ref{lemma_possible_critical_values} so that the reduced space $M_{-2 + \epsilon}$ 
	near the minimum is an $S^2$-bundle over $S^2$. Recall from \eqref{equation_basis_trivial} and \eqref{equation_basis_twist} that
	$x$ and $y$ are the dual classes of the fiber and the base of the bundle $M_{-2 + \epsilon} \rightarrow Z_{\min}$,
	respectively, such that 
	\[
		\begin{cases}
					\langle xy, [M_{-2 + \epsilon}] \rangle = 1, \quad \langle x^2, [M_{-2 + \epsilon}] \rangle = \langle y^2, [M_{-2 + \epsilon}] \rangle = 0 &
						\text{if $M_{-2 + \epsilon} \cong S^2 \times S^2$} \\
					\langle xy, [M_{-2 + \epsilon}] \rangle = 1, \quad \langle x^2, [M_{-2 + \epsilon}] \rangle = 0, \quad \langle y^2, [M_{-2 + \epsilon}] \rangle = -1 & 
						\text{if $M_{-2 + \epsilon} \cong E_{S^2}$} \\		
		\end{cases}
	\]
	where $E_{S^2}$ denotes the non-trivial $S^2$-bundle over $S^2$. 
	With these notations, we have 
	\begin{equation}\label{equation_chern_class}
		c_1(T(S^2 \times S^2)) = 2x + 2y, \quad c_1(TE_{S^2}) = 3x+2y.
	\end{equation}

	Observe that the only possible non-extremal critical values are $\{-1, 0\}$
	and each non-extremal fixed component $Z$ is either 
	\[
		\begin{cases}
			\text{$Z$ = pt} \hspace{1cm} \text{if $H(Z) = - 1$, \quad or} \\
			\text{$\dim Z = 2$} \quad \text{if $H(Z) = 0$.}
		\end{cases}
	\]
	by Lemma  \ref{lemma_possible_critical_values}. 
	
	Let $m = |Z_{-1}|$ be the number of isolated fixed points (of index two). Then $M_0 \cong M_1$ is the $k$ points blow-up of $M_{-2 + \epsilon}$
	by Proposition \ref{proposition_topology_reduced_space}. 
	Moreover, since $(M_0, \omega_0)$ is a monotone symplectic four manifold, $M_0$ is diffeomorphic to 
	a del Pezzo surface by \cite{OO2} and so $M_0$ is either $\p^2, S^2 \times S^2$, or $X_k$ for $k \leq 8$ 
	where $X_k$ denotes the blow-up of $\p^2$ at $k$ generic points. Thus we obtain
	\begin{equation}\label{equation_number_of_points}
		0 \leq m \leq 7
	\end{equation}
	Set
	\begin{itemize}
		\item $E_1, \cdots, E_m$ the dual classes of the exceptional divisors on $M_0$, 
		\item $\mathrm{PD}(Z_0) = ax + by + \sum_{i=1}^m c_i E_i$ the dual class of $Z_0$ in $M_0$ for some integers $a,b,c_1,\cdots,c_m$,
		\item $e(P_{-2 + \epsilon}) = kx - y$ (for $k \in \Z$) the Euler class of $P_{-2}^+$. (See Lemma \ref{lemma_volume}.)
	\end{itemize}
	Then Corollary \ref{corollary_volume} implies that
	\begin{equation}\label{equation_I_k}
		\begin{cases}
			k \geq -1 & ~\text{if $b_{\min}$ is odd} \\
			k \geq 0 & ~\text{if $b_{\min}$ is even} 
		\end{cases}
	\end{equation}

	We collect equations in $\{a,b,c_1, \cdots, c_m, k\}$ as follows. Since $c_1(TM_0) = [\omega_0]$ by Proposition \ref{proposition_monotonicity_preserved_under_reduction}, 
	we see that
	\begin{equation}\label{equation_I_symplectic_form}
		[\omega_0] = \begin{cases}
			3x + 2y - \sum_{i=1}^m E_i & ~\text{if $b_{\min}$ is odd,} \\ 
			2x + 2y - \sum_{i=1}^m E_i & ~\text{if $b_{\min}$ is even.} \\ 
		\end{cases}
	\end{equation}
	So, we have
	\begin{equation}\label{equation_I_volume}
		\mathrm{Vol}(Z_0) = \langle [\omega_0], [Z_0] \rangle = \begin{cases}
			2a + b + \sum_{i=1}^m c_i \geq 1 & ~\text{if $b_{\min}$ is odd,} \\ 
			2a + 2b + \sum_{i=1}^m c_i \geq 1& ~\text{if $b_{\min}$ is even.} \\ 
		\end{cases}
	\end{equation}
	Moreover. since $e(P_0^+) = kx - y + \sum_{i=1}^m E_i + \mathrm{PD}(Z_0) = (a+k)x + (b-1)y + \sum_{i=1}^m (c_i + 1)E_i$ by Lemma \ref{lemma_Euler_class}, we obtain
	\begin{equation}\label{equation_I_symplectic_form_level_t}
		[\omega_t] = [\omega_0] - e(P_0^+)t = 
		\begin{cases}
			(3 - (a + k)t)x + (2 + (1-b)t)y - \ds \sum_{i=1}^m (1 + (c_i+1) t)E_i & ~\text{if $b_{\min}$ is odd,} \\
			(2 - (a + k)t)x + (2 + (1-b)t)y - \ds \sum_{i=1}^m (1 + (c_i+1) t)E_i & ~\text{if $b_{\min}$ is even.} 
		\end{cases}
	\end{equation}
	for $t \in [0,1]$.
	In particular, we have
	\begin{equation}\label{equation_I_symplectic_form_level_1}
		[\omega_1] = [\omega_0] - e(P_0^+) = 
		\begin{cases}
			(3 - a - k)x + (3-b)y - \ds \sum_{i=1}^m (c_i + 2)E_i & ~\text{if $b_{\min}$ is odd,} \\
			(2 - a - k)x + (3-b)y - \ds \sum_{i=1}^m (c_i + 2)E_i & ~\text{if $b_{\min}$ is even.} 
		\end{cases}
	\end{equation}
	
	Note that the Chern number of $(M, \omega)$ can be calculated by the following lemma. 
		
	\begin{lemma}\label{lemma_I_Chern_number}
		Suppose that $\dim Z_{\min} = 2$ and $\dim Z_{\max} = 4$ and let $m = |Z_{-1}|$. Then
		\begin{equation}\label{equation_I_Chern}
			\int_M c_1(TM)^3 = 24 + 4b_{\min} -m + \langle 3c_1(TM_0)^2 - 3c_1(TM_0)e(P_0^+) + e(P_0^+)^2, [M_0] \rangle.
		\end{equation}
	\end{lemma}
	
	\begin{proof}
		Applying the ABBV localization theorem, we have
		\[
			\begin{array}{ccl}\vs{0.3cm}
				\ds \int_M c_1^{S^1}(TM)^3 & = &  \ds  
							\int_{Z_{\min}} \frac{\left(c_1^{S^1}(TM)|_{Z_{\min}}\right)^3}{e_{Z_{\min}}^{S^1}} + m \frac{\overbrace{\lambda^3}^{Z_{-1} ~\text{term}}}
							{-\lambda^3}
							+ \int_{Z_0} \frac{\overbrace{\left(c_1^{S^1}(TM)|_{Z_0}\right)^3}^{= 0}}{e_{Z_0}^{S^1}}
							 + \int_{Z_{\max}} \frac{\left(c_1^{S^1}(TM)|_{Z_{\max}}\right)^3}
							 {e_{Z_{\max}}^{S^1}} \\ \vs{0.2cm}
							 & = & \ds \int_{Z_{\min}} \frac{(2\lambda + (2+b_{\min})u)^3}{\lambda^2 + b_{\min}u\lambda} - m + 
							 \int_{Z_{\max}} \frac{\left(c_1^{S^1}(TM)|_{Z_{\max}}\right)^3} {e_{Z_{\max}}^{S^1}} \\ \vs{0.2cm}
							& = &  (24 + 4b_{\min}) - m + \ds \int_{M_0} \frac{\left(c_1(TM_0) - e - \lambda \right)^3}{-e - \lambda}\quad \quad 
							(c_1(TZ_{\max}) = c_1(TM_0), ~e := e(P_0^+) = e(P_1^-)) \\ \vs{0.2cm}
							& = & 24 + 4b_{\min} - m + \ds \int_{M_0} \frac{\left(c_1(TM_0) - e - 
							\lambda \right)^3(e^2 - e\lambda + \lambda^2)}{-(e+\lambda)(e^2 - e\lambda + \lambda^2)}
							\\ \vs{0.2cm}
							& = & 24 + 4b_{\min} - m + \langle 3c_1(TM_0)^2 - 3c_1(TM_0)e) + e^2, [M_0] \rangle.
			\end{array}			
		\]
	\end{proof}

	We divide into four cases: $\mathrm{Crit} ~H = \underbrace{\{1,-2\}}_{\text{({\bf I-1})}}$ , $\underbrace{\{1,0, -2\}}_{\text{({\bf I-2})}}$, 
	$\underbrace{\{1,-1, -2\}}_{\text{({\bf I-3})}}$, and $\underbrace{\{1,0,-1,-2\}}_{\text{({\bf I-4})}}$.

	\begin{theorem}[Case {\bf I-1}]\label{theorem_Sec8_I_1}
		Let $(M,\omega)$ be a six-dimensional closed monotone semifree Hamiltonian $S^1$-manifold such that 
		$\mathrm{Crit} H = \{ 1, -2\}$. 
		Then the list of all possible topological fixed point data is given in Table \ref{table_Sec8_I_1}
		\begin{table}[h]
			\begin{tabular}{|c|c|c|c|c|c|c|}
				\hline
				    & $(M_0, [\omega_0])$ & $e(P_{-2}^+)$ & $Z_{-2}$ & $Z_1$ & $b_2(M)$ & $c_1^3(M)$ \\ \hline \hline
				    {\bf (I-1-1.1)} & $(E_{S^2}, 3x + 2y)$ & $-x-y$  & $S^2$ & 
				    	$E_{S^2}$
					     & $2$ & $54$ \\ \hline    
				    {\bf (I-1-2.1)} & $(S^2 \times S^2, 2x + 2y)$ & $-y$  & $S^2$ & 
				    	$S^2 \times S^2$
					     & $2$ & $54$ \\ \hline    
				    {\bf (I-1-2.2)} & $(S^2 \times S^2, 2x + 2y)$ & $x-y$  & $S^2$ & 
				    	$S^2 \times S^2$
					     & $2$ & $54$ \\ \hline    
			\end{tabular}		
			\vs{0.5cm}			
			\caption{\label{table_Sec8_I_1} Topological fixed point data for $\mathrm{Crit} H = \{1, -2\}$.}
		\end{table}				   
	\end{theorem}

	\begin{proof}
		 We divide the proof into two cases: $M_0 \cong E_{S^2}$ {\bf (I-1-1)} and $M_0 \cong S^2 \times S^2$ {\bf (I-1-2)}.
		\vs{0.3cm}
		
	\noindent	
	{\bf (I-1-1)}: $M_0 \cong E_{S^2}$. \vs{0.3cm}
	
	\noindent
		In this case, we have $b_{\mathrm{min}} = 2k + 1$ for some $k \in \Z$ by Lemma \ref{lemma_volume}.
		Recall that 
		\[
			[\omega_t] = (3 - kt)x + (2 + t)y, \quad t \in [0,1]
		\]
		by \eqref{equation_I_symplectic_form_level_t}.
		If $k \geq 0$, then 
		\[
			\langle [\omega_t], y \rangle = 1 - (k+1)t
		\]
		and a blow-down occurs at level $0< t = \frac{1}{k+1} \leq 1$ where the vanishing exceptional divisor corresponds to $y$. 
		Since $M_0 \cong M_t \cong M_1$, this cannot be happened and therefore $k \leq -1$ which implies that $k=-1$ 
		by \eqref{equation_I_k}.
		See Table \ref{table_Sec8_I_1}: {\bf (I-1-1.1)}. The Chern number is computed from Lemma \ref{lemma_I_Chern_number} by
		\[
			\int_M c_1(TM)^3 = 24 + (-4) + 0 + 3\cdot 8 - 3 (-3) + 1 = 54.
		\]
	\vs{0.3cm}
	
	\noindent
	{\bf (I-1-2)}: $M_0 \cong S^2 \times S^2$. \vs{0.3cm}

	\noindent
		In this case, $b_{\min}$ is even by Lemma \ref{lemma_volume} so that $b_{\mathrm{min}} = 2k$ 
		for some $k \in \Z$. Since $[\omega_1] = (2 - k)x + 3y$ by \eqref{equation_I_symplectic_form_level_1}, 
		\[
			\int_{M_1} \omega_1^2 = 6(2-k) > 0. 
		\]
		Also since $k \geq 0$ by \eqref{equation_I_k}, we have $k=0$ or $1$. 
		See {\bf (I-1-2.1)} and {\bf (I-1-2.2)} in Table \ref{table_Sec8_I_1} for $k=0$ and $k=1$, respectively.
		The Chern numbers for each cases immediately follow from Lemma \ref{lemma_I_Chern_number}.
	\end{proof}

	\begin{example}[Fano varieties of type {\bf (I-1)}]\label{example_Sec8_I_1} We describe Fano varieties of type {\bf (I-1)} listed in Theorem \ref{theorem_Sec8_I_1}
	as follows.
		
		\begin{itemize}

			\item {\bf (I-1-1.1)} \cite[No.33 in Section 12.3]{IP} : Let $M$ be the blow-up of $\p^3$ along a $T^3$-invariant line and the circle action generated 
			by $\xi = (0,0,-1) \in \frak{t}^* \cong \R^3$. 
			
				 \begin{figure}[H]
	           	 		\scalebox{0.8}{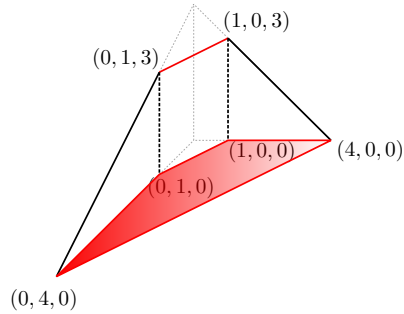}
		           	 	\caption{\label{figure_I_1_2_2} Blow-up of $\p^3$ along a $T^3$-invariant line}
		           	 \end{figure}
			\noindent
			As in Figure \ref{figure_I_1_2_2}, the volume of $Z_{\min} = 1$ (or equivalently $b_{\min} = -1$ by 
			Corollary \ref{corollary_volume}). Also, the volume of $Z_{\max}$ equals $15 (= 16 - 1)$ and so it coincides with 
			$\langle [\omega_1]^2, [Z_1] \rangle = \rangle (4x + 3y)^2, [Z_1] \rangle = 15$. Thus one can easily check that the topological fixed point data 
			of the action on $M$ coincides with {\bf (I-1-1.1)} in Table \ref{table_Sec8_I_1}.
			See also Example \ref{example_Sec6_3_3} and Example \ref{example_III}.\vs{0.3cm}
					           	 
	           	 \item {\bf (I-1-2.1)} \cite[No.34 in Section 12.3]{IP}  : Consider $M = \p^1 \times \p^2$ with the standard Hamiltonian $T^3$-action. Take the circle subgroup 
	           	 of $T^3$ generated by $\xi = (1,1,0)$. Then the fixed point set can be described as in Figure \ref{figure_I_1_1_1} (colored by red). In this case, the volume 
	           	 of $Z_{\min} = 2$ so that $b_{\min} = 0$ (i.e., $k=0$). Thus this example corresponds to {\bf (I-1-2.1)} in Table \ref{table_Sec8_I_1}.
			See also Example \ref{example_Sec6_2_3} (2). 
			
	           		 \begin{figure}[H]
	           	 		\scalebox{0.8}{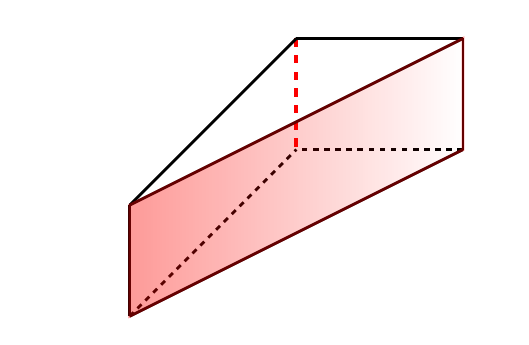}
		           	 	\caption{\label{figure_I_1_1_1} $\p^1 \times \p^2$}
		           	 \end{figure}

			\item {\bf (I-1-2.2)} \cite[No.33 in Section 12.3]{IP} : Let $M$ be the same as {\bf (I-1-1.1)}, that is, the blow-up of $\p^3$ along a $T^3$-invariant line. 
			In this case, we take another circle subgroup of $T^3$ generated by $\xi = (-1-1,0)$. Then the fixed point set can be described as in Figure \ref{figure_I_1_2_1}.
			Also, since the volume of $Z_{\min} = 4$, we see that $b_{\min} = 2$ by Corollary \ref{corollary_volume}, and hence we have $k=1$. 
			See also {\bf (I-1-1.1)}, Example \ref{example_Sec6_3_3} and Example \ref{example_III}.
			
	           		 \begin{figure}[H]
	           	 		\scalebox{0.8}{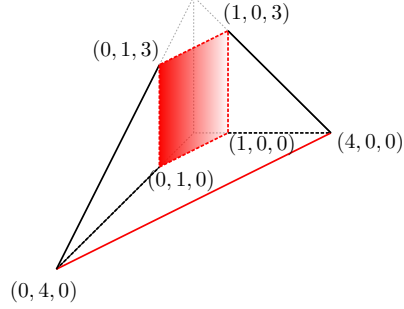}
		           	 	\caption{\label{figure_I_1_2_1} Blow-up of $\p^3$ along a $T^3$-invariant line}
		           	 \end{figure}
			
		\end{itemize}		
	\end{example}

	\begin{theorem}[Case {\bf I-2}]\label{theorem_Sec8_I_2}
		Let $(M,\omega)$ be a six-dimensional closed monotone semifree Hamiltonian $S^1$-manifold such that $\mathrm{Crit} H = \{ 1, -1, -2\}$. 
		Then there is only one possible topological fixed point data as in Table \ref{table_Sec8_I_2}
		\begin{table}[h]
			\begin{tabular}{|c|c|c|c|c|c|c|c|}
				\hline
				    & $(M_0, [\omega_0])$ & $e(P_{-2}^+)$ & $Z_{-2}$ & $Z_{-1}$ & $Z_1$ & $b_2(M)$ & $c_1^3(M)$ \\ \hline \hline
				    {\bf (I-2)} & $(E_{S^2}, 3x + 2y - E_1)$ & $-x-y$  & $S^2$ & $\mathrm{pt}$ & 
				    	$E_{S^2} \# ~\overline{\p^2}$
					     & $3$ & $46$ \\ \hline    
			\end{tabular}		
			\vs{0.5cm}			
			\caption{\label{table_Sec8_I_2} Topological fixed point data for $\mathrm{Crit} H = \{1, -1, -2\}$.}
		\end{table}				   
	\end{theorem}

	\begin{proof}
		Similar to the proof of Theorem \ref{theorem_I_1}, we divide into two cases: $M_0 \cong E_{S^2}$ and $M_0 \cong S^2 \times S^2$. \vs{0.3cm}
		Note that $m > 0$ by our assumption.
	\noindent	
	{\bf Case (1)}: $M_0 \cong E_{S^2}$. \vs{0.3cm}
	
	\noindent
	Let $b_{\mathrm{min}} = 2k + 1$ for some $k \in \Z$ (by Lemma \ref{lemma_volume}). If $m > 1$, then $x + y - E_1 - E_2$ becomes an exceptional class by Lemma
	 \ref{lemma_list_exceptional}
	and the volume of the corresponding exceptional divisor $C$ on $M_1$ equals 
	\[
		\langle [\omega_1], [C] \rangle = \langle (x + y - E_1 - E_2) \cdot ((3-k)x + 3y - 2E_1 - 2E_2), [M_1] \rangle  = -1 - k \leq 0
	\]
	by \eqref{equation_I_k} and \eqref{equation_I_symplectic_form_level_t}.
	That is, the blow-down occurs at some level $t \leq 1$. (Indeed, the blow-down occurs at level $t = \frac{1}{2+k}$ by \eqref{equation_I_symplectic_form_level_t}.)
	This is impossible since $M_0 \cong M_1$ and hence we have $m = 1$. Moreover, the symplectic volume of the exceptional class $y$ on $M_1$ 
	should be positive and equal to 
	\[
		\langle y \cdot ((3-k)x + 3y - 2E_1 - 2E_2), [M_1] \rangle = -k >0. 
	\]
	Thus we get $k = -1$ by \eqref{equation_I_k}. See Table \ref{table_Sec8_I_2}: {\bf (I-2)}.\vs{0.3cm}
	
	\noindent
	{\bf Case (2)}: $M_0 \cong S^2 \times S^2$. \vs{0.3cm}
	
	\noindent
	Let $b_{\mathrm{min}} = 2k$ (for some $k \geq 0$ by \eqref{equation_I_k}) by Lemma \ref{lemma_volume}.  
	In this case, we observe that 
	\[
		\langle (y - E_1) \cdot [\omega_1], [M_1] \rangle = \langle (y - E_1) \cdot ((2 - k)x + 3y - \sum_{i=1}^m 2 E_i), [M_1] \rangle = -k \leq 0
	\]
	so that the exceptional divisor representing $y - E_1$ vanishes at $M_t$ where $t = \frac{1}{k+1} \leq 1$. Thus no such manifold exists.

	\end{proof}

	\begin{example}[Fano variety of type {\bf (I-2)}]\label{example_Sec8_I_2}\cite[No.26 in Section 12.4]{IP}
	Consider the standard $T^3$-action on $\p^3$ and let $M$ be the blow-up of $\p^3$ along a fixed point $p$ and a $T^3$-invariant sphere not containing $p$. The induced $T^3$-action
	on $M$ admits a moment map whose image is given in Figure \ref{figure_I_2}. Taking a circle subgroup of $T^3$ generated by $\xi = (-1,0,0)$, we obtain a semifree Hamiltonian 
	$S^1$-action on $M$ whose fixed point set is depicted by red faces as in Figure \ref{figure_I_2}. 
	
	The volume of $Z_{\min}$ equals the length of the red edge, 1, which also equals $b_{\min} + 2$  by Corollary \ref{corollary_volume}. So, one can see that the fixed point data
	coincides with {\bf (I-2)} in Table \ref{table_Sec8_I_2}. See also Example \ref{example_Sec6_3_4} (3). 
			           	 
		\begin{figure}[H]
	           	\scalebox{0.8}{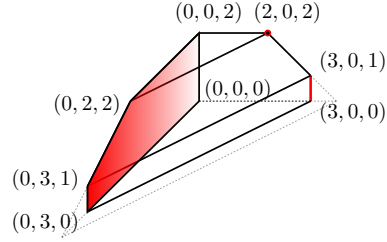}
			\caption{\label{figure_I_2} Blow-up of $\p^3$ along a point and a sphere}
		\end{figure}
			
	\end{example}

	\begin{theorem}[Case {\bf (I-3-1)}]\label{theorem_Sec8_I_3_1}
		Let $(M,\omega)$ be a six-dimensional closed monotone semifree Hamiltonian $S^1$-manifold such that $\mathrm{Crit} H = \{ 1, 0, -2\}$. 
		If $M_0 \cong E_{S^2}$, then the list of all possible topological fixed point data is given in the Table \ref{table_Sec8_I_3_1}.
		\begin{table}[h]		
		\begin{adjustbox}{max width=.9\textwidth}		
			\begin{tabular}{|c|c|c|c|c|c|c|c|}
				\hline
				    & $(M_0, [\omega_0])$ & $e(P_{-2}^+)$ & $Z_{-2}$ & $Z_0$ & $Z_1$ & $b_2(M)$ & $c_1^3(M)$ \\ \hline \hline
				    {\bf (I-3-1.1)} & $(E_{S^2}, 3x + 2y)$ & $-y$  & $S^2$ & \makecell{$Z_0 \cong S^2$ \\ $\mathrm{PD}(Z_0) = y$} & $E_{S^2}$ & $3$ & $52$ \\ \hline    
				    {\bf (I-3-1.2)} & $(E_{S^2}, 3x + 2y)$ & $-x-y$  & $S^2$ & \makecell{$Z_0 \cong S^2$ \\ $\mathrm{PD}(Z_0) = y$} 
				    															& $E_{S^2}$ & $3$ & $50$ \\ \hline    
				    {\bf (I-3-1.3)} & $(E_{S^2}, 3x + 2y)$ & $-y$  & $S^2$ &  
				    				\makecell{$Z_0 \cong S^2 ~\dot\cup ~S^2$ \\ $\mathrm{PD}(Z_0^1) = x+y$ \\ $\mathrm{PD}(Z_0^2) = y$} 				    
				    															& $E_{S^2}$ & $4$ & $44$ \\ \hline    
				    {\bf (I-3-1.4)} & $(E_{S^2}, 3x + 2y)$ & $-x-y$  & $S^2$ & \makecell{$Z_0 \cong S^2$ \\ $\mathrm{PD}(Z_0) = x + y$} 
				    								& $E_{S^2}$ & $3$ & $44$ \\ \hline    
				    {\bf (I-3-1.5)} & $(E_{S^2}, 3x + 2y)$ & $-x-y$  & $S^2$ & 
				    							\makecell{$Z_0 \cong S^2 ~\dot\cup ~S^2$ \\ $\mathrm{PD}(Z_0^1) = x+y$ \\ $\mathrm{PD}(Z_0^2) = y$} 
				    								& $E_{S^2}$ & $4$ & $40$ \\ \hline    
				    {\bf (I-3-1.6)} & $(E_{S^2}, 3x + 2y)$ & $-x-y$  & $S^2$ & \makecell{$Z_0 \cong S^2$ \\ $\mathrm{PD}(Z_0) = 2x + 2y$} & $E_{S^2}$ & $3$ & $36$ \\ \hline    
			\end{tabular}		
			\end{adjustbox}
			\vs{0.5cm}			
			\caption{\label{table_Sec8_I_3_1} Topological fixed point data for $\mathrm{Crit} H = \{1, 0, -2\}$ with $M_0 \cong E_{S^2}$.}
		\end{table}				   
	\end{theorem}

	\begin{proof}
		Let $\mathrm{PD}(Z_0) = ax + by \in H^2(M_0; \Z)$. \vs{0.3cm}
		Observe that, since $[\omega_1] = (3 - a - k)x + (3-b)y$ by \eqref{equation_I_symplectic_form_level_1}, we have
		\[
			\underbrace{a+k \leq 2 \text{\hs{0.1cm} and \hs{0.1cm}} b \leq 2}_{\langle [\omega_1]^2,  [M_1] \rangle > 0}, \quad 
			\underbrace{b \geq a + k + 1}_{\langle [\omega_1]\cdot y, [M_1] \rangle \geq 1}, \quad \underbrace{2a + b \geq 1}_{\mathrm{Vol}(Z_0) \geq 1}, \quad 
			\underbrace{k \geq -1}_{\text{by \eqref{equation_I_k}}}.
		\]
		From the inequalities above, we have $a + k \leq 1$, $0 \leq a \leq 2$, and $k \geq -1$. Thus we have 9 solutions: 
		\[
			a = 0: \quad \begin{cases}
				k = 1 & b = 2 \\
				k = 0 & b = 1, 2 \\
				k = -1 & b = 1,2 \\
			\end{cases},
			\quad \quad 
			a = 1: \quad \begin{cases}
				k = 0 & b = 2 \\
				k = -1 & b = 1, 2 \\
			\end{cases},
			\quad \quad 
			a = 2: \quad \begin{cases}
				k = -1 & b = 2.
			\end{cases}						
		\]
		On the other hand, for the case of $a = 0$ and $b=2$, we have 
		\begin{itemize}
			\item $\langle c_1(TM_0), [Z_0] \rangle = \mathrm{Vol}(Z_0) = 2$ and so there are at most two connected components in $Z_0$, and 
			\item $[Z_0] \cdot [Z_0] = -4$ (so that there are at least three sphere components by the adjunction formula)
		\end{itemize}
		which contradict to each other. Therefore, there are 6 possible cases: 
			\begin{table}[H]
				\begin{tabular}{|c|c|c|c|c|c|c|}
					\hline
					    & {\bf (I-3-1.1)} & {\bf (I-3-1.2)} & {\bf (I-3-1.3)} & {\bf (I-3-1.4)} & {\bf (I-3-1.5)} & {\bf (I-3-1.6)} \\ \hline \hline
					    $(a,b,k)$ & $(0,1,0)$ & $(0,1,-1)$  & $(1,2,0)$ & $(1,1,-1)$ & $(1,2,-1)$ & $(2,2,-1)$ \\ \hline
					    $[Z_0]\cdot[Z_0]$ & $-1$ & $-1$ & $0$ & $1$ & $0$ & $4$ \\ \hline
					    \makecell{$\langle c_1(TM_0), [Z_0] \rangle$ \\ $= ~\mathrm{Vol}(Z_0)$}  & $1$ & $1$ & $4$ & $3$ & $4$ & $6$ \\ \hline
				\end{tabular}		
			\end{table}
		\noindent
		Applying the adjunction formula and Lemma \ref{lemma_list_exceptional} to each cases, we obtain Table \ref{table_Sec8_I_3_1}. 
		We only provide the proof for the case {\bf (I-3-1.3)} where the other five cases can be proved in a similar way and so we left it to the reader.
		
		In case of {\bf (I-3-1.3)}, we have $\mathrm{Vol}(Z_0) = 4$ and $[Z_0] \cdot [Z_0] = 0$ so that 
		there are at least two sphere components, namely $Z_0^1$ and $Z_0^2$, such that $\mathrm{Vol}(Z_0^1) + \mathrm{Vol}(Z_0^2) \leq 4$.
		Then,
		\begin{itemize}
			\item  $Z_0^1$ and $Z_0^2$ cannot have volume one simultaneously (otherwise $Z_0^1$ and $Z_0^2$ represent $-1$ spheres so that 
			$\mathrm{PD}(Z_0^1) = \mathrm{PD}(Z_0^2) = y$ by Lemma \ref{lemma_list_exceptional}
				and this violates the fact that $[Z_0^1] \cdot [Z_0^2] = 0$),
			\item  any $Z_0^i$ cannot have volume two. Otherwise, $[Z_0^i] \cdot [Z_0^i]  = 0$ by the adjunction formula and so one would obtain 
				either $\mathrm{PD}(Z_0^i) = c(x + 2y)$ for some integer $c$ (where the volume of $Z_0^i$ becomes $4c$) or 
				$\mathrm{PD}(Z_0^i) = x$ (so that $(\mathrm{PD}(Z_0) - \mathrm{PD}(Z_0^i)) \cdot \mathrm{PD}(Z_0^i) = 2y \cdot x \neq 0$).
		\end{itemize}
		Therefore, the only possibility is that  $\mathrm{Vol}(Z_0^1) = 1$ and  $\mathrm{Vol}(Z_0^2) = 3$
		and hence $\mathrm{PD}(Z_0^1) = y$ and $\mathrm{PD}(Z_0^2) = x + y$. (See Table \ref{table_Sec8_I_3_1} {\bf (I-3-1.3)}.)

	\end{proof}

	\begin{example}[Fano variety of type {\bf (I-3-1)}]\label{example_Sec8_I_3_1}  We describe Fano varieties of type {\bf (I-3-1)} in Theorem \ref{theorem_Sec8_I_3_1} as follows. 
		
		\begin{itemize}
	           	 \item {\bf (I-3-1.1)} \cite[No.31 in Section 12.4]{IP}  : Let $M = \mathbb{P}(\mcal{O} \oplus \mcal{O}(1,1))$ equipped with a $T^3$-action whose moment polytope is given 
	           	 in Figure \ref{figure_I_3_1_1}. If we consider a circle subgroup of $T^3$ generated by $\xi = (1,0,1)$, then the $S^1$-action becomes semifree and has a fixed point set 
	           	 where its image is colored by red. The volume of $Z_{\min}$ equals three as in Figure \ref{figure_I_3_1_1} and so we have $b_{\min} = 1$ (i.e., $k = 0$). 
	           	 Also, $Z_0 \cong S^2$
	           	 has volume $1$, and therefore the fixed point data equals {\bf (I-3-1.1)} in Table \ref{table_Sec8_I_3_1}.  See also Example \ref{example_Sec6_1_3}.
	           	 
	           		 \begin{figure}[H]
	           	 		\scalebox{0.8}{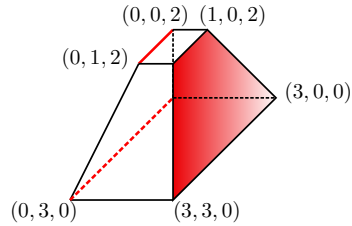}
		           	 	\caption{\label{figure_I_3_1_1} $M = \mathbb{P}(\mcal{O} \oplus \mcal{O}(1,1))$}
		           	 \end{figure}	 	
	           	 	
			\item {\bf (I-3-1.2)} \cite[No.30 in Section 12.4]{IP} : Consider $V_7$, the blow-up of $\p^3$ at a fixed point and let $M$ be the blow-up of $V_7$ along a $T^3$-invariant sphere 
			$C$ (depicted in Figure \ref{figure_I_3_1_2} (a)) passing through the exceptional divisor of $V_7 \rightarrow \p^3$. A moment polytope associated to the induced $T^3$-action on $M$ 
			is given by Figure \ref{figure_I_3_1_2} (b). 
			
			Take a circle subgroup generated by $\xi = (1,1,1)$. The fixed point set for the $S^1$-action corresponds to the faces painted by red. The volume of $Z_{\min}$ and $Z_0$ are both
			one (and hence $b_{\min} = -1$). So, the fixed point data for the $S^1$-action should coincides with {\bf (I-3-1.2)} in Table \ref{table_Sec8_I_3_1}. See also 
			Example \ref{example_Sec6_3_4} (2) and Example \ref{example_IV_2}.
	           	 
	           		 \begin{figure}[H]
	           	 		\scalebox{0.8}{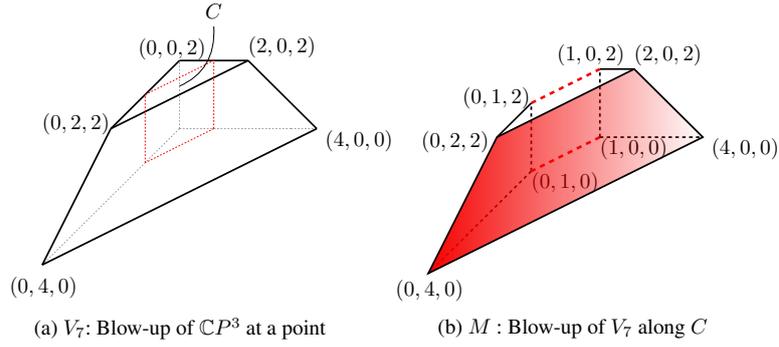}
		           	 	\caption{\label{figure_I_3_1_2} Blow-up of $V_7$ along a $T^2$-invariant sphere passing through the exceptional divisor}
		           	 \end{figure}					
		           
			\item {\bf (I-3-1.3)} \cite[No.11 in Section 12.5]{IP} : Consider the toric variety $\p^1 \times X_1$ and let $M$ be the blow-up of $\p^1 \times X_1$ along a 
			$T^3$-invariant sphere $t \times E$ where $t \in \p^1$ is a fixed point (for the $S^1$-action on $\p^1$) and $E$ is the exceptional curve in $X_1$. Then Figure \ref{figure_I_3_1_3}
			is a moment polytope for the induced $T^3$-action on $M$. 

	           		 \begin{figure}[H]
	           	 		\scalebox{1}{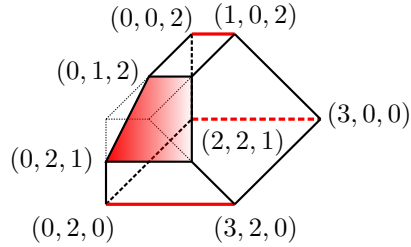}
		           	 	\caption{\label{figure_I_3_1_3} Blow-up of $\p^1 \times X_1$ along $t \times E$}
		           	 \end{figure}					
			
			Now, consider the circle subgroup of $T^3$ generated by $\xi = (0,1,1)$. Then the moment map image of the fixed point set for the $S^1$-action is red faces
			as in Figure \ref{figure_I_3_1_3} (where $Z_0^1$ and $Z_0^2$ correspond to $\overline{(0,0,2)~(1,0,2)}$ and $\overline{(0,2,0)~(3,2,0)}$, respectively).
			Also, we have $\mathrm{Vol}(Z_{\min}) = 3$ (i.e., $b_{\min} =1$ and $k=0$), $\mathrm{Vol}(Z_0^1) = 1$, and $\mathrm{Vol}(Z_0^2) = 3$ so that the fixed point data is exactly the same as 
			{\bf (I-3-1-2)} in Table \ref{table_Sec8_I_3_1}. See also Example \ref{example_Sec6_2_4} (1). \vs{0.3cm}

	           	 \item {\bf (I-3-1.4)} \cite[No.25 in Section 12.4]{IP} : Consider $\p^3$ with the standard $T^3$-action and let $Y$ be the blow-up of $\p^3$ with two disjoint $T^3$-invariant spheres.
	           	 Then the induced $T^3$-action has a moment polytope illustrated in Figure \ref{figure_I_3_1_4}. 

	           		 \begin{figure}[H]
	           	 		\scalebox{0.8}{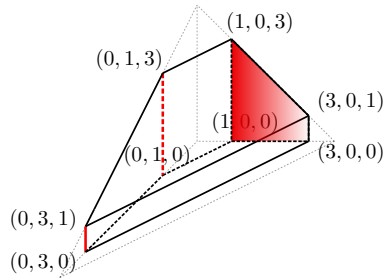}
		           	 	\caption{\label{figure_I_3_1_4} $M$ : blow-up of $\p^3$ along two disjoint curves}
		           	 \end{figure}

	           	 Take the circle subgroup generated by $\xi = (0,-1,0)$. Then the fixed point set for the $S^1$-action corresponds to red faces in Figure \ref{figure_I_3_1_4}. 
	           	 One can easily see that $\mathrm{Vol}(Z_{\min}) = 1$ (so that $b_{\min} = -1$ and $k = -1$) and also $\mathrm{Vol}(Z_0) = 3$. Thus the fixed point data 
	           	 coincides with {\bf (I-3-1.4)} in Table \ref{table_Sec8_I_3_1}. 
	           	 See also Example \ref{example_III}. \vs{0.3cm}

	           	 \item {\bf (I-3-1.5)} \cite[No.9 in Section 12.5]{IP} : Let $M$ be the blow-up of $\p^2$ along two disjoint $T^3$-invariant spheres given in {\bf (I-3-1.4)}. For the induced $T^3$-action, 
	           	 let $\widetilde{M}$ be the blow-up of $M$ along a $T^3$-invariant curve (labeled by $A$ in Figure \ref{figure_I_3_1_5})
	           	 lying on an exceptional divisor where the corresponding moment polytope is given on the right in Figure \ref{figure_I_3_1_5}. 
	      
		     		 \begin{figure}[H]
	           	 		\scalebox{0.8}{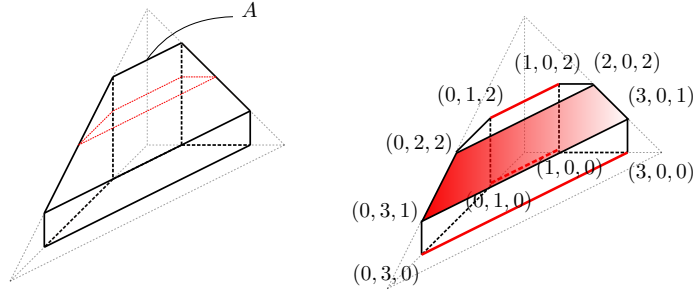}
		           	 	\caption{\label{figure_I_3_1_5} Blow-up of $M$ (in {\bf (I-3-1.4)}) along a $T^3$-invariant exceptional sphere}
		           	 \end{figure}
	      	            \vs{-0.5cm}	 
	           	 Take a circle subgroup of $T^3$ generated by $\xi = (1,1,1)$. Then the minimal fixed component $Z_{\min}$ has volume 1 so that $b_{\min} = -1$ and $k = -1$. Also the volume of 
	           	 $Z_0^1$ and $Z_0^2$ corresponding to $\overline{(0,1,0)~(1,0,0)}$ and $\overline{(3,0,0)~(0,3,0)}$ are respectively $1$ and $3$. Thus the fixed point data for the $S^1$-action
	           	 should be equal to {\bf (I-3-1.5)} in Table \ref{table_Sec8_I_3_1}.
	           	 See also Example \ref{example_IV_1} (4).
	           	 \vs{0.3cm}
	      
	           	 \item {\bf (I-3-1.6)} \cite[No.18 in Section 12.4]{IP} : Consider the closed monotone semifree Hamiltonian $S^1$-manifold $(M,\omega)$ given in 
	           	 Example \ref{example_Sec8_I_1}  {\bf (I-1-1.1)}. Then the maximal fixed component $Z_{\max}$ is the one-point blow-up of $\p^2$. 
	           	 
	           	 Let $C$ be a conic in $Z_{\max}$ disjoint from the exceptional divisor and $\widetilde{M}$ be the $S^1$-equivariant blow-up of $M$ along $C$. Then one can check 
	           	 that the induced 
	           	 $S^1$-action is semifree. Since the blow-up does not affect outside the neighborhood of the exceptional divisor, we have $\mathrm{Vol}(\widetilde{Z}_{\min}) = 1$ and so 
	           	 $b_{\min} = -1$ and $k = -1$. Also, since $\widetilde{Z}_0$ is a conic in $M_0 \cong \p^2$, the volume of $\widetilde{Z}_0$ equals six. Thus the fixed point data for 
	           	 the $S^1$-action
	           	 coincides with {\bf (I-3-1.6)} in Table \ref{table_Sec8_I_3_1}. See figure \ref{figure_I_3_1_6}.
	           	 
	           		 \begin{figure}[H]
	           	 		\scalebox{0.8}{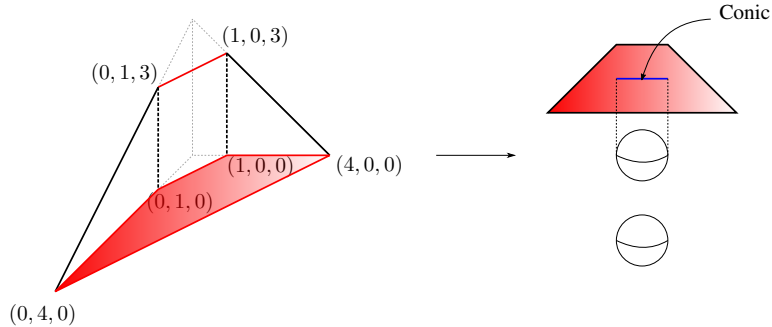}
		           	 	\caption{\label{figure_I_3_1_6} $M$ : blow-up of $\p^3$ along a disjoint union of a line and a conic}
		           	 \end{figure}

		\end{itemize}		
	\end{example}

	\begin{theorem}[Case {\bf (I-3-2)}]\label{theorem_Sec8_I_3_2}
		Let $(M,\omega)$ be a six-dimensional closed monotone semifree Hamiltonian $S^1$-manifold such that $\mathrm{Crit} H = \{ 1, 0, -2\}$. 
		If $M_0 \cong S^2 \times S^2$, then the list of all possible topological fixed point data is given in the Table \ref{table_Sec8_I_3_2}.
		\begin{table}[h]
			\begin{tabular}{|c|c|c|c|c|c|c|c|}
				\hline
				    & $(M_0, [\omega_0])$ & $e(P_{-2}^+)$ & $Z_{-2}$ & $Z_0$ & $Z_1$ & $b_2(M)$ & $c_1^3(M)$ \\ \hline \hline
				    {\bf (I-3-2.1)} & $(S^2 \times S^2, 2x + 2y)$ & $-y$  & $S^2$ & \makecell{$Z_0 \cong S^2$ \\ $\mathrm{PD}(Z_0) = y$} & $S^2 \times S^2$ & $3$ & $48$ \\ \hline    
				    {\bf (I-3-2.2)} & $(S^2 \times S^2, 2x + 2y)$ & $x-y$  & $S^2$ & \makecell{$Z_0 \cong S^2$ \\ $\mathrm{PD}(Z_0) = y$} 
				    															&$S^2 \times S^2$ & $3$ & $50$ \\ \hline    
				    {\bf (I-3-2.3)} & $(S^2 \times S^2, 2x + 2y)$ & $-y$  & $S^2$ &  
				    				\makecell{$Z_0 \cong S^2 ~\dot\cup ~S^2$ \\ $\mathrm{PD}(Z_0^1) = y$ \\ $\mathrm{PD}(Z_0^2) = y$} 				    
				    															&$S^2 \times S^2$ & $4$ & $42$ \\ \hline    
				    {\bf (I-3-2.4)} & $(S^2 \times S^2, 2x + 2y)$ & $x-y$  & $S^2$ & 
				    								\makecell{$Z_0 \cong S^2 ~\dot\cup ~S^2$ \\ $\mathrm{PD}(Z_0^1) = y$ \\ $\mathrm{PD}(Z_0^2) = y$} 
				    								& $S^2 \times S^2$ & $4$ & $46$ \\ \hline    
				    {\bf (I-3-2.5)} & $(S^2 \times S^2, 2x + 2y)$ & $-y$  & $S^2$ & 
				    							\makecell{$Z_0 \cong S^2$ \\ $\mathrm{PD}(Z_0) = x$}
				    								& $S^2 \times S^2$ & $3$ & $46$ \\ \hline    
				    {\bf (I-3-2.6)} & $(S^2 \times S^2, 2x + 2y)$ & $-y$  & $S^2$ & 
				    							\makecell{$Z_0 \cong S^2$ \\ $\mathrm{PD}(Z_0) = x + y$}
				    								& $S^2 \times S^2$ & $3$ & $42$ \\ \hline    
				    {\bf (I-3-2.7)} & $(S^2 \times S^2, 2x + 2y)$ & $-y$  & $S^2$ & 
				    							\makecell{$Z_0 \cong S^2$ \\ $\mathrm{PD}(Z_0) = x + 2y$}
				    								& $S^2 \times S^2$ & $3$ & $38$ \\ \hline    				    								
			\end{tabular}		
			\vs{0.5cm}			
			\caption{\label{table_Sec8_I_3_2} Topological fixed point data for $\mathrm{Crit} H = \{1, 0, -2\}$ with $M_0 \cong S^2 \times S^2$.}
		\end{table}				   
	\end{theorem}

	\begin{proof}
		In this case, we have $b_{\min} = 2k$ for some $k \in \Z$ by Lemma \ref{lemma_volume}. Let $\mathrm{PD}(Z_0) = ax + by$ for some $a,b \in \Z$
		Since $[\omega_1] = (2 - a - k)x + (3-b)y$ by \eqref{equation_I_symplectic_form_level_1}, we have 
		\[
			\underbrace{a+k \leq 1 \text{\hs{0.1cm} and \hs{0.1cm}} b \leq 2}_{\langle [\omega_1]^2,  [M_1] \rangle > 0}, \quad 
			\underbrace{2a + 2b \geq 2}_{\mathrm{Vol}(Z_0)}, \quad 
			\underbrace{k \geq 0}_{\text{by \eqref{equation_I_k}}}.
		\]
		From the inequalities above, we have 10 solutions: 
		\[
			a = -1: \quad \begin{cases}
				k = 0 & b = 2 \\
				k = 1 & b = 2 \\
				k = 2 & b = 2 \\
			\end{cases},
			\quad \quad 
			a = 0: \quad \begin{cases}
				k = 0 & b = 1,2 \\
				k = 1 & b = 1,2 \\
			\end{cases},
			\quad \quad 
			a = 1: \quad \begin{cases}
				k = 0 & b = 0,1,2.
			\end{cases}						
		\]		 
		On the other hand, in case of $(a,b) = (-1,2)$ (i.e., $\mathrm{PD}(Z_0) = -x + 2y$), we have 
		\begin{itemize}
			\item $[Z_0] \cdot [Z_0] = -4$,
			\item $\mathrm{Vol}(Z_0) = 2$
		\end{itemize}	
		so that it violates the adjunction formula. Thus we have 7 possibilities:
			\begin{table}[H]
				\begin{tabular}{|c|c|c|c|c|c|c|c|}
					\hline
					    & {\bf (I-3-2.1)} & {\bf (I-3-2.2)} & {\bf (I-3-2.3)} & {\bf (I-3-2.4)} & {\bf (I-3-2.5)} & {\bf (I-3-2.6)} & {\bf (I-3-2.7)}\\ \hline \hline
					    $(a,b,k)$ & $(0,1,0)$ & $(0,1,1)$  & $(0,2,0)$ & $(0,2,1)$ & $(1,0,0)$ & $(1,1,0)$ & (1,2,0) \\ \hline
					    $[Z_0]\cdot[Z_0]$ & $0$ & $0$ & $0$ & $0$ & $0$ & $2$  & $4$\\ \hline
					    \makecell{$\langle c_1(TM_0), [Z_0] \rangle$ \\ $= ~\mathrm{Vol}(Z_0)$}  & $2$ & $2$ & $4$ & $4$ & $2$ & $4$ & $6$\\ \hline
				\end{tabular}		
			\end{table}
		\noindent
		Similar to the proof of Theorem \ref{theorem_Sec8_I_3_1}, one can check, by the adjunction formula and the fact that the volume of each fixed component of $Z_0$ is even,
		that each $(a,b,k)$ in the table above determines the corresponding topological fixed point data 
		in Table \ref{table_Sec8_I_3_2} uniquely (where each Chern numbers are obtained from Lemma \ref{lemma_I_Chern_number}).
	\end{proof}
	
	\begin{example}[Fano varieties of type {\bf (I-3-2)}]\label{example_Sec8_I_3_2}  We describe Fano varieties of type {\bf (I-3-2)} in Theorem \ref{theorem_Sec8_I_3_2} as follows. 
		
		\begin{itemize}
	           	 \item {\bf (I-3-2.1)} \cite[No.28 in Section 12.4]{IP} : Let $M = \p^1 \times X_1$ equipped with the $T^3$-action induced from the standard toric action on 
	           	 $\p^1 \times \p^2$. The moment polytope is described in Figure \ref{figure_I_3_2_1}. If we take a circle subgroup of $T^3$ generated by $\xi = (1,0,1)$, then
	           	 the $S^1$-action is semifree and the fixed point components corresponds to the faces colored by red in Figure \ref{figure_I_3_2_1}. 
	           	 One can check at a glance from the figure that $\mathrm{Vol}(Z_{\min}) = \mathrm{Vol}(Z_0) = 2$ (and so $b_{\min} = 0$ and $k=0$) so that the fixed point data 
	           	 for the $S^1$-action coincides with {\bf (I-3-2.1)} in Table \ref{table_Sec8_I_3_2}.
	           	 See also Example \ref{example_II_2} (1). 
	           	 
	           		 \begin{figure}[H]
	           	 		\scalebox{1}{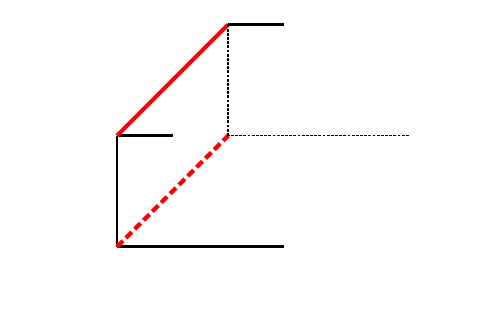}
		           	 	\caption{\label{figure_I_3_2_1} $M = \p^1 \times X_1$}
		           	 \end{figure}
	           	 	
	           	 \item {\bf (I-3-2.2)} \cite[No.30 in Section 12.4]{IP} : Consider $M$ in {\bf (I-3-1.2)}, the blow-up of $V_7$\footnote{Note that $V_7$ denotes the $T^3$-equivariant 
	           	 blow-up of $\p^3$ at a fixed point.} along a $T^3$-invariant sphere passing through the exceptional divisor. In this case, we take another circle subgroup of $T^3$
	           	 generated by $\xi = (-1,-1,0)$. Then we see from Figure \ref{figure_I_3_2_2} that $\mathrm{Vol}(Z_{\min}) = 4$ (and so $b_{\min} = 2$ and $k=1$) and 
	           	 $\mathrm{Vol}(Z_0) = 2$. So, the corresponding fixed point data should equal {\bf (I-3-2.2)}.
	           	 See also Example \ref{example_Sec6_3_4} (2), Example 8.5 \ref{example_IV_2}, and Example \ref{example_Sec8_I_3_1} {\bf (I-3-1.2)}.

	           		 \begin{figure}[H]
	           	 		\scalebox{0.8}{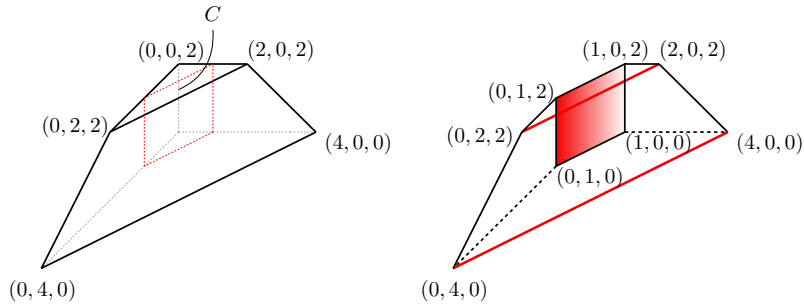}
		           	 	\caption{\label{figure_I_3_2_2} Blow up of $V_7$ along a $T^3$-invariant sphere passing through the exceptional divisor}
		           	 \end{figure}

	           	 \item {\bf (I-3-2.3)} \cite[No.10 in Section 12.5]{IP} : Let $M = \p^1 \times X_2$ with the $T^3$-action induced from the standard $T^3$-action on 
	           	 $\p^1 \times \p^2$. We then take a circle subgroup generated by $\xi = (1,0,1)$ where the moment map image of the fixed point set is depicted by red
	           	 in Figure \ref{figure_I_3_2_3}. One can immediately check that $\mathrm{Vol}(Z_{\min}) = 2$ (so that $b_{\min} = 0$ and $k=0$) and 
	           	 $\mathrm{Vol}(Z_0^1) = \mathrm{Vol}(Z_0^2) = 2$. Thus the fixed point data coincides with {\bf (I-3-2.2)} in Table \ref{table_Sec8_I_3_2}.
	           	 See also Example \ref{example_Sec6_2_4} (2) and Example \ref{example_IV_2}. 
	           	 
	           		 \begin{figure}[H]
	           	 		\scalebox{1}{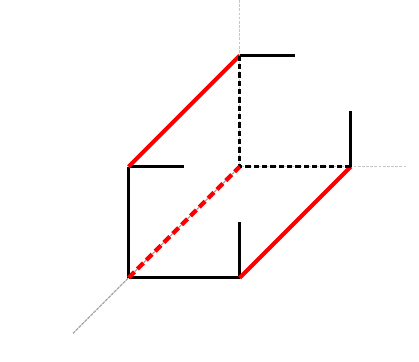}
		           	 	\caption{\label{figure_I_3_2_3} $\p^1 \times X_2$}
		           	 \end{figure}

	           	 \item {\bf (I-3-2.4)} \cite[No.12 in Section 12.5]{IP} : Let $M$ be the $T^3$-equivariant blow-up of $\p^3$ along a $T^3$-invariant line given in 
	           	 Example \ref{example_Sec8_I_1} {\bf (I-1-1.1)} (and {\bf (I-1-2.2)}). 
	           	 Let $\widetilde{M}$ denotes the $T^3$-equivariant blow-up of $M$ along two $T^3$-invariant exceptional lines (that is, lines lying on the exceptional 
	           	 divisor). Then the induced $T^3$-action on $\widetilde{M}$ has a moment map whose image is illustrated in Figure \ref{figure_I_3_2_4} on the right.
	           	 	           	 
	           		 \begin{figure}[H]
	           	 		\scalebox{0.8}{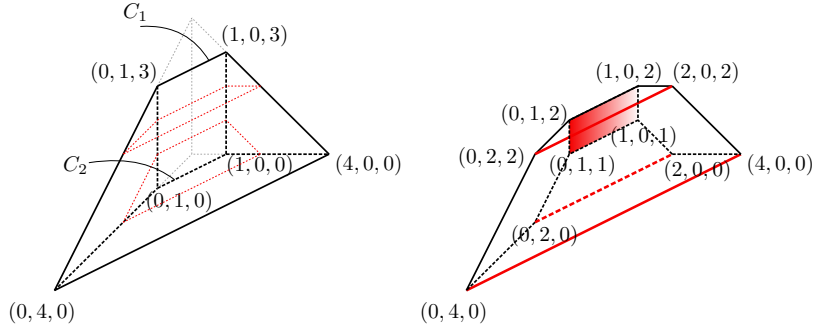}
		           	 	\caption{\label{figure_I_3_2_4} Blow-up of $Y$ along two exceptional lines}
		           	 \end{figure}
		           	 \noindent
	           	 Take a circle subgroup of $T^3$ generated by $\xi = (-1,-1,0)$. Then the action becomes semifree and we can immediately check that 
	           	 $\mathrm{Vol}(Z_{\min}) = 4$ (so that $b_{\min} = 2$ and $k=1$) and $\mathrm{Vol}(Z_0^1) = \mathrm{Vol}(Z_0^2) = 2$. So, the fixed point data of the action
	           	 is equal to {\bf (I-3-2.4)} in Table \ref{table_Sec8_I_3_2}. See also Example \ref{example_Sec6_3_4} (5) and Example \ref{example_IV_2}.\vs{0.3cm}

	           	 \item {\bf (I-3-2.5)} \cite[No.26 in Section 12.4]{IP} : Let $M$ be the $T^3$-equivariant blow-up of $\p^3$ along a disjoint union of a fixed point and a $T^3$-invariant line.
	           	 If we take a circle subgroup generated by $\xi = (-1,-1,0)$, then the action is semifree and the moment map image of the fixed point set can be described as in 
	           	 Figure \ref{figure_I_3_2_5}. The volume of the minimal fixed component $Z_{\min}$ is 2 so that $b_{\min} = 2k = 0$ and also $\mathrm{Vol}(Z_0) = 2$. 
	           	 Therefore the fixed point data for the $S^1$-action equals {\bf (I-3-2.5)} in Table \ref{table_Sec8_I_3_2}.
	           	 See also Example \ref{example_Sec6_3_4} (3) and Example \ref{example_IV_2}.
	           	 
	           		 \begin{figure}[H]
	           	 		\scalebox{0.8}{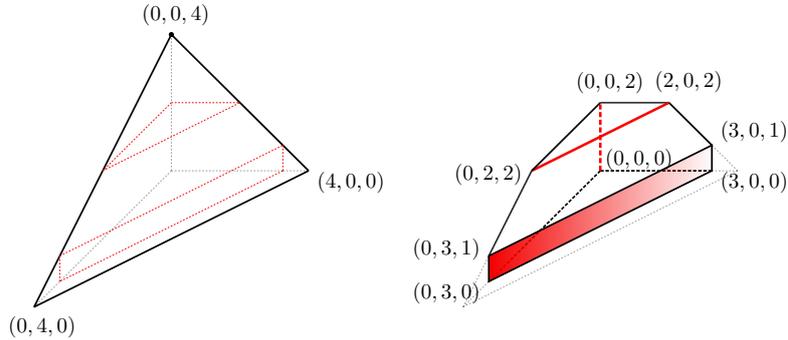}
		           	 	\caption{\label{figure_I_3_2_5} Blow-up of $\p^3$ along a disjoint union of a point and a line}
		           	 \end{figure}

	           	 \item {\bf (I-3-2.6)} \cite[No.24 in Section 12.4]{IP} : Consider the complete flag manifold $\mcal{F}(3)$ with the $T^2$-action induced from the standard $T^3$-action 
	           	 on $\C^3$ where the moment map image, together with the images of $T^2$-invariant spheres, is described on the left of Figure \ref{figure_I_3_2_6}. Let $M$ be the 
	           	 $T^2$-equivariant blow-up of $\mcal{F}(3)$ along a $T^2$-invariant curve $C$ where the corresponding edge is $\overline{(0,0)~(0,2)}$. Then the image of the 
	           	 moment map for the induced $T^2$-action on $M$ is given on the right of Figure \ref{figure_I_3_2_6}. 
	           	 (Note that the red edge $\overline{(1,0)~(1,3)}$ corresponds to the exceptional divisor of the blow-up.) 

	           		 \begin{figure}[H]
	           	 		\scalebox{1}{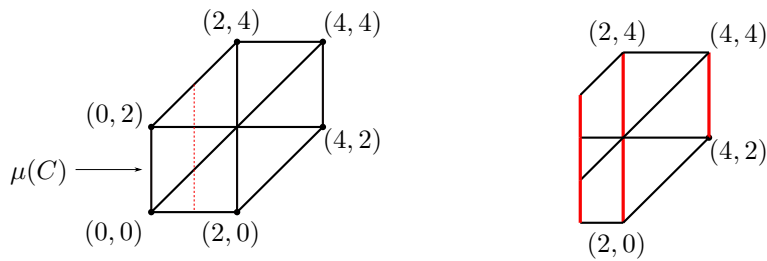}
		           	 	\caption{\label{figure_I_3_2_6} Blow up of $\mcal{F}(3)$ along $C$}
		           	 \end{figure}
	           	 
	           	 Take a circle subgroup of $T^2$ generated by $\xi = (-1,0)$. Then $\mathrm{Vol}(Z_{\min}) = 2$ (so that $b_{\min} = 2k = 0$) and $\mathrm{Vol}(Z_0) = 4$. 
	           	 Therefore the fixed point data for the $S^1$-action should be equal to {\bf (I-3-2.6)} in Table \ref{table_Sec8_I_3_2}.
	           	 See Example \ref{example_IV_2}.
			\vs{0.3cm}	           	 

	           	 \item {\bf (I-3-2.7)} \cite[No.21 in Section 12.4]{IP} : Consider $\p^1 \times \p^2$ with the standard $T^3$-action and consider the $S^1$ subgroup of $T^3$ 
	           	 generated by $\xi = (0,-1,0)$. Then the fixed point set for the $S^1$-action
	           	 can be described as in the first of Figure \ref{figure_I_3_2_7}. The maximal fixed component $Z_{\max}$ is diffeomorphic to  $S^2 \times S^2$. Take 
	           	 a curve $C$ of bidegree $(1,2)$. 
	
           		 \begin{figure}[H]
	           	 		\scalebox{1}{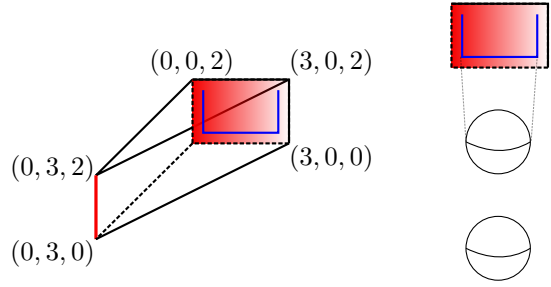}
		           	 	\caption{\label{figure_I_3_2_7} blow-up of $\p^1 \times \p^2$ along a curve of bidegree $(2,1)$}
	           	 \end{figure}
		            \noindent	 
	           	 Let $M$ be the $S^1$-equivariant blow-up of $\p^1 \times \p^2$ along the curve $C$. Since $C$ has volume six in $M_0$ and the blow-up operation does not 
	           	 affect the neighborhood of the minimal fixed component $Z_{\min}$, the volume of $Z_{\min}$ is still 2 and so $b_{\min} = 2k = 0$. This coincides with 
	           	 {\bf (I-3-2.7)} in Table \ref{table_Sec8_I_3_2}. 
	           	 	           	 	
		\end{itemize}
	\end{example}

	\begin{theorem}[Case {\bf (I-4)}]\label{theorem_Sec8_I_4}
		Let $(M,\omega)$ be a six-dimensional closed monotone semifree Hamiltonian $S^1$-manifold such that $\mathrm{Crit} H = \{ 1, 0, -1, -2\}$. 
		Then the list of all possible topological fixed point data is given in the Table \ref{table_Sec8_I_4}
		\begin{table}[h]
			\begin{tabular}{|c|c|c|c|c|c|c|c|c|}
				\hline
				    & $(M_0, [\omega_0])$ & $e(P_{-2}^+)$ & $Z_{-2}$ & $Z_{-1}$ & $Z_0$ & $Z_1$ & $b_2(M)$ & $c_1^3(M)$ \\ \hline \hline
				    {\bf (I-4-1.1)} & \makecell{$(E_{S^2} \# \overline{\p^2},$ \\ $3x + 2y - E_1)$} & $-x-y$  & $S^2$ & $\mathrm{pt}$ & \makecell{ $Z_0 \cong S^2$  \\ $\mathrm{PD}(Z_0) = x + y - E_1$} & 
				    	$X_2$
					     & $4$ & $40$ \\ \hline    
				    {\bf (I-4-1.2)} & \makecell{$(E_{S^2} \# 2\overline{\p^2},$ \\ $3x + 2y - E_1 - E_2)$} & $-x-y$  & $S^2$ & $\mathrm{2 ~pts}$ & \makecell{ $Z_0 \cong S^2$  \\ 
				    $\mathrm{PD}(Z_0) = x + y - E_1 - E_2$} & 
				    	$X_3$
					     & $5$ & $36$ \\ \hline    					     
				    {\bf (I-4.2)} & \makecell{$(S^2 \times S^2 \# \overline{\p^2},$ \\  $2x + 2y - E_1)$} & $-y$  & $S^2$ & $\mathrm{pt}$ & \makecell{ $Z_0 \cong S^2$  \\ $\mathrm{PD}(Z_0) = y - E_1$} & 
				    	$X_2$
					     & $4$ & $44$ \\ \hline    					     					     
			\end{tabular}		
			\vs{0.5cm}			
			\caption{\label{table_Sec8_I_4} Topological fixed point data for $\mathrm{Crit} H = \{1, 0, -1, -2\}$.}
		\end{table}				   
	\end{theorem}

	\begin{proof}
		 Let $m = |Z_{-1}|$ so that $M_0$ is the $m$-times blow-up of $M_{-2+\epsilon}$ with the exceptional divisors whose dual classes are 
		 denoted by $E_1, \cdots, E_m$. 
		 Let
		 \[
		 	\mathrm{PD}(Z_0) = ax + by + \sum_{i=1}^m c_i E_i, \quad \quad a,b,c_1, \cdots, c_m \in \Z.
		 \]
		 We divide the proof into two cases for $M_{-2 + \epsilon}$: $E_{S^2}$ and $S^2 \times S^2$.\vs{0.3cm}
		 
  		 \noindent
    		 {\bf (I-4-1) :} $M_{-2 + \epsilon} \cong E_{S^2}$ \vs{0.3cm}
		
		\noindent 
		In this case, we have 
		\[
			b_{\min} = 2k + 1, \quad \quad [\omega_1] = (3 - a - k)x + (3 - b)y - \sum_{i=1}^m (c_i + 2) E_i.
		\]
		Then we have the following inequalities:
		\begin{enumerate}
			\item $2a + b + \sum_{i=1}^m c_i \geq 1$ \quad ($\because$ ~ $\mathrm{Vol}(Z_0) \geq 1$)
			\item $a + k \leq 2$ \quad ($\because$ ~ $\int_{M_t} \omega_t^2 > 0$)
			\item $b \leq 2$ \quad ($\because$ ~ $\int_{M_t} \omega_t^2 > 0$)
			\item $c_i \geq -1$ \quad ($\because$ ~ $E_i$ does not vanish on $M_t$)
			\item $b+c_i \leq 0$ \quad ($\because$ ~ $x - E_i$ does not vanish on $M_t$)
			\item $k+a-b \leq -1$ \quad ($\because$ ~ $y$ does not vanish on $M_t$)
			\item $k \geq -1$ \quad (by \eqref{equation_I_k})
		\end{enumerate}
		Combining these inequalities, we have 
		\begin{equation}\label{equation_I_4}
			a \leq b \leq 1
		\end{equation}
		where the first one comes from (6) and (7), and the latter one follows from (4),(5).
		
		We claim that $m \leq 2$. If $m=3$, then (1) and (5) imply that 
		\[
			2a - 2b \geq 2a - 2b + 3b + \sum_{i=1}^3 c_i \geq 1
		\]
		which contradicts that $a - b \leq 0$ in \eqref{equation_I_4}.
		Also if $m \geq 4$, a new exceptional divisor $C$ appears where $\mathrm{PD}(C) = 2x + y - E_{i_1} - E_{i_2} - E_{i_3} - E_{i_4}$
		for any distinct $i_1, i_2, i_3, i_4$ as in Lemma \ref{lemma_list_exceptional}.
		Then 
		\[
			\mathrm{Vol}(C) = \langle (2x + y - E_{i_1} - E_{i_2} - E_{i_3} - E_{i_4}) \cdot [\omega_1], [M_1] \rangle = -2 - a - b - k - \sum_{j=1}^4 c_{i_j} \geq 1
		\]
		which implies that 
		\[
			-3 \geq a + b + k + \sum_{j=1}^4 c_{i_j}. 
		\]		
		Summing those equations for all possible quadruples $(i_1, i_2, i_3, i_4)$, we obtain
		\[
			-3 {m \choose 4} \geq {m \choose 4}a + {m \choose 4}b + {m \choose 4}k + {{m-1} \choose 3} \sum_{i=1}^m c_i		
		\]
		which can be simplified as 
		\begin{equation}\label{equation_imsi}
			\begin{array}{ccl}\vs{0.1cm}
				-3m & \geq & ma + mb + mk + 4\sum_{i=1}^m c_i \\ \vs{0.1cm}
					& = &  4(2a + b + k + \sum_{i=1}^m c_i) + (m-8)a + (m-4)b + (m-4)k \\ \vs{0.1cm}
					& \geq & \ds (m-4)(b-a) + (2m - 12)a - (m-4) \\ \vs{0.1cm}
					& \geq & \ds (2m - 12)a - (m-4).
			\end{array}
		\end{equation}
		by (1) and (7). So, $-2m - 4 \geq (2m - 12)a$ and it is impossible when $4\leq m\leq 6$ since $a \leq 1$ by \eqref{equation_I_4}. 
		Therefore the only possible case is when $m=7$. However in this case, a new divisor appears, namely $D$, where 
		$\mathrm{PD}(D) = 6x + 3y - 2(E_1 + \cdots + E_7)$ by Lemma \ref{lemma_list_exceptional}.
		 Applying the inequality $\langle \mathrm{PD}(D) \cdot [\omega_1], [M_1] \rangle \geq 1$,(1),(7), and \eqref{equation_I_4},
		we get 
		\[
			-11 \geq 3(a+k+b) + 2\sum_{i=1}^7 c_i \geq -a + b + k \geq -1
		\]
		which leads to a contradiction.

		Consequently, we have $m = 1$ or $2$. If $m=1$, then by (1), (5), and \eqref{equation_I_4}, we have $a=1, b=1$ (and hence $k = -1$ by (6) and (7)). Also $c = -1$ by (4) and (5).
		So, 
		\[
			\mathrm{PD}(Z_0) = x + y - E_1, \quad Z_0 \cong S^2 ~\text{by the adjunction formula.}
		\]
		See Table \ref{table_Sec8_I_4}: {\bf (I-4-1.1)}. 
		
		Finally, when $m=2$, we similarly obtain $a=b=1$ by (1),(5), and \eqref{equation_I_4}, and hence $c_1 = c_2 = -1$ 
		by (4),(5), and $k=-1$ by (6),(7). Thus
		\[
			\mathrm{PD}(Z_0) = x+ y - E_1 - E_2, \quad Z_0 \cong S^2
		\]
		by the adjunction formula again. See Table \ref{table_Sec8_I_4}: {\bf (I-4-1.2)}.
		\vs{0.3cm}
		
		\noindent
		 {\bf (I-4-2) :} $M_{-2 + \epsilon} \cong S^2 \times S^2$ \vs{0.3cm}
		
		\noindent 
		In this case, we have 
		\[
			b_{\min} = 2k \quad (k \geq 0 ~\text{by \eqref{equation_I_k}}), \quad \quad [\omega_1] = (2 - a - k)x + (3 - b)y - \sum_{i=1}^m (c_i + 2) E_i.
		\]		
		We similarly have the following inequalities :
		\begin{enumerate}
			\item $2a + 2b + \sum_{i=1}^m c_i \geq 1$ \quad ($\mathrm{Vol}(Z_0) \geq 1$.)
			\item $1\geq a+k$ \quad ($\int_{M_t} [\omega_t]^2 > 0$.)
			\item $2 \geq b$ \quad ($\int_{M_t} [\omega_t]^2 > 0$.)
			\item $c_i  \geq -1$ \quad($E_i$ does not vanish on $M_t$)
			\item $b+c_i \leq 0$ \quad ($x - E_i$ does not vanish on $M_t$)
			\item $k+a+c_i \leq -1$ \quad ($y-E_i$ does not vanish on $M_t$)
			\item $k \geq 0$ 
		\end{enumerate}
		Note that by (4),(5),(6),(7), we have 
		\begin{equation}\label{equation_I_4_2}
			a \leq 0, \quad b \leq 1, \quad \underbrace{a+b \leq 0}_{\Leftrightarrow ~(a,b) \neq (0,1)} ~(\text{in case of}~ m \geq 2)
		\end{equation}
		otherwise it contradicts the equation (1),(4),(5).

		We claim that $m=1$. If $m=2$, then (1) and (5) imply that 
		\[
			1 \leq 2a + 2b + c_1 + c_2 \leq 2a 
		\]
		which contradicts that $a \leq 0$ in \eqref{equation_I_4_2}.
		
		Now we assume that $m \geq 3$ (and $m \leq 7$ by \eqref{equation_number_of_points}). 
		Then we have an exceptional divisor $C$ with 
		$\mathrm{PD}(C) = x + y - E_i - E_j - E_k$ for distinct $i,j,k$ by Lemma \ref{lemma_list_exceptional}. 
		The condition $\langle [\omega_1], C \rangle > 0$ implies that 
		\[
			k + a + b + c_i + c_j + c_k \leq -2
		\] for each triple $(i,j,k)$. By summing these inequalities, we obtain
		\[
			{{m} \choose {3}} k+ {{m} \choose {3}} a + {{m} \choose {3}} b + {{m-1} \choose {2}} \sum_{i=1}^m c_i \leq -2 {{m} \choose {3}}
		\] 
		or equivalently, $mk + ma + mb + 3\sum_{i=1}^m c_i \leq -2m$.
		Then, 
		\[
			\begin{array}{ccl}\vs{0.2cm}
				-2m & \geq & 3(2a + 2b + \sum_{i=1}^m c_i) + (m-6)a + (m-6)b +mk  \\ \vs{0.2cm}
					& \geq & 3 + (m-6)(a+b) 
			\end{array}
		\]
		which is possible unless $m=7$ (since $a + b \leq 0$ by \eqref{equation_I_4_2}). When $m=7$, we have an exceptional divisor $D$ with 
		$\mathrm{PD}(D) = 4x + 3y - E_1 - 2(E_2 + \cdots + E_7)$, the last item in the list of Lemma \ref{lemma_list_exceptional}. Then 
		\[
			\langle \mathrm{PD}(D) \cdot [\omega_1], [M_1] \rangle = -8 - 3(a+k) - 4b + c_1 - 2\sum_{i=1}^7 c_i \geq 1.
		\]
		So, by (1), (6), and (7), we have 
		\[
			-9 \geq -c_1 + 3(a+k) + 4b + 2\sum_{i=1}^7 c_i \geq -c_1 - a + 2 \geq k + 1 + 2 \geq 3
		\] which leads to a contradiction.
		
		Therefore, we have $m=1$. By (1),(5), and \eqref{equation_I_4_2}, we obtain
		\[
			1 - c \leq 2a + 2b \leq 2a - 2c \leq -2c, \quad \quad (\Rightarrow ~c \leq -1)
		\]
		which implies that $c = -1$ by (4). Then (1) and \eqref{equation_I_4_2} induces $a = 0, b=1$ (and hence $k=0$ by (6) and (7)). So, 
		\[
			\mathrm{PD}(Z_0) = y - E_1, \quad Z_0 \cong S^2
		\]
		by the adjunction formula. See Table \ref{table_Sec8_I_4}: {\bf (I-4-2)}. This completes the proof.
		
	\end{proof}
			
	\begin{example}[Fano varieties of type {\bf (I-4)}]\label{example_Sec8_I_4}  We describe Fano varieties of type {\bf (I-4)} in Theorem \ref{theorem_Sec8_I_4} as follows. 
		
		\begin{itemize}
	           	 \item {\bf (I-4-1.1)} \cite[No.9 in Section 12.5]{IP}: Let $M$ be the toric variety given in Example \ref{example_Sec8_I_3_1} {\bf (I-3-1.4)}; it is the 
	           	 $T^3$-equivariant blow-up of $\p^3$ along two $T^3$-invariant spheres whose moment polytope is given on the left of Figure \ref{figure_I_4_1_1}.
	           	 Let $\widetilde{M}$ be the $T^3$-equivariant blow-up of $M$ along a $T^3$-equivariant exceptional sphere. Then the corresponding moment polytope
	           	 can be described as on the right of Figure \ref{figure_I_4_1_1}. 

	           		 \begin{figure}[H]
	           	 		\scalebox{0.8}{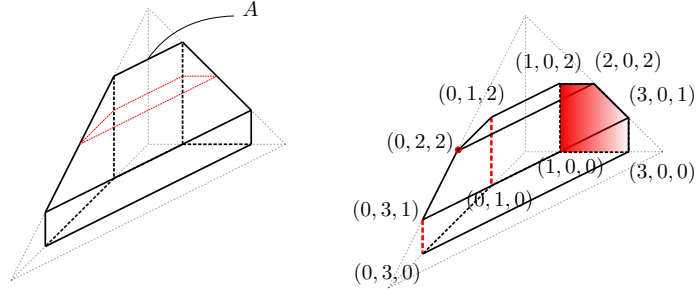}
		           	 	\caption{\label{figure_I_4_1_1} Blow-up of $M$ (in {\bf (I-3-1.4)}) along an exceptional curve}
		           	 \end{figure}
	           	 \noindent
	           	 Take the circle subgroup of $T^3$ generated by $\xi = (0,-1,0)$. The red faces correspond to the fixed components of the action and we have 
	           	 $\mathrm{Vol}(Z_{\min}) = 1$ (and so $b_{\min} = -1$ and $k = -1$). Also, there is a one isolated fixed point of index two and $Z_0$ is a sphere 
	           	 having volume 2. Therefore, the fixed point data should be equal to {\bf (I-4-1.1)} in Table \ref{table_Sec8_I_4}.
	           	 See also Example \ref{example_IV_1} (4).
			\vs{0.3cm}	           	 
	           	 	
	           	 \item {\bf (I-4-1.2)} \cite[No.2 in Section 12.6]{IP} : Let $M$ be the same as in Example \ref{example_Sec8_I_3_1} {\bf (I-3-1.4)}. 
	           	 Now, we let $\widetilde{M}$ be the $T^3$-equivariant 
	           	 blow-up of $M$ along two $T^3$-invariant exceptional spheres lying on the same exceptional component where the moment polytope for the induced $T^3$-action
	           	 is described in Figure \ref{figure_I_4_1_2}.
	           	 
	           		 \begin{figure}[H]
	           	 		\scalebox{0.8}{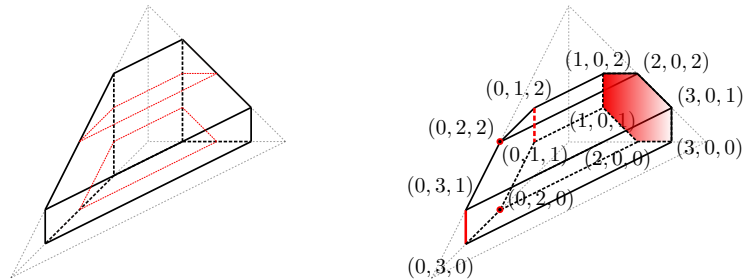}
		           	 	\caption{\label{figure_I_4_1_2} Blow-up of $M$ along two invariant spheres on the same exceptional components}
		           	 \end{figure}
			\noindent
			Take the circle subgroup generated by $\xi = (0,-1,0)$. One can check that the action is semifree and the fixed components corresponds to the red faces 
			in Figure \ref{figure_I_4_1_2}. We can check at a glance that the volume of $Z_{\min}$ and $Z_0$ are both 1 and there are two isolated fixed points 
			of index two. Thus the fixed point data for the $S^1$-action coincides with {\bf (I-4-1.2)} in Table \ref{table_Sec8_I_4}. 
			See also Example \ref{example_IV_1} (1). \vs{0.3cm}
	           	 	
	           	 \item {\bf (I-4-2)} \cite[No.11 in Section 12.5]{IP} :  $M$ be given in {\bf (I-3-1.3)} in Example \ref{example_Sec8_I_3_1}, the $T^3$-equivariant blow-up of 
	           	 $\p^1 \times X_1$
	           	 along $t \times E$ where $t \in \p^1$ is the fixed point for the $S^1$-action and $E$ is the exceptional curve in $X_1$. 
	           	 
	           	 Take the circle subgroup
	           	 generated by $\xi = (-1,0,0)$. Then the action is semifree and the fixed point set can be illustrated by the red faces in Figure \ref{figure_I_4_2}. One can immediately
	           	 check that $\mathrm{Vol}(Z_{\min}) = 2$ (so that $b_{\min} = 2k = 0$)
	           	  and $\mathrm{Vol}(Z_0) = 1$. Also there is exactly one isolated fixed point of index two. Thus the fixed point data for the action
	           	 should be equal to {\bf (I-4-2)} in Table \ref{table_Sec8_I_4}. 
			See also Example \ref{example_Sec6_2_4} (1) and Example \ref{example_Sec8_I_3_1} {\bf (I-3-1.3)}. \vs{0.3cm}
	           	 
	           		 \begin{figure}[H]
	           	 		\scalebox{0.8}{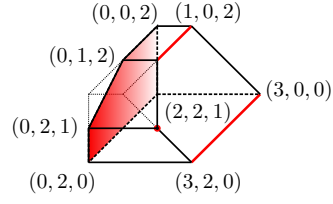}
		           	 	\caption{\label{figure_I_4_2} Blow-up of $\p^1 \times X_1$ along $t \times E$}
		           	 \end{figure}

		\end{itemize}
	
	\end{example}			

\subsection{Case II : $\dim Z_{\min} = 4$}
\label{ssecCaseIIDimZMin4}

	In this section we deal with the case where $\dim Z_{\min} = \dim Z_{\max} = 4$. 
	Since the only possible interior critical value is $0$, there is no isolated fixed point and $Z_{\min} \cong Z_0 \cong Z_{\max}$.
	Also, all possible candidates of $M_0$ are $\C P^2, S^2 \times S^2,$ or $X_k$ (blow-up of $\p^2$ at $k$  generic points) for $1 \leq k \leq 8$
	by Proposition \ref{proposition_monotonicity_preserved_under_reduction}. \vs{0.1cm}
	
	We begin with the following lemma for the computation of Chern numbers. 
		
	\begin{lemma}\label{lemma_II_Chern_number}
		Suppose that $\dim Z_{\min} = 4$ and $\dim Z_{\max} = 4$. Then
		\[
			\int_M c_1(TM)^3 =  \langle 2e^2 + 6c_1^2 - 3c_1[Z_0] + 2e[Z_0] + [Z_0]^2, [M_0] \rangle
		\]
		where $e := e(P_{-1 + \epsilon}^+)$, $c_1 := c_1(TM_0)$, and $[Z_0] := \mathrm{PD}(Z_0)$. 
	\end{lemma}
	
	\begin{proof}
		It is straightforward from the localization theorem \ref{theorem_localization}:
		\[
			\begin{array}{ccl}\vs{0.3cm}
				\ds \int_M c_1^{S^1}(TM)^3 & = &  \ds  
							\int_{Z_{\min}} \frac{\left(c_1^{S^1}(TM)|_{Z_{\min}}\right)^3}{e_{Z_{\min}}^{S^1}} 
							+ \int_{Z_0} \frac{\overbrace{\left(c_1^{S^1}(TM)|_{Z_0}\right)^3}^{= 0}}{e_{Z_0}^{S^1}}
							 + \int_{Z_{\max}} \frac{\left(c_1^{S^1}(TM)|_{Z_{\max}}\right)^3}
							 {e_{Z_{\max}}^{S^1}} \\ \vs{0.2cm}
							 & = & \ds \int_{Z_{\min}} \frac{(\lambda + e + c_1)^3}{\lambda + e}  + 
							 \int_{Z_{\max}} \frac{(-\lambda - (e + [Z_0]) + c_1)^3} {-\lambda - (e + [Z_0]) }  \\ \vs{0.2cm}
							& = &  \langle 2e^2 + 6c_1^2 - 3c_1[Z_0] + 2e[Z_0] + [Z_0]^2, [M_0] \rangle
			\end{array}			
		\]
	\end{proof}

	We divide the classification into two cases: $\mathrm{Crit} \mathring{H} = \emptyset$ {\bf (II-1)} and $\mathrm{Crit} \mathring{H} = \{0\}$ {\bf (II-2)}.
	Note that, since the moment map can be reversed by taking $-\omega$ instead of $\omega$, we may assume that
	\begin{equation}\label{equation_II_1_assumption}
		\mathrm{Vol}(Z_{\min}) \geq \mathrm{Vol}(Z_{\max}).
	\end{equation}
	We first consider the case that $\mathrm{Crit} \mathring{H} = \emptyset$. 		
	
	\begin{theorem}[Case {\bf (II-1)}]\label{theorem_Sec8_II_1}
		Let $(M,\omega)$ be a six-dimensional closed monotone semifree Hamiltonian $S^1$-manifold such that $\mathrm{Crit} H = \{ 1,-1\}$. 
		Then the list of all possible topological fixed point data is given in the Table \ref{table_Sec8_II_1}
		\begin{table}[h]
			\begin{tabular}{|c|c|c|c|c|c|c|c|c|}
				\hline
				    & $(M_0, [\omega_0])$ & $e(P_{-1}^+)$ & $Z_{-1}$  & $Z_0$ &  $Z_1$ & $b_2(M)$ & $c_1^3(M)$ \\ \hline \hline
				    {\bf (II-1-1.1)} & $(\p^2, 3u)$ & $0$  & $\p^2$  & & 
				    	$\p^2$
					     & $2$ & $54$ \\ \hline    
				    {\bf (II-1-1.2)} & $(\p^2, 3u)$ & $u$  & $\p^2$  && 
				    	$\p^2$
					     & $2$ & $56$ \\ \hline    
				    {\bf (II-1-1.3)} & $(\p^2, 3u)$ & $2u$  & $\p^2$  && 
				    	$\p^2$
					     & $2$ & $62$ \\ \hline    
					     					     					     
				    {\bf (II-1-2.1)} & $(S^2 \times S^2, 2x + 2y)$ & $0$  & $S^2 \times S^2$  && 
				    	$S^2 \times S^2$
					     & $3$ & $48$ \\ \hline    					     
				    {\bf (II-1-2.2)} & $(S^2 \times S^2, 2x + 2y)$ & $x$  & $S^2 \times S^2$ && 
				    	$S^2 \times S^2$
					     & $3$ & $48$ \\ \hline    					     
				    {\bf (II-1-2.3)} & $(S^2 \times S^2, 2x + 2y)$ & $x+y$  & $S^2 \times S^2$ && 
				    	$S^2 \times S^2$
					     & $3$ & $52$ \\ \hline    					     
				    {\bf (II-1-2.4)} & $(S^2 \times S^2, 2x + 2y)$ & $x-y$  & $S^2 \times S^2$ && 
				    	$S^2 \times S^2$
					     & $3$ & $44$ \\ \hline    					     
					     					     					     					     
				    {\bf (II-1-3.1)} & $(E_{S^2}, 3x + 2y)$ & $0$  & $E_{S^2}$ && 
				    	$E_{S^2}$
					     & $3$ & $48$ \\ \hline    					     					     
				    {\bf (II-1-3.2)} & $(E_{S^2}, 3x + 2y)$ & $x + y$  & $E_{S^2}$ && 
				    	$E_{S^2}$
					     & $3$ & $50$ \\ \hline    					     					     
					     					     
				    \makecell{{\bf (II-1-4.k)} \\ {\bf k = 2$\sim$8}} & $(X_k, 3u - \sum_{i=1}^k E_i)$ & $0$  & $X_k$ && 
				    	$X_k$
					     & $k+2$ & $54-6k$ \\ \hline    					     					     					     
			\end{tabular}		
			\vs{0.5cm}			
			\caption{\label{table_Sec8_II_1} Topological fixed point data for $\mathrm{Crit} H = \{1,-1\}$.}
		\end{table}				   
	\end{theorem}
	
	\begin{proof}
		Let $e := e(P_{-1 + \epsilon}^+) = e(P_{1 - \epsilon}^-)$. \vs{0.3cm}
	
	\noindent
	{\bf (II-1-1)} $M_0 \cong \p^2$. \vs{0.3cm}
	
	Note that $[\omega_t] = 3u - te$ for $t \in [-1, 1]$. Since $\int_{M_0} \omega_t^2 > 0$, we get $e = 0, u,$ or $2u$ (by our assumption \eqref{equation_II_1_assumption}).
	Each cases can be described as
	\begin{table}[H]
		\begin{tabular}{|c|c|c|c|}
			\hline
				& {\bf (II-1-1.1)} & {\bf (II-1-1.2)} & {\bf (II-1-1.3)} \\ \hline \hline
				$e$ & $0$ & $u$  & $2u$  \\ \hline
				$\mathrm{Vol}(Z_{\min})$ & $9$ & $16$ & $25$ \\ \hline
				$\mathrm{Vol}(Z_{\max})$ & $9$ & $4$ & $1$ \\ \hline
		\end{tabular}		
	\end{table}
	\noindent	
	and those are listed in Table \ref{table_Sec8_II_1}.
	\vs{0.3cm}

	\noindent
	{\bf (II-1-2)} $M_0 \cong S^2 \times S^2$. \vs{0.3cm}
	
	Letting $e = ax + by$ for $a,b \in \Z$, 
	we have $[\omega_t] = (2-at)x + (2-bt)y$ for $t \in [-1,1]$. Again by $\int_{M_t} \omega_t^2 > 0$, we obtain $2 \pm a \geq 1$ and $2 \pm b \geq 1$. 
	So, up to permutation on $\{x, y\}$ and by the condition \eqref{equation_II_1_assumption}, 
	we have $(a,b) = (0,0), (1,0), (1,1),$ and $(1,-1)$. So, we obtain
	\begin{table}[H]
		\begin{tabular}{|c|c|c|c|c|}
			\hline
				& {\bf (II-1-2.1)} & {\bf (II-1-2.2)} & {\bf (II-1-2.3)} & {\bf (II-1-2.4)} \\ \hline \hline
				$e$ & $0$ & $x$  & $x+y$ & $x - y$  \\ \hline
				$\mathrm{Vol}(Z_{\min})$ & $8$ & $12$ & $18$ & $6$ \\ \hline
				$\mathrm{Vol}(Z_{\max})$ & $8$ & $4$ & $2$ & $6$ \\ \hline
		\end{tabular}		
	\end{table}
	\vs{0.3cm}

	\noindent
	{\bf (II-1-3)} $M_0 \cong E_{S^2}$. \vs{0.3cm}
	
	Let $u = x + y$ and $E_1 = y$ so that $\langle u^2, [M_0] \rangle = 1$, $\langle u\cdot E_1, [M_0] \rangle = 0$, and $\langle E_1^2, [M_0] \rangle = -1$.
	Set $e = au + bE_1$ ($a,b \in \Z$). 
	Since $\int_{M_t} \omega_t^2 > 0$ for every $t \in [-1,1]$ where $[\omega_t] = (3-at)u - (1 + bt)E_1$ for $t \in [-1,1]$, we have
	\begin{enumerate}
		\item $3-a > 1+b \geq 1$ \quad ($\langle [\omega_1]^2, [M_1] \rangle >0$ and $\langle [\omega_1] \cdot E_1, [M_1] \rangle > 0$),  
		\item $3+a > 1-b \geq 1$ \quad ($\langle [\omega_{-1}]^2, [M_{-1}] \rangle >0$ and $\langle [\omega_{-1}] \cdot E_1, [M_{-1}] \rangle > 0$) . 
	\end{enumerate}
	From (1) and (2), we obtain $b=0$ and hence $(a,b) = (1,0)$, $(0,0)$ (where the case $(a,b) = (-1,0)$ can be recovered by \eqref{equation_II_1_assumption}). Therefore,
	\begin{table}[H]
		\begin{tabular}{|c|c|c|}
			\hline
				& {\bf (II-1-3.1)} & {\bf (II-1-3.2)} \\ \hline \hline
				$e$ & $0$ & $x + y$   \\ \hline
				$\mathrm{Vol}(Z_{\min})$ & $8$ & $15$  \\ \hline
				$\mathrm{Vol}(Z_{\max})$ & $8$ & $3$ \\ \hline
		\end{tabular}		
	\end{table}
	\vs{0.3cm}

	\noindent
	{\bf (II-1-4)} $M_0 \cong X_k$ (for $2 \leq k \leq 8$). \vs{0.3cm}
	
	In this case, we have $[\omega_0] = 3u - \sum_{i=1}^{k} E_i$. Let $\mathrm{PD}(Z_0) = au + \sum_{i=1}^k b_i E_i$ for some $a, b_1, \cdots, b_k \in \Z$.
	Then, as $\int_{M_t} \omega_t^2 > 0$ for every $t \in [-1,1]$ where $[\omega_t] = (3-at)u - \sum_{i=1}^k (1 + b_it)E_i$ for $t \in [-1,1]$, we get
	\begin{enumerate}
		\item $3-a > 1+b_i \geq 1$, \quad ($\langle [\omega_1]^2, [M_1] \rangle >0$ and $\langle [\omega_1] \cdot E_i, [M_1] \rangle > 0$)  
		\item $3+a > 1-b_i \geq 1$, \quad ($\langle [\omega_{-1}]^2, [M_{-1}] \rangle >0$ and $\langle [\omega_{-1}] \cdot E_i, [M_{-1}] \rangle > 0$)  
	\end{enumerate}
	for every $i=1,\cdots,k$. From (1) and (2), we see that every $b_i$ vanishes and $-1 \leq a \leq 1$. On the other hand, since $u - E_i - E_j$ is an exceptional class for any $i \neq j$, we also have
	\[
		\langle [\omega_1] \cdot (u - E_i - E_j), [M_1] \rangle = 1 - a \geq 1, \quad \langle [\omega_{-1}] \cdot (u - E_i - E_j), [M_{-1}] \rangle = 1 + a \geq 1, 
	\]
	i.e., $a = 0$ (and hence $e = 0$). We denote by {\bf (II-1-4.k)} for each case : $M_0 \cong X_k$ where $2 \leq k \leq 8$. See Table \ref{table_Sec8_II_1}.

	The Chern number computation for each case can be easily obtained from Lemma \ref{lemma_II_Chern_number}.

	\end{proof}

	\begin{example}[Fano varieties of type {\bf (II-1)}]\label{example_Sec8_II_1}  We describe Fano varieties of type {\bf (II-1)} in Theorem \ref{theorem_Sec8_II_1} as follows. \vs{0.1cm}
		
		\begin{itemize}
	           	 \item {\bf (II-1-1)} \cite[No. 34, 35, 36 in Section 12.3]{IP} 
			For {\bf (II-1-1.1)}, let $M = \p^1 \times \p^2$ with the standard $T^3$-action where the moment polytope is described in the first of Figure \ref{figure_II_1_1}. 
			The circle subgroup generated by $\xi = (0,0,1)$ acts on the first factor of $\p^1 \times \p^2$ and is semifree and the fixed components for the action 
			correspond to the red faces. It immediately follows from the figure that $\mathrm{Vol}(Z_{\min}) = \mathrm{Vol}(Z_{\max}) = 16$. Thus the fixed point data 
			coincides with {\bf (II-1-1.1)} in Table \ref{table_Sec8_II_1}. See also Example \ref{example_Sec6_2_3} (2) and Example \ref{example_Sec8_I_2} {\bf (I-1-2.1)}.  \vs{0.3cm}

			For {\bf (II-1-1.2)}, Let $M = V_7$, the $T^3$-equivariant blow-up of $\p^3$ at a fixed point. The induced $T^3$-action has a moment polytope as in the middle 
			of Figure \ref{figure_II_1_1}. If we take the circle subgroup generated by $\xi = (0,0,1)$, then the action becomes semifree and there are two fixed components
			(colored by red in the figure) such that $\mathrm{Vol}(Z_{\min}) = 16$ and $\mathrm{Vol}(Z_{\max}) = 4$. So, this corresponds to {\bf (II-1-1.2)} in 
			Table \ref{table_Sec8_II_1}.
			See also Example \ref{example_Sec6_3_2} and Example \ref{example_II_1} (2).
			\vs{0.3cm}

			For {\bf (II-1-1.3)}, we consider the polytope given on the right of Figure \ref{figure_II_1_1}. The polytope corresponds to the toric variety 
			$M = \mathbb{P}(\mcal{O} \oplus \mcal{O}(2))$. Take a circle subgroup generated by $\xi = (0,0,1)$. Then two fixed components appears as in the figure such that 
			$\mathrm{Vol}(Z_{\min}) = 25$ and$\mathrm{Vol}(Z_{\max}) = 1$. Thus the fixed point data for the action shuold equal {\bf (II-1-1.3)} in Table \ref{table_Sec8_II_1}.
			See also Example \ref{example_Sec6_2_3} (1). 
	           	 
	           		 \begin{figure}[H]
	           	 		\scalebox{0.8}{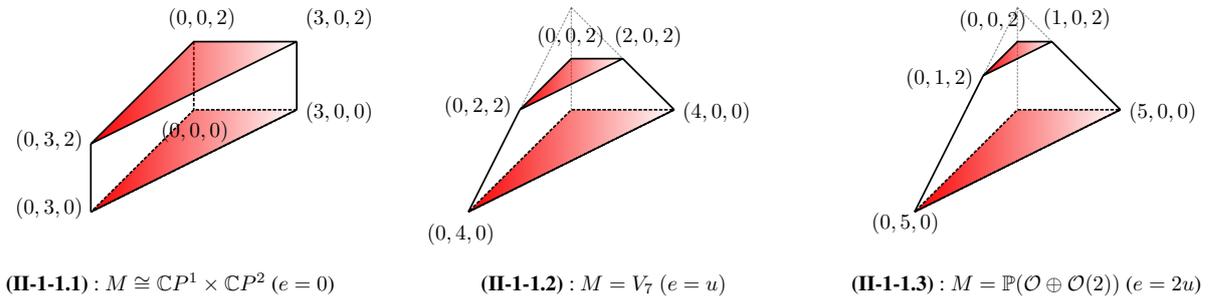}
		           	 	\caption{\label{figure_II_1_1} Fano varieties of type {\bf (II-1-1)}}
		           	 \end{figure}
		           	 
	           	 \item {\bf (II-1-2)} \cite[No. 27,28,31,25 in Section 12.4]{IP}  
	           	 For {\bf (II-1-2.1)}, we consider $M = \p^1 \times \p^1 \times \p^1$ with the standard $T^3$-action whose moment polytope is given in the first of Figure 
	           	 \ref{figure_II_1_2} and a circle subgroup generated by $\xi = (0,0,1)$. Then the action is semifree and there are two copies of $\p^1 \times \p^1$ as 
	           	 the fixed components whose volumes are both $8$. Thus the fixed point data is equal to {\bf (II-1-2.1)} in Table \ref{table_Sec8_II_1}. 
	           	 See also Example \ref{example_Sec6_1_2} and Example \ref{example_II_1} (3). \vs{0.3cm}

	           		 \begin{figure}[H]
	           	 		\scalebox{0.8}{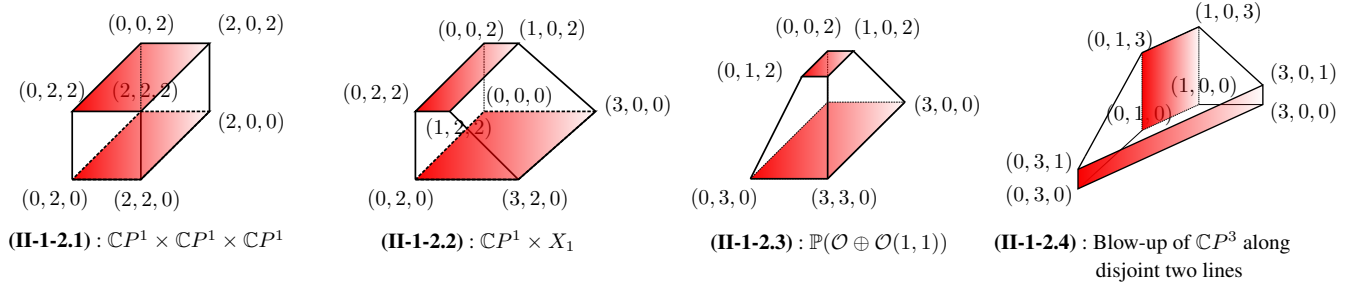}
		           	 	\caption{\label{figure_II_1_2} Fano varieties of type {\bf (II-1-2)}}
		           	 \end{figure}
	           	 
	           	 In case of {\bf (II-1-2.2)}, let $M = \p^1 \times X_1$ equipped with the standard $T^3$-action where the $T^2$-action on $X_1$ is induced from the standard 
	           	 $T^2$-action on $\p^2$. The moment polytope for the action is given in the second of Figure \ref{figure_II_1_2}. Take the circle subgroup of $T^3$ generated 
	           	 by $\xi = (0,0,1)$ so that the action is semifree and there are two fixed  components corresponding to red faces of the figure. One can immediately check that 
	           	 $\mathrm{Vol}(Z_{\min}) = 12$ and $\mathrm{Vol}(Z_{\max}) = 4$. So, the corresponding fixed point data should be {\bf (II-1-2.2)} in Table \ref{table_Sec8_II_1}. 
	           	 See also Example \ref{example_II_2} (1) and Example \ref{example_Sec8_I_3_2} {\bf (I-3-2.1)}. \vs{0.3cm}
	           	 
	           	  For {\bf (II-1-2.3)}, consider the polytope given in the third of Figure \ref{figure_II_1_2} where the corresponding toric variety is 
	           	  $M = \mathbb{P}(\mcal{O} \oplus \mcal{O}(1,1))$. Then the circle subgroup of $T^3$ generated by $\xi = (0,0,1)$ acts on $M$ semifreely where the fixed point set 
	           	  consists of two extremal components such that $\mathrm{Vol}(Z_{\min}) = 18$ and $\mathrm{Vol}(Z_{\max}) = 2$. 
	           	  Therefore the fixed point data should be equal to {\bf (II-1-2.3)} in Table \ref{table_Sec8_II_1}. 
	           	  See also Example \ref{example_Sec6_1_3} and Example \ref{example_Sec8_I_3_1} {\bf (I-3-1.1)}. \vs{0.3cm}

			  Finally in case of {\bf (II-1-2.4)}, let $M$ be the $T^3$-equivariant blow-up of $\p^3$ along two $T^3$-invariant lines. Then the corresponding moment polytope is 
			  depicted in the last of Figure \ref{figure_II_1_2}. Take the circle subgroup of $T^3$ generated by $\xi = (1,1,0)$ so that the fixed point set for the $S^1$-action consists
			  of two extremal components. One can check that both of two fixed components have volume 6, and hence the fixed point data is exactly {\bf (II-1-2.4)} in Table 
			  \ref{table_Sec8_II_1}. See also	Example \ref{example_III} (2) and Example \ref{example_Sec8_I_3_1} {\bf (I-3-1.4)}.\vs{0.3cm}

	           	 \item {\bf (II-1-3)} \cite[No. 28, 30 in Section 12.4]{IP}  For {\bf (II-1-3.1)}, we let $M = \p^1 \times X_1$ with the standard $T^3$-action as in the case of 
	           	 {\bf (II-1-2.2)}. We take the circle subgroup of $T^3$ generated by $\xi = (0,1,0)$. Then the $S^1$-action is semifree and the fixed components are two copies 
	           	 of $X_1$ where the corresponding faces in the moment polytope are colored by red in Figure \ref{figure_II_1_3}. It immediately follows that the volume of the both
	           	 fixed components are 8, and so the fixed point data equals {\bf (II-1-3.1)} in Table \ref{table_Sec8_II_1}. 
	           	 See also Example \ref{example_II_2} (1), Example \ref{example_Sec8_I_3_2} {\bf (I-3-2.1)}, and \ref{example_Sec8_II_1} {\bf (II-1-2.2)}. \vs{0.3cm}
	           	 
	           	 In case that {\bf (II-1-3.2)}, we consider $V_7$, the $T^3$-equivariant blow-up of $\p^3$ at a fixed point, and let $M$ be the $T^3$-equivariant blow-up of 
	           	 $V_7$ along a $T^3$-invariant line passing through the exceptional divisor of $V_7 \rightarrow \p^3$. Then the moment polytope can be illustrated as on the right 
	           	 of Figure \ref{figure_II_1_3}. Take the circle subgroup of $T^3$ generated by $\xi = (0,0,1)$. One can easily see that $\mathrm{Vol}(Z_{\min}) = 15$ and 
	           	 $\mathrm{Vol}(Z_{\max}) = 3$, and therefore the fixed point data is exactly {\bf (II-1-3.2)} in Table \ref{table_Sec8_II_1}.
	           	 See also Example \ref{example_Sec6_3_4} (2), Example \ref{example_IV_2}, Example \ref{example_Sec8_I_3_1} {\bf (I-3-1.2)}, and Example 
	           	 \ref{example_Sec8_I_3_2} {\bf (I-3-2.2)}. 
	           	 
	           		 \begin{figure}[H]
	           	 		\scalebox{0.8}{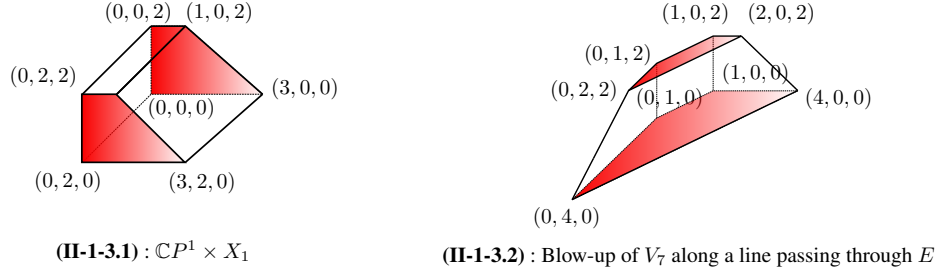}
		           	 	\caption{\label{figure_II_1_3} Fano varieties of type {\bf (II-1-3)}}
		           	 \end{figure}
	           	 
	           	 \item {\bf (II-1-4)} \cite[No. 10 in Section 12.5, No. 3,4,5,6,7,8 in Section 12.6]{IP} : For each $2 \leq k \leq 8$, let $M = X_k \times \p^1$ 
	           	 where the $S^1$-action is given on the second factor. Then the Euler class of each level set of a moment map is trivial. 
	           	 In particular two fixed components $X_k \times t_1$ and $X_k \times t_2$ have the volume $9-k$. Therefore, the fixed point data coincides with
	           	 {\bf (II-1-4.k)} in Table \ref{table_Sec8_II_1}.
		\end{itemize}
	\end{example}

	Before to proceed the classification for remaining cases, we will show that $M_0$ cannot be $X_k$ for $k > 1$ whenever $Z_0$ is non-empty as follows.

	\begin{proposition}\label{proposition_II_2}
		Let $(M,\omega)$ be a six-dimensional closed monotone semifree Hamiltonian $S^1$-manifold such that $\mathrm{Crit} H = \{ 1,0,-1\}$. Then 
		$M_0 \not \cong X_k$ for any $k \geq 2$. 
	\end{proposition}
	
	\begin{proof}
		Suppose that $M_0 \cong X_k$ for some $k \geq 2$ with $[\omega_0] = 3u - \sum_{i=1}^{k} E_i$ where $u \in H^2(X_k; \Z)$ such that 
		$\langle u^2, [X_k] \rangle = 1$ and $u \cdot E_i  = 0$ for every $i=1,\cdots,k$.
		
		Denote by $\mathrm{PD}(Z_0) = xu + \sum_{i=1}^k y_iE_i$ and $e = au + \sum_{i=1}^k b_iE_i$ for some $a, b_1, \cdots, b_k, x, y_1, \cdots, y_k \in \Z$.
		By the Duistermaat-Heckman theorem \ref{theorem_DH}, we get
		\[
			[\omega_{-1}] = (3+a)u - \sum (1 - b_i)E_i \quad \text{and} \quad [\omega_{1}] = (3-a-x)u - \sum (1 + b_i + y_i)E_i.
		\]			
		Thus we have the following inequalities.
		\begin{enumerate}
			\item $3x + \sum_{i=1}^k y_i \geq 1$ \quad ($\mathrm{Vol}(Z_0) \geq 1$):
			\item $b_i \leq 0$ \quad ($\langle [\omega_{-1}] \cdot E_i, [M_{-1}] \rangle \geq 1$),
			\item $a + b_i + b_j \geq 0$ \quad ($\langle [\omega_{-1}] \cdot (u - E_i - E_j), [M_{-1}] \rangle \geq 1$),
			\item $b_i + y_i \geq 0$ \quad ($\langle [\omega_{1}] \cdot E_i, [M_{1}] \rangle \geq 1$),
			\item $a+x+b_i+b_j+y_i+y_j \leq 0$ \quad ($\langle [\omega_1] \cdot (u-E_i-E_j), [M_1] \rangle \geq 1$).
		\end{enumerate}
		Note that $\underbrace{a \geq 0}_{=: (6)}$ from (2) and (3), and $\underbrace{a + x \leq 0}_{=: (7)}$ by (4) and (5).
		By summing all equations of the form (5), we obtain
		\[
			{k \choose 2} (a + x) + (k-1) \left( \sum_{i=1}^k b_i + \sum_{i=1}^k y_i \right) \leq 0 \quad \Rightarrow \quad \underbrace{\frac{k}{2}a + \frac{k}{2}x + 
			\sum_{i=1}^k b_i + \sum_{i=1}^k y_i \leq 0.}_{=: \text{(5')}}. 
		\]	
		Similarly, from (3), we have
		\[
			\underbrace{\frac{k}{2}a + \sum_{i=1}^k b_i \geq 0.}_{=: \text{(3')}}.
		\]
		So, by (3') and (5'), we get $\underbrace{\ds \frac{k}{2}x + \sum_{i=1}^k y_i \leq 0}_{=: (8)}$. 
		Since $x \leq 0$ by (6) and (7), we have $k > 6$ (otherwise (1) and (8) contradict to each other).
		
		For $k \geq 7$, there are exceptional divisors of the form $3u - 2E_1 - E_{234567}$ by Lemma \ref{lemma_list_exceptional} where 
		$E_{234567} := E_{2} + \cdots + E_{7}$. Then 
		\begin{itemize}
			\item $\langle [\omega_{-1}] \cdot (3u - 2E_1 - E_{234567}), [M_{-1}] \rangle \geq 1$ \quad $\Rightarrow$ 
			\quad $3a + 2b_1 + b_{234567} \geq 0$.
			\item $\langle [\omega_1] \cdot (3u - 2E_1 - E_{234567}), [M_1] \rangle \geq 1$ \quad $\Rightarrow$ 
			\quad $3a + 3x + 2b_1 + 2y_1 + b_{234567} + y_{234567} \leq 0$
		\end{itemize}
		where $b_{234567} =  b_{2} + \cdots + b_{7}$ and $y_{234567} = y_{2} + \cdots + y_{7}$, respectively.
		Then we obtain 
		\begin{equation}\label{equation_k_7}
			3x  + 2y_1 + y_{234567} \leq 0 
		\end{equation}
		Since the inequality \eqref{equation_k_7} holds up to permutation on $\{1,\cdots,8\}$, we may assume that $y_1$ is  the maximal among all $y_i$'s. Then, 
		\[
			1 \leq 3x  + (y_1 + y_8) + y_{234567} \leq 3x  + 2y_1 + y_{234567} \leq 0 
		\]
		by (1) and this leads  to a contradiction. Therefore, no such manifold exists. 
	\end{proof}
				
	By Proposition \ref{proposition_II_2}, we only to consider the three cases $M_0 \cong \p^2$ {\bf (II-2-1)}, $S^2 \times S^2$ {\bf (II-2-2)}, and $X_1$ {\bf (II-2-3)}. 			

	\begin{theorem}[Case {\bf (II-2-1)}]\label{theorem_Sec8_II_2_1}
		Let $(M,\omega)$ be a six-dimensional closed monotone semifree Hamiltonian $S^1$-manifold such that $\mathrm{Crit} H = \{ 1,0,-1\}$. 
		Suppose that $M_0 \cong \p^2$. Then the list of all possible topological fixed point data is given in the Table \ref{table_Sec8_II_2_1}
		\begin{table}[H]
			\begin{tabular}{|c|c|c|c|c|c|c|c|c|}
				\hline
				    & $(M_0, [\omega_0])$ & $e(P_{-1}^+)$ & $Z_{-1}$  & $Z_0$ &  $Z_1$ & $b_2(M)$ & $c_1^3(M)$ \\ \hline \hline
				    {\bf (II-2-1.1)} & $(\p^2, 3u)$ & $0$  & $\p^2$  & \makecell{$Z_0 \cong S^2$, \\ $\mathrm{PD}(Z_0) = u$} & 
				    	$\p^2$
					     & $3$ & $46$ \\ \hline    
				    {\bf (II-2-1.2)} & $(\p^2, 3u)$ & $u$  & $\p^2$  & \makecell{$Z_0 \cong S^2$, \\ $\mathrm{PD}(Z_0) = u$} & 
				    	$\p^2$
					     & $3$ & $50$ \\ \hline    
				    {\bf (II-2-1.3)} & $(\p^2, 3u)$ & $-u$  & $\p^2$  & \makecell{$Z_0 \cong S^2$, \\ $\mathrm{PD}(Z_0) = 2u$} & 
				    	$\p^2$
					     & $3$ & $38$ \\ \hline    
					     					     					     
				    {\bf (II-2-1.4)} & $(\p^2, 3u)$ & $0$  & $\p^2$  & \makecell{$Z_0 \cong S^2$, \\ $\mathrm{PD}(Z_0) = 2u$} & 
				    	$\p^2$
					     & $3$ & $40$ \\ \hline    					     
				    {\bf (II-2-1.5)} & $(\p^2, 3u)$ & $-u$  & $\p^2$ & \makecell{$Z_0 \cong T^2$, \\ $\mathrm{PD}(Z_0) = 3u$} & 
				    	$\p^2$
					     & $3$ & $32$ \\ \hline    					     
				    {\bf (II-2-1.6)} & $(\p^2, 3u)$ & $-2u$  & $\p^2$ & \makecell{$Z_0 \cong \Sigma_3$, \\ $\mathrm{PD}(Z_0) = 4u$} & 
				    	$\p^2$
					     & $3$ & $26$ \\ \hline    					     
			\end{tabular}		
			\vs{0.5cm}			
			\caption{\label{table_Sec8_II_2_1} Topological fixed point data for $\mathrm{Crit} H = \{1,0,-1\}$ with $M_0 = \p^2$}
		\end{table}				   
	\end{theorem}

	\begin{proof}
		Let $e = au$ and $\mathrm{PD}(Z_0) = pu$ for $a, p \in \Z$. Then the Duistermaat-Heckman theorem yields
		\[
			[\omega_{-1}] = (3+a)u, \quad [\omega_1] = (3 - a - p)u \quad \text{so that} \quad 
			\begin{cases}
				3 + a \geq 1 \\
				3 - a - p \geq 1.
			\end{cases}
		\]
		Note that $p$ is positive by $\langle [\omega_0], [Z_0] \rangle  > 0$ and $3 + a \geq 3 - a - p$ by the assumption \eqref{equation_II_1_assumption}.
		Therefore all possible solutions of $(p,a)$
		are listed as follows.
		\begin{table}[H]
			\begin{tabular}{|c|c|c|c|c|c|c|}
				\hline
					& {\bf (II-2-1.1)} & {\bf (II-2-1.2)} & {\bf (II-2-1.3)} & {\bf (II-2-1.4)} & {\bf (II-2-1.5)} & {\bf (II-2-1.6)} \\ \hline \hline
					$(p,a)$ & $(1,0)$ & $(1,1)$  & $(2,-1)$ & $(2,0)$ & $(3,-1)$ & $(4,-2)$ \\ \hline
					$\mathrm{Vol}(Z_{\min})$ & $9$ & $16$ & $4$ & $9$ & $4$ & $1$\\ \hline
					$\mathrm{Vol}(Z_{\max})$ & $4$ & $1$ & $4$ & $1$ & $1$& $1$\\ \hline
					\makecell{$\langle c_1(TM_0), [Z_0] \rangle$ \\ $= \mathrm{Vol}(Z_0)$} & $3$ & $3$ & $6$ & $6$ & $9$& $12$\\ \hline
					$[Z_0] \cdot [Z_0]$ & $1$ & $1$ & $4$ & $4$ & $9$& $16$\\ \hline
					$g$ : genus of $Z_0$ & $0$ & $0$ & $0$ & $0$ & $1$& $3$\\ \hline					
			\end{tabular}		
		\end{table}
		\noindent
		Thus we obtain Table \ref{table_Sec8_II_2_1} using the adjunction formula and Lemma \ref{lemma_II_Chern_number}.
		\vs{0.3cm}		
	\end{proof}

	\begin{example}[Fano varieties of type {\bf (II-2-1)}]\label{example_Sec8_II_2_1}
	  We describe Fano varieties of type {\bf (II-2-1)} in Theorem \ref{theorem_Sec8_II_2_1} as follows. \vs{0.3cm}
	  
	  \begin{itemize}
	  	\item {\bf (II-2-1.1), (II-2-1.4)} \cite[No. 26, 22 in Section 12.4]{IP} : 
	  	Let $M = \p^1 \times \p^2$ with the standard $T^3$-action and we consider the circle subgroup generated by $\xi = (0,0,-1)$.
	  	Then the maximal and minimal fixed component are both $\p^2$. Take (any) line $C_1$ and a conic $C_2$ in $Z_{\max}$ and denote the $S^1$-equivariant 
	  	blow-up of $M$ along $C_i$ by $M_i$ for $i=1,2$. 
	  	
	  		 \begin{figure}[H]
	           	 		\scalebox{0.8}{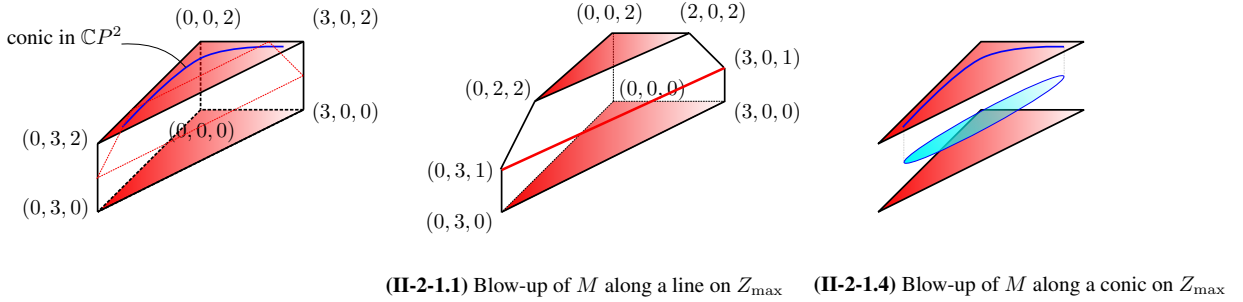}
		           	 	\caption{\label{figure_II_2_1_1} Fano varieties of type {\bf (II-2-1.1)} and {\bf (II-2-1.4)}}
		           \end{figure}
	  	\noindent
	  	Then the induced $S^1$-action on each $M_i$ is semifree and the fixed point data is described in Figure \ref{figure_II_2_1_1}. (The first one is a moment polytope of 
	  	$\p^1 \times \p^2$, the second one is for $M_1$ where $C_1$ is taken to be $T^3$-equivariant, and  the third one is a {\em conceptual images} of the fixed point set
	  	where the {\em blue ellipse} indicates the fixed component on the zero level set in $M_2$. If we denote by $Z_{\min}^i$ and $Z_0^i$ 
	  	the minimal fixed component and the interior fixed component of $M_i$ respectively, then one can 
	  	immediately check that $\mathrm{Vol}(Z_{\min}^1) = \mathrm{Vol}(Z_{\min}^2) = 9$ and $\mathrm{Vol}(Z_0^1) = 3$ and $\mathrm{Vol}(Z_0^2) = 6$ (since $Z_0^2$ is a 
	  	conic in $M_0$).
	  	Thus the corresponding fixed point data coincide with {\bf (II-2-1.1)} and {\bf (II-2-1.4)} in Table \ref{table_Sec8_II_2_1}, respectively. 
	  	See also Example \ref{example_Sec8_I_2} {\bf (I-2)}, Example \ref{example_Sec8_I_3_2} {\bf (I-3-2.5)}, Example \ref{example_Sec6_3_4} (3), 
	  	and Example \ref{example_IV_2} for {\bf (II-2-1.1)}.
	  	\vs{0.3cm}
	  		  	
	  	\item {\bf (II-2-1.2), (II-2-1.3), (II-2-1.5)} \cite[No. 29, 19, 14 in Section 12.4]{IP} : 
	  	Consider $V_7$, the $T^3$-equivariant blow-up of $\p^3$ at a fixed point, and the $S^1$-action on $V_7$
	  	generated by $\xi = (0,0,1)$ where the fixed components for the action are colored by red in the first of Figure \ref{figure_II_2_1_2}
	  	where $\mathrm{Vol}(Z_{\min}) = 4$ and $\mathrm{Vol}(Z_{\min}) = 16$.
	  	
	  		 \begin{figure}[H]
	           	 		\scalebox{0.8}{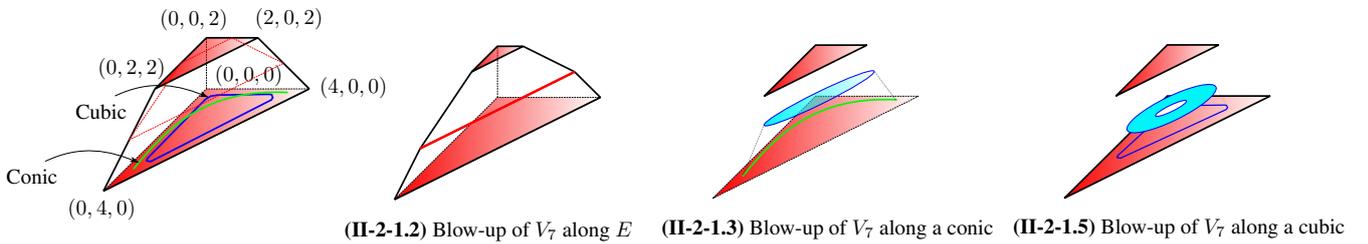}
		           	 	\caption{\label{figure_II_2_1_2} Fano varieties of type {\bf (II-2-1.2)}, {\bf (II-2-1.3)} and {\bf (II-2-1.5)}}
		           \end{figure}
	  	\noindent
	  	There are three ways of obtaining a semifree $S^1$-Fano variety by blowing-up $V_7$ as follows. 
		Let $C_1$ be degree-one curve lying on the exceptional divisor $Z_{\max} \cong \p^2$. Also, we let $C_2$ be a conic, and $C_3$ be a cubic in $Z_{\min}$. 
		Denote by $M_i$ the $S^1$-equivariant blow-up of $V_7$ along $C_i$ for each $i=1,2,3$. If we denote by $Z_{\min}^i$ and $Z_{\max}^i$ the minimal and maximal
		fixed components on $M_i$, respectively, then we can easily see that 
		\begin{itemize}
			\item $\mathrm{Vol}(Z_{\min}^1) = 16$, $\mathrm{Vol}(Z_{\max}^1) = 1$,
			\item $\mathrm{Vol}(Z_{\min}^2) = 9$, $\mathrm{Vol}(Z_{\max}^2) = 4$,
			\item $\mathrm{Vol}(Z_{\min}^3) = 4$, $\mathrm{Vol}(Z_{\max}^3) = 4$.
		\end{itemize}
		Therefore, the fixed point data for each case coincides with one described in Table \ref{table_Sec8_II_2_1}. See also Example \ref{example_Sec6_3_4} (1) and
		Example \ref{example_IV_2} for {\bf (II-2-1.2)}.
		
		\begin{remark}
			If we consider a line $C$ in $Z_{\min}$ in $V_7$ and take the $S^1$-equivariant blow-up of $V_7$ along $C$, then the resulting manifold $\widetilde{M}$ is also Fano 
			and $\mathrm{Vol}(\widetilde{Z}_{\min}) = 9$, $\mathrm{Vol}(\widetilde{Z}_{\max}) = 4$. This model can be taken as a toric model and so one can check that 
			$\widetilde{M}$ is isomorphic to the blow-up of $\p^1 \times \p^2$ along a line, see Example \ref{example_Sec8_II_2_1} {\bf (II-2-1.1)}. 
			
			We also note that, in Mori-Mukai's classification, the manifold $M_2$ is described as the $T^2$-equivariant blow-up of a smooth quadric $Q$ in $\p^4$ at two isolated
			extrema. See in Figure \ref{figure_II_2_1_2_1}.
	  	
  		 \begin{figure}[H]
	           	 \scalebox{1}{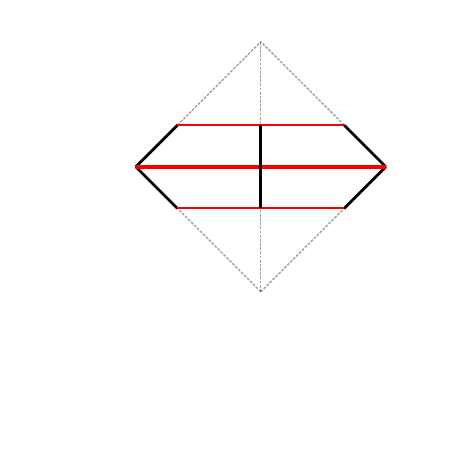}
		           \caption{\label{figure_II_2_1_2_1} Another description of {\bf (II-2-1.3)}}
	           \end{figure}
		\end{remark}

		\vs{0.3cm}
	  	
	  	\item {\bf (II-2-1.6)} \cite[No. 9 in Section 12.4]{IP} : Let $M = \mathbb{P}(\mcal{O} \oplus \mcal{O}(1,1))$ with the semifree $S^1$-action 
	  	as we have seen in Example \ref{example_II_1} {\bf (II-1-1.3)}.
	  	It is a toric variety whose moment polytope is given in the first of Figure \ref{figure_II_2_1_3}. Take $C$ a quartic in $Z_{\min} \cong \p^2$ and denote the $S^1$-equivariant 
	  	blow-up of $M$ along $C$ by $\widetilde{M}$. Then we may easily check that $\mathrm{Vol}(\widetilde{Z}_{\min}) = \mathrm{Vol}(\widetilde{Z}_{\max}) = 1$ and this 
	  	coincides with {\bf (II-2-1.6)} in Table \ref{table_Sec8_II_2_1}.

	           		 \begin{figure}[H]
	           	 		\scalebox{0.8}{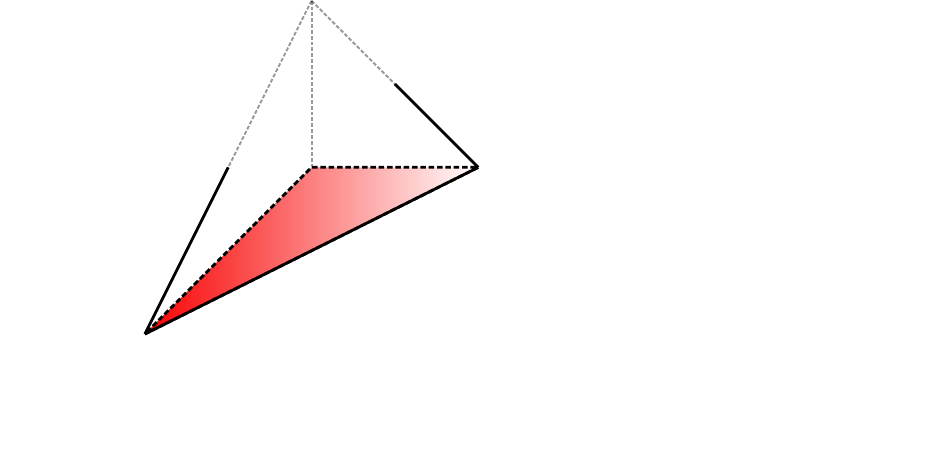}
		           	 	\caption{\label{figure_II_2_1_3} Fano varieties of type {\bf (II-2-1.6)}}
		           	 \end{figure}
		           	 
		\begin{remark}
			It is also possible that we take any curve $C$ of degree less than four in $Z_{\min}$ and take the $S^1$-equivariant blow-up of $M$ along $C$. Namely, if we let
			$C_i$ be any smooth curve of degree $i$ in $Z_{\min}$, then the $S^1$-equivariant blow-up, denoted by $M_i$, is also a semifree $S^1$-Fano variety where the fixed 
			point data
			is given by 
		\begin{itemize}
			\item $\mathrm{Vol}(Z_{\min}^1) = 16$, $\mathrm{Vol}(Z_{\max}^1) = 1$,
			\item $\mathrm{Vol}(Z_{\min}^2) = 9$, $\mathrm{Vol}(Z_{\max}^2) = 1$,
			\item $\mathrm{Vol}(Z_{\min}^3) = 4$, $\mathrm{Vol}(Z_{\max}^3) = 1$.
		\end{itemize}
		Those fixed point data exactly match up with {\bf (II-2-1.2), (II-2-1.4), (II-2-1.5)}, respectively. We will see later that 
		each $M_i$' are $S^1$-equivariantly symplectomorphic to the Fano varieties of type {\bf (II-2-1.2), (II-2-1.4), (II-2-1.5)} given above.

		\end{remark}
		  \end{itemize}
	\end{example}
				
	\begin{theorem}[Case {\bf (II-2-2)}]\label{theorem_Sec8_II_2_2}
		Let $(M,\omega)$ be a six-dimensional closed monotone semifree Hamiltonian $S^1$-manifold such that $\mathrm{Crit} H = \{ 1,0,-1\}$. 
		Suppose that $M_0 \cong S^2 \times S^2$. Then the list of all possible topological fixed point data is given in the Table \ref{table_Sec8_II_2_2}
		\begin{table}[H]
			\begin{tabular}{|c|c|c|c|c|c|c|c|c|}
				\hline
				    & $(M_0, [\omega_0])$ & $e(P_{-1}^+)$ & $Z_{-1}$  & $Z_0$ &  $Z_1$ & $b_2(M)$ & $c_1^3(M)$ \\ \hline \hline
				    {\bf (II-2-2.1)} & $(S^2 \times S^2, 2x + 2y)$ & $-x -y$  & $S^2 \times S^2$  & \makecell{$Z_0 \cong T^2$, \\ $\mathrm{PD}(Z_0) = 2x+2y$} & 
				    	$S^2 \times S^2$
					     & $4$ & $28$ \\ \hline    
				    {\bf (II-2-2.2)} & $(S^2 \times S^2, 2x + 2y)$ & $-x$  & $S^2 \times S^2$  & \makecell{$Z_0 \cong S^2$, \\ $\mathrm{PD}(Z_0) = 2x + y$} & 
				    	$S^2 \times S^2$
					     & $4$ & $32$ \\ \hline    
				    {\bf (II-2-2.3)} & $(S^2 \times S^2, 2x + 2y)$ & $-x$  & $S^2 \times S^2$  & \makecell{$Z_0 \cong S^2 ~\dot \cup ~S^2$, \\ 
				    $\mathrm{PD}(Z_0^1) = \mathrm{PD}(Z_0^2) = x$} & 
				    	$S^2 \times S^2$
					     & $5$ & $36$ \\ \hline    			     					     
				    {\bf (II-2-2.4)} & $(S^2 \times S^2, 2x + 2y)$ & $-x$  & $S^2 \times S^2$  & \makecell{$Z_0 \cong S^2$, \\ $\mathrm{PD}(Z_0) = x+y$} & 
				    	$S^2 \times S^2$
					     & $4$ & $36$ \\ \hline    					     
				    {\bf (II-2-2.5)} & $(S^2 \times S^2, 2x + 2y)$ & $-x + y$  & $S^2 \times S^2$  & \makecell{$Z_0 \cong S^2 ~\dot \cup ~S^2$, \\ 
				    $\mathrm{PD}(Z_0^1) = \mathrm{PD}(Z_0^2) = x$} & 
				    	$S^2 \times S^2$
					     & $5$ & $36$ \\ \hline    			     					     
				    {\bf (II-2-2.6)} & $(S^2 \times S^2, 2x + 2y)$ & $-x + y$  & $S^2 \times S^2$  & \makecell{$Z_0 \cong S^2$, \\ $\mathrm{PD}(Z_0) = x$} & 
				    	$S^2 \times S^2$
					     & $4$ & $40$ \\ \hline    					     					     
				    {\bf (II-2-2.7)} & $(S^2 \times S^2, 2x + 2y)$ & $0$  & $S^2 \times S^2$ & \makecell{$Z_0 \cong S^2$, \\ $\mathrm{PD}(Z_0) = x+y$} & 
				    	$S^2 \times S^2$
					     & $4$ & $38$ \\ \hline    					     
				    {\bf (II-2-2.8)} & $(S^2 \times S^2, 2x + 2y)$ & $0$  & $S^2 \times S^2$ & \makecell{$Z_0 \cong S^2$, \\ $\mathrm{PD}(Z_0) = x$} & 
				    	$S^2 \times S^2$
					     & $4$ & $42$ \\ \hline    					     
				    {\bf (II-2-2.9)} & $(S^2 \times S^2, 2x + 2y)$ & $y$  & $S^2 \times S^2$ & \makecell{$Z_0 \cong S^2$, \\ $\mathrm{PD}(Z_0) = x$} & 
				    	$S^2 \times S^2$
					     & $4$ & $44$ \\ \hline    					     					     					     
			\end{tabular}		
			\vs{0.5cm}			
			\caption{\label{table_Sec8_II_2_2} Topological fixed point data for $\mathrm{Crit} H = \{1,0,-1\}$ with $M_0 = S^2 \times S^2$}
		\end{table}				   
	\end{theorem}
				
	\begin{proof}
		Let $e = ax + by$ and $\mathrm{PD}(Z_0) = px + qy$ for $a,b,p,q \in \Z$. Then we have 
		\[
			[\omega_{-1}] = (2+a)x + (2 + b)y \quad \text{and} \quad [\omega_1] = (2 - (a+p))x + (2 - (b + q))y
		\]
		and so 
		\[
			2 + a \geq 1, \quad 2 + b \geq 1, \quad 2 - a - p \geq 1, \quad 2 - b - q \geq 1
		\]
		by $\int_{M_t} \omega_t^2 > 0$. We also have $\mathrm{Vol}(Z_0) = 2p + 2q \geq 2$ and $(2+a)(2+b) \geq (2 - a - p)(2 - b - q)$ by \eqref{equation_II_1_assumption}. 
		Furthermore, we may assume that 
		\begin{equation}\label{equation_II_2_2}
			\begin{cases}
				p > q, ~\text{or}  \\
				p = q ~\text{and}~ b \geq a
			\end{cases}
		\end{equation}
		as there are two choices of $\{x,y\}$, i.e., other cases ($p < q$ or $p=q$ and $a > b$) can be recovered by swapping $x$ and $y$.
		Summing up, we get
		\[
			-2 \leq a + b \leq 1, \quad -1 \leq a,b, \quad a+p, b+q \leq 1 \quad p + q \geq 1, \quad p \geq q, \quad (2+a)(2+b) \geq (2 - a - p)(2 - b - q).
		\]
		Using the adjunction formula, we obtain the following table. (Note that, two solutions $(a,b,p,q) = (0,-1,1,1)$ and $(-1,1,2,-1)$ are dropped from the table below
		due to \eqref{equation_II_2_2} and the adjunction formula.)
		
		\begin{table}[H]
		\begin{adjustbox}{max width=\textwidth}
			\begin{tabular}{|c|c|c|c|c|c|c|c|c|c|c|}
				\hline
					& {\bf (II-2-2.1)} & {\bf (II-2-2.2)} & {\bf (II-2-2.3)} & {\bf (II-2-2.4)} & {\bf (II-2-2.5)} & {\bf (II-2-2.6)} & {\bf (II-2-2.7)} & 
					{\bf (II-2-2.8)} & {\bf (II-2-2.9)} \\ \hline \hline
					$e = (a,b)$ & $(-1,-1)$ & $(-1,0)$  & $(-1,0)$ & $(-1,0)$ & $(-1,1)$ & $(-1,1)$ &$(0,0)$ & $(0,0)$  & $(0,1)$\\ \hline
					$\mathrm{PD}(Z_0)= (p,q)$ & $(2,2)$ & $(2,1)$  & $(2,0)$ & $(1,1)$ & $(2,0)$  & $(1,0)$ &$(1,1)$ & $(1,0)$ & (1, 0) \\ \hline
					$\mathrm{Vol}(Z_{\min})$ & $2$ & $4$ & $4$ & $4$ & $6$ & $6$ &$8$ & $8$ & $12$ \\ \hline
					$\mathrm{Vol}(Z_{\max})$ & $2$ & $2$ & $4$ & $4$ & $2$ & $4$ &$2$ & $4$ & $2$\\ \hline
					\makecell{$\langle c_1(TM_0), [Z_0] \rangle$ \\ $= \mathrm{Vol}(Z_0)$} & $8$ & $6$ & $4$ & $4$ & $4$ & $2$ &$4$ & $2$ & $2$\\ \hline
					$[Z_0] \cdot [Z_0]$ & $8$ & $4$ & $0$ & $2$ & $0$ & $0$ &$2$ & $0$ & $0$\\ \hline
					\# components of $Z_0$ & $1$ & $1$ & $2$ & $1$ & $2$ & $1$ &$1$ & $1$ & $1$\\ \hline
					$g$ : genus of $Z_0$ & $1$ & $0$ & $(0,0)$ & $0$ & $(0,0)$ &$0$ &$0$ & $0$  & $0$\\ \hline					
			\end{tabular}		
		\end{adjustbox}
		\end{table}
		\noindent
		The Chern numbers in Table \ref{table_Sec8_II_2_2} directly follows from simple computation using Lemma \ref{lemma_II_Chern_number}.
	\end{proof}				
				
	\begin{example}[Fano varieties of type {\bf (II-2-2)}]\label{example_II_2_2}\cite[No. 2,5,7,8,9,10,11 in Section 12.5, No. 2,3 in Section 12.6]{IP} 
	  We describe Fano varieties of type {\bf (II-2-2)} in Theorem \ref{theorem_Sec8_II_2_2} as follows. 
			           	 
	\begin{itemize}
		\item {\bf (II-2-2.1), (II-2-2.2), (II-2-2.7)} \cite[No. 2,5,11 in Section 12.5]{IP} : Consider the toric variety $M = \pp(\mcal{O} \oplus \mcal{O}(1,1))$ whose moment polytope 
		is described in the first of Figure \ref{figure_II_2_2_1}. For the $S^1$-action generated by $\xi = (0,0,1)$, the minimal fixed component $Z_{\min}$ is diffeomorphic to 
		$\p^1 \times \p^1$ and it has volume 18, see also 
		Example \ref{example_Sec8_II_1} {\bf (II-1-2.3)}.
		
	           		 \begin{figure}[H]
	           	 		\scalebox{0.8}{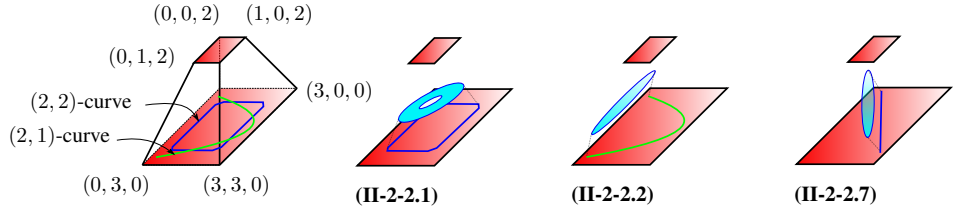}
		           	 	\caption{\label{figure_II_2_2_1} Fano varieties of types {\bf (II-2-2.1), (II-2-2.2), (II-2-2.7)}}
		           	 \end{figure}
		\noindent
		Let $C_{(a,b)}$ be a smooth curve of bidegree $(a,b)$ on $Z_{\min}$. Since $[\omega_1] = 3x + 3y$ as in Example \ref{example_Sec8_II_1} {\bf (II-1-2.3)}, we may blow-up 
		of $M$ along $C_{(a,b)}$ where $(a,b) = (1,0), (2,0), (1,1),(2,1)$ and $(2,2)$. Denote by $M_{(a,b)}$ the $S^1$-equivariant blow-up of $M$ along $C_{(a,b)}$. Then
		each $M_{(a,b)}$ has three fixed components and one can check that 
		
		\begin{table}[H]
		\begin{adjustbox}{max width=\textwidth}
			\begin{tabular}{|c|c|c|c|c|c|}
				\hline
					$(a,b)$ & $(1,0)$ & $(1,1)$ & $(2,0)$ & $(2,1)$ & $(2,2)$ \\ \hline \hline
					$\mathrm{Vol}(Z_{\min})$ & $12$ & $8$ & $6$ & $4$ & $2$ \\ \hline
					$\mathrm{Vol}(Z_{\max})$ & $2$ & $2$ & $2$ & $2$ & $2$ \\ \hline
					$\mathrm{Vol}(Z_0)$ & $2$ & $4$ & $4$ & $6$ & $8$ \\ \hline
					TFD type & {\bf (II-2-2.9)} & {\bf (II-2-2.7)} & {\bf (II-2-2.5)} & {\bf (II-2-2.2)} & {\bf (II-2-2.1)} \\ \hline
			\end{tabular}		
		\end{adjustbox}
		\end{table}
		\noindent
		We visualize each cases in Figure \ref{figure_II_2_2_1} except for the cases {\bf (II-2-2.5)} and {\bf (II-2-2.9)}, where the latter two types can be 
		represented by toric Fano varieties as we will see below.
		
		\begin{remark}
			Let $M = \p^1 \times \p^1 \times \p^1$ with the standard $T^3$-action and consider the $S^1$-action generated by $\xi = (0,0,1)$. 
			Then there are two fixed components that are both diffeomorphic to $S^2 \times S^2$ with the same volume 8. If $C$ is a curve of bidegree $(1,1)$
			on $Z_{\max}$, then one can take the $S^1$-equivariant blow-up of $M$ along $C$ where we denote the resulting manifold by $\widetilde{M}$. 
			Then we have 
			\[
				\mathrm{Vol}(\widetilde{Z}_{\max}) = 2, \quad \mathrm{Vol}(\widetilde{Z}_{\min}) = 8, \quad \mathrm{Vol}(\widetilde{Z}_{0}) = 4
			\]
			and so the fixed point data is of type {\bf (II-2-2.7)}. See Figure \ref{figure_II_2_2_1_1}. We will see in Section \ref{secMainTheorem} that 
			$\widetilde{M}$ is $S^1$-equivariantly symplectomorphic to the blow-up of $\pp(\mcal{O} \oplus \mcal{O}(1,1))$ along a curve of bidegree $(1,1)$
			described in Figure \ref{figure_II_2_2_1}.
			
	           		 \begin{figure}[H]
	           	 		\scalebox{0.8}{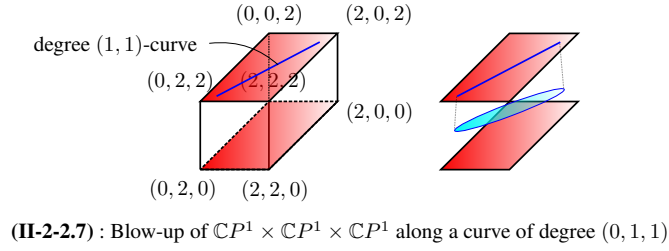}
		           	 	\caption{\label{figure_II_2_2_1_1} Fano varieties of types {\bf (II-2-2.7)}}
		           	 \end{figure}
			
		\end{remark}
		\vs{0.3cm}
		
		\item {\bf (II-2-2.3), (II-2-2.4), (II-2-2.8), (II-2-2.9)} \cite[No. 7,8 in Section 12.5, No. 2,3 in Section 12.6]{IP} : 			
		Now we let $M = \p^1 \times X_1$ admitting the standard $T^3$-action where the moment polytope is given in Figure \ref{figure_II_2_2_2}. With respect to 
		the $S^1$-action generated by $\xi = (0,0,1)$, we have two fixed components such that $\mathrm{Vol}(Z_{\min}) = 12$ and $\mathrm{Vol}(Z_{\max}) = 4$.
		
	           		 \begin{figure}[H]
	           	 		\scalebox{0.8}{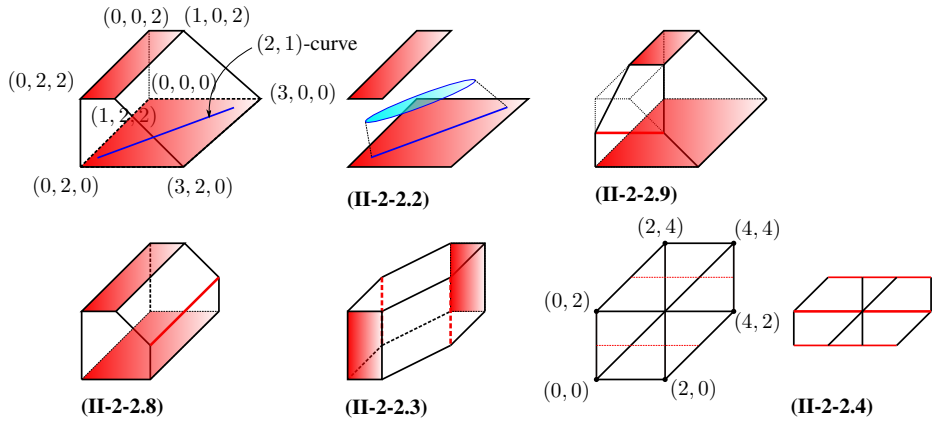}
		           	 	\caption{\label{figure_II_2_2_2} Fano varieties of types {\bf (II-2-2.3), (II-2-2.4), (II-2-2.8), (II-2-2.9)}}
		           	 \end{figure}
		
		\noindent
		Note that $[\omega_{-1}] = x + 2y$ and $[\omega_1] = 3x + 2y$. So, we may obtain a Fano variety by blowing up $M$ along 
		\begin{itemize}
			\item a curve of bidegree $(a,b)$ on $Z_{\min}$ where $(a,b) = (1,0), (2,0), (1,1),$ or $(2,1)$, or 
			\item a curve of bidegree $(0,1)$ on $Z_{\max}$.
		\end{itemize}
		and each of them has fixed point set as follows.
		\begin{table}[H]
		\begin{adjustbox}{max width=\textwidth}
			\begin{tabular}{|c|c|c|c|c|c|}
				\hline
					$(a,b)$ & $(1,0)$ & $(2,0)$ & $(1,1)$ & $(2,1)$ & $(0,1)$ \\ \hline \hline
					$\mathrm{Vol}(Z_{\min})$ & $8$ & $4$ & $4$ & $2$ & $12$ \\ \hline
					$\mathrm{Vol}(Z_{\max})$ & $4$ & $4$ & $4$ & $4$ & $2$ \\ \hline
					$\mathrm{Vol}(Z_0)$ & $2$ & $4$ & $4$ & $6$ & $2$ \\ \hline
					TFD type & {\bf (II-2-2.8)} & {\bf (II-2-2.3)} & {\bf (II-2-2.4)} & {\bf (II-2-2.2)} & {\bf (II-2-2.9)} \\ \hline
			\end{tabular}		
		\end{adjustbox}
		\end{table}
		\noindent
		
		\begin{remark}
			Note that 
			\begin{itemize}
				\item the blow-up of $\pp(\mcal{O} \oplus \mcal{O}(1,1))$ along a curve of bidegree $(2,1)$ and
				\item the blow-up of $\p^1 \times X_1$ along a curve of bidegree $(2,1)$
			\end{itemize}
			have the same fixed point data {\bf (II-2-2.2)}. We will see in Section \ref{secMainTheorem} 
			that they are $S^1$-equivariantly symplectormophic to each other.
		\end{remark}

		\begin{remark}
			Consider the complete flag variety $\mcal{F}(3)$ equipped with the $T^2$-action where the corresponding moment polytope is given 
			on the bottom-right in Figure \ref{figure_II_2_2_2}. By taking $T^2$-equivariant blow-up along two $T^2$-invariant spheres $C_{\min}$ and $C_{\max}$ corresponding
			the edges $\overline{(0,0) ~(2,0)}$ and $\overline{(2,4) ~(4,4)}$, we obtain a Fano variety and it has three fixed components such that 
			\begin{itemize}
				\item $Z_{\min} \cong Z_{\max} \cong S^2 \times S^2$ (since $\mathrm{Vol}(C_{\min}) = \mathrm{Vol}(C_{\max})  = 2$ and by Lemma \ref{lemma_volume})
				\item $\mathrm{Vol}(Z_0) = 4$ and $\mathrm{Vol}(Z_{\min}) = \mathrm{Vol}(Z_{\max}) = 4$. 
			\end{itemize}
			So, the fixed point data is of type {\bf (II-2-2.4)}. Consequently, two semifree $S^1$-Fano varieties 
			\begin{itemize}
				\item the blow-up of $\p^1 \times X_1$ along a curve of bidegree $(1,1)$ on the mnimum, and
				\item the blow-up of $\mcal{F}(3)$ along the disjoint union of two spheres
			\end{itemize}
			have the same fixed point data. We will see in Section \ref{secMainTheorem} that they are $S^1$-equivariantly symplectomorphic. 
		\end{remark}

		\item {\bf (II-2-2.5), (II-2-2.6)} \cite[No. 9,10 in Section 12.5]{IP} : In this case, we consider $M$, the $T^3$-equivariant blow-up of $\p^3$ along two disjoint $T^3$-invariant curves
		which we described in Example \ref{example_Sec8_I_3_1} {\bf (I-3-1.4)}. We fix one exceptional divisor $D$ on $M$ and let $C_1$, $C_2$ be two disjoint $T^3$-invariant curves. Then, 
		we obtain two Fano varieties
		\begin{itemize}
			\item $M_1$ : the $T^3$-equivariant blow-up of $M$ along $C_1 ~\dot \sqcup ~C_2$.
			\item $M_2$ : the $T^3$-equivariant blow-up of $M$ along $C_1$,
		\end{itemize}
		One can check that the $S^1$-action on each $M_i$ generated by $\xi = (-1,-1,0)$ is semifree. Then it immediately follows from Figure \ref{figure_II_2_2_3} that   
		\begin{table}[H]
		\begin{adjustbox}{max width=\textwidth}
			\begin{tabular}{|c|c|c|}
				\hline
							& $M_1$ & $M_2$ \\ \hline \hline
					$\mathrm{Vol}(Z_{\min})$ & $6$ & $6$ \\ \hline
					$\mathrm{Vol}(Z_{\max})$ & $2$ & $4$ \\ \hline
					$\mathrm{Vol}(Z_0)$ & $(2,2)$ & $2$ \\ \hline
					TFD type & {\bf (II-2-2.5)} & {\bf (II-2-2.6)} \\ \hline
			\end{tabular}		
		\end{adjustbox}
		\end{table}
		\noindent
		where the fixed component $Z_0$ in $M_1$ consists of two spheres corresponding to the edges $\overline{(0,1,1) ~(1,0,1)}$ and $\overline{(0,2,0) ~(2,0,0)}$ in the middle of 
		Figure \ref{figure_II_2_2_3} whose affine distances are both two. See also Example \ref{example_Sec8_I_3_1} {\bf (I-3-1.5)} and Example \ref{example_IV_1} (4).
	\end{itemize}					           	 
	
	           		 \begin{figure}[H]
	           	 		\scalebox{0.8}{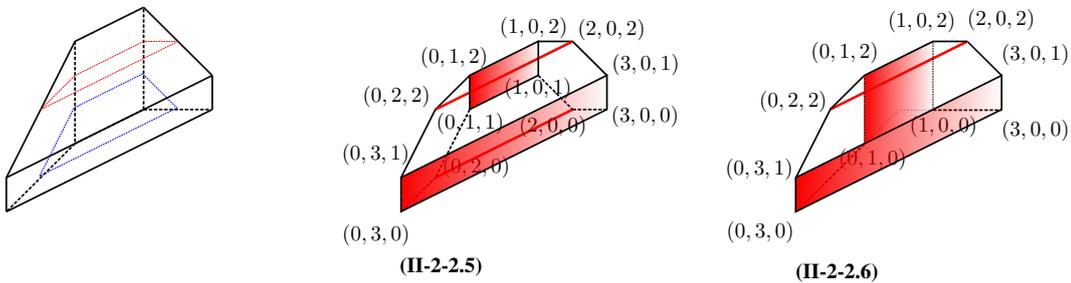}
		           	 	\caption{\label{figure_II_2_2_3} Fano varieties of type {\bf (II-2-2.5) and (II-2-2.6)}}
		           	 \end{figure}

	\end{example}
								
	\begin{theorem}[Case {\bf (II-2-3)}]\label{theorem_Sec8_II_2_3}
		Let $(M,\omega)$ be a six-dimensional closed monotone semifree Hamiltonian $S^1$-manifold such that $\mathrm{Crit} H = \{ 1,0,-1\}$. 
		Suppose that $M_0 \cong X_1$. Then the list of all possible topological fixed point data is given in the Table \ref{table_Sec8_II_2_3}
		\begin{table}[H]
			\begin{tabular}{|c|c|c|c|c|c|c|c|c|}
				\hline
				    & $(M_0, [\omega_0])$ & $e(P_{-1}^+)$ & $Z_{-1}$  & $Z_0$ &  $Z_1$ & $b_2(M)$ & $c_1^3(M)$ \\ \hline \hline
				    {\bf (II-2-3.1)} & $(X_1, 3u - E_1)$ & $-u$  & $X_1$  & \makecell{$Z_0 \cong S^2$, \\ $\mathrm{PD}(Z_0) = 2u$} & 
				    	$X_1$
					     & $4$ & $32$ \\ \hline    
				    {\bf (II-2-3.2)} & $(X_1, 3u - E_1)$ & $-E_1$  & $X_1$  & \makecell{$Z_0 \cong S^2 ~\dot \cup ~ S^2$, \\ $\mathrm{PD}(Z_0^1) = u$ \\ $\mathrm{PD}(Z_0^2) = E_1$} & 
				    	$X_1$
					     & $5$ & $36$ \\ \hline    
				    {\bf (II-2-3.3)} & $(X_1, 3u - E_1)$ & $0$  & $X_1$  & \makecell{$Z_0 \cong S^2 $, \\ 
				    $\mathrm{PD}(Z_0) = u$} & 
				    	$X_1$
					     & $4$ & $40$ \\ \hline    			     					     
				    {\bf (II-2-3.4)} & $(X_1, 3u - E_1)$ & $u - E_1$  & $X_1$  & \makecell{$Z_0 \cong S^2$, \\ $\mathrm{PD}(Z_0) = E_1$} & 
				    	$X_1$
					     & $4$ & $46$ \\ \hline    					     
				    {\bf (II-2-3.5)} & $(X_1, 3u - E_1)$ & $0$  & $X_1$  & \makecell{$Z_0 \cong S^2 $, \\ 
				    $\mathrm{PD}(Z_0) = E_1$} & 
				    	$X_1$
					     & $4$ & $44$ \\ \hline    			     					     
			\end{tabular}		
			\vs{0.5cm}			
			\caption{\label{table_Sec8_II_2_3} Topological fixed point data for $\mathrm{Crit} H = \{1,0,-1\}$ with $M_0 = X_1$}
		\end{table}				   
	\end{theorem}

	\begin{proof}
		Let $e = au + bE_1$ and $\mathrm{PD}(Z_0) = pu + qE_1$ for some $a,b,p,q \in \Z$. By the Dustermaat-Heckman theorem, we get 
		\[
			[\omega_{-1}] = (3+a)u + (b-1)E_1 \quad \text{and} \quad [\omega_1] = (3 - (a+p))u + (-1 - (b + q))E_1. 
		\]
		By the following inequalities $\int_{M_t} \omega_t^2 > 0$, $\langle [\omega_t], E_1 \rangle > 0$, and $\mathrm{Vol}(Z_0) > 0$, we obtain
		\[
			3+ a > 1 - b \geq 1, \quad 3- a - p > 1 + b + q \geq 1, \quad 3p + q \geq 1.
		\]
		Moreover, the condition \eqref{equation_II_1_assumption} allows us to assume that
		\[
			(3+a)^2 - (b-1)^2 \geq (3 - a - p)^2 - (1 + b + q)^2.
		\]
		Combining those equations, we obtain 
		\[
			p+q \leq 2, \quad 3p + q \geq 1, \quad q \geq 0
		\]
		whose integral solutions are $(2,0), (1,0), (0,1), (0,2), (1,1)$. Using the adjunction formula, we obtain the following table. 
		which coincides with Table \ref{table_Sec8_II_2_3}. 
		\begin{table}[H]
		\begin{adjustbox}{max width=\textwidth}
			\begin{tabular}{|c|c|c|c|c|c|c|c|}
				\hline
					& {\bf (II-2-2.1)} & {\bf (II-2-2.2)} & {\bf (II-2-2.3)} & {\bf (II-2-2.4)} & {\bf (II-2-2.5)} & {\bf (II-2-2.6)} & {\bf (II-2-2.7)}\\ \hline \hline
					$\mathrm{PD}(Z_0)= (p,q)$ & $(2,0)$ & $(1,1)$ & $(1,0)$ & $(0,1)$  & $(0,1)$ & $(0,2)$ & $(0,2)$  \\ \hline					
					$e = (a,b)$ & $(-1,0)$ & $(0, -1)$ & $(0,0)$ & $(1,-1)$ & $(0,0)$ & $(1,-2)$ & $(0,-1)$ \\ \hline
					$\mathrm{Vol}(Z_{\min})$ & $3$ & $5$ & $8$ & $12$ & $8$ & $7$ & $5$ \\ \hline
					$\mathrm{Vol}(Z_{\max})$ & $3$ & $3$ & $3$ & $3$ & $5$ & $3$ & $5$ \\ \hline
					\makecell{$\langle c_1(TM_0), [Z_0] \rangle$ \\ $= \mathrm{Vol}(Z_0)$} & $6$ & $4$ & $3$ & $1$ & $1$ & $2$ & $2$ \\ \hline
					$[Z_0] \cdot [Z_0]$ & $4$ & $0$ & $1$ & $-1$ & $-1$ & $-4$ & $-4$ \\ \hline
					\# components & $1$ & $2$ & $1$ & $1$ & $1$ & $\times$ & $\times$ \\ \hline
					$g$ : genus of $Z_0$ & $0$ & $(0,0)$ & $0$ & $0$ & $0$ & $\times$ & $\times$ \\ \hline					
			\end{tabular}		
		\end{adjustbox}
		\end{table}
		\noindent
		Note that the Chern numbers in Table \ref{table_Sec8_II_2_3} can be obtained from Lemma \ref{lemma_II_Chern_number} by direct computation.
	\end{proof}
				
	\begin{example}[Fano varieties of type {\bf (II-2-3)}]\label{example_II_2_3}\cite[No. 4,9,11,12 in Section 12.5, No. 2 in Section 12.6]{IP} 
	  We describe Fano varieties of type {\bf (II-2-3)} in Theorem \ref{theorem_Sec8_II_2_3} as follows. 
			           	 
	\begin{itemize}
		\item {\bf (II-2-3.1), (II-2-3.2), (II-2-3.3), (II-2-3.4)} \cite[No. 4,9,12 in Section 12.5, No. 2 in Section 12.6]{IP} : 		
		Let $M$ be the $T^3$-equivariant blow-up of $V_7$ along a $T^3$-invariant line passing through the exceptional divisor of $V_7 \rightarrow \p^3$
		where the moment polytope of $M$ is depicted on the left of Figure \ref{figure_II_2_3_1}. 
		Then the $S^1$-subgroup of $T^3$ generated by $\xi = (0,0,1)$ acts on $M$ semifreely such that $\mathrm{Vol}(Z_{\min}) = 15$ and $\mathrm{Vol}(Z_{\max}) = 3$.		
		See also Example \ref{example_Sec8_II_1} {\bf (II-1-3.2)}. 			
		
			 \begin{figure}[H]
	           	 		\scalebox{0.8}{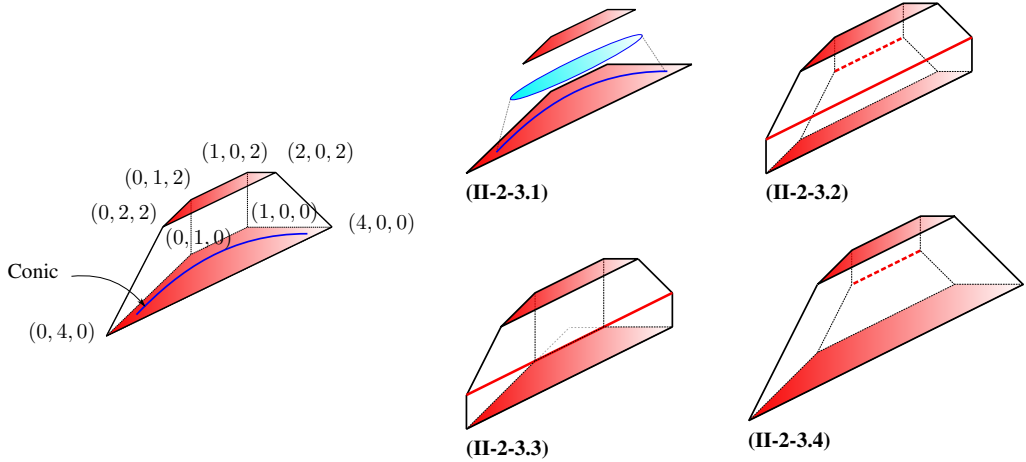}
		           	 	\caption{\label{figure_II_2_3_1} {\bf (II-2-3.1)} : Blow-up of $M$ along the conic}
	           	 \end{figure}
	           	 
		\noindent
		There are four possible ways of blowing $M$ up to obtain a semifree $S^1$-Fano variety. Let $C_{(a,b)}$ be a smooth (possibly disconnected) curve 
		of bidegree $(a,b)$ in $Z_{\min} \cong X_1$. Since $[\omega_{-1}] = 4u - E_1$ and by the adjunction formula, we can easily check that all possible $(a,b)$'s are $(1,0), (2,0), (0,1), (1,1)$.
		(Note that $(4u - E_1) - (au  + bE_1)$ becomes the symplectic form on the minimal fixed component of the blow-up of $M$ along $C_{(a,b)}$.) Denote by $M_{(a,b)}$ the $S^1$-equivariant 
		blow-up of $M$ along $C_{(a,b)}$. Then
		\begin{itemize}
			\item $M_{(1,0)}$ is the blow-up of $M$ along a line $H$ (corresponding to the edge $\overline{(4,0,0)~(0,4,0)}$),
			\item $M_{(2,0)}$ is the blow-up of $M$ along a conic,
			\item $M_{(0,1)}$ is the blow-up of $M$ along an exceptional line $E$ (corresponding to the edge $\overline{(1,0,0)~(0,1,0)}$),
			\item $M_{(1,1)}$ is the blow-up of $M$ along a disjoint union of an exceptional line $E$ and a line $H$.
		\end{itemize}
		Moreover, the volume of the extremal fixed component for each $M_{(a,b)}$ can be computed directly from Figure \ref{figure_II_2_3_1} and therefore we obtain the following.
		\vs{0.3cm}	      
		\begin{table}[H]
		\begin{adjustbox}{max width=\textwidth}
			\begin{tabular}{|c|c|c|c|c|}
				\hline
							& $M_{(1,0)}$ & $M_{(2,0)}$ &  $M_{(0,1)}$ & $M_{(1,1)}$ \\ \hline \hline
					$\mathrm{Vol}(Z_{\min})$ & $8$ & $3$ & $12$ & $5$ \\ \hline
					$\mathrm{Vol}(Z_{\max})$ & $3$ & $3$ & $3$ & $3$  \\ \hline
					$\mathrm{Vol}(Z_0)$ & $3$ & $(3,3)$ & $1$ & $4$ \\ \hline
					TFD type & {\bf (II-2-3.3)} & {\bf (II-2-3.1)} & {\bf (II-2-3.4)} & {\bf (II-2-3.2)} \\ \hline
			\end{tabular}		
		\end{adjustbox}
		\end{table}
		\noindent
	           	 
		\item {\bf (II-2-3.5)} \cite[No. 11 in Section 12.5]{IP} : Consider $M = \p^1 \times X_1$ with the standard $T^3$-action and let $\widetilde{M}$ be the $T^3$-equivariant blow-up 
		along the $T^3$-invariant curve 	$t \times E$ where $t$ is a fixed point of the $S^1$-action on $\p^1$ and $E$ is the exceptional divisor of $X_1$. (Note that the curve $t \times E$
		corresponds to the edge $\overline{(0,2,2) ~(1,2,2)}$ in Figure \ref{figure_II_2_3_2}.)
		
			 \begin{figure}[H]
	           	 		\scalebox{0.8}{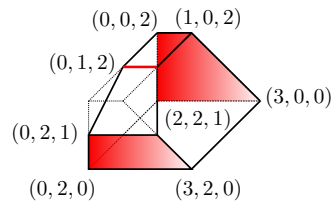}
		           	 	\caption{\label{figure_II_2_3_2} {\bf (II-2-3.5)} : Blow-up of $\p^1 \times X_1$ along $t \times E$}
	           	 \end{figure}		
	           \noindent
	           The $S^1$ subgroup of $T^3$ generated by $\xi = (0,1,0)$ is semifree and the fixed point set satisfies 
	           \[
	           	\mathrm{Vol}(Z_{\min}) = 8, \quad \mathrm{Vol}(Z_{0}) = 1, \quad \mathrm{Vol}(Z_{\max}) = 5
	           \]
	           and therefore the fixed point data coincides with {\bf (II-2-3.5)} in Table \ref{table_Sec8_II_2_3}.
	           	 
	\end{itemize}			           	 

	\end{example}

		{\footnotesize
		\begin{longtable}{|c|c|c|c|c|c|c|c|c|}
			\hline
			& $(M_0, [\omega_0])$ & $e(P_{-2}^+)$ & $Z_{-2}$ & $Z_{-1}$ &  $Z_0$ & $Z_1$ & $b_2$ & $c_1^3$ \\ \hline \hline
		{\bf (I-1-1.1)} & $(E_{S^2}, 3x + 2y)$  & $-x-y$  & $S^2$ & & & $E_{S^2}$ & $2$ & $54$ \\ \hline    
		{\bf (I-1-2.1)} & $(S^2 \times S^2, 2x + 2y)$ & $-y$  & $S^2$ & & & $S^2 \times S^2$ & $2$ & $54$ \\ \hline    
		{\bf (I-1-2.2)} & $(S^2 \times S^2, 2x + 2y)$ & $x-y$  & $S^2$ & & & $S^2 \times S^2$ & $2$ & $54$ \\ \hline    
		{\bf (I-2)} & $(E_{S^2}, 3x + 2y - E_1)$ & $-x-y$  & $S^2$ & $\mathrm{pt}$ &  & $E_{S^2} \# ~\overline{\p^2}$ & $3$ & $46$ \\ \hline    
		{\bf (I-3-1.1)} & $(E_{S^2}, 3x + 2y)$ & $-y$ & & $S^2$ & \makecell{$Z_0 \cong S^2$ \\ $\mathrm{PD}(Z_0) = y$} & $E_{S^2}$ & $3$ & $52$ \\ \hline    
		{\bf (I-3-1.2)} & $(E_{S^2}, 3x + 2y)$ & $-x-y$  & $S^2$ & & \makecell{$Z_0 \cong S^2$ \\ $\mathrm{PD}(Z_0) = y$} & $E_{S^2}$ & $3$ & $50$ \\ \hline    
		{\bf (I-3-1.3)} & $(E_{S^2}, 3x + 2y)$ & $-y$  & $S^2$ &  &
				    				\makecell{$Z_0 \cong S^2 ~\dot\cup ~S^2$ \\ $\mathrm{PD}(Z_0^1) = x+y$ \\ $\mathrm{PD}(Z_0^2) = y$} 				    
				    															& $E_{S^2}$ & $4$ & $44$ \\ \hline    
		{\bf (I-3-1.4)} & $(E_{S^2}, 3x + 2y)$ & $-x-y$  & $S^2$ & & \makecell{$Z_0 \cong S^2$ \\ $\mathrm{PD}(Z_0) = x + y$} 
				    								& $E_{S^2}$ & $3$ & $44$ \\ \hline    
		{\bf (I-3-1.5)} & $(E_{S^2}, 3x + 2y)$ & $-x-y$  & $S^2$ & &
				    							\makecell{$Z_0 \cong S^2 ~\dot\cup ~S^2$ \\ $\mathrm{PD}(Z_0^1) = x+y$ \\ $\mathrm{PD}(Z_0^2) = y$} 
				    								& $E_{S^2}$ & $4$ & $40$ \\ \hline    
		{\bf (I-3-1.6)} & $(E_{S^2}, 3x + 2y)$ & $-x-y$  & $S^2$ & &\makecell{$Z_0 \cong S^2$ \\ $\mathrm{PD}(Z_0) = 2x + 2y$} & $E_{S^2}$ & $3$ & $36$ \\ \hline    
		{\bf (I-3-2.1)} & $(S^2 \times S^2, 2x + 2y)$ & $-y$  & $S^2$ & &\makecell{$Z_0 \cong S^2$ \\ $\mathrm{PD}(Z_0) = y$} & $S^2 \times S^2$ & $3$ & $48$ \\ \hline    
		{\bf (I-3-2.2)} & $(S^2 \times S^2, 2x + 2y)$ & $x-y$  & $S^2$ & &\makecell{$Z_0 \cong S^2$ \\ $\mathrm{PD}(Z_0) = y$} 
				    															&$S^2 \times S^2$ & $3$ & $50$ \\ \hline    
		{\bf (I-3-2.3)} & $(S^2 \times S^2, 2x + 2y)$ & $-y$  & $S^2$ &  &
				    				\makecell{$Z_0 \cong S^2 ~\dot\cup ~S^2$ \\ $\mathrm{PD}(Z_0^1) = y$ \\ $\mathrm{PD}(Z_0^2) = y$} 				    
				    															&$S^2 \times S^2$ & $4$ & $42$ \\ \hline    
		{\bf (I-3-2.4)} & $(S^2 \times S^2, 2x + 2y)$ & $x-y$  & $S^2$ & &
				    								\makecell{$Z_0 \cong S^2 ~\dot\cup ~S^2$ \\ $\mathrm{PD}(Z_0^1) = y$ \\ $\mathrm{PD}(Z_0^2) = y$} 
				    								& $S^2 \times S^2$ & $4$ & $46$ \\ \hline    
		{\bf (I-3-2.5)} & $(S^2 \times S^2, 2x + 2y)$ & $-y$  & $S^2$ & &
				    							\makecell{$Z_0 \cong S^2$ \\ $\mathrm{PD}(Z_0) = x$}
				    								& $S^2 \times S^2$ & $3$ & $46$ \\ \hline    
		{\bf (I-3-2.6)} & $(S^2 \times S^2, 2x + 2y)$ & $-y$  & $S^2$ & &
				    							\makecell{$Z_0 \cong S^2$ \\ $\mathrm{PD}(Z_0) = x + y$}
				    								& $S^2 \times S^2$ & $3$ & $42$ \\ \hline    
		{\bf (I-3-2.7)} & $(S^2 \times S^2, 2x + 2y)$ & $-y$  & $S^2$ & &
				    							\makecell{$Z_0 \cong S^2$ \\ $\mathrm{PD}(Z_0) = x + 2y$}
				    								& $S^2 \times S^2$ & $3$ & $38$ \\ \hline    				    									{\bf (I-4-1.1)} & \makecell{$(E_{S^2} \# \overline{\p^2},$ \\ $3x + 2y - E_1)$} & $-x-y$  & $S^2$ & $\mathrm{pt}$ & 
				    								\makecell{ $Z_0 \cong S^2$  \\ $\mathrm{PD}(Z_0) = x + y - E_1$} & $X_2$ & $4$ & $40$ \\ \hline    
		{\bf (I-4-1.2)} & \makecell{$(E_{S^2} \# 2\overline{\p^2},$ \\ $3x + 2y - E_1 - E_2)$} & $-x-y$  & $S^2$ & $\mathrm{2 ~pts}$ & 
			\makecell{ $Z_0 \cong S^2$  \\  $\mathrm{PD}(Z_0) = x + y - E_1 - E_2$} & $X_3$ & $5$ & $36$ \\ \hline    					     
		{\bf (I-4.2)} & \makecell{$(S^2 \times S^2 \# \overline{\p^2},$ \\  $2x + 2y - E_1)$} & $-y$  & $S^2$ & $\mathrm{pt}$ & 
		\makecell{ $Z_0 \cong S^2$  \\ $\mathrm{PD}(Z_0) = y - E_1$} & $X_2$ & $4$ & $44$ \\ \hline    					     					     			
				    {\bf (II-1-1.1)} & $(\p^2, 3u)$ & $0$  & & $\p^2$  & & 
				    	$\p^2$
					     & $2$ & $54$ \\ \hline    
				    {\bf (II-1-1.2)} & $(\p^2, 3u)$ & $u$  & & $\p^2$  && 
				    	$\p^2$
					     & $2$ & $56$ \\ \hline    
				    {\bf (II-1-1.3)} & $(\p^2, 3u)$ & $2u$  &  &$\p^2$  && 
				    	$\p^2$
					     & $2$ & $62$ \\ \hline    
					     					     					     
				    {\bf (II-1-2.1)} & $(S^2 \times S^2, 2x + 2y)$ & $0$  & & $S^2 \times S^2$  && 
				    	$S^2 \times S^2$
					     & $3$ & $48$ \\ \hline    					     
				    {\bf (II-1-2.2)} & $(S^2 \times S^2, 2x + 2y)$ & $x$  & & $S^2 \times S^2$ && 
				    	$S^2 \times S^2$
					     & $3$ & $48$ \\ \hline    					     
				    {\bf (II-1-2.3)} & $(S^2 \times S^2, 2x + 2y)$ & $x+y$  & & $S^2 \times S^2$ && 
				    	$S^2 \times S^2$
					     & $3$ & $52$ \\ \hline    					     
				    {\bf (II-1-2.4)} & $(S^2 \times S^2, 2x + 2y)$ & $x-y$  & &$S^2 \times S^2$ && 
				    	$S^2 \times S^2$
					     & $3$ & $44$ \\ \hline    					     
					     					     					     					     
				    {\bf (II-1-3.1)} & $(E_{S^2}, 3x + 2y)$ & $0$  & & $E_{S^2}$ & & 
				    	$E_{S^2}$
					     & $3$ & $48$ \\ \hline    					     					     
				    {\bf (II-1-3.2)} & $(E_{S^2}, 3x + 2y)$ & $x + y$  & & $E_{S^2}$ && 
				    	$E_{S^2}$
					     & $3$ & $50$ \\ \hline    					     					     
					     					     
				    \makecell{{\bf (II-1-4.k)} \\ {\bf k = 2$\sim$8}} & $(X_k, 3u - \sum_{i=1}^k E_i)$ & $0$  & & $X_k$ & &
				    	$X_k$
					     & $k+2$ & $54-6k$ \\ \hline    					     					     					     	
				    {\bf (II-2-1.1)} & $(\p^2, 3u)$ & $0$  & &$\p^2$  & \makecell{$Z_0 \cong S^2$, \\ $\mathrm{PD}(Z_0) = u$} & 
				    	$\p^2$
					     & $3$ & $46$ \\ \hline    
				    {\bf (II-2-1.2)} & $(\p^2, 3u)$ & $u$  & & $\p^2$  & \makecell{$Z_0 \cong S^2$, \\ $\mathrm{PD}(Z_0) = u$} & 
				    	$\p^2$
					     & $3$ & $50$ \\ \hline    
				    {\bf (II-2-1.3)} & $(\p^2, 3u)$ & $-u$  & & $\p^2$  & \makecell{$Z_0 \cong S^2$, \\ $\mathrm{PD}(Z_0) = 2u$} & 
				    	$\p^2$
					     & $3$ & $38$ \\ \hline    
					     					     					     
				    {\bf (II-2-1.4)} & $(\p^2, 3u)$ & $0$  & & $\p^2$  & \makecell{$Z_0 \cong S^2$, \\ $\mathrm{PD}(Z_0) = 2u$} & 
				    	$\p^2$
					     & $3$ & $40$ \\ \hline    					     
				    {\bf (II-2-1.5)} & $(\p^2, 3u)$ & $-u$  & & $\p^2$ & \makecell{$Z_0 \cong T^2$, \\ $\mathrm{PD}(Z_0) = 3u$} & 
				    	$\p^2$
					     & $3$ & $32$ \\ \hline    					     
				    {\bf (II-2-1.6)} & $(\p^2, 3u)$ & $-2u$  & & $\p^2$ & \makecell{$Z_0 \cong \Sigma_3$, \\ $\mathrm{PD}(Z_0) = 4u$} & 
				    	$\p^2$
					     & $3$ & $26$ \\ \hline    					     
				    {\bf (II-2-2.1)} & $(S^2 \times S^2, 2x + 2y)$ & $-x -y$  & & $S^2 \times S^2$  & \makecell{$Z_0 \cong T^2$, \\ $\mathrm{PD}(Z_0) = 2x+2y$} & 
				    	$S^2 \times S^2$
					     & $4$ & $28$ \\ \hline    
				    {\bf (II-2-2.2)} & $(S^2 \times S^2, 2x + 2y)$ & $-x$  & & $S^2 \times S^2$  & \makecell{$Z_0 \cong S^2$, \\ $\mathrm{PD}(Z_0) = 2x + y$} & 
				    	$S^2 \times S^2$
					     & $4$ & $32$ \\ \hline    
				    {\bf (II-2-2.3)} & $(S^2 \times S^2, 2x + 2y)$ & $-x$  & & $S^2 \times S^2$  & \makecell{$Z_0 \cong S^2 ~\dot \cup ~S^2$, \\ 
				    $\mathrm{PD}(Z_0^1) = \mathrm{PD}(Z_0^2) = x$} & 
				    	$S^2 \times S^2$
					     & $5$ & $36$ \\ \hline    			     					     
				    {\bf (II-2-2.4)} & $(S^2 \times S^2, 2x + 2y)$ & $-x$  & & $S^2 \times S^2$  & \makecell{$Z_0 \cong S^2$, \\ $\mathrm{PD}(Z_0) = x+y$} & 
				    	$S^2 \times S^2$
					     & $4$ & $36$ \\ \hline    					     
				    {\bf (II-2-2.5)} & $(S^2 \times S^2, 2x + 2y)$ & $-x + y$  & & $S^2 \times S^2$  & \makecell{$Z_0 \cong S^2 ~\dot \cup ~S^2$, \\ 
				    $\mathrm{PD}(Z_0^1) = \mathrm{PD}(Z_0^2) = x$} & 
				    	$S^2 \times S^2$
					     & $5$ & $36$ \\ \hline    			     					     
				    {\bf (II-2-2.6)} & $(S^2 \times S^2, 2x + 2y)$ & $-x + y$  & & $S^2 \times S^2$  & \makecell{$Z_0 \cong S^2$, \\ $\mathrm{PD}(Z_0) = x$} & 
				    	$S^2 \times S^2$
					     & $4$ & $40$ \\ \hline    					     					     
				    {\bf (II-2-2.7)} & $(S^2 \times S^2, 2x + 2y)$ & $0$  & & $S^2 \times S^2$ & \makecell{$Z_0 \cong S^2$, \\ $\mathrm{PD}(Z_0) = x+y$} & 
				    	$S^2 \times S^2$
					     & $4$ & $38$ \\ \hline    					     
				    {\bf (II-2-2.8)} & $(S^2 \times S^2, 2x + 2y)$ & $0$  & & $S^2 \times S^2$ & \makecell{$Z_0 \cong S^2$, \\ $\mathrm{PD}(Z_0) = x$} & 
				    	$S^2 \times S^2$
					     & $4$ & $42$ \\ \hline    					     
				    {\bf (II-2-2.9)} & $(S^2 \times S^2, 2x + 2y)$ & $y$  & & $S^2 \times S^2$ & \makecell{$Z_0 \cong S^2$, \\ $\mathrm{PD}(Z_0) = x$} & 
				    	$S^2 \times S^2$
					     & $4$ & $44$ \\ \hline    					     					     					     
				    {\bf (II-2-3.1)} & $(X_1, 3u - E_1)$ & $-u$  & & $X_1$  & \makecell{$Z_0 \cong S^2$, \\ $\mathrm{PD}(Z_0) = 2u$} & 
				    	$X_1$
					     & $4$ & $32$ \\ \hline    
				    {\bf (II-2-3.2)} & $(X_1, 3u - E_1)$ & $-E_1$  & & $X_1$  & 
				    \makecell{$Z_0 \cong S^2 ~\dot \cup ~ S^2$, \\ $\mathrm{PD}(Z_0^1) = u$ \\ $\mathrm{PD}(Z_0^2) = E_1$} & 
				    	$X_1$
					     & $5$ & $36$ \\ \hline    
				    {\bf (II-2-3.3)} & $(X_1, 3u - E_1)$ & $0$  & & $X_1$  & \makecell{$Z_0 \cong S^2 $, \\ 
				    $\mathrm{PD}(Z_0) = u$} & 
				    	$X_1$
					     & $4$ & $40$ \\ \hline    			     					     
				    {\bf (II-2-3.4)} & $(X_1, 3u - E_1)$ & $u - E_1$  & & $X_1$  & \makecell{$Z_0 \cong S^2$, \\ $\mathrm{PD}(Z_0) = E_1$} & 
				    	$X_1$
					     & $4$ & $46$ \\ \hline    					     
				    {\bf (II-2-3.5)} & $(X_1, 3u - E_1)$ & $0$  & & $X_1$  & \makecell{$Z_0 \cong S^2 $, \\ 
				    $\mathrm{PD}(Z_0) = E_1$} & 
				    	$X_1$
					     & $4$ & $44$ \\ \hline  	     					     					     
		\multicolumn{3}{l}{} \\[-1pt]					     
		\caption{ \label{table_list_3} Topological fixed point data for $\dim Z_{\min} = \dim Z_{\max} = 4$}
	\end{longtable}
	}

\section{Main Theorem}
\label{secMainTheorem}

	In this section, we will prove Theorem \ref{theorem_main}.
	
	\begin{theorem}[Theorem \ref{theorem_main}]
		Let $(M,\omega)$ be a six-dimensional closed monotone symplectic manifold equipped with a semifree Hamiltonian 
		circle action. 
		Then $(M,\omega)$ is $S^1$-equivariantly symplectomorphic to some 
		K\"{a}hler Fano manifold with some holomorphic Hamiltonian circle action. 
	\end{theorem}

	We notice that, according to our classification result of topological fixed point data, any reduced space of $(M,\omega)$ in Theorem \ref{theorem_main} is either 
	$\p^2$, $\p^1 \times \p^1$, or $\p^2 \#~k~ \overline{\p^2}$ for $1 \leq k \leq 4$ except for the cases 
	{\bf (II-1-4.k)} where $5 \leq k \leq  8$. See Table \ref{table_list}, \ref{table_list_2}, and \ref{table_list_3}. 
	The following theorems then imply that every reduced space (except for {\bf (II-1-4.k)})
	is symplectically rigid (in the sense of \cite[Definition 2.13]{McD2} or \cite[Definition 1.4]{G}). (See also Section \ref{secFixedPointData}.)

	\begin{theorem}\cite[Theorem 1.2]{McD4}\label{theorem_uniqueness}
		Let $M$ be a blow-up of a rational or a ruled symplectic four manifold. Then any two cohomologous and deformation equivalent\footnote{Two symplectic forms $\omega_0$ and $\omega_1$
		are said to be {\em deformation equivalent} if there exists a family of symplectic forms $\{ \omega_t  ~|~  0 \leq t \leq 1 \}$ connecting $\omega_0$ and $\omega_1$. We also say that 
		$\omega_0$ and $\omega_1$ are {\em isotopic} if such a family can be chosen such that $[\omega_t]$ is a constant path in $H^2(M; \Z)$.}
		symplectic forms on $M$ are isotopic.
	\end{theorem}

	\begin{theorem}\cite[Lemma 4.2]{G}\label{theorem_symplectomorphism_group}
	For any of the following symplectic manifolds, the group of symplectomorphisms  which act trivially on homology is path-connected. 
		\begin{itemize}
			\item $\p^2$ with the Fubini-Study form. \cite[Remark in p.311]{Gr}
			\item $\p^1 \times \p^1$ with any symplectic form. \cite[Theorem 1.1]{AM}
			\item $\p^2 \# ~k~\overline{\p^2}$ with any blow-up symplectic form for $k \leq 4$. \cite[Theorem 1.4]{AM}, \cite{E}, \cite{LaP}, \cite{Pin}, \cite{LLW}
		\end{itemize}
	\end{theorem}
	
	\begin{remark}
		In \cite{Se}, Seidel proved that the symplectic rigidity of $X_k$ fails for $k = 5$.
	\end{remark}
	
	From now on, we discuss how a topological fixed point data determines a fixed point data. 
	Note that a topological fixed point data only records homology classes of fixed components regarded as embedded submanifolds of reduced spaces. 
	In general, we cannot rule out the possibility that there are many distinct
	fixed point data which have the same topological fixed point data. 
	
	Recall that any non-extremal component of a topological fixed point data in Table \ref{table_list} is one of the forms
	\[
		(M_c, [\omega_c], [Z_c^1], \cdots, [Z_c^{k_c}]), \quad c = -1, 0, 1. 
	\]
	If $c = \pm 1$, then all $Z_c^i$'s are isolated points. In this case, the topological fixed point data determines a fixed point data uniquely, since if 
	\[
		(M_c, \omega_c, p_1, \cdots, p_r) \quad \text{and} \quad (M_c, \omega_c' , q_1, \cdots, q_r), \quad \quad p_i, q_j : \text{points},\quad  [\omega_c] = [\omega_c'],
	\]
	then it follows from the symplectic rigidity of $M_c$ (obtained by Theorem \ref{theorem_uniqueness} and Theorem \ref{theorem_symplectomorphism_group})
	that there exists a symplectomorphism $\phi : (M_c, \omega_c) \rightarrow (M_c, \omega_c')$ sending $p_i$ to $q_i$ for $i=1,\cdots,r$. 
	(See \cite[Proposition 0.3]{ST}.)  
	
	For $c= 0$, it is not clear whether a topological fixed point data determines a fixed point data uniquely. On the other hand, the following theorems guarantee that 
	any symplectic embedding $Z_0 \hookrightarrow M_0$ in Table \ref{table_list}, \ref{table_list_2}, and \ref{table_list_3}
	can be identified with an algebraic embedding. Note the followings.
	\begin{itemize}
		\item In Table \ref{table_list}, every $Z_0^i$, except for the case {\bf (III-3.3)}, in Table \ref{table_list} is a sphere with self intersection greater than equal to $-1$. 
			Moreover, in case of $M_0 \cong \p^2$, 
			the degree of $Z_0$ is less than or equal to $3$. In particular, $Z_0 \cong T^2$ in {\bf (III-3.3)} is of degree $3$, i.e., cubic, in $\p^2$.
		\item In Table \ref{table_list_2}, every $Z_0^i$ is a sphere with self intersection greater than equal to $-1$. 
		\item In Table \ref{table_list_3}, every $Z_0^i$ is a sphere with self intersection greater than equal to $-1$ except for the cases {\bf (II-2-1.5), (II-2-1.6), (II-2-2.1)}.
		For {\bf (II-2-1.5)} and {\bf (II-2-1.6)}, $Z_0 \subset \p^2$ is a Riemann surface of genus $1$ and $3$, respectively, and they are of degree less than 4.
		For {\bf (II-2-2.1)}, we see that $Z_0 \cong T^2$ living in $S^2 \times S^2$ with $\mathrm{PD}(Z_0) = 2x + 2y$. In particular $Z_0$ has relative degree two.

	\end{itemize} 
	
	\begin{theorem}\cite[Theorem C]{ST}\label{theorem_ST}
		Any symplectic surface in $\p^2$ of degree $d \leq17$ is symplectically isotopic to an algebraic curve.
		Also, any symplectic surface in $\p^1 \times \p^1$ of relative degree  $d \leq 7$ is symplectically isotopic to an algebraic curve.
	\end{theorem}
	Here, a {\em relative degree} of a curve $C$ in $\p^1 \times \p^1$ means the degree of the map $p|_C : C \rightarrow \p^1$ where $p : \p^1 \times \p^1 \rightarrow \p^1$
	is a projection on the left (or right) factor. Equivalently, the relative degree of $C$ is the intersection number of $C$ and the fiber of $p$.
	
	\begin{theorem}\cite[Proposition 3.2]{LW}\cite[Theorem 6.9]{Z}\label{theorem_Z}
		Any symplectic sphere $S$ with self-intersection $[S]\cdot[S] \geq 0$ in a symplectic four manifold $(M,\omega)$ is symplectically isotopic to an (algebraic) rational curve.
		Any two homologous spheres with self-intersection $-1$ are symplectically isotopic to each other.
	\end{theorem}
	
	Now we are ready to prove Theorem \ref{theorem_main}
	
	\begin{proof}[Proof of Theorem \ref{theorem_main}]
		We first assume that the topological fixed point data of $(M,\omega)$ is not of type {\bf (II-1-4.k)} for $k \geq 5$.
		Since every reduced space is symplecticaly rigid (by Theorem \ref{theorem_uniqueness} and Theorem \ref{theorem_symplectomorphism_group}), 
		it is enough to show that for each $(M,\omega,H)$, there exists a smooth Fano 3-fold admitting semifree holomorphic Hamiltonian $S^1$-action whose fixed point data 
		equals $\frak{F}(M,\omega,H)$. Then the proof immediately follows from Theorem \ref{theorem_Gonzalez}. 
		
		For any $(M,\omega,H)$ satisfying the conditions in Theorem \ref{theorem_main}, we have already seen in Section \ref{secClassificationOfTopologicalFixedPointDataDimZMin},
		 \ref{secClassificationOfTopologicalFixedPointDataDimZMinDimZMax2}, \ref{secClassificationOfTopologicalFixedPointDataDimZMinGeq2AndDimZMax4} that
		there exists a smooth Fano 3-fold $(X, \omega_X, H_X)$ with a holomorphic Hamiltonian $S^1$-action
		whose topological fixed point data equals $\frak{F}_{\mathrm{top}} (M,\omega,H)$. 
		By Theorem \ref{theorem_ST} and Theorem \ref{theorem_Z}, we may assume that every $(M_c, \omega_c, Z_c) \in \frak{F}(M,\omega,H)$ is an algebraic tuple, that is, 
		$Z_c$ is a complex (and hence K\"{a}hler) submanifold of $M_c$ for every critical value $c$ of the balanced moment map $H$. 
		Note that every reduced space of $(M,\omega,H)$ is either $\p^2$, $\p^1 \times \p^1$, or $\p^2 \# ~k \overline{\p^2}$ for $k \leq 4$, see Table \ref{table_list}. 
		In particular, any reduced space is birationally equivalent to $\p^2$, and therefore $H^1(M_c, \mcal{O}_{M_c}) = 0$. 
		Then Lemma \ref{lemma_isotopic} (stated below) implies that $(M_c, \omega_c, Z_c)$ is equivalent to the fixed point data $(X_c, (\omega_X)_c, (Z_X)_c)$ of $X$ at level $c$. 

		On the other hand, if $(M,\omega)$ is of type {\bf (II-1-4.k)} with $k \geq 5$, 
		then the TFD has no extremal fixed component and the Euler class of each level set (as a principal $S^1$-bundle)
		vanishes, see Table \ref{table_list_2}. Moreover, since any reduced space of $M$ is $X_k$, which is in particular a rational surface, any two deformation equivalent cohomologous 
		symplectic forms are isotopic by Theorem \ref{theorem_uniqueness}.
		
		 If $(M', \omega')$ is another semifree Hamiltonian $S^1$-manifold having the same TFD of $(M,\omega)$'s, 
		we obtain cohomologous families of reduced symplectic forms $\{\omega_t\}$ from $M$ and $\{\omega_t'\}$ from $M'$ such that $[\omega_t] = [\omega_t']$ for every $t \in [-1,1]$. 
		Note that both $\{\omega_t\}$ and $\{\omega_t'\}$ can be regarded as families of symplectic forms on $M_0$ since $M_0 \cong M_t$ for every $t \in [-1,1]$. Moreover, 
		by the uniqueness of symplectic structure on $M_{-1}$ (Theorem \ref{theorem_uniqueness}), we may assume that $\omega_{-1} = \omega_{-1}'$. 
		
		Define a two-parameter family of symplectic forms on $M_0$:
		\[
			\omega_{s,t} := \begin{cases}
				\omega_{(1-2s)t} & \text{if $0 \leq s \leq \frac{1}{2}$} \\
				\omega_{(2s - 1)t}' & \text{if $\frac{1}{2} \leq s \leq 1$} \\
			\end{cases}
		\]
		Then $[\omega_{s,t}]$ is a constant class in $H^2(M_0)$ by the vanishing of the Euler class $e(P_{-1}^+)$, see Table \ref{table_list_2} {\bf (II-1-4.k)}. Then 
		$\{ \omega_{s,t} \}$ satisfies 
		\begin{itemize}
			\item $\omega_{0, t} = \omega_t$ and $\omega_{1,t} = \omega_t'$,
			\item $\frac{d}{ds} [\omega_{s,t}] = 0$ 
		\end{itemize}
		for any $0 \leq s \leq 1$ and $-1 \leq t \leq 1$. (Equivalently, $\{\omega_t\}$ and $\{\omega_t'\}$ are equivalent in the sense of \cite[Definition 3.3]{G}.)
		Therefore, we can apply the Moser type argument to show that there exists a family of symplectomorphisms 
		\[
			\varphi_t : (M_0, \omega_t) \rightarrow (M_0, \omega_t'), \quad \varphi_t^*  \omega_t' =  \omega_t
		\]
		for every $t \in [-1,1]$. This family $\{\varphi_t\}$ induces an $S^1$-equivariant symplectomorphism
		\[
			\varphi : P \times [-1,1]  \rightarrow P \times [-1,1], \quad \varphi^*\widetilde{\omega}' = \widetilde{\omega}
		\]
		where 
		\begin{itemize}
			\item $P$ is the principal $S^1$-bundle $S^1 \times M_0 \rightarrow M_0$,
			\item $\widetilde{\omega}$ and $\widetilde{\omega}'$ are two $S^1$-invariant symplectic forms on $P \times I$ determined by $\{ \omega_t\}$ and $\{\omega_t'\}$, respectively.
		\end{itemize}
		See \cite[Proposition 5.8]{McS1}.
		Then $\varphi$ implies the symplectomorphism from $(M,\omega)$ to $(M,\omega')$ induced by the symplectic cut of
		$(P \times I, \widetilde{\omega})$ and $(P \times I, \widetilde{\omega}')$ respectively. 
		This completes the proof.

	\end{proof}

	The following lemma can be obtained by composing Pereira's post \cite{MO} in MathOverflow (originally given in \cite[Remark 2]{To}) and \cite[Proposition 0.3]{ST}.

	\begin{lemma}\label{lemma_isotopic}
		Suppose that $X$ is a smooth projective surface with $H^1(X, \mcal{O}_X) = 0$. Let $H_1$ and $H_2$ be two smooth curves of $X$ representing the same homology class.
		Then $H_1$ is symplectically isotopic to $H_2$ with respect to the symplectic form $\omega_X = \omega_{\mathrm{FS}}|_X$ on $X$.
	\end{lemma}
	
	\begin{proof}
		Using 
		\begin{itemize}
			\item \cite[Proposition 0.3]{ST}, and 
			\item the fact that every smooth algebraic curve of $X$ is a symplectic submanifold of $(X,\omega_X)$, 
		\end{itemize}
		It is enough to find a family $\{H_t\}_{1 \leq t \leq 2}$ of smooth algebraic curves in $X$ which induces a constant homology class. 
		Note that the set of effective divisors in the class $[H_1]$ contains both $H_1$ and $H_2$ since the Chern class map 
		\[
			c : H^1(X, \mcal{O}_X^*) \rightarrow H^2(X; \Z)
		\]
		is injective by our assumption that $H^1(X, \mcal{O}_X) = 0$. 
		
		Now consider the complete linear system $|H_1|$ isomorphic to some projective space $\p^N$. 
		Since the set of smooth divisors in $|H_1|$ is Zariski open in $\p^N$, it is connected. This completes the proof.
	\end{proof}

\appendix

\section{Monotone symplectic four manifolds with semifree $S^1$-actions}
\label{secMonotoneSymplecticFourManifoldsWithSemifreeS1Actions}

In this section, we classify semifree Hamiltonian $S^1$-actions on compact monotone symplectic four manifolds up to $S^1$-equivariant symplectomorphism.

Let $(M,\omega)$ be a four dimensional closed monotone semifree Hamiltonian $S^1$-manifold such that $c_1(TM) = [\omega]$ with the balanced moment map. 
Then, similar to Lemma \ref{lemma_possible_critical_values}, we have the following. 

\begin{lemma}
	All possible critical values of $H$ are $\pm 2, \pm 1$, and $0$. Moreover, any connected component $Z$ of $M^{S^1}$ satisfies one of the followings : 
		\begin{table}[h]
			\begin{tabular}{|c|c|c|c|}
			\hline
			    $H(Z)$ & $\dim Z$ & $\mathrm{ind}(Z)$ & $\mathrm{Remark}$ \\ \hline 
			    $2$ &  $0$ & $4$ & $Z = Z_{\max} = \mathrm{point}$ \\ \hline
			    $1$ &  $2$ & $2$ & $Z = Z_{\max} \cong S^2$ \\ \hline
			    $0$ &  $0$ & $2$ & $Z = \mathrm{point}$ \\ \hline
			    $-1$ &  $2$ & $0$ & $Z = Z_{\min} \cong S^2$ \\ \hline
			    $-2$ &  $0$ & $0$ & $Z = Z_{\min} = \mathrm{point}$ \\ \hline
			\end{tabular}
			\vspace{0.2cm}
			\caption{\label{table_fixed_4dim} List of possible fixed components}
		\end{table}
\end{lemma}

\begin{proof}
	The dimension and the Morse-Bott index of each fixed component can be obtained from Corollary \ref{corollary_sum_weights_moment_value}. Also, if $H(Z) = \pm 1$, then 
	$Z$ should be an extremal fixed component since the weights of the $S^1$-action at $Z$ is $(\mp 1, 0)$. Furthermore, $Z \cong S^2$ 
	by \cite{Li1} since $M$ is diffeomorphic to some del Pezzo surface (which is simply connected). 
\end{proof}

So, we may divide into three cases (where other cases are recovered by taking reversed orientation of $M$) : 

\begin{itemize}
	\item $H(Z_{\min}) = -2$ and $H(Z_{\max}) = 2$.
	\item$H(Z_{\min}) = -2$ and $H(Z_{\max}) = 1$.
	\item $H(Z_{\min}) = -1$ and $H(Z_{\max}) = 1$.
\end{itemize}

\begin{theorem}\label{theorem_classification_4dim}
	Let $(M,\omega)$ be a four dimensional closed monotone semifree Hamiltonian $S^1$-manifold such that $c_1(TM) = [\omega]$ with the balanced moment map $H$. 
	Then all possible $(M,\omega)$, together with their fixed point data, are list as follows : 
	\begin{table}[H]
		\begin{tabular}{|c|c|c|c|c|c|c|c|}
			\hline
			                &  $M$                            & $Z_{-2}$ & $Z_{-1}$ &  $Z_0$          & $Z_1$ & $Z_2$ & $e(P_{\min}^+)$\\ \hline \hline
			{\bf (I-1)} &  $\p^1 \times \p^1$ & {\em pt} &                    & {\em 2 pts} &              & {\em pt} & $-u$ \\ \hline    
			
			{\bf (II-1)} &  $\p^2$ & {\em pt} &                    & 	 & $\p^1$             &  & $-u$ \\ \hline    
			{\bf (II-2)} &  $\p^2 \# ~\overline{\p^2}$ & {\em pt} &                    & {\em pt}	 & $\p^1$             & & $-u$\\ \hline    
			{\bf (II-3)} &  $\p^2 \# ~2 ~\overline{\p^2}$ & {\em pt} &                    & {\em 2 pts}	 & $\p^1$        &     & $-u$ \\ \hline    
			
			{\bf (III-1)} &  $\p^1 \times \p^1$                      &  & $\p^1$ &                    & $\p^1$  &  & $0$\\ \hline    
			{\bf (III-2)} &  $\p^2 \# ~\overline{\p^2}$                &  & $\p^1$ &                    & $\p^1$  & & $-u$  \\ \hline    
			{\bf (III-3)} &  $\p^2 \# ~2~\overline{\p^2}$ &  & $\p^1$ &                  {\em pt} & $\p^1$ &  & $0$ \\ \hline    
			{\bf (III-4)} &  $\p^2 \# ~3~\overline{\p^2}$ &  & $\p^1$ &                  {\em 2 pts} & $\p^1$ &  & $-u$\\ \hline    			
		\end{tabular}
		\vs{0.3cm}
		\caption {List of topological fixed point data} \label{table_list_4dim} 
	\end{table}
	\noindent
	Moreover, the corresponding $S^1$-actions for each cases are described in Figure \ref{figure_4dim}.
\end{theorem}

\begin{figure}[h]
	\scalebox{0.9}{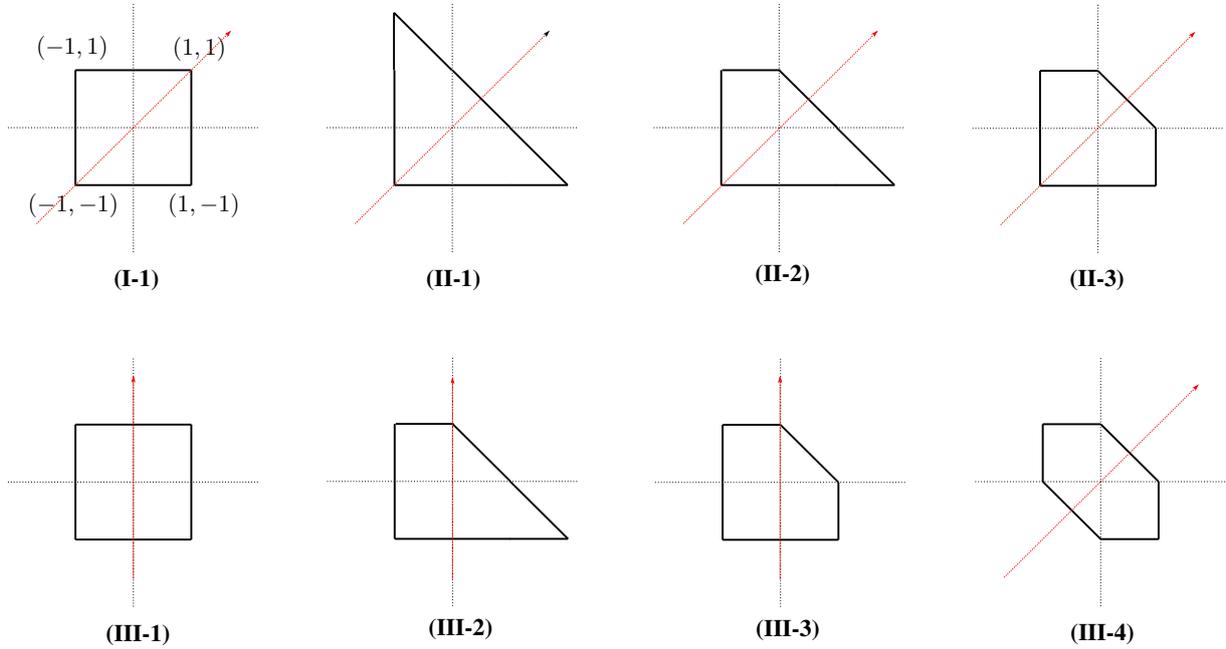}
	\vs{0.3cm}
	\caption{\label{figure_4dim} Classification in dimension four}
\end{figure}

\begin{proof}
	First note that $(M_0, [\omega_0]) = (\p^1, 2u)$ by Proposition \ref{proposition_monotonicity_preserved_under_reduction}.
	Let $k = |Z_0|$ be the number of interior fixed points. \vs{0.1cm}
	
	\noindent {\bf Case I : $H(Z_{\min}) = -2$ and $H(Z_{\max}) = 2$.} In this case, we have $Z_{\min}$ = $Z_{\max}$ = pt and 
	the Euler classes of the principal bundles 
	\[
		H^{-1}(-2 + \epsilon) \rightarrow M_{-2 + \epsilon} \quad \text{and} \quad H^{-1}(2 - \epsilon) \rightarrow M_{2 - \epsilon}
	\]
	are $-u$ and $u$, respectively. By Lemma \ref{lemma_Euler_class}, we have 
	\[
		e(P_{2}^-) = e(P_{- 2}^+) + \mathrm{PD}(Z_0) \quad \Leftrightarrow \quad u = -u + ku \quad \Leftrightarrow \quad k = 2.
	\]
	
	\noindent {\bf Case II : $H(Z_{\min}) = -2$ and $H(Z_{\max}) = 1$.} In this case, we have $Z_{\min}$ = pt and $Z_{\max}$ = $\p^1$. 
	Note that if $k \geq 3$, then 
	\[
		[\omega_1] = 2u - (k-1)u = (3 - k)u 
	\]
	so that $\int_{Z_{\max}} \omega_1 \leq 0$ which is impossible. Thus we get $k \leq 2$.
	\vs{0.2cm}
	
	\noindent {\bf Case III : $H(Z_{\min}) = -1$ and $H(Z_{\max}) = 1$.} In this case, we have $Z_{\min}$ = $Z_{\max}$ = $\p^1$. Set 
	\[
		[\omega_{-1}] = au, \quad  e(P_{-1}^+) = bu, \quad a > 0 , ~b \in \Z.
	\]
	Then, 
	\[
		[\omega_0] = 2u = (a-b)u, \quad [\omega_1] = a - b - (b+k)u = (a - 2b - k)u > 0,
	\]
	where all possible tuples $(a,b,k)$ are 
	\begin{itemize}
		\item $\{(1, -1, k) ~|~ 0 \leq k \leq 2\},$
		\item $\{(2, 0,  k) ~|~ 0 \leq k \leq 1\},$
		\item $\{(3, 1,0 )\}$. 
	\end{itemize}	
	Note that $(2, 0, 1)$ and $(1, -1, 1)$ are essentially same where one can be obtained from another by taking an opposite orientation $M$. 
	Similarly $(3,1,0)$ and $(1, -1, 0)$ induce the same fixed point data. 
	Therefore, 
	there are four possibilities 
	\[
		\underbrace{(2,0,0)}_{\text{\bf (III-1)}}, \quad \underbrace{(1,-1,0)}_{\text{\bf (III-2)}}, \quad \underbrace{(2,0,1)}_{\text{\bf (III-3)}}, \quad \underbrace{(1,-1,2)}_{\text{\bf (III-4)}}. 
	\]
	
	To show that $(M,\omega)$ is one of those in Table \ref{table_list_4dim} (or in Figure \ref{figure_4dim}), note that every reduced space (at any regular level) is diffeomorphic to $S^2$, 
	and hence it is symplectically rigid by the classical Moser Lemma. Also, the topological fixed point data determines the fixed point data uniquely since 
	\[
		(S^2, \omega_0, p_1, \cdots, p_k) \underbrace{\cong}_{\text{symplectic isotopic}} (S^2, \omega_0, q_1, \cdots, q_k) 
	\]
	for any distinct $k$ points $\{p_1, \cdots, p_k\}$ and $\{q_1, \cdots, q_k\}$ on $S^2$ by \cite[Proposition 0.3]{ST}. Furthermore, each toric Fano manifold with the specific choice of the 
	$S^1$-action described in Figure \ref{figure_4dim} has the same fixed point data as the corresponding one in Table \ref{table_list_4dim}. So, our theorem follows from
	Theorem \ref{theorem_Gonzalez}. (See also Remark \ref{remark_Gonzalez_5}.)
\end{proof}

\section{Symplectic capacities of smooth Fano 3-folds}
\label{secSymplecticCapacitiesOfSmoothFano3Folds}

In this section, we compute two kinds of symplectic capacities, namely the {\em Gromov width} and the {\em Hofer-Zehnder capacity}, of symplectic manifolds given in Table \ref{table_list}. 

Recall that the {\em Gromov width} of a closed $2n$-dimensional symplectic manifold $(M,\omega)$ is defined as 
\[
	w_G(M,\omega) := \sup ~\{ \pi a^2 ~|~ \text{$B(a)$ is symplectically embedded in $(M,\omega)$} \}
\]
where $B(a) \subset \C^n$ is a $2n$-dimensional open ball of radius $a$ with a standard symplectic structure on $\C^n$. The Hofer-Zehnder capacity of $(M,\omega)$ is defined to be 
\[
	c_{HZ}(M,\omega) := \sup ~\{ K_{\max} - K_{\min} ~|~ K : \text{admissible} \}.
\]
Here, a smooth function $K : M \rightarrow \R$ is said to be {\em admissible} if there are open subsets $U, V$ in $M$ such that 
\begin{itemize}
	\item $K|_U = K_{\max}$ and $K|_V = K_{\min}$,
	\item there is no non-constant periodic orbit of $K$ whose period is less than one.
\end{itemize}
(We refer to \cite{McS1} for more details.) In \cite{HS}, Hwang and Suh proved the following. 

\begin{theorem}\cite[Theorem 1.1]{HS}\label{theorem_HS}
	Let $(M,\omega)$ be a closed monotone symplectic manifold with a semifree Hamiltonian circle action such that $c_1(TM) = [\omega]$ and let $H$ be a moment map.
	If the minimal fixed component is an isolated point, then the Gromov width and the Hofer-Zehnder capacity of $(M,\omega)$ are respectively given by $H_{\mathrm{smin}} - H_{\min}$ 
	and $H_{\max} - H_{\min}$. Here, $M_{\mathrm{smix}}$ is the second minimal critical value of the moment map $H$.
\end{theorem}

Using Theorem \ref{theorem_HS}, we obtain the followings. 
\begin{proposition}
	For each smooth Fano 3-fold admitting semifree Hamiltonian circle action, the Gromov width and the Hofer-Zehnder capacity can be compute as follows : 
	\begin{table}[H]
		\begin{tabular}{|c|c|c|c|c|c|}
			\hline
			    & $H_{\max}$ & $H_{\mathrm{smin}}$ & $H_{\min}$ & $w_G$ & $c_{HZ}$ \\ \hline \hline
			    {\bf (I-1)} & 3 & -1 &  -3& 2 & 6 \\ \hline    
			   {\bf (I-2)} &  3 & -1 &  -3& 2 & 6 \\ \hline    
			   
			   {\bf (I-3)} & 3 & -1 &  -3& 2 & 6 \\ \hline    
			    
			   {\bf (II-3.1)} & 2 & -1 &  -3& 2 & 5 \\ \hline    
			   {\bf (II-3.2)} & 2 & -1 &  -3& 2 & 5 \\ \hline    
			   {\bf (II-3.3)} & 2 & -1 &  -3& 2 & 5 \\ \hline    
			    
			   {\bf (II-4.1)} & 2 & -1 &  -3& 2 & 5 \\ \hline    
			   {\bf (II-4.2)} & 2 & -1 &  -3& 2 & 5 \\ \hline    
		    		    
			   {\bf (III-1)} & 1 & 1 &  -3& 4 & 4 \\ \hline    
			   {\bf (III-2)} & 1 & -1 &  -3& 2 & 4 \\ \hline    
			   {\bf (III-3.1)} & 1 & 0 &  -3& 3 & 4 \\ \hline    
			   {\bf (III-3.2)} & 1 & 0 &  -3& 3 & 4 \\ \hline    
			   {\bf (III-3.3)} & 1 & 0 &  -3& 3 & 4 \\ \hline    
			   
			   {\bf (III-4.1)} &  1 & -1 &  -3& 2 & 4 \\ \hline
			   {\bf (III-4.2)} &  1 & -1 &  -3& 2 & 4 \\ \hline
			   {\bf (III-4.3)} &  1 & -1 &  -3& 2 & 4 \\ \hline
			   {\bf (III-4.4)} &  1 & -1 &  -3& 2 & 4 \\ \hline
			   {\bf (III-4.5)} &  1 & -1 &  -3& 2 & 4 \\ \hline
		\end{tabular}
		\vs{0.3cm}
		\caption {Gromov width and Hofer-Zehnder capacity of smooth Fano 3-folds} \label{table_list_capacities} 
	\end{table}
	
\end{proposition}

\bibliographystyle{annotation}

\end{document}